\documentclass[11pt]{amsart}
\usepackage[margin=0.9in]{geometry}

\def\uvuciautor{\hspace{-0.4cm}}

\usepackage[dvipsnames]{xcolor}
\usepackage[centertableaux]{ytableau}
\usepackage{graphicx}

\usepackage{wrapfig}
\usepackage{eso-pic}
\usepackage{mathtools}
\usepackage{float} 
\usepackage{amsmath, amsthm, amsfonts, amssymb, amscd, bm}
\usepackage{mathrsfs}
\usepackage{multicol}
\usepackage{xcolor}
\definecolor{OxBlue}{HTML}{002147}

\usepackage[colorlinks]{hyperref}
\hypersetup{colorlinks = true, citecolor={OxBlue}, urlcolor  ={OxBlue}, linkcolor={OxBlue} }

\usepackage{perpage}
\MakePerPage{footnote}

\usepackage[normalem]{ulem}

\usepackage{enumerate}
\usepackage[capitalise, noabbrev]{cleveref}

\usepackage{tikz-cd} 
\usepackage{subfigure}

\usepackage[all,cmtip]{xy}
\usepackage[utf8]{inputenc}
\usepackage[T1]{fontenc}

\newtheorem{thm}{Theorem}[section]
\newtheorem{cor}[thm]{Corollary}

\newtheorem{prop}[thm]{Proposition}
\newtheorem{lm}[thm]{Lemma}

\theoremstyle{definition}
\newtheorem{de}[thm]{Definition}
\newtheorem{ex}[thm]{Example}

\theoremstyle{remark}
\newtheorem{rmk}[thm]{Remark}

\newcommand{\upsidedownY}{\rotatebox[origin=c]{180}{Y}}

\newcommand*{\myplus}{\ensuremath{\boldsymbol{\pmb{+}}}}

\def\im{\text{im}}
\def\det{\text{det}}

\def\char{\text{char}}

\def\z{\zeta}

\def\Q{\mathbb Q}
\def\R{\mathbb R}
\def\C{\mathbb C}
\def\N{\mathbb N}

\def\Z{\mathbb Z}

\def\Kh{K\"{a}hler}
\def\P{{\mathbb P}}

\def\a{\alpha}
\def\b{\beta}
\def\c{\beta}

\def\Fi{\varphi}

\def\M{\mathfrak{M}}

\def\CM0{\C[\M_0]}

\def\F{\mathfrak{F}}

\def\Con1{Con_1(\M)}

\def\FF{\mathscr{F}}
\def\PP{\mathscr{P}}
\def\EE{\mathscr{E}}

\def\MM0{\mathfrak{M}_{\tiny{(\zeta_{\mathbb{R}},0})}(Q,{\normalfont\textbf{v}},{\normalfont\textbf{w}})}

\def\MG0{\mathcal{M}_{0,\z_\C}(Q,{\normalfont\textbf{v}},{\normalfont\textbf{w}})}

\def\k{\mathbb{K}}

\def\ph{pseudoholomorphic}

\def\Fil{\mathscr{F}}

\def\ku{\k [\![u]\!]}
\def\kuu{\k(\!(u)\!)}
\def\karo{\diamondsuit}

\def\MM{\mathcal{M}}

\def\CP{\mathbb{C}P}

\def\MB{Morse--Bott}

\def\MBF{Morse--Bott--Floer}

\def\ord{\mathrm{ord}}

\makeatletter
\newcommand{\doublewidetilde}[1]{{%
  \mathpalette\double@widetilde{#1}%
}}
\newcommand{\double@widetilde}[2]{%
  \sbox\z@{$\m@th#1\widetilde{#2}$}%
  \ht\z@=.9\ht\z@
  \widetilde{\box\z@}%
}
\makeatother
\renewcommand{\arraystretch}{0.85}

\begin{document}

\title
[Filtrations on equivariant quantum cohomology]
{Filtrations on equivariant quantum 
\\ 
cohomology and Hilbert-Poincar\'{e} series}
\author{Alexander F. Ritter}
\address{\uvuciautor A. F. Ritter, 
Mathematical Institute, University of Oxford, 
OX2 6GG, U.K.}
\email{ritter@maths.ox.ac.uk}
\author{Filip Živanović}
\address{\uvuciautor F. T. Živanović, 
Simons Center for Geometry and Physics, 
Stony Brook, NY 11794-3636, U.S.A.}
\email{fzivanovic@scgp.stonybrook.edu}

\begin{abstract} 
We prove that Floer theory induces a filtration by ideals on equivariant quantum cohomology of symplectic manifolds equipped with a $\C^*$-action.
In particular, this gives rise to Hilbert-Poincar\'{e} polynomials on ordinary cohomology that depend on Floer theory.
En route, the paper develops structural properties of filtrations on three versions of equivariant Floer cohomology.
We obtain an explicit presentation for equivariant symplectic cohomology in the Calabi-Yau and Fano settings.
\end{abstract}

\maketitle
\setcounter{secnumdepth}{3}
\setcounter{tocdepth}{1}

\tableofcontents  %

\section{Introduction}\label{Introduction}
\subsection{Motivation}
The goal of this paper is to study structural properties of equivariant quantum cohomology for non-compact symplectic manifolds $(Y,\omega)$ admitting Hamiltonian $S^1$-actions $\Fi$.
This will involve foundational work on equivariant Hamiltonian Floer cohomology.
In order to make sense of Floer theory, we make some further assumptions. We assume that $\Fi$ extends to a {\ph} $\C^*$-action for some $\omega$-compatible almost complex structure $I$, because the majority of interesting examples arise in this way, from Algebraic Geometry and Geometric Representation Theory.\footnote{Nakajima quiver varieties;
moduli spaces of Higgs bundles;
{\Kh} quotients of $\C^n$ including semiprojective 
toric manifolds;
hypertoric varieties;
cotangent bundles of flag varieties;
negative complex vector bundles;
crepant resolutions of quotient singularities $\C^n/G$ for finite subgroups $G\subset SL(n,\C)$; all Conical Symplectic Resolutions; and most
equivariant resolutions of affine singularities.}
Unfortunately, the past Symplectic Topology literature only dealt with non-compact symplectic manifolds that are ``convex at infinity'', so symplectomorphic to $(\Sigma \times [R_0,\infty), d(R\alpha))$ at infinity, where $(\Sigma,\alpha)$ is a closed contact manifold.
This almost never applies in the above examples.
Instead, the spaces belong to a class we introduced in \cite{RZ1}, called symplectic $\C^*$-manifolds, whose defining property is the existence of 
a proper {\ph}\footnote{On $\Sigma \times [R_0,\infty)$, a choice of $d(R\alpha)$-compatible almost complex structure of contact type is made.} map
$\Psi: Y^{\mathrm{out}}\to \Sigma \times [R_0,\infty)$
defined outside of the interior of a compact set, such that $\Psi$ sends the Hamiltonian $S^1$-vector field to the Reeb vector field.

For example, Conical Symplectic Resolutions (in particular all quiver varieties) arise in this way and are much studied in Algebraic Geometry. In that case, $c_1(Y)=0$, and quantum cohomology equals ordinary cohomology because $Y$ can be deformed to an affine variety.
An interesting feature however, is that they are usually not $\C^*$-equivariantly deformable to an affine variety, and often their equivariant quantum cohomology is not classical. From this point of view, the equivariant theory for symplectic $\C^*$-manifolds is of greater interest than the non-equivariant theory, as it detects non-trivial Gromov-Witten invariants. This was the main motivation for our work. We also highlight recent work by Jae Hee Lee \cite{JaeHeeLee2024} that investigates the mod $p$ equivariant quantum cohomology of these spaces. 

For any symplectic $\C^*$-manifold, we constructed Hamiltonian Floer cohomology $HF^*(H_{\lambda})$ and symplectic cohomology $SH^*(Y)=\varinjlim HF^*(H_{\lambda})$ in \cite{RZ1}. Here $\lambda$ is a slope parameter that grows to infinity in the limit. Recall that the generators of $HF^*(H_{\lambda})$ are Hamiltonian $1$-orbits of the Hamiltonian $H_{\lambda}:Y \to \R$; these correspond to $S^1$-orbits of $\Fi$ of various periods. We choose $H_{\lambda}=c(H)$ to be a suitable function of the moment map $H$, of slope $c'(H)=\lambda$ for large $H$, so that the $1$-orbits are of two types: constant $1$-orbits at points of the components $\F_{\a}$ of the fixed locus $\F = Y^{S^1}=Y^{\C^*}$, and {\MB} manifolds $B_{p,\beta}$ of non-constant $1$-orbits corresponding to $S^1$-orbits of various periods $p=c'(H)$. In \cite{RZ2} we construct {\MBF} spectral sequences that effectively compute $HF^*(H_{\lambda})$ and $SH^*(Y)$ in terms of the cohomologies of the {\MB} manifold $\F_\a$ and $B_{p,\beta}$.

By \cite{RZ1}, quantum cohomology $QH^*(Y)$ is filtered by graded ideals $\FF^{p}$ ordered by  $p\in \R\cup\{\infty\}$,
\begin{equation}\label{Equation introduction filtration}
\FF^{p} :=\bigcap_{\mathrm{generic}\,\lambda\geq p} \left(\ker c_{\lambda}^*:QH^*(Y)\to HF^*(H_{\lambda})\right), \qquad \FF^{\infty}:=QH^*(Y),
\end{equation} 
where $c_\lambda^*$ is a Floer continuation map, a grading-preserving $QH^*(Y)$-module homomorphism.
That filtration is an invariant of  $Y$ up to isomorphism of symplectic $\C^*$-manifolds ({\ph } $\C^*$-equivariant symplectomorphisms, without conditions on the $\Psi$ map).
The real parameter $p$ is a part of that invariant, and has a geometric interpretation: $x\in \Fil^p$ means that $x$ can be represented as a Floer chain involving non-constant $S^1$-orbits of period $\leq p.$

\subsection{Overview of $S^1$-equivariant chain models}
$S^1$-equivariant Floer theory involves moduli spaces of maps into the Borel model $S^{\infty} \times_{S^1} Y$ that can contribute arbitrarily high powers of the equivariant parameter $u\in H^2(\C P^{\infty})$ in the construction of the Floer chain differential,
$$ 
d = \delta_0 + u \delta_1 + u^2 \delta_2 + \cdots,
$$
where $\delta_i$ does not involve $u$, and $\delta_0$ is the non-equivariant Floer differential.
For this reason, there are three natural equivariant complexes one can build\footnote{$[\![u]\!]$ means we take a completed tensor product with the $\k$-algebra $\ku$ of power series in $u$; whereas $(\!(u)\!)$ means we $u$-localise the $[\![u]\!]$-version, so formal Laurent series in $u$. The symbol $(\cdots)_u$ always denotes $u$-localisation.} from a non-equivariant Floer chain complex $C^*$:
$$
E^-C^*:=C^*[\![u]\!],
\quad\;
E^{\infty}C^* := C^*(\!(u)\!) = (E^-C^*)_u,
\quad\;
E^+C^* := (E^-C^*)_u/uE^-C^*
\cong E^-C^*\otimes_{\ku} \mathbb{F}, 
$$
where $\mathbb{F}:=\kuu/u\ku$. Here $\k$ is a suitable Novikov field of coefficients, which we assume is of characteristic zero. 
The above are all $\ku$-modules; and $E^{\infty}C^*$ is also a $\kuu$-module.
On cohomology, $(E^-H^*)_u\cong E^{\infty}H^*$ are canonically isomorphic for algebraic reasons (flatness of $u$-localisation).

For small slopes, $\lambda=0^+$, using $C^* = CF^*(H_{0^+})$, yields Morse-Bott models for ``classical'' equivariant cohomology $E^{\karo}H^*(Y)$ for each of the three models $\karo\in \{-,\infty,+\}$, indeed only finitely many $\delta_j$ are non-zero and $E^-HF^*(H_{0^+})\cong H^*_{S^1}(Y)\otimes_{\k[u]}\ku$ is the Borel construction of $S^1$-equivariant cohomology; the other two models are classically motivated by cyclic homology.

Increasing the slope $\lambda$ gives rise to continuation maps: $\ku$-module homomorphisms 
$$
E^{\karo}c^*_{\lambda}: E^{\karo}QH^*(Y)\to E^{\karo}HF^*(H_{\lambda}),
$$
defined for generic $\lambda>0$,
abbreviated $c^*_{\lambda}$ when there is no ambiguity. $E^{\karo}c^*_{\lambda'}$ factors through $E^{\karo}c^*_{\lambda}$ whenever $\lambda<\lambda'$, and the direct limit as $\lambda \to \infty$ defines 
\begin{equation}\label{Equation cstar map}
E^{\karo}c^*: E^{\karo}QH^*(Y)\to E^{\karo}SH^*(Y),
\end{equation}
often abbreviated $c^*$.
For $\karo\in \{-,\infty\}$, the equivariant quantum product on equivariant cohomology can be defined by
parametrising the quantum product via the Borel model (e.g.\;see \cite{liebenschutz2020intertwining,liebenschutz2021shift}). For $\karo = +$, there is no product as there is no multiplication map $\mathbb{F}\times \mathbb{F} \to \mathbb{F}$. 
In general, there is no product on $E^{\karo}SH^*(Y)$, but we discuss an exception to this in \cref{Subsection weight 10 case product}.
We abbreviate $E^{\karo}QH^*(Y):=E^{\karo}H^*(Y;\k)$, with the convention that for $\karo\in \{-,\infty\}$ we use the equivariant quantum product.
\begin{ex}\label{Example S1 equiv cohomologies calculation}
For the space $M=S^1$, with the standard $S^1$-rotation action, the classical Morse models $E^-,E^{\infty},E^+$ for equivariant cohomology are: 
\begin{center}
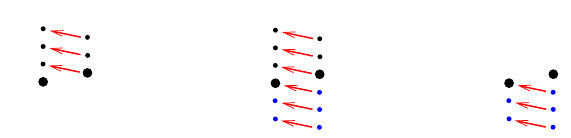
\vspace{1mm}
\end{center}
where the Morse function has two critical points: a minimum $x$ and a maximum $y$, so grading $|x|=0$, $|y|=1$, $|u|=2$.
Thus $E^-SH^*(M)=\k\cdot x$ whereas $E^{\infty}SH^*(M)=0$, and $E^+SH^*(M)=\k\cdot y$ (all three are torsion $\ku$-modules with trivial $u$-action).
This is consistent with Borel's localisation theorem: 
\begin{equation}\label{Equation Borel localisation}
H^*_{S^1}(M)_u \stackrel{\cong}{\longrightarrow} (H^*_{S^1}(\mathrm{Fixed\; locus}))_u= H^*(\mathrm{Fixed\; locus})(\!(u)\!).
\end{equation}
The same holds for the weight $k\notin \Z\setminus\{0\}$ circle action: the red arrows are multiplication by $k$, so isomorphisms, as $\mathrm{char}\,\k=0$. This is the lowest-order approximation of local Floer cohomology near an isolated $S^1$-family of Hamiltonian $1$-orbits. If $\mathrm{char}\,\k\neq 0$, torsion arises and Borel's theorem fails.
Even the simplest examples display interesting but  complicated torsion phenomena, see \cite{liebenschutz2020intertwining}. We therefore chose to assume $\char\, \k=0$ in order to obtain clean general statements and computations.
\end{ex}

Equivariant Floer theory actually involves two $S^1$-actions acting on $1$-periodic Hamiltonian orbits $x=x(t):S^1 \to Y$. The {\bf $\Fi$-action} on $Y$, and the {\bf loop-action} that reparametrises the domain of $x$:
$$
(w\cdot x)(t) = \Fi_w(x(t)), \quad \textrm{ and } \quad (w\cdot x)(t) = x(t+\theta),
$$
where $w=e^{2\pi i \theta} \in S^1\subset \C^*.$ Abbreviate $\Fi_{e^{2\pi i\theta}}$ by $\Fi_{\theta}$, for $\theta \in S^1:=[0,1]/(0\sim 1)$, and parametrise loops by $t\in [0,1]/(0\sim 1)$. A {\bf weight} $(a,b)\in \Z^2$ defines an $S^1$-action on free loops $x=x(t)\in \mathcal{L}Y$:
\begin{equation}\label{Equation Introduction action on loops}
 \theta \cdot x : t\mapsto \Fi_{a\theta}(x(t-b\theta)).
\end{equation}
\begin{rmk}\label{Remark characteristic and mixed action}
We chose to consider just one mixed $S^1$-action 
at a time, as in \eqref{Equation Introduction action on loops}, yielding one equivariant parameter $u$, rather than the full $S^1\times S^1$-action which would produce two equivariant parameters.
In \cref{Subsection Comparing different weights via equivariant Seidel isomorphisms} we comment on what happens when we vary $(a,b).$
\end{rmk}

The $\ku$-module $E^-QH^*(Y)$ (arising, as explained, from small slopes $\lambda>0$), only depends on the $\Fi^a$-action on $Y$,  not\footnote{the loop-action is trivial on constant Hamiltonian $1$-orbits, which are the generators of the Morse-Bott complex.} on $b$. Let $QH^*_{\Fi^a}(Y)$ be equivariant quantum cohomology for the $S^1$-action $\Fi^a$, working over $\k[u]$-coefficients. For grading reasons, we have a canonical $\ku$-algebra isomorphism
$$
E^-QH^*(Y) \cong QH^*_{\Fi^a}(Y)\otimes_{\k[u]}\ku.
$$
If $x$ is a Hamiltonian $1$-orbit of $H_{\lambda}$, then $x$ is a fixed point of \eqref{Equation Introduction action on loops} precisely when $x$ lies in a {\MB} manifold $B_{p,\beta}$ with $p=\tfrac{a}{b}$: the $\Fi$-rotation and loop-action perfectly cancel out.

So we call $(a,b)$ a {\bf free weight} if $\frac{a}{b}$ is not equal to an $S^1$-period $p$ for which a {\MB} manifold $B_{p,\beta}$ exists. We call those ``bad'' $p$ the 
{\bf non-generic slopes}, and by \cite{RZ1} they are precisely: 
\begin{equation}\label{Equation bad slope}
\{\tfrac{k}{m}: m \in W^+, k\geq 1 \in \Z \},
\end{equation}
where $W^+:=\{$strictly positive weights $m$ of the fixed components $\F_\a\}$. The discrete subset $W^+\subset [0,\infty)$ can be easily computed \cite{RZ1}.
So: $(a,b)$ free $\Leftrightarrow aW^+ \cap b \Z_{\geq 1}=\emptyset$.

For example, $(a,b)\neq (0,0)$ is free if $a\cdot b\leq 0$. Also $(1,0)$ is free, and $(1,0)$ is the only weight for which there is a pair-of-pants product on $E^{\karo}SH^*(Y)$ for $\karo\in \{-,\infty\}$, so \eqref{Equation cstar map} is a $\ku$-algebra homomorphism. That product arises as a direct limit of products $E^{\karo}HF^*(H_{\lambda_1})\otimes E^{\karo}HF^*(H_{\lambda_2}) \to E^{\karo}HF^*(H_{\lambda_1+\lambda_2})$.

\begin{rmk}[Zhao's work]
The case $(a,b)=(0,1)$ is also free.
This case has been studied extensively in the literature because it only uses the loop-action; no $\Fi$-action is needed.
In particular, Zhao \cite{zhang2017multiplicativity} 
investigated this case for Liouville manifolds $M$.
We remark how Zhao's terminology relates to ours: ``localised periodic'' symplectic cohomology $HP^*_{S^1,\mathrm{loc}}(M)=E^{\infty}SH^*(M)$, and ``periodic'' symplectic cohomology $HP^*_{S^1}(M)=E^{\infty}_{\textrm{telescope}}SH^*(M)$ involves a telescope model, see \cref{Subsection A remark about the telescope versus the exhaustive model}.

Zhao proved a localisation theorem in that setting: using $\char\,\k=0$,
the canonical homomorphism $E^{\infty}c^*:H^*(M)(\!(u)\!) \cong E^{\infty}SH^*(M)$ is an isomorphism. 
One of our main theorems, \cref{Theorem intro localisation theorem}, is a general localisation theorem for all symplectic $\C^*$-manifolds, for any free weight. We use a different proof than Zhao's explicit one: we deduce it from Borel's theorem \eqref{Equation Borel localisation} applied to the equivariant analogue of {\MBF} spectral sequences, from \cite{RZ2}. 
\end{rmk}

\subsection{Equivariant formality and computation of equivariant symplectic cohomology}
By adapting a classical result of Kirwan \cite[Sec.5.8]{Ki84}, we prove: 

\begin{lm}[Equivariant formality]\label{Lemma equivariant formality}
There are non-canonical $\ku$-module isomorphisms
    $$
    E^-QH^*(Y)\cong H^*(Y)[\![u]\!]
    ,\quad
    E^{\infty}QH^*(Y)\cong  H^*(Y)(\!(u)\!)
    ,
    \quad
    E^{+}QH^*(Y)\cong H^*(Y)\otimes_{\k} \mathbb{F}.
    $$
\end{lm}

Just as in the non-equivariant case, the $c^*$ map \eqref{Equation cstar map} can be unpredictable. In the ``Calabi-Yau'' setting $c_1(Y)=0$, we showed $SH^*(Y)=0$ in \cite{RZ1}, which motivated our interest in the filtration \eqref{Equation introduction filtration}.
In the ``monotone'' or ``Fano'' setting $c_1(Y)\in \R_{>0}[\omega]$, the non-equivariant $c^*$ map is a quotient map onto
$$
SH^*(Y)\cong QH^*(Y)/E_0, \textrm{ where }E_0 =\ker Q_{\Fi}^y,
$$
where $y:=\dim_{\k}\,H^*(Y)$, and $Q_{\Fi}\in QH^*(Y)$ is a ``rotation class'' \cite{RZ1} acting by quantum multiplication on $QH^*(Y)$ (a generalisation of the Seidel element for $\Fi$ \cite{Sei97} to the non-compact setting \cite{R14,R16}). 
One can also think of $E_0$ in terms of ``positive symplectic cohomology'', arising from the quotient complexes in which only non-constant Hamiltonian $1$-orbits are considered,
$$E_0\cong SH^{*}_{\myplus}(Y)[-1].$$
The notation $[-1]$ means we shift the grading up by one.
We now summarise the main relations that we prove between the equivariant and non-equivariant theories for any symplectic $\C^*$-manifold $Y$.

\begin{lm}
There is a short exact sequence, commuting with the maps from \eqref{Equation cstar map}, 
\begin{equation}\label{Equation short exact sequence for ESH}
0\longrightarrow  E^-SH^*(Y) \stackrel{u}{\longrightarrow} E^{\infty}SH^*(Y) \longrightarrow E^+SH^*(Y)   \longrightarrow 0.
\end{equation}
There are long exact sequences
\begin{align}
  \cdots \longrightarrow SH^*(Y)[-1] \longrightarrow E^-SH^*(Y)[-2] \stackrel{u}{\longrightarrow} E^-SH^*(Y)
  \;\;-\!\!\!\!\!\!\!\!\stackrel{\textrm{set }u=0\;\;\;}{\longrightarrow}  
  \!\!\!
  SH^*(Y)\longrightarrow \cdots
 \\
   \strut\;\;\;\cdots \longrightarrow SH^*(Y)[-2] \;\;-\!\!\!\!\!\!\!\stackrel{u^0\textrm{-incl}\;\;\;} \longrightarrow \!\!\! E^+SH^*(Y)[-2] \stackrel{u}{\longrightarrow} E^+SH^*(Y) \longrightarrow SH^*(Y)[-1]\longrightarrow \cdots\!\!\!\!
\end{align}
Vanishing criterion: $SH^*(Y)=0\Leftrightarrow E^+SH^*(Y)=0$, in which case $E^-SH^*(Y)\cong E^{\infty} SH^*(Y).$
\end{lm}

\begin{thm}\label{Theorem intro ESH computation main}
Let $(a,b)$ be any free weight.
Then $E^{\infty}SH^*(Y) \cong E^{\infty}QH^*(Y).$

If $c_1(Y)\in \R_{\geq 0}[\omega]$, there is a commutative diagram of $\ku$-modules
$$\!
\xymatrix{
E^-QH^*(Y) 
 \ar@{->}[r]_-{\textrm{injective}}^{E^-c^*}
  \ar@{->}[d]_-{\cong} &
E^-SH^*(Y) 
 \ar@{->}[r]_-{\textrm{injective}}^{u}
  \ar@{->}[d]_-{\cong} &
E^{\infty}SH^*(Y)
\ar@{->}^-{\cong}[d] 
\\
\mathbf{\!SH^*(Y)[\![u]\!]} \!\oplus\! E_0[\![u]\!]\!
\ar@{->}[r]^-{\subset} & 
\mathbf{\!SH^*(Y)[\![u]\!]} \!\oplus\! E_0(\!(u)\!)\!
\ar@{->}[r]^-{u} & \mathbf{\!SH^*(Y)(\!(u)\!)} \!\oplus\! E_0(\!(u)\!) \!
}
$$
and the cokernel of the horizontal maps on the right is: $$E^+SH^*(Y)\cong \mathbf{SH^*(Y)}\otimes_{\k} \mathbb{F}.$$

The vertical arrows are non-canonical isomorphisms, however the copy of $E_0(\!(u)\!)$ inside $E^-SH^*(Y)$ is canonical: it is the largest $\ku$-submodule on which multiplication by $u$ acts as an invertible map.

If $c_1(Y)=0$, the above simplifies to:
$
E^-SH^*(Y) \cong E^{\infty}SH^*(Y)
$
and
$E^+SH^*(Y)=0.
$
\end{thm}

\begin{rmk}[Further explanations]
Using \cref{Lemma equivariant formality}, there is an injection 
$QH^*(Y)\hookrightarrow E^-QH^*(Y)$, whose image freely generates $E^-QH^*(Y)$ as a $\ku$-module. 
The inverse of the first vertical arrow in the theorem arises by identifying $SH^*(Y)\cong \mathrm{image}(Q_{\Fi}^y)\subset QH^*(Y)$ and $E_0= \ker Q_{\Fi}^y\subset QH^*(Y)$, and then uses that injection combined with the free $\ku$-action.

The map $E^-c^*$ equals the natural $u$-localisation map but only $u$-localising the submodule $E_0[\![u]\!]$.
The second horizontal map is $u$ times the canonical $u$-localisation map, $u$-localising everything.

The middle vertical arrow can be thought of as a short exact sequence of $\ku$-modules,
$$
0 \to E_0(\!(u)\!) \to E^-SH^*(Y) \to SH^*(Y)[\![u]\!] \to 0.
$$
By splitting, so by picking a lift of a basis of the free $\ku$-module $SH^*(Y)[\![u]\!]$, we see there are many copies of $SH^*(Y)[\![u]\!]$ inside $E^-SH^*(Y)$, even though in the proof we construct a specific copy. By contrast, the copy of $E_0(\!(u)\!)$ is canonical, and we build generators by adding very specific higher order $u$-corrections to a $\k$-linear basis of $E_0\subset QH^*(Y)\hookrightarrow E^-QH^*(Y)$.
\end{rmk}

\begin{ex}\label{Example Cn example intro} 
For $Y=\C^n$, we have $QH^*(\C^n)\cong \k$, $SH^*(\C^n)=0$. 
In this case we could simply use $\k=\Q(\!(T)\!)$, with grading $|T|=0$. 
For the standard $S^1$-action, and free $(a,b)$ (equivalently, $a\notin b \Z_{\geq 1}$):
$$
QH^*_{S^1}(\C^n)\cong \k[u], 
\;\
E^-QH^*(\C^n) \cong \ku,
\;\
E^-SH^*(\C^n)\cong E^{\infty}SH^*(\C^n)\cong \kuu,
\;\
E^+SH^*(\C^n)=0.
$$  
\end{ex}

\begin{ex} 
The total space $Y$ of the negative line bundle $\pi:\mathcal{O}(-k) \to \C\P^m$ for $1\leq k \leq m$ is monotone, $c_1(Y)=(1+m-k)x\neq 0$ where $x:=\pi^*[\omega_{\C\P^m}]$. To simplify notation, in this case we can use $\k=\Q(\!(T)\!),$ with grading 
$|T|=2$.
Abbreviate
$y:=x^{1+m-k}-(-k)^kT^{1+m-k}$.
By \cite[Cor.4.14]{R16}, 
$$QH^*(Y)=\k[x]/(x^k y)\cong \k^{1+m},\quad Q_{\varphi}=-kx,\quad E_0=\langle y \rangle \cong \k^{k} ,\quad SH^*(Y)\cong \k[x]/(y)\cong \k^{1+m-k},$$
where $\Fi$ is the standard $\C^*$-action on fibres.
Therefore, for free weights $(a,b)$ (equivalently $a\notin b \Z_{\geq 1}$),
\begin{equation*}
\begin{split}
QH^*_{S^1}(Y)\cong \k[u]^{1+m}, 
\;
E^-QH^*(\C^n) \cong \ku^{1+m},
\;
E^{\infty}QH^*(\C^n) \cong \kuu^{1+m},
\;
E^+QH^*(\C^n) \cong \mathbb{F}^{1+m},
\\
E^-SH^*(Y)\cong \ku^{1+m-k}\oplus \kuu^{k},
\qquad
E^{\infty}SH^*(Y)\cong \kuu^{1+m},
\qquad
E^+SH^*(Y)=\mathbb{F}^{1+m-k}.
\quad
\end{split}
\end{equation*}  
\end{ex}

Abbreviate: $\qquad$
$
s=\dim_{\k} SH^*(Y),  \qquad e=\dim_{\k}E_0, \qquad y=\dim_{\k} H^*(Y) = s+e.
$
\begin{thm}\label{Cor SH calculation for monotone in coords}
Assume $c_1(Y)\!\in\! \R_{\geq 0}[\omega]$.
For any weight $(a,b)\neq (0,0)$, whether free or not, we have:
$$
E^-SH^*(Y)\cong \ku^s \oplus \kuu^e,
\qquad
E^{\infty}SH^*(Y) \cong E^{\infty}QH^*(Y)\cong \kuu^{y},
\qquad
E^+SH^*(Y) \cong \mathbb{F}^{s},
$$
which simplifies when $c_1(Y)=0$: 
$E^-SH^*(Y)\cong \kuu^y \cong E^{\infty}SH^*(Y)$ and $
E^+SH^*(Y)=0$.
\end{thm}

In the free weight case, this follows from \cref{Theorem intro ESH computation main}, moreover $c^*:E^-QH^*(Y)\to E^-SH^*(Y)$ corresponds to the inclusion $\ku^s \oplus \ku^e \hookrightarrow \ku^s \oplus \kuu^e$. The non-free weight case satisfies the same isomorphisms, but the latter $c^*$ map can have kernel as $\ku^e \to \kuu^e$ may no longer be an inclusion. To avoid complicating the introduction, we do not explain the general result further (see \cref{Theorem u localisation for monotone any weight case} and \cref{Subsection Growth rate of the filtration polynomial nonfree weight case}), and instead we describe the simplest example.

\begin{ex}[Non-free weights]\label{Example nonfree weight for Cn}
In \cref{Example Cn example intro}, consider the non-free weight $(a,b)=(1,1)$. Each point of $S^{2n-1}\subset \C^n$ is the initial point of a great circle, which are orbits of the $S^1$-action. This corresponds to a {\MB} manifold $B_1\cong S^{2n-1}$ of $1$-orbits for slope value $1$, which are all fixed points of \eqref{Equation Introduction action on loops}. The maximum of an auxiliary Morse function on $B_1$ will kill the unit $1\in H^0(Y)$ in the equivariant spectral sequence for $E^-SH^*(Y)$: indeed, that edge differential exists already in the non-equivariant spectral sequence for $SH^*(Y)$. So $E^-c^*_{1^+}:E^-QH^*(Y)\to E^-HF^*(H_{1^+})$ and $E^{\infty}c^*_{1^+}:E^{\infty}QH^*(Y)\to E^{\infty}HF^*(H_{1^+})$ both vanish (in stark contrast to the free weight case, where such maps are injective).
However, the minimum of the auxiliary Morse function on $B_1$ also creates a new $\ku$-module generator in degree $-2n$, indeed 
$E^-HF^*(H_{1^+})\cong E^-QH^*(Y)[2n]$.
The continuation maps $E^{\infty}HF^*(H_{1^+})\to E^{\infty}HF^*(H_{\lambda})$ all have to be injective for generic slopes $\lambda > 1$ (for general reasons, by \cref{Theorem injectivity theorem 2}). So $E^{\infty}SH^*(Y) \cong E^{\infty}QH^*(Y)[2n] \cong \kuu$ even though the map $E^{\infty}c^*$ in \eqref{Equation cstar map} is zero! From there it is not too difficult to see that $E^-SH^*(Y) \cong E^{\infty}SH^*(Y)$ and $E^+SH^*(Y)=0.$
\end{ex}

\subsection{Localisation theorem, injectivity and torsion-freeness}\label{Subsection intro Localisation}
Beyond those descriptions of equivariant cohomology, we are interested in how equivariant Floer theory induces filtrations on ordinary and quantum cohomology.
Recall we considered \eqref{Equation introduction filtration} due to the phenomenon that kernels of  continuation maps grow in an interesting way.
The same phenomenon occurs in the $E^+SH^*(Y)$ theory, by considering the kernels $E_{1}^{\lambda}:=\ker (E^+c_{\lambda}^*:\; E^+QH^*(Y)\to E^+HF^*(H_{\lambda}))$. Indeed there is more structure:
for $j\in \Z\cup \{\pm\infty\}$, $p\in [0,\infty]$, we define $\EE_{j}^p :=\bigcap_{\mathrm{generic}\,\lambda\geq p} E_{j}^{\lambda}$ where
\begin{equation}\label{Equation E+ filtration}
E_{j\leq 0}^{\lambda}:=\ker (u^{|j|+1} \cdot E^+c_{\lambda}^*), \qquad 
E_{j> 0}^{\lambda}:=u^{j-1}\ker E^+c_{\lambda}^*, \qquad 
E_{-\infty}^{\lambda}:=E^+QH^*(Y), \;
E_{\infty}^{\lambda}:=\{0\}.
\end{equation}

The $\ku$-submodules $\EE_{j}^p\subset E^+QH^*(Y)$ satisfy
$$\EE_{j}^p\subset \EE_{j}^{p'} \;\textrm{ for }\;p<p', \qquad u\EE_{j}^p\subset \EE_{j+1}^p \subset \EE_{j}^p,\qquad u\EE_{j}^{p} =  \EE_{j+1}^{p}.$$
(The second equation is a general property of all filtrations that we will define, whereas the third equation is a special strengthening of the first inclusion of the second equation).

\begin{rmk}[Persistence module and barcode] Just as in the non-equivariant discussion in \cite[Sec.1.6]{RZ1}, we can build a persistence module $(V_p)_{p\in \R\cup \{\infty\}}$ with $V_{\lambda}:=E^+HF^*(H_{\lambda})$ for generic $\lambda$, with $V_{\infty}=E^+SH^*(Y)$, with $\EE_{1}^p=\ker E^+c^*_{p^+}$ playing the role that $\FF_p^{\Fi}$ was playing in \cite[Sec.1.6]{RZ1}. In particular, this assigns a barcode to the $\C^*$-action $\Fi$, and one can use the bottleneck distance to compare barcodes of different $\C^*$-actions.
For weights $(a,0)$, the persistence module has the same periodicity property mentioned in \cite[Sec.1.6]{RZ1}, in view of \cref{Cor sequence of r maps equiv}.
\end{rmk}

We saw above that often $E^+SH^*(Y)=0$, and then the above $\EE$-kernels encode very interesting and complicated structural properties of moduli spaces of families of Floer solutions.
What caught us by surprise, is that the precise opposite phenomenon occurs for the kernels when $\karo\in \{-,\infty\}$:

\begin{prop}\label{Prop injectivity theorem intro}
    If the $u$-localisation of $c^*\!:E^-QH^*(Y) \to E^-SH^*(Y)$ is an isomorphism, then the continuation maps $E^-QH^*(Y) \to E^-HF^*(H_{\lambda})$ and $E^-QH^*(Y) \to E^-SH^*(Y)$ are injective.
\end{prop}

That this occurs rather often, is due to the following, which is at the heart of our paper:

\begin{thm}[Localisation theorem]\label{Theorem intro localisation theorem}
Let $(a,b)$ be a free weight. The following diagram of isomorphisms commutes, where the top row is the $u$-localisation of $E^-c^*$.
    $$
\begin{tikzcd}[column sep=0.6in]
E^-QH^*(Y)_u 
 \arrow[r, "(E^-c^*)_u","\cong"'] 
 \arrow[d, "\cong","u\otimes 1"']
 &
 E^{-} SH^*(Y)_u  
\arrow[d, "\cong"',"u\otimes 1"]
\\
E^{\infty}QH^*(Y)
\arrow[r, "E^{\infty}c^*","\cong"'] 
& E^{\infty} SH^*(Y)
\end{tikzcd}
$$
\end{thm}

\subsection{Young diagrams and Hilbert-Poincar\'{e} series}

\cref{Prop injectivity theorem intro} means that the equivariant analogue of \eqref{Equation introduction filtration} for $\karo\in \{-,\infty\}$ would be trivial. So we seek instead a filtration arising from injective homomorphisms. 
The choice in \cref{Remark characteristic and mixed action} ensures the basic but consequential fact that $\ku$ is a discrete valuation ring. Therefore, any $\ku$-module homomorphism $Q:V \to W$ of free $\ku$-modules of finite rank admits a Smith normal form\footnote{$\mathrm{diag}$ refers to the main diagonal of a possibly non-square matrix. So $a= \min(\mathrm{rank}_{\ku}V,\mathrm{rank}_{\ku}W)$.} $\mathrm{diag}(u^{j_1},\ldots,u^{j_a})$ where $0\leq j_1 \leq \cdots \leq j_a\leq \infty$, abusively allowing the formal symbol $u^{\infty}:=0$ (this does not arise when $Q$ is injective). We call $u^j$ the {\bf invariant factors} of $Q$.
Consider first slopes just above an integer, so $\lambda=k^+$, $k\in \N$ (the Hamiltonian flow for slope $k$ would give rise to $k$ ``full rotations'' of the $S^1$-action, at infinity).

\begin{lm}
The $\ku$-module $E^-HF^*(H_{k^+})$ is free, indeed it is a copy of $E^-QH^*(Y)$ shifted in grading.
Thus, we obtain invariant factors $u^{j_i(k)}$ associated to the homomorphism
\begin{equation}\label{Equation continuation for integer slope k}
c^*_{k^+}: E^-QH^*(Y) \to E^-HF^*(H_{k^+}).
\end{equation} 
These {\bf j-indices} cannot decrease in $k$, and they satisfy $0\leq j_1(k) \leq\ldots \leq j_y(k)\leq \infty$, $y=\mathrm{dim}_{\k}H^*(Y).$
\end{lm}

To an injective homomorphism $Q:V \to W$ of free finite rank $\ku$-modules, we associate two polynomials which we call {\bf filtration polynomial} $f_Q(t)$ and {\bf slice polynomial} $s_Q(t)$:
\begin{align}
f_Q= f_0 + f_1 t + \cdots + f_j t^j+\cdots
\;
\textrm{ where }
\;
f_j = \#\{\textrm{invariant factors }u^j\} = \# \{i: j_i= j\}.
\\
s_Q= d_0 + d_1 t + \cdots + d_j t^j\!+\cdots
\;
\textrm{ where }
\;
d_j = \#\{\textrm{invariant factors }u^{\geq j}\} = \# \{i: j_i\geq j\}
\end{align}
They are recoverable from each other, using $f_j = d_j-d_{j+1}$ and $d_0=\mathrm{rank}_{\ku}\,V.$
The derivative $f'_Q(0)=\sum_{j\geq 1} d_j=\dim_{\k}\textrm{Torsion}(\mathrm{coker}\, Q)$, which becomes $\dim_{\k}\mathrm{coker}\,Q$ when $V,W$ have equal ranks.
\\
Now view $W$ as sitting inside the $u$-localisation $V_u$ by using $Q_u^{-1}=\mathrm{adj}\,Q\otimes \tfrac{1}{\det Q}:W \hookrightarrow V\otimes \kuu = V_u$. 

\begin{ex}\label{Example Young diagram and dual}
$Q=\mathrm{diag}(u,u^3,u^4):V=\ku^3 \to \ku^3=W$, so the j-indices are $j_1=1$, $j_2=3$, $j_3=4$. View $W$ as the $\ku$-submodule $u^{-1}\ku\oplus u^{-3}\ku \oplus u^{-4}\ku \subset V_u=\kuu^3$.
$$
f_Q=t+t^3+t^4
 \qquad \textrm{ and } \qquad \;
s_Q=3+3t+2t^2+2t^3+t^4.
$$
Pictorially, $V=\ku^3$ is a rectangular shape of vertically stacked dots corresponding to copies of $\k$, with an upward $u$-multiplication map; and $W$  ``attaches'' a Young diagram's worth of new copies of $\k$:
\begin{center}
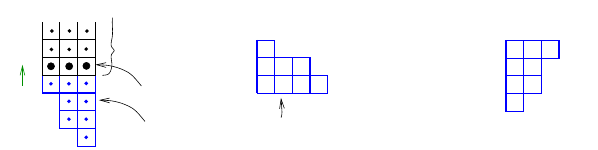
\end{center}
The Young diagram also encodes $f_{j}=\#$(rows of length $j$), so $f_1=1,f_3=1,f_4=1$ and $0$ otherwise.
\end{ex}

As becomes clear from the example, 
$$d_j=(u^{-j}\textrm{-slice dimension})=\dim_{\k} (W\cap u^{-j}V)/(W\cap u^{-j+1}V).$$
Moreover, $W\cap u^{-j}V \cong V\cap u^j W $ via $w\mapsto u^j w$ has image
\begin{equation}\label{Equation Intro Fn in free case}
F_j := V\cap u^j W 
=
\{ v\in V: Q(v) \textrm{ is }u^j\textrm{-divisible in }W\} 
=
\ker (V \stackrel{Q}{\longrightarrow} W/u^jW).
\end{equation}
These $\ku$-submodules $F_j\subset V$ yield a filtration of $V$, satisfying 
\begin{equation}\label{Equation intro Fn relations}
u F_j = F_{j+1}\cap uV \subset F_{j+1}\subset F_j.
\end{equation}
The $u^{-j}$-slice $(W\cap u^{-j}V)/(W\cap u^{-j+1}V)$ is canonically $F_j/uF_{j-1}$ via $w \mapsto u^j w.$ Abbreviate 
\begin{equation}\label{Equation Intro Pn in free case}
P_j := F_j/uF_{j-1} = F_j/(F_j\cap uV) \qquad \textrm{ and } \qquad P:=V/uV,
\end{equation}
and note that $P_j$ filter $P$, with associated graded module $\mathrm{gr}\,P:=\oplus P_j/P_{j+1}$. We deduce: 
$$
d_j = \dim_{\k} P_j
\;\;
\textrm{ and }
\;\;
f_j 
= \dim_{\k} P_j/P_{j+1}=\dim_{\k} F_j/(F_{j+1} + u F_{j-1}).
$$
In particular $f_Q$, by definition, is the {\bf Hilbert-Poincar\'{e} series} of $P$ for the filtration $P_j$.

\begin{ex}\label{Example intro Fn picture}
In \cref{Example Young diagram and dual}, we have:
\begin{center}
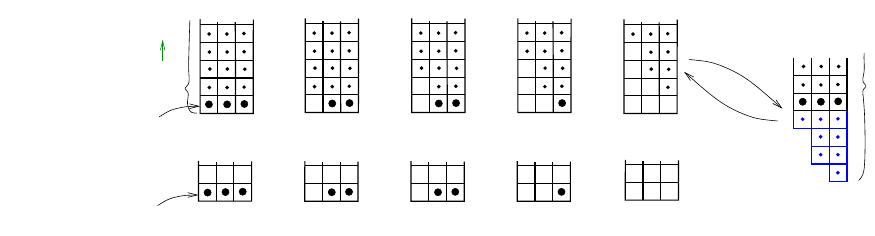
\end{center}
Note that for all sufficiently large $j$, we have $F_j = u^j W \subset V$ (see the $j=5$ case in the picture). 
\end{ex}

In our setting, \eqref{Equation continuation for integer slope k} plays the role of $Q$, so $V=E^-QH^*(Y)$. We have a canonical isomorphism $V/uV \cong QH^*(Y)$ to non-equivariant quantum cohomology. So $f_Q$ is a Hilbert-Poincar\'{e} series for classical cohomology, for a filtration that depends on  equivariant quantum and Floer theory.

For general slopes, $E^-HF^*(H_{\lambda})$ may not be free, unlike the $\lambda=k^+$ case. It is however finitely generated over the PID $\ku$, so the quotient by the torsion submodule is free. 
For $W$ finitely generated, we can therefore apply the previous setup to $Q: V \to W/T$, where $T=\mathrm{Torsion}(W)$. Thus,
$$
F_j:= \{ v\in V: Q(v) \textrm{ is }u^j\textrm{-divisible in } W/T\} = \ker (V \to W/(T+u^jW))
$$
satisfies \eqref{Equation intro Fn relations}, and we leave definition \eqref{Equation Intro Pn in free case} unchanged.
For $Q: V \to W$ injective, the map $Q:V \to W/T$ may have non-trivial kernel $Q^{-1}(T)\subset F_j$. This will not happen in our setting, when $Q$ is the map $c_{\lambda}: E^-QH^*(Y) \to E^-HF^*(H_{\lambda})$. Indeed, for $k>\lambda$, $Q$ composed with the continuation $E^-HF^*(H_{\lambda}) \to E^-HF^*(H_{k^+})$ equals \eqref{Equation continuation for integer slope k}, which is injective. 

\subsection{Filtrations in equivariant Floer theory}
We now summarise the outcomes in Floer theory.
Abbreviate $E^-_{\lambda}:=E^-HF^*(H_{\lambda})$ with its torsion $T_{\lambda},$ and $E^-:=\varinjlim E^-_{\lambda}=E^-SH^*(Y)$ with torsion $T$.
\begin{de}
For any $p\in [0,\infty)$, define the {\bf $\mathbf{E^-}$-filtration}
\begin{equation}
    \FF_{j}^p:=\!\!\!\!\!\!\!\!\bigcap_{\textrm{generic }\lambda\geq p} \!\!\!\!\!\! \ker \left(E^-QH^*(Y) \stackrel{c_{\lambda}^*}{\longrightarrow} E^-_{\lambda}/(T_{\lambda} + u^jE^-_{\lambda})\right) \quad(\textrm{so }\FF_{0}^{\lambda}=E^-QH^*(Y),\; \FF_{\infty}^{\lambda}=\ker c_{\lambda}^*),
    \end{equation}
for $j\in \Z_{\geq 0}$, and $\FF_{j}^p:=E^-QH^*(Y)$ for $j\in \Z_{< 0}$.
For generic $p$, we can ignore the intersection and let $\lambda=p$. 
Also, $\FF_j^p=\ker c_{p^+}^*$ for all $p^+>p$ sufficiently close to $p$ (such $p^+$ are generic).
For $p=\infty$, we use $E^-:=E^-SH^*(Y)$ and $c^*=\varinjlim c_{\lambda}^*$, so $\FF_j^{\infty}:=\ker (c^*:E^-QH^*(Y)\to E^-/(T+u^j E^-))$.
\end{de}

\begin{prop}
For the free weight $(a,b)=(1,0)$, the $\ku$-submodules $\FF_j^{p} \subset E^-QH^*(Y)$ are ideals with respect to the equivariant quantum product.
\end{prop}

\begin{cor}
    The $\ku$-submodules $\FF_{j}^p\subset E^-QH^*(Y)$ satisfy
$$\FF_{j}^p\subset \FF_{j}^{p'} \textrm{ for }p<p', \qquad
 u\FF_j^{p}   \subset \FF_{j+1}^{p} \subset \FF_j^{p}, \qquad \textrm{ and } \qquad
u\FF_j^{p} =  \FF_{j+1}^{p} \cap u \,E^-QH^*(Y).
$$
\end{cor}

\begin{rmk}\label{Remark intro yound diagram nesting}
$\FF_{j+1}^{p} \subset \FF_j^{p}$ corresponds to a naive inclusion of $\k$-summands in pictures like those from \cref{Example intro Fn picture}; $\FF_{j}^p\subset \FF_{j}^{p'}$ on the other hand is highly non-trivial and is often not the naive inclusion of such pictures. The pictures arise from two steps: first, a map of type $Q_u^{-1}:W/T \hookrightarrow V_u$ is applied, and then a basis of the free $\ku$-module $Q_u^{-1}(W/T)\subset V_u$ needs to be chosen. This second step is typically not compatible with varying $p$, because the basis of $V$ needed to put $Q: V \to W/T$ into Smith normal form (for a suitable basis of $W/T$) typically depends on $p$.  
\end{rmk}

\begin{de}
For $p\in [0,\infty]$, the {\bf slice dimensions}, {\bf slice series}, and  {\bf filtration polynomial} are
    $$
    d_{j}^{p}:=\dim_{\k} \FF_{j}^{p}/u\FF_{j-1}^{p},
    \qquad
    s_{p}(t) := \sum_{j\geq 0} d_j^{p} t^j \qquad \textrm{ and } \qquad 
    f_{p}(t) := \sum_{j\geq 0} f_j^{p} t^j = \sum_{j\geq 0} (d_j^{p}-d_{j+1}^{p}) t^j.
    $$
\end{de}
\begin{cor}
The $d_{j}^{p}$ and $f_j^p$ are finite invariants, and they satisfy the following properties:
\begin{align*}
0\leq d_j^{p} \leq y,
\qquad\quad
d_j^p \leq d_{j}^{p'} \;&\textrm{ for }\; p\leq p', \qquad\quad
d_j^p \leq d_{j+1}^p,
\\
f_j^p = \#\{\textrm{invariant factors }u^j\textrm{ of } c_{p^+}\}, 
\qquad
 &\textrm{ and } \qquad
d_j^p = \#\{\textrm{invariant factors }u^{\geq j}\textrm{ of } c_{p^+}\} 
\end{align*}
where $y=\dim_{\k}H^*(Y)=\mathrm{rank}_{\ku}E^-QH^*(Y)$, and
the invariant factors refer to the Smith normal form of the map
$
c_{p^+}: E^-QH^*(Y)\to E^-_{p^+}/T_{p^+}.
$
In particular, $f_Q$ is always a polynomial, and $s_Q$ is a polynomial precisely when $c_{p^+}$ is injective (which holds when $(a,b)$ is free, by \cref{Subsection intro Localisation}). 
\end{cor}

\begin{de}
    Abbreviate $\PP := E^-QH^*(Y)/uE^-QH^*(Y)$. This is filtered, for $j\in \Z$, by 
    $$\PP_j^p:=\FF_{j}^{p}/u\FF_{j-1}^{p}
    =
    \FF_{j}^{p}/(\FF_{j}^{p} \cap u \,E^-QH^*(Y)), \qquad \PP_{\infty}^{p}:=0, \qquad (\textrm{so }\PP_j^p=\PP\textrm{ for }j\leq 0),$$ 
    satisfying $\PP_{j+1}^p\subset \PP_{j}^p$ and $\PP_j^p\subset \PP_j^{p'}$ for $p\leq p'.$
    Define $\mathrm{gr}^p_j\PP:=\PP_j^p/\PP_{j+1}^p$ and $\mathrm{gr}^p\,\PP:=\oplus_{j\geq 0} \mathrm{gr}^p_j\PP$.
\end{de}

\begin{cor}
There is a canonical identification
$\PP\cong QH^*(Y)$ with non-equivariant quantum cohomology. Thus quantum cohomology is filtered by $\PP_j^p$ and its Hilbert-Poincar\'{e} series equals $f_{p}(t)$: 
$$f_j^{p}=\dim_{\k} \PP_j^{p}/\PP_{j+1}^{p} = \dim_{\k}\mathrm{gr}_j^p \PP.$$
\end{cor}

When $(a,b)$ is not free, $c_{p^+}$ may be non-injective: $s_{p}:=\dim_{\k}\ker c_{p^+}>0$ contributes to each $d_j^{p}$, and $d_j^{p}=s$ for all large $j$. One would then allow infinite size Young diagrams in the following definition.

\begin{de}
    Let $j_1^p,j_2^p,\ldots$ be the distinct exponents that arise in the invariant factors $u^j$ of $c_{p^+}$. The {\bf Young diagram} $\mathcal{Y}_p$ is the Young diagram of shape $(j_1^p,\ldots,j_1^p,j_2^p,\ldots)$, where each $j_i^p$ is repeated $f_{j_i}^p$ times.
    Switching rows and columns, we obtain the {\bf dual Young diagram} $\mathcal{Y}_p^*$, and it is the Young diagram of shape $(d_1^p,d_2^p,\ldots)$ in French notation (non-increasing row-lengths $d_j^p$ as in \cref{Example Young diagram and dual}).

    To $\mathcal{Y}_p^*$ we associate $\mathcal{Y}_p^*(\PP):=(\PP=\PP_0\supset \PP_1^p\supset \PP_2^p \supset \cdots)$, which is a filtration of $\PP$ by $\k$-linear subspaces. To $\mathcal{Y}_p$ we associate the associated graded $\k$-vector space $\mathcal{Y}_p(\PP):=\mathrm{gr}^p\PP=\mathrm{gr}_0^p \oplus \mathrm{gr}_1^p \oplus \cdots$.
\end{de}

\begin{cor}
For $i,j\geq 0$, $p\in [0,\infty]$, we have associations: 
\begin{align*}
(i\textrm{-th row in }\mathcal{Y}_p\textrm{ of length }j_i)&\mapsto \mathcal{Y}_p(i):=\mathrm{gr}_{j_i}^p\PP\subset \mathrm{gr}^p\PP  \;\textrm{ of dimension} = f_{j_i}^p = \#(\textrm{rows of length }j_i),
\\
(j\textrm{-th row in }\mathcal{Y}_p^*\textrm{ of length }d_j^p) &\mapsto \mathcal{Y}_p^*(j):=\PP_j^p\subset \PP \;\textrm{ of dimension} = d_j^p.
\end{align*}

Vertical upward inclusion of rows in $\mathcal{Y}_p^*$ corresponds to inclusions $\mathcal{Y}_p^*(j)=\PP_j^p\supset \PP_{j+1}^p= \mathcal{Y}_p^*(j+1)$. 
\end{cor}

\begin{rmk}\label{Remark choosing bases}
    Choose a basis for each $\mathcal{Y}_p(i)$. By successive basis extensions, choose lifted bases for $\cdots \subset \mathcal{Y}_p^*(j)\subset \mathcal{Y}_p^*(j-1)\subset \cdots \subset \mathcal{Y}_p^*(0)$. The resulting $\ku$-module basis for $V$ determines a picture as in \cref{Example Young diagram and dual}. In that example, the basis for $\mathcal{Y}_p^*(2)$ spans $0\oplus \k^2\subset \k^3$ and its product by $u^{-2}$ gives a basis for the $u^{-2}$-slice. In general, the boxes of $\mathcal{Y}_p^*$ correspond to a $\k$-linear basis for the complement of $V\subset W$, so that vertical upward inclusion of boxes corresponds to $u$-multiplication on the basis, and rows correspond to slices. However, substantial choices were required to pass from the abstract spaces to the picture, e.g.\;there is no basis-independent choice of $u^{-j}$-slice viewed inside $W$.
\end{rmk}

\begin{lm}\label{Lemma filtrations are compatible with LES}
The filtration $\FF_{j}^p$ on $E^-QH^*(Y)$ also induces a filtration on $E^{\infty}QH^*(Y)$, so that
the following natural short exact sequence respects the filtrations, using $\EE_j^p$ for $E^{+}QH^*(Y),$ 
$$0 \longrightarrow E^-QH^*(Y)[-2] \stackrel{u}{\longrightarrow} E^{\infty}QH^*(Y) \longrightarrow E^+QH^*(Y) \longrightarrow 0.$$
\end{lm}

\begin{rmk}[$u$-torsion]
There is also a canonical long exact sequence $$\cdots \longrightarrow  E^-HF^*(H_{\lambda})[-2] \stackrel{u}{\longrightarrow} E^{\infty}HF^*(H_{\lambda}) \longrightarrow E^+HF^*(H_{\lambda})  \longrightarrow E^-HF^{*+1}(H_{\lambda})[-1] \longrightarrow \cdots.$$
One can define a filtration also for $E^{\karo}HF^*(H_{\lambda})$, by considering the continuation maps $E^{\karo}HF^*(H_{\lambda}) \to E^{\karo}HF^*(H_{\lambda'})$ for $\lambda<\lambda'$, and the LES respects filtrations.
As explained previously, $E^-HF^*(H_{k^+})$ is a torsion-free $\ku$-module, so $E^-SH^*(Y)$ is torsion-free. 
For general slopes, $E^-HF^*(H_{\lambda})$ may not be torsion-free, but in many examples it is. Indeed, in the notation from \cite{RZ1}, $H^*(Y)\cong \oplus H^*(\F_\a)[-\mu_\a]$ often lies entirely in even degrees; if so, we deduce as in the non-equivariant proof in \cite{RZ1} that
$$E^-HF^*(H_{\lambda})\cong E^-HF^*(\lambda H)\cong \oplus H^*(\F_\a)[-\mu_{\lambda}(\F_\a)] \otimes_{\k}\ku,$$
using that all generators are in even grading; thus $E^-HF^*(H_{\lambda})$ is free in this case.

Note that $E^{\infty}$ models are always torsion-free, and $E^+$ models are entirely torsion. In the LES, it is precisely the torsion of 
$E^-HF^*(H_{\lambda})$ which is hit by $E^+HF^{*-1}(H_{\lambda})$. Since the direct limit $E^-SH^*(Y)$ is torsion-free, it follows that the direct limit of the LES becomes an SES: \cref{Equation short exact sequence for ESH}.
\end{rmk}

The {\bf $u$-valuation} $\mathrm{val}^{\lambda}(v) = \max \{j\in \N\cup \infty: v\in \FF_j^{\lambda}\}$ for $v\in E^-QH^*(Y)$ relates the filtrations:
 \begin{equation}\label{Equation relating valuation and plus theory}
\begin{split}
 (v \in \FF_{j+1}^{\lambda}) 
 \Longleftrightarrow (\mathrm{val}^{\lambda}(v)\geq j+1)
 \Longleftrightarrow 
 E^+c_{\lambda}(v\otimes u^{-j})=0 \in E^+HF^*(H_{\lambda}).
 \\
 \mathrm{val}^{\lambda}(v)=j \Longleftrightarrow \left(\;E^+c_{\lambda}(v\otimes u^{-j})\neq 0  \quad \textrm{but}\quad   u\cdot E^+c_{\lambda}(v\otimes u^{-j})=0\;\right).
 \end{split}
\end{equation}

This is the key idea to the following result.
Let $\mathbb{F}_j:=QH^*(Y)\otimes u^{-j}\ku/u\ku$ be the obvious\footnote{Via 
\cref{Lemma equivariant formality},
or intrinsically: $\mathbb{F}_j=\ker (u^{j+1}:E^+QH^*(Y) \to E^+QH^{*+2j+2}(Y))$.} nested $\ku$-submodules of $E^+QH^*(Y)$, for $j\geq 0$. Abbreviate $V:=E^-QH^*(Y)$. 

\begin{thm}\label{Theorem relation between F fil and E fil}
For all $j\geq 0$ and $p\in [0,\infty]$, we have an isomorphism of $\ku$-modules
$$
\FF_{j+1}^p/u^{j+1}V \to \EE_1^p\cap \mathbb{F}_j, \;\; v \mapsto [v \otimes u^{-j}],
$$
whose images define an exhaustive nested filtration $\cup_{j\geq 0} (\EE_1^p\cap \mathbb{F}_j) = \EE_1^p= \ker E^+c_{p^+}.$
This stabilises to $\EE_1^p=\EE_1^p\cap \mathbb{F}_j$ when $j\geq \max\{j: u^{j}$ is an invariant factor of $E^-c_{p^+}\}$ and $E^-c_{p^+}$ is injective.

A similar statement holds for $\FF_{j+1}^{p}(V^{-})/u^{j+1}V^- \to \EE_{m}^p(V^{+}) \cap \mathbb{F}_{j-m+1}$, $v \mapsto [v \otimes u^{m-1-j}]$.
\end{thm}

\subsection{Simple examples: $\C, T^*\C P^1, \mathcal{O}_{\C P^1}(-1)$}

\begin{ex}
Let $Y=\C$ with the standard $\C^*$-action. Let $(a,b)$ be free, equivalently $a \notin b\Z_{\geq 1}$. 
\begin{equation}
\begin{array}{ccccc}
\!\!\!E^-\!SH^*(\C)\cong \kuu, 
&
\strut\quad
&
\!\!\!\!\!\!E^{\infty}SH^*(\C)\cong \kuu, 
&
\strut\quad
&
\!\!\!\!\!\!\!\!\!\!\!\!\!\!\!E^{+}\!SH^*(\C) = 0,
\\
E^-HF^*(H_{k^+})= \ku\cdot x_k,
&
\strut\quad
&
E^{\infty}HF^*(H_{k^+})\cong \kuu\cdot x_0,
&
\strut\quad
&
E^{\infty}HF^*(H_{k^+}) \cong \mathbb{F}[2k],
\end{array}
\end{equation}
where $x_0$ comes from the unit $1\in QH^0_{S^1}(Y)$, and $x_k=\mathcal{S}^{-k}(x_0)\in E^{-}HF^{-2k}(H_{k^+})$ where $\mathcal{S}$ is a map that we discuss later in \cref{Subsection Comparing different weights via equivariant Seidel isomorphisms}.
 
Let $p \in [k,k+1)$, for $k\geq 0 \in \N.$ Then 
$
\mathrm{val}^{p}(x_0) = k =\lfloor p \rfloor,
$
 $$\FF_j^{p}=E^-QH^*(\C)=\ku\cdot x_0 \textrm{ for }j\leq k,\; \textrm{ and }\;\FF_j^{p}=\ku \cdot u^{j-k}x_0\textrm{ for }j>k.$$
The filtration polynomial is 
$
f_{p}(t) = t^k,
$
and the
slice polynomial is
$
s_{p}(t) = 1+t+\cdots+t^k.
$
\end{ex}

\begin{ex}
Let $Y$ be the blow-up of $\C^2$ at the origin. So $Y$ is the total space of $\mathcal{O}(-1) \to \C P^1$, and consider the standard $\C^*$-action on fibres.
Let $(a,b)$ be a free weight, equivalently $a\notin b\Z_{\geq 1}$.
\begin{equation*}
\begin{array}{ccccc}
\strut\hspace{-4.5ex}
E^-\!SH^*(Y)\cong \ku \oplus \kuu, 
&
\strut%
&
\hspace{3ex}E^{\infty}\!SH^*(Y)\cong \kuu^2, 
&
\strut\qquad %
&
\hspace{3.5ex}E^{+}\!SH^*(Y) \cong \mathbb{F},
\\
E^-\!HF^*(H_{k^+})\cong \ku\cdot x_0\oplus \ku\cdot x_k,
&
\strut \qquad %
&
E^{\infty}\!HF^*(H_{k^+})\cong \kuu^2,
&
\strut \qquad %
&
E^{\infty}\!HF^*(H_{k^+}) \cong \mathbb{F}
\end{array}
\end{equation*}
where $x_0$ comes from the unit $1\in QH^0_{S^1}(Y)$, and $x_k=\mathcal{S}^{-k}(x_0)$ in grading $|x_k|=-2k$. 
More explicitly, 
 $$E^-SH^*(Y)= \ku \cdot x_0 \oplus \kuu\cdot [\omega+T x_0+u^{\geq 1}\textrm{-corrections}],\;\quad SH^*(Y)\cong QH^*(Y)/\k[\omega+T x_0]\cong \k\cdot x_0.$$

For $p\in [k,k+1)$, the filtration polynomial is
$f_p = 1+t^k$, and the
slice polynomial is $s_p=2+t+\cdots + t^k.$
\end{ex}

\begin{ex}\label{Example Introduction TCP1 picture}
Let $Y=T^*\C P^1$ (the total space of $\mathcal{O}(-2)\to \C P^1$), with the standard $\C^*$-action on fibres. Let $(a,b)$ be a free weight, equivalently  $a \notin b\Z_{\geq 1}$. 
Then
\begin{equation*}
\strut\hspace{0ex}\begin{array}{ccccc}
\strut\hspace{-13ex}E^-\!SH^*(Y)\cong \kuu^2, 
&
\strut\quad\!\!\!\!\!\!\!\!\!\!\!\!\!\!\!
&
\hspace{3ex}E^{\infty}\!SH^*(Y)\cong \kuu^2, 
&
\strut\quad\!\!\!\!\!\!\!
&
\hspace{-13ex}E^{+}\!SH^*(Y) = 0,
\\
E^-\!HF^*(H_{k^+})= \ku\cdot x_k\oplus \ku\cdot x_{k-1},
&
\strut\quad\!\!\!\!\!\!\!\!
&
E^{\infty}\!HF^*(H_{k^+})\cong \kuu^2,
&
\strut\quad\!\!\!\!\!\!\!\!\!\!\!\!\!\!
&
E^{\infty}\!HF^*(H_{k^+}) \cong \mathbb{F}[2k]\oplus \mathbb{F}[2k-2],
\end{array}
\end{equation*}
where $x_0$ comes from the unit $1\in QH^0_{S^1}(Y)$, and $x_k = \mathcal{S}^{-k}(x_0)$ has grading $|x_k|=-2k$.

The slice and filtration polynomials, and the invariant factors of $c_{k^+}$, land in one of two cases:
\begin{equation*}
\begin{split}
s_k=N_k(t) &:= 2+2t+\cdots+2t^{k-1}+t^k+t^{k+1},\;\;\;\hspace{1ex}f_k=t^{k-1}+t^{k+1},
\;\;\hspace{1ex} \textrm{invariant factors }(u^{k-1},u^{k+1}); \\
s_k=Z_k(t) &:=  2+2t+\cdots+2t^{k-1}+2t^k,
\qquad\;\;\;\;\;\;\hspace{1ex} f_k=2t^{k}, \hspace{10ex}\textrm{invariant factors } (u^{k},u^{k}).
\end{split}
\end{equation*}
For example $s_0=Z_0=2$, but $s_1=N_1=2+t+t^2$.
We prove that $s_k=N_k$ holds for infinitely many $k\in \N$.
The $N_k$ case is the interesting one: Floer theory gives rise to a different filtration than the canonical one that measures $u$-divisibility in $E^-QH^*(Y)$ (which is the filtration that we get in the $Z_k$ case, up to translating the valuation by $k$).
The $Z_k$ case occurs if $x_0$ is more $u$-divisible than expected: $x_0$ is a multiple of $u^k x_k$ in $E^-HF^*(H_{k^+})$, so
$x_0 \in \FF_{k}^{k}$, equivalently $[x_0\otimes u^{-k+1}] = 0 \in E^+HF^*(H_{k^+})$.

The following picture\footnote{The dots are copies of $\k$. The black dots give a $\k$-linear basis for $E^-QH^*(Y)\cong H^*(\C P^1)\otimes_{\k} \ku$, the new $\k$-basis elements of $E^-HF^*(H_{3^+})\cong (H^*(\C P^1)\otimes_{\k} \ku)[6]$ are in blue. The fat blue dots are the $\ku$-module generators $x_2,x_3$ in grading $-4,-6$. The red circles around pairs of classes means that some $\k$-linear combination may be involved.} illustrates how subtle the filtration is, by comparing the invariant-factor picture from before, with the natural picture one draws when taking into account the grading.
\begin{center}
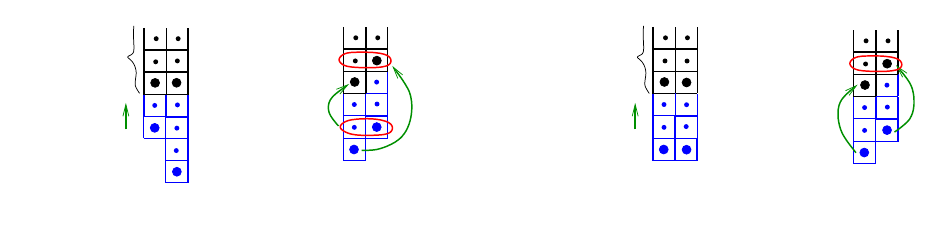
\end{center}
\end{ex}

\begin{ex}\label{Example twisted TCP1}
Let $Y=T^*\C P^1$ but now make $\C^*$ act with weight one on both the fibres and the $\C P^1$ base, using the twisted action described in \cite[Example 7.6]{RZ2}. 
The fixed locus consists of two points in the zero section, and the positive weights mentioned in \eqref{Equation bad slope} are $W^+=\{1,3\}$.

Let $(a,b)$ be a free weight, equivalently  $3a \notin b\Z_{\geq 1}$. 
Then
\begin{equation*}
\begin{array}{ccccc}
E^-\!SH^*(Y)\cong \kuu^2, 
&
\strut\qquad
&
E^{\infty}\!SH^*(Y)\cong \kuu^2, 
&
\strut\qquad
&
E^{+}\!SH^*(Y) = 0,
\end{array}
\end{equation*}
holds as before, but the $HF^*(H_{\lambda})$ groups are more complicated. 
Unlike the previous example, there is a $\Z/3$-torsion submanifold in $Y$, arising as the fiber over the fixed point that is not the minimum of the moment map. It gives rise to {\MB} manifolds diffeomorphic to $S^1$ for fractional slope values $(3k+1)/3$ and $(3k+2)/3$, for all $k\in \N.$
The {\MBF} spectral sequence that converges to $E^-SH^*(Y)$ is shown in \cref{Eq sp seq for T^CP^1 all together}; it has converged on the $E_1$-page. On the $E_0$-page, the fractional slope-valued columns are shifted copies of $H^*(S^1)[\![u]\!]$, whereas the positive integer slope-valued ones are shifted copies of $H^*(\R P^3)[\![u]\!]$ (using $\mathrm{char}\,\k=0$). The dots on the $E_1$-page in columns with slope values $\leq p$ are $\k$-linear generators of $HF^*(H_{p^+})$.

\begin{figure}[H]%
				\centering
				{
					\includegraphics[scale=1]{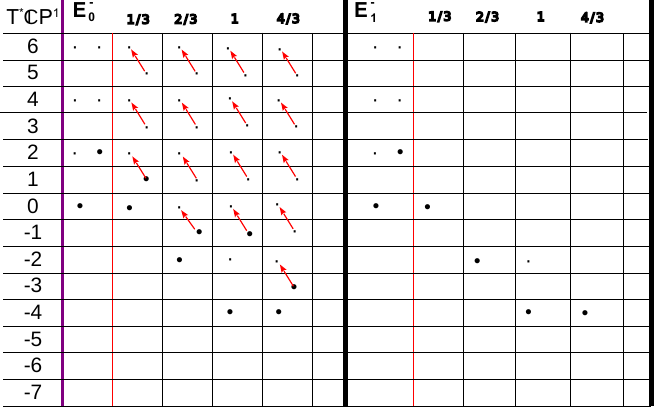}
					\caption{The $E_0$, $E_1$ pages for the spectral sequence converging to $E^-SH^*(T^*\CP^1)$, using a twisted action on $T^*\C P^1.$ The slope values are in bold. The zeroth column is $E^-QH^*(Y).$ Each dot is a copy of $\k$. The fat dots are the $u^0$-contributions: these generate the non-equivariant spectral sequence which converges to $SH^*(Y)=0$. 
     }
\label{Eq sp seq for T^CP^1 twisted}
				}
\end{figure}

The $E_1$-page in \cref{Eq sp seq for T^CP^1 all together} shows how the rank of $HF^*(H_{\lambda})$ steadily grows with $\lambda$, as it slowly converges to $\kuu^2\cong E^{\infty}QH^*(Y)\cong E^{\infty}SH^*(Y)$.
How the new generators appearing in higher columns arrange themselves into the Young diagram $\mathcal{Y}_p$ is difficult to extrapolate because the spectral sequence forgets the $u$-module action. We can however deduce the following properties of the slice polynomial.
Let $Z_k,N_k$ be as before, and let $X_k:=2+2t+\cdots+2t^k+t^{k+1}.$ Then $s_0 = Z_0 = 2$, and for $k\in \N$:
\begin{equation}\label{Equation slice poly twisted TCP1}
s_{k+\frac{1}{3}}= X_{2k}, \quad \textrm{and either }\; (s_{k+\frac{2}{3}},s_{k+1})=(N_{2k+1},Z_{2k+2})
\;\;\textrm{ or }\;\;
(s_{k+\frac{2}{3}},s_{k+1})=(Z_{2k+1},N_{2k+2}).
\end{equation}
This encodes the invariant factors of all $E^-c^*_{\lambda}$ maps, and also determines the filtration polynomial.
\end{ex}

\subsection{Comparing different weights via equivariant Seidel isomorphisms}
\label{Subsection Comparing different weights via equivariant Seidel isomorphisms}

Recall by \cref{Remark characteristic and mixed action} that the weight $(a,b)$ singles out a specific choice of $S^1$-action $S^1\hookrightarrow S^1\times S^1$. 
The Floer complexes 
$E^{\karo}_{(a,b)}HF^*(H_{\lambda})$  are related by chain isomorphisms, called {\bf Seidel isomorphisms}, which change the data $a,H_{\lambda}$.
These are discussed in detail in the Appendix. The main theorem about these is \cref{Theorem Seidel iso}; here we just mention the simplest example, at cohomology level:
\begin{equation}\label{Equation Introduction Seidel iso 2}
    \mathcal{S}: E^{\karo}_{(a,b)}HF^*(H_{\lambda}) \cong E^{\karo}_{(a-b,b)}HF^{*}(H_{\lambda-1})[2\mu], \qquad (\textrm{1-orbit }x) \mapsto \Fi^*x,
    \end{equation}    
where $(\Fi^*x)(t)=\Fi_{-t}(x(t))$, and where $\mu$ is the Maslov index of $\Fi$.
It follows that for $k\in \N,$
$$
\mathcal{S}^k: 
E^{\karo}_{(a,b)}HF^*(H_{k^+}) \cong
E^{\karo}_{(a-kb,b)}QH^*(Y)[2k\mu],
$$
which as a $\ku$-module is isomorphic to $H^*(Y)\otimes_{\k}E^{\karo}H^*(\mathrm{point})$, by \cref{Lemma equivariant formality}.

\begin{cor}\label{Cor sequence of r maps equiv}
$E_{(a,b)}^{\karo}SH^*(Y)$ can be identified with a direct limit of $\ku$-module maps 
\begin{equation}
\begin{tikzcd}[row sep=large, column sep=large]
E_{(a,b)}^{\karo} QH^*(Y)
\arrow[r, "r_{a,b}"] 
&
E_{(a-b,b)}^{\karo} QH^*(Y)[2\mu]
\arrow[r, "{r_{a-b,b}}"] 
&
E_{(a-2b,b)}^{\karo} QH^*(Y)[4\mu]
\arrow[r, "{r_{a-2b,b}}"] 
&
\cdots
\end{tikzcd}
\end{equation}
The $u^0$-part of the $r_{a,b}$ maps is the rotation map $r: QH^*(Y)\to QH^{*+2\mu}(Y)$ from Ritter \cite{R14}, which corresponds to 
quantum product by the class $Q_{\Fi^a}\in QH^{2\mu}(Y)$ described in \cite{RZ1}. 
\end{cor}

The continuation map $c_{k^+}:E^{\karo}_{(a,b)}QH^*(Y) \to E^{\karo}_{(a,b)} HF^*(H_{k^+})$, suitably composed with Seidel isomorphisms, defines a $\ku$-linear map:
\begin{equation}
ER_k:=\mathcal{S}^k\circ c_{k^+} = r_{a-(k-1)b,b} \circ \cdots \circ r_{a,b}
: E^-_{(a,b)}QH^*(Y) \to E^-_{(a-kb,b)}QH^*(Y)[2k\mu].
\end{equation}
We call these {\bf rotation maps}, to distinguish them from the Seidel isomorphisms: they need not be isomorphisms, nor injective in general, since we composed Seidel isomorphisms with continuation maps.
The rotation maps are the key ingredient for the proof of
\cref{Theorem intro ESH computation main}. The filtration and slice polynomials for the map $ER_k$ equal 
$f_k$ and $s_k$ (since an isomorphism, $\mathcal{S}^k$, does not affect these).

We conclude with a description of how the rotation maps are compatible with the filtrations. We label the three filtrations in \cref{Lemma filtrations are compatible with LES} by $\FF_{j,(a,b)}^{\karo,p}$ for $\karo\in \{-,\infty,+\}.$

\begin{thm}
For all $\karo\in \{-,\infty,+\}$,  $j\in \Z$, $p\in [0,\infty]$, $k\in \N$, the filtrations satisfy:
$$
ER_k^{-1}(\FF_{j,(a-kb,b)}^{\karo,p})
=
\FF_{j,(a,b)}^{\karo,p+k}.
\qquad \textrm{ In particular, }\;ER_k:\FF_{j,(a,b)}^{\karo,p+k} \to \FF_{j,(a-kb,b)}^{\karo,p}.
$$
\end{thm}

\begin{rmk}
    In this paper, for reasons of space, we have not used the machinery due to Liebenschutz-Jones \cite{liebenschutz2020intertwining,liebenschutz2021shift}, called ``intertwining relation'', which implies algebraic constraints on the $ER_k$ maps. In examples, that relation can sometimes recursively determine $ER_k$ from the $u^0$-theory (the non-equivariant quantum theory) and the $T^0$-theory (the non-quantum equivariant theory).
\end{rmk}

\noindent \textbf{Acknowledgements.} 
We thank Mohammed Abouzaid and Paul Seidel 
for helpful conversations. 

\section{Algebra I: filtrations on $\ku$-modules}
Throughout the general algebra sections, $\k$ denotes any field of characteristic zero. In Floer-theoretic applications, it will be the Novikov field.
\subsection{Basic prelimiaries about $\ku$}
\label{Subsection Basic prelimiaries about ku}
The $\k$-algebra of power series in $u$ over $\k$ is denoted $\ku$.
Its units, $\ku^{\times}\subset \ku$, are the series with non-zero $u^0$-coefficient, and any non-zero element of $\ku$ can be written as $u^k \cdot \textrm{unit}$ for a unique $k\in \N$. Indeed, there is a function
\begin{equation}\label{Definition nu}
\nu: \kuu:=\{\textrm{Laurent series}\}\to \Z\cup \{\infty\}, \;\nu(u^k\cdot \ku^{\times})=k, \;\textrm{ and }\;\nu(0):=\infty,
\end{equation}
which is a discrete valuation, and $\ku=\{\nu\geq 0\}$ is its valuation ring. Thus, 
$\ku$ is a discrete valuation ring (DVR) with maximal ideal $(u)$.
In particular, $\ku$ is a principal ideal domain (PID), indeed its non-zero ideals are classified: $(u^k)$ for $k\in \N$.

For any $\ku$-linear map $Q:V_1 \to V_2$ of free $\ku$-modules of finite rank, the Smith normal form gives rise to a choice of bases for $V_1,V_2$ in which the matrix $Q$ only has non-zero entries along the main diagonal, $\mathrm{diag}(u^{a_1},\ldots, u^{a_r},0,\ldots,0)$, with $0\leq a_1 \leq \cdots \leq a_r$. The $u^{a_j}$ are the {\bf invariant factors} of $Q$: the $a_j$ are unique, independent of choices, as invariant factors are unique up to rescaling by units, and
\begin{equation}\label{Equation invariant factors in terms of gcds}
\begin{split}
u^{a_1+\cdots + a_j} & = \mathrm{gcd}\{\textrm{all }j\times j \textrm{ minors of }Q\}, 
\\
a_1+\cdots + a_j
&=  \min \{\nu(\det A): A \textrm{ is a non-singular }
j\times j \textrm{ submatrix of }Q\},
\end{split}
\end{equation}
noting that greatest common divisors are meaningful only up to rescaling by units (here, $\ku^{\times}$). That relation follows from noting that the determinants of $\ku$-module automorphisms lie in $\ku^{\times}$.

\begin{ex}\label{Example invariant factors in the 2x2 case}
    When $V_1=V_2=\ku^2$, so $Q=(Q_{ij})$ is a $2\times 2$ matrix over $\ku$, the invariant factors are $q:=\mathrm{gcd}(Q_{ij})$ and $\tfrac{1}{q}\det Q$. For example, if $Q_{ij}\neq 0$ then $Q_{ij}=\textrm{unit}\cdot u^{n_{ij}}$, so $q=u^{\min n_{ij}}$.
\end{ex} 

Since $\ku$ is a PID, we often use the structure theorem for finitely generated $\ku$-modules $W$:
\begin{equation}\label{Equation torsion piece notation}
\begin{split}
& W\cong T\oplus \ku^r \quad\textrm{ and }\quad  T\cong T_{k_1}\oplus \cdots \oplus T_{k_s} \;\; (T=\textrm{Torsion}(W),\textrm{ but }\ku^r\textrm{ is non-canonical})
\\
&
\textrm{where }
T_k:=\ku/u^k\ku. \qquad \textrm{Thus, } T_1=\ku/u\ku \cong \k \textrm{ with the trivial $u$-action},
\end{split}
\end{equation}
where  $T\subset W$ is the torsion submodule, $r=\mathrm{rank}_{\ku}\,W$, and the $k_1\leq \cdots \leq k_s$ in $\N$ are unique.

We also need the following Lemma in \cref{Subsection The monotone case for free weights}. This technical result is needed because invariant factors of $\ku$-linear maps between free $\ku$-modules need not behave well under composition.

\begin{lm}\label{Lemma invariant factors of id block uorder 1 blocks}
Let $I$ be an identity matrix of size $t\times t$.
Suppose we have a sequence of square matrices
\begin{equation}\label{Equation id block plus higher u}
f_j:=\left(
\begin{smallmatrix} 
I + A_j u & B_j u \\
C_j u & D_j u
\end{smallmatrix}
\right)
\end{equation}
of the same size, for $j=1,2,\ldots,$ 
where $A_j,B_j,C_j,D_j$ are matrices over $\ku$.
Then the invariant factors of the composite $F_k:=f_k\circ f_{k-1}\circ \cdots \circ f_1$ are $1,\ldots,1,u^{\geq k},\ldots,u^{\geq k}$ with $t$ copies of $1$. Here we allow $u^{\infty}:=0$, and the number of zero invariant factors is the nullity, $\mathrm{rank}\, \ker F_k$.
\end{lm}
\begin{proof}
Let $w_1,\ldots,w_t,v_1,\ldots,v_s$ denote the standard $\ku$-module basis of $\ku^{t+s}$, where $(t+s)\times (t+s)$ is the size of $f_k$. Thus $f_k(w_i)=w_i + u^{\geq 1}$-terms, and $f_k(v_i) = u^{\geq 1}$-terms.
We change basis on the codomain of $f_1$ and domain of $f_2$: the new basis will be $f_1(w_1),\ldots,f_1(w_t),v_1,\ldots,v_s$. Similarly, on the codomain of $f_2$ and domain of $f_3$ we use: $f_2(f_1(w_1)),\ldots,f_2(f_1(w_t)),v_1,\ldots,v_s$, etc. Notice these are still $\ku$-module bases for $\ku^{s+t}$, because their $u^0$-part is still the original standard basis $w_1,\ldots,v_1,\ldots$.

In these new bases, for some new matrices $B_j',D_j',E_k$ over $\ku$, we get
\begin{equation}\label{Equation change of basis trick}
f_j:=\left(
\begin{smallmatrix} 
I \; & B_j' u \\
0 \; & D_j' u
\end{smallmatrix}
\right),
\quad
\textrm{ and }
\quad
F_k = \left(
\begin{smallmatrix} 
I \;\; & E_k u \\
0 \;\; & D_k'D_{k-1}'\cdots \cdot D_1' u^k
\end{smallmatrix}.
\right)
\end{equation}
The fact that we changed the basis on the codomain of $F_k$ does not affect its invariant factors.\footnote{the Smith normal form is not affected by applying automorphisms on the domain and/or codomain.} The invariant factors of the last matrix above are obviously $1,\ldots,1,u^{\geq k},\ldots,u^{\geq k}$, as required.
\end{proof}

\begin{rmk}\label{Remark basis for ker in Smith argument id plus ugeqone}
The final expression of $F_k$ implies that after changing $v_1,\ldots,v_s$ by a $\k$-linear change of basis, we can produce a $\ku$-module basis for $\ker F_k$ of the the form $v_{t+s-n+1},v_{t+s-n+1},\ldots,v_{t+s}$, where $n=\mathrm{rank}\,\ker F_k$; whereas a basis on the domain that (for some basis on the codomain) will give the non-zero invariant factors $u^{\geq k},\ldots,u^{\geq k}$ can be chosen to have $u^0$-parts $v_1,\ldots,v_{t+s-n}$.
\end{rmk}

\subsection{Localisation and direct limits for $\ku$-algebras}

Let $V$ be a unital $\ku$-algebra. Let $Q: V \to V$ be the $\ku$-module homomorphism given by multiplication by some $Q\in V$. 
Then
\begin{equation}\label{Equation V to V to V multiplication by Q}
V \stackrel{Q}{\longrightarrow} V \stackrel{Q}{\longrightarrow} V \stackrel{Q}{\longrightarrow} \cdots
\end{equation}
is a directed system, and we label the copies of $V$ by $V_0,V_1,V_2,\ldots$, so $Q:V_k\to V_{k+1}.$

\begin{rmk}[Motivation]
When $V$ is equivariant quantum cohomology, we show in \cref{Subsection weight 10 case product}
that \eqref{Equation V to V to V multiplication by Q} can arise.
In the setting of $\k$-vector spaces, \eqref{Equation V to V to V multiplication by Q} arose in \cite{R14,R16} for quantum cohomology $V:=QH^*(Y)$ with a (typically non-invertible) element $Q$ acting by quantum multiplication. In that setting, for dimension reasons the repeated images will stabilise, i.e.\;for sufficiently large $k$ we have $Q^kV=Q^{k+1}V$ and 
 $Q:Q^kV\to Q^kV$ is an isomorphism, in particular $Q$-localisation $V \to V_Q\cong \varinjlim V_k \cong Q^kV \cong V/\ker Q^k$ is a surjective map that quotients out the generalised $0$-eigenspace $\ker Q^k$ of $Q$.
This surjectivity may fail for $\ku$-algebras $V$ if the images fail to stabilise.\\
   \emph{Example:} $V=\ku$ and $Q=u$, then $V \to \varinjlim V = \kuu$ is the inclusion map.
\end{rmk}

\begin{lm}\label{Lemma direct limit of V copies}
The direct limit $\varinjlim V_k$ is the localisation $V_Q$ of $V$ at $Q$. In particular, the natural map $V_0\to \varinjlim V_k$ induces an isomorphism $(V_0)_Q\cong \varinjlim V_k$ where $(V_0)_Q$ is the localisation at $Q$.
\end{lm}
\begin{proof}
Formally write $\tfrac{v}{Q^k}$ to mean $v\in V_k$.
In the direct limit, $v\in V_k,v'\in V_{k'}$ are identified if they become equal after mapping both into $V_{k''}$ for large enough $k''$.
Note $Q^{k'}v,Q^{k}v'\in V_{k+k'}$, so that equality condition becomes $Q^{k''-k-k'}(Q^{k'}v-Q^{k}v')=0$. Localisation of $V$ at $Q$ is also defined by formal symbols $\tfrac{v}{Q^k}$ subject to the equivalence $\tfrac{v}{Q^k}\sim \tfrac{v'}{Q^{k'}}\Leftrightarrow Q^l(Q^{k'}v-Q^{k}v')=0$, some $l\in \N$. 
This shows that $V_Q\cong \varinjlim V_k$ are naturally identified, and the identification map is precisely the one induced by $V_0 \to \varinjlim V_k$, $v\mapsto \tfrac{v}{Q^0}$.
The claim follows.
\end{proof}

\begin{lm}
If $Q^k \in uV$ for some $k\in \N$, then $u\cdot 1$ is a unit, thus multiplication by $u$ on $\varinjlim V_k$ is an invertible map and $\varinjlim V_k$ is in fact a $\kuu$-module.
\end{lm}
\begin{proof}
    By \cref{Lemma direct limit of V copies}, $Q$ is invertible in $\varinjlim V_k$ since it is invertible in $V_Q$ with inverse $\tfrac{1}{Q}$, where $1\in V$ denotes the unit of $V$. Thus $Q^k$ is invertible, and in particular, any factor of $Q^k$ is invertible. Finally, $Q^k\in uV$ means that $u\cdot 1\in V$ is a factor of $Q^k$.
\end{proof}

\subsection{Localisation and direct limits for $\ku$-modules}

Assume now the weaker condition that $V$ is a $\ku$-module of finite rank $r$. Let
\begin{equation}\label{Equation V to V to V module maps Qj}
V \stackrel{Q_0}{\longrightarrow} V \stackrel{Q_1}{\longrightarrow} V \stackrel{Q_2}{\longrightarrow} \cdots
\end{equation}
be a sequence of $\ku$-linear homomorphisms $Q_k: V_k \to V_{k+1}$ where we label the copies of $V$ by $V_k$.

We first look at the case when the $Q_j$ are injective, postponing the general case to \cref{Subsection Non-injective homomorphisms}.

\begin{rmk}[Motivation]
In general equivariant Floer settings, \eqref{Equation V to V to V module maps Qj} arises rather than \eqref{Equation V to V to V multiplication by Q}: 
the maps can keep varying due to shifts in the equivariant weights.
In the $E^-$-equivariant setting, we often find that all $Q_j$ are injective and their $u^0$-parts are the same nilpotent map.
\\
   \emph{Example:} $V=\ku^2$ and $Q_j=\mathrm{diag}(1,u)$, then $V \to \varinjlim V = \ku\oplus \kuu$ is the inclusion map. 
\end{rmk}

\begin{de}[Adjugates]
For any $\ku$-module homomorphism $Q: V \to V$ (for $V$ as above), the {\bf adjugate} of $Q$ is the $\ku$-module homomorphism $\mathrm{adj}(Q): V \to V$ arising as the adjoint\footnote{$\mathrm{adj}(Q)(v_1)\wedge v_2 \wedge \cdots \wedge v_r=v_1 \wedge (\Lambda^{r-1}Q)(v_2\wedge \cdots \wedge v_r)=v_1 \wedge Q v_2\wedge \cdots \wedge Qv_r.$} of the exterior power $\Lambda^{r-1}Q$ for the pairing $V\times \Lambda^{r-1}V \stackrel{\wedge}{\longrightarrow} \Lambda^r V\cong \ku$. 
After a choice of basis of $V$, the adjugate is the transpose of the cofactor matrix of $Q$.
Moreover,
\begin{equation}\label{Equation adjugate}
Q\circ \mathrm{adj}(Q) = \det Q \cdot \mathrm{id} = \mathrm{adj}(Q) \circ Q.
\end{equation}
We recall that $\mathrm{adj}(Q_2 \circ Q_1)=\mathrm{adj}(Q_1)\circ \mathrm{adj}(Q_2)$ and $\mathrm{adj}(\lambda Q) = \lambda^{r-1}\mathrm{adj}(Q).$
Also, $Q$ is injective precisely if $\det Q\in \ku^{\times}$, in which case \eqref{Equation adjugate} determines $\mathrm{adj}(Q)$: the localisation $Q_u=Q\otimes 1: V_u\to V_u$ is a linear map on the $\kuu$-vector space $V_u=V \otimes_{\ku} \kuu$, and $\mathrm{adj}(Q)=\det Q \cdot Q_u^{-1}$.
\end{de}

From 
\eqref{Equation V to V to V module maps Qj}, we get adjugate maps
\begin{equation}
V_0 \stackrel{\mathrm{adj}Q_0}{\longleftarrow} V_1 \stackrel{\mathrm{adj}Q_1}{\longleftarrow} V_2 \stackrel{\mathrm{adj}Q_2}{\longleftarrow} \cdots
\end{equation}
For $a\leq b$, abbreviate by $Q^a_b: V_a \to V_{b+1}$ the map
$$
\qquad
Q^a_b:=Q_b\circ Q_{b-1}\circ \cdots\circ Q_a,
\quad
\textrm{ whose determinant is }
\quad
q^a_b:= \det Q^a_b = \det Q_a\cdots \det Q_b.
$$
By the properties of adjugates, $\mathrm{adj}(Q^a_b)=\mathrm{adj}(Q_a)\circ \cdots \circ \mathrm{adj}(Q_b):V_{b+1} \to V_a$, and $\mathrm{adj}(Q^a_b)Q^a_b=q^a_b\cdot \mathrm{id}$.

\begin{lm}\label{Lemma localisation when have ku mods}
When all $Q_j$ are injective, the direct limit of \eqref{Equation V to V to V module maps Qj} can be identified with a $\ku$-submodule of the $u$-localisation $(V_0)_u:= V_0 \otimes_{\ku} \kuu$,
$$
\varinjlim V_j \subset (V_0)_u \;\textrm{ with canonical maps }\;  \mathrm{adj}(Q^0_j)\otimes \tfrac{1}{q^0_{j}} = (Q^0_j)_u^{-1}: V_{j+1} \to (V_0)_u,
$$
and canonical map $\mathrm{id}\otimes 1:V_0 \to \varinjlim V_j \subset (V_0)_u$.
\end{lm}
\begin{proof}
The direct limit of the finite sequence of injective maps $V_0 \to \cdots \to V_{n+1}$ is identifiable with $V_{n+1}$, with canonical maps $Q^j_n:V_{j}\to V_{n+1}$. We compose this with the injective map $\mathrm{adj}(Q^0_n):V_{n+1}\to V_0$ so that we can view the direct limit inside $V_0$. The canonical maps from $V_j$ become  $\mathrm{adj}(Q^0_n)\circ Q^{j}_n = q^j_n\cdot \mathrm{adj}(Q^0_{j-1})$. To make this independent of $n$ we rescale by $(q^0_n)^{-1}$ at the cost of changing the target $V_0$ to $(V_0)_u$. The canonical map $V_j \to (V_0)_u$ becomes $\mathrm{adj}\,Q^0_{j-1}\otimes \tfrac{1}{q^0_{j-1}}.$

This procedure is equivalent to replacing \eqref{Equation V to V to V module maps Qj} by a sequence of inclusions
\begin{equation}\label{Equation V to V to V multiplication by Q version 2}
V_0
\;
\stackrel{\mathrm{incl}}{\longrightarrow} 
\;
(\tfrac{1}{q^0_0}\mathrm{adj}(Q^0_0)V_1\subset \tfrac{1}{q^0_0}V_0)
\;
\stackrel{\mathrm{incl}}{\longrightarrow} 
\;
(\tfrac{1}{q^0_1}\mathrm{adj}(Q^0_1)V_2\subset \tfrac{1}{q^0_1}V_0)
\;
\stackrel{\mathrm{incl}}{\longrightarrow} 
\;
\cdots
\stackrel{\mathrm{incl}}{\longrightarrow} 
\kuu \otimes_{\ku} V_0. \qedhere
\end{equation}
\end{proof}

\begin{ex}
Consider \eqref{Equation V to V to V module maps Qj} but with $\Z$-modules $V_j=\Z$ and multiplication maps $Q_j=(j+1)\cdot \mathrm{id}$.
The adjugates equal $1$ and 
$q^0_n=(n+1)!$. So 
$\tfrac{1}{q^0_{n}}\mathrm{adj}(Q^0_n)=\tfrac{1}{(n+1)!}\mathrm{id}: V_{n+1}=\Z \to \tfrac{1}{(n+1)!}\Z\subset  \Q=V_0\otimes_{\Z}\Q$.
The analogue of \eqref{Equation V to V to V multiplication by Q version 2} are the inclusions
$\Z \subset \tfrac{1}{1}\Z \subset \tfrac{1}{2}\Z  \subset \tfrac{1}{3!}\Z \subset \cdots \subset \Q.$
The direct limit is $\Q$.
\end{ex}

\begin{de}
For any $\ku$-module homomorphism $Q: V \to V$ (with $V$ free of finite rank), we can write $Q=P_0 + uP_1 + \cdots$, where $P_i$ are endomorphisms not involving $u$. We call $P_0$ the {\bf $u^0$-part} of $Q$.
\end{de}

\begin{lm}\label{Lemma Definition order-factors}
If $Q:V \to V$ is injective with nilpotent $u^0$-part, then $\det\, Q \in u^k \ku^{\times}$ for a unique $k\geq 1 \in \N$, namely $k=\nu(\det\,Q)$. We call $k=\nu(\det\,Q)$ the {\bf order-factor} of $Q$.
\end{lm}
\begin{proof}
    Composing $Q$ enough times we see that some $Q^N$ has zero $u^0$-part. Hence $\det Q^N \notin \ku^{\times}$. As $Q$ is injective, $\det Q\neq 0$.
    The claim follows by taking $k:=\nu(\det\, Q)$ in the notation of \cref{Subsection Basic prelimiaries about ku}.
\end{proof}

\begin{prop}\label{Prop Qj injective get direct lim}
Suppose that the $Q_j$ in \eqref{Equation V to V to V module maps Qj} are injective, and their $u^0$-part is a nilpotent map independent of $j$.
    Then the direct limit of \eqref{Equation V to V to V module maps Qj} is isomorphic to $u$-localisation,
    $$
    \varinjlim V_j \cong V_u:= V\otimes_{\ku}\kuu,
    $$
with canonical maps $\mathrm{id}\otimes 1: V_0 \to (V_0)_u$ and 
$$
(Q^0_j)_u^{-1}=u^{-(k_0+\cdots +k_j)}\cdot \mathrm{unit}_j\cdot \mathrm{adj}(Q^0_j): V_{j+1} \to (V_0)_u,
$$
where $k_j=\nu(\det\,Q_j)$ are the order-factors, and $q^0_j = \mathrm{unit}_j^{-1}\cdot u^{k_0+\cdots +k_j}$ for some $\mathrm{unit}_j\in \ku^{\times}$.
\end{prop}
\begin{proof}
\cref{Lemma Definition order-factors} yields order-factors $k_j:=\nu(\det\,Q_j)$.
Define $\mathrm{unit}_j:=u^{\nu(q^0_j)}(q^0_j)^{-1}\in \ku^{\times}$, and note $\nu(q^0_j) = \nu(\det Q_0 \cdots \det Q_j)= k_0+\cdots +k_j$.
Now apply \cref{Lemma localisation when have ku mods} to view $\varinjlim V_j \subset (V_0)_u$, with canonical maps as in the claim. It remains to show that this inclusion is an equality.\footnote{The condition $\nu(\det\,Q_j)\geq 1$, without the nilpotency assumption, is not enought to guarantee equality, e.g.\;$Q_j:=\mathrm{diag}(1,u^2)$, $V=\ku^2$, $\nu(\det\,Q_j)=2$, $\varinjlim V_j=\ku \oplus \kuu\neq \kuu^2$. Assuming that  the $u^0$-parts of the $Q_j$ are nilpotent (but not the same nilpotent map) is not enough: e.g.\,$Q_{\mathrm{odd}}=\left(\begin{smallmatrix} 0 & 1 \\ u & 0 \end{smallmatrix}\right)$
and $Q_{\mathrm{even}}=\left(\begin{smallmatrix} 0 & u \\ 1 & 0 \end{smallmatrix}\right)$
have product $\mathrm{diag}(1,u^2).$
}
Let $P$ denote the $u^0$-part of $Q_j$, so $P^{b-a+1}$ is the $u^0$-part of $Q^a_b$. Note $P^r=0$ since $P$ is nilpotent and $V$ has rank $r$ (e.g. by considering the $\kuu$-vector space problem for the $u$-localised map $P_u$, using $\dim V_u=r$).
It follows that $Q^0_{r-1},Q^{r}_{2r-1},\ldots$ are divisible by $u$. Let $m:=Nr-1$, then given $v\in V_0$ we have $Q^0_{m}(v)=u^N\widetilde{v}$ for some $\widetilde{v}\in V_{Nr}$. So $\mathrm{adj}(Q^0_{m})(u^N\widetilde{v})=q^0_{m}v$, therefore
$\tfrac{1}{q^0_{m}}\mathrm{adj}(Q^0_{m})(\widetilde{v})=u^{-N}v\in \varinjlim V_j\subset (V_0)_u$ (using the map from \cref{Lemma localisation when have ku mods}). As $N\in \N$, $v\in V_0$ were arbitrary, $\varinjlim V_j \subset (V_0)_u$ is thus an equality.
\end{proof}

We continue to assume $V$ is a free $\ku$-module of rank $r$, and $V_k=V$ for $k=0,1,\ldots$. The motivation for the following algebra machinery is illustrated in the simple example in \cref{Remark example of slice dimensions}.

\begin{lm}\label{Lemma presentation Rj inverse of V}
Let $R_k: V_0 \to V_k$ be an injective $\ku$-module homomorphism.
Let $(u^{j_1},\ldots,u^{j_r})$ be the invariant factors, where $0\leq j_1 \leq \cdots \leq j_r$.
Then
$$
(R_k)_u^{-1}(V_k) \cong u^{-j_1}\ku \oplus \cdots \oplus u^{-j_r} \ku \subset \kuu^r \cong (V_0)_u.
$$
\end{lm}
\begin{proof}
Choose a basis of $V=V_0=V_k$, 
then the Smith normal form applies automorphisms to the bases of $V_0,V_k$ to yield $R_k = \mathrm{diag}(u^{j_1},\ldots,u^{j_r})$. The result the follows immediately. 
\end{proof}

\begin{de}\label{Definition slice dimensions}
Let $W$ be a $\ku$-submodule of $\kuu^r$. Extend \eqref{Definition nu} to tuples: $\nu:\kuu^r \to \Z\cup\{\infty\}$,
$$
\nu(u^{j_1}\ku^{\times},\ldots,u^{j_r}\ku^{\times})
= \min (j_1,\ldots,j_r),
$$
and omit $j_i$ in the $\min$  when there is a zero entry in place of $u^{j_i}\ku^{\times}$.
If an element in $\kuu^n$ has $\nu$-value $m\in \Z$, then the {\bf leading part} is obtained by forgetting all $u^{>m}$-terms in each entry. Call {\bf leading entries} those entries which become non-zero entries in the leading part.
For example, $x:=(2u^{-2}+u^{-1},5u^{-3}+u^2)$ has $m=\nu(x)=-3$, leading part $(0,5u^{-3})$, and $x_2$ is a leading entry.

For $n\in \Z$, the $u^{-n}$\textbf{-slice} of $W\subset \kuu^r$ is $W\cap u^{-n}\k^r$; its dimension is the {\bf $u^{-n}$-slice dimension}:
$$
\dim_{-n}\, W
= \dim_{\k} W\cap u^{-n}\k^r
=
\dim_{\k} \,\{ \textrm{leading parts of all }w\in W\textrm{ with }\nu(w)=-n  \}.
$$
Equivalently, it is $\dim_{\k}\,(W/u^{-n+1}\ku^r)\cap (u^{-n}\ku^r/u^{-n+1}\ku^r)$; and it is also the number of $\k$-linearly independent elements of $W\cap \k^r\subset \kuu^r$ which are divisible in $W$ by $u^n$.

Note $\dim_{-n}\,W\leq r$. In view of the injective $\ku$-module action, $\dim_{-n}\,W$ cannot increase as we increase $n$, so it stabilises for large $n$, and we call it {\bf slice dimension},
$$
\dim_{-\infty}W := \lim_{n\to \infty} \dim_{-n} W \in \{0,1,\ldots,r\}.
$$
Observe that $\dim_{-\infty} W=0$ is equivalent to $W\subset u^{-n}\ku^r$ for large enough $n$, and so it is equivalent to the condition that $W$ is free (note $u^{-n}\ku$ is free, but $\kuu$ is not).
\end{de}

\begin{cor}\label{Cor structure theorem for V in kuu power}
Given injections $\ku^r \hookrightarrow W \hookrightarrow W_u\cong \kuu^r$, let $s:=\dim_{-\infty}\,W$, then
$$
W \cong u^{-j_1}\ku \oplus \cdots \oplus u^{-j_{r-s}} \oplus \kuu^s,
$$
so that $\ku^r\hookrightarrow W$ becomes an isomorphism $\ku^r \to \ku^r\subset u^{-j_1}\ku\oplus \cdots \oplus u^{-j_{r-s}}\ku \oplus \kuu^s.$
\end{cor}
\begin{proof}
Pick $e_1,\ldots,e_s\in W$, 
with $\nu(e_i)=n$, whose leading parts are a basis 
for all leading parts of $w\in W$ with $\nu(w)=n$ (here we assume $n$ is sufficiently large, so that this basis exists by definition of slice dimension). 
Replace $e_i$ by $u^ne_i$, so that $e_i\in \ku^n$ with $\nu(e_i)=0$.
After a $\k$-linear automorphism of $\ku^n$, we may assume that the leading parts of $e_1,\ldots,e_s$ are the last $s$ vectors of the standard basis of $\ku^r.$
By quotienting $\ku^r$ by $0\oplus \ku^s$, and $W$ by $\kuu^s$, we may now assume that $s=0$.
Therefore $W$ is free, and in fact $W\cong \ku^r$ since the $u$-localisation of the injections in the claim give isomorphisms of $\kuu^r$. 
Now apply \cref{Lemma presentation Rj inverse of V} to the given injection $\ku^r\hookrightarrow W \cong \ku^r$, which yields the invariant factors $u^{-j_i}$ appearing in the claim.
\end{proof}

Let $W := u^{-j_1}\ku \oplus \cdots \oplus u^{-j_{r}} \ku$. Abbreviate
$$
d_j := \#\{\textrm{invariant factors }u^{\leq -j}\} = \# \{i: j_i\geq j\},
$$
By definition, $d_j\geq d_{j+1}$, and $d_0=r$. 
We let these be the coefficients of a polynomial:
$$s_W(t)=r\cdot 1 + d_1 t + \cdots + d_j t^j+\cdots$$ 
We study such ``slice polynomials'' in \cref{Corollary filtration polynomial}. See \cref{Example Young diagram and dual} for a picture.

\begin{lm}\label{Lemma slice dimensions give coeffs of filtration polynomial}
$d_n = \dim_{-n}\,W$ for $n\geq 0$. In particular, $\deg s_W = \min \{n: \dim_{-n}\,W \neq 0\}$.
Moreover, $d_0+d_1+\cdots +d_r = r+j_1 + j_2 + \cdots + j_r.$
\end{lm}
\begin{proof}
Assuming $j_1\leq j_2 \leq \cdots \leq j_r$, and letting $i$ be the smallest index with $j_i\geq j$, we see that the tail of the standard basis $e_{j_i},e_{j_i+1},\ldots,e_r$ of $\kuu^r$, rescaled by $u^{-j_i}$, is a basis for the leading parts of $W$ with $\nu=-j_i$. The first claim follows.
For the final claim, 
by inspection we see that $\dim_{0}\,W+\dim_{-1}\,W+\cdots +\dim_{-j_r}\,W=\dim_{\k} W/u\ku^r=r+j_1+\cdots+j_r$.
\end{proof}

\begin{cor}\label{Cor structure theorem 2}
Given injections $\ku^r \cong W_0 \hookrightarrow W_1 
\hookrightarrow W_2 \hookrightarrow \cdots \hookrightarrow \kuu^r$, for free $W_k$, then
\begin{equation}\label{Equation direct limit structure theorem 2}
W_k \cong u^{-j_1(k)}\ku \oplus \cdots \oplus u^{-j_{r}(k)} \ku,
\end{equation}
so that $\ku^r\hookrightarrow W_k$ becomes an isomorphism $\ku^r \to \ku^r$. The $0\leq j_1(k)\leq \cdots \leq j_r(k)$ satisfy $j_i(k)\leq j_i(k+1)$. Equivalently, the coefficients of $s_{W_k}=\sum d_j(k)t^j$ satisfy $0\leq d_j(k)\leq d_j(k+1) \leq r$.

For some $s,j_i$, letting $k\to \infty$ we have $j_i(k)\to j_i$ for $i\leq r-s$, $j_i(k)\to \infty$ for $i\geq r-s+1$, and
$$
\varinjlim W_n \cong u^{-j_1}\ku \oplus \cdots \oplus u^{-j_{r-s}}\oplus \kuu^s.
$$
Similarly, $d_j(k)\to d_j=s+\#\{i:j_i\geq j\} = \dim_{-j}\,\varinjlim W_n$, yielding a series
$$
\lim_{k\to \infty} s_{W_k} = r\cdot 1 + d_1 t + d_2 t^2 + \cdots + s t^{j_{r-s}+1} + \cdots
$$
\end{cor}
\begin{proof}
This follows from the observations preceding the claim and the proof of \cref{Cor structure theorem for V in kuu power}.
\end{proof}

\begin{rmk}\label{Remark prototypical example}
    Abbreviate by $W^-$ the direct limit in \eqref{Equation direct limit structure theorem 2}, $W^{\infty}:=\kuu^r$, and $W^+:=W^{\infty}/uW^-$. Then $W^+\cong \mathbb{F}^s$, where $\mathbb{F}:=\kuu/u\ku.$ 
    This is the prototypical example of what happens in the short exact sequence of \cref{Corollary torsion freeness 2 free case}.
\end{rmk}

\subsection{Non-injective homomorphisms}
\label{Subsection Non-injective homomorphisms}

Let $Q: V \to W$ be a $\ku$-module homomorphism  between free $\ku$-modules. Abbreviate:
$$
v = \mathrm{rank}\,V,
\qquad 
w=\mathrm{rank}\,W,
\qquad 
n=\mathrm{null}\,Q=\mathrm{rank}\,\ker Q, 
\qquad 
q=v-n=\mathrm{rank}\,Q.
$$
By taking the Smith normal form, $Q$ becomes
$$
\mathrm{diag}(u^{j_1},\ldots,u^{j_{q}},0,\ldots,0):
\ku^{q} \oplus \ku^{n}
\to
\ku^{q} \oplus \ku^{w-q},
$$
where $0\leq j_1 \leq \cdots \leq j_{q}$,
diag refers to the main diagonal of the (possibly non-square) matrix, and $n$ invariant factors $u^{\infty}:=0$ arose on the diagonal because of the kernel.
We can apply the previous machinery to the injective block $\mathrm{diag}(u^{j_1},\ldots,u^{j_{q}}):
\ku^{q} \to \ku^{q},
$
to deduce that 
$$
V_u=V\otimes_{\ku}\kuu \supset \mathrm{image}(\mathrm{adj}\,Q) \otimes \tfrac{1}{j_1\cdots j_q}
\cong u^{-j_1}\ku \oplus \cdots \oplus u^{-j_{q}} \ku.$$
Unlike the injective case discussed previously, that image is not a copy of $W$, it forgets the summand
$$
\ku^{w-q}\cong \mathrm{coker}\, Q / \mathrm{Torsion}(\mathrm{coker}\, Q).
$$

Let $R_k=Q_{k-1}\circ \cdots \circ Q_0: V_0 \to V_k$ be the composite of the first $k$ maps in \eqref{Equation V to V to V module maps Qj} (not assuming injectivity), where $V=V_0=V_k$ is free of rank $r$.
Let $(u^{j_1(k)},\ldots,u^{j_r(k)})$ be the invariant factors of $R_k$, where $0\leq j_1(k) \leq \cdots \leq j_r(k)\leq \infty$, where we allow $u^{\infty}:=0$.
Abbreviate the map
$$
A_k:=\mathrm{adj}\,R_k \otimes \prod_{j_i(k)\neq \infty} u^{-j_i(k)}  : V_k \to V_0 \otimes_{\ku} \kuu = (V_0)_u.
$$
\begin{cor}\label{Lemma presentation Rj inverse of V noninjective case}
There is a $\ku$-module isomorphism
$$
V_k \cong \Bigg(\bigoplus_{j_i(k)\neq\infty} u^{-j_i(k)} \ku\Bigg) \oplus \ker R_k \cong \mathrm{im}\,A_k \oplus \ker R_k \subset (V_0)_u.
$$
The exponents satisfy $j_i(k)\leq j_i(k+1)$, and 
$\#\{i: j_i(k)=\infty\}=\mathrm{null}\,R_k$ ($=\mathrm{null}\,\varinjlim R_k$ for $k\geq r$).
\end{cor}
\begin{proof}
This follows by the previous discussion, noting: $v = w = \mathrm{rank}\,V$, $n(k)=\mathrm{rank}\,\ker R_k$, 
and $q(k)=\mathrm{rank}\,R_k=v-n(k)$.
The kernels of the $R_k$ maps are nested free $\ku$-submodules of $V$, so for sufficiently large $k$ ($k\geq r$ suffices), the numbers $n(k),q(k)$ become independent of $k$.
\end{proof}

\begin{rmk}
The motivation for the next result, needed in \cref{Subsection Growth rate of the filtration polynomial nonfree weight case}, 
is that for non-free weights $(a,b)\neq (0,0)$, only one of the Floer continuation maps that define the direct limit $E^{\infty}SH^*(Y)$ can fail to be  injective, in view of
\cref{Theorem injectivity theorem 2}. The geometric reason is explained in \cref{Remark non-free weight MB contributions}.
\end{rmk}

Consider the free $\ku$-module
$$C_k=\mathrm{coker}\,R_k/\mathrm{Torsion}(\mathrm{coker}\,R_k).$$
Note that $Q_j$ induces a natural map 
$\mathrm{coker}\,Q_{j-1} \to \mathrm{coker}\,Q_j$ since $Q_j\circ R_{j-1}=R_j.$ 
If $Q_j$ is injective, then that natural map will send torsion to torsion, so it descends to a map on quotients: $Q_j':C_{j-1} \to C_j$.

Suppose that only one map $Q_{p-1}:V_{p-1}\to V_{p}$ in \eqref{Equation V to V to V module maps Qj} fails to be injective.
Then $\ker R_k=\ker R_p$ for $k\geq p$.
We can make a choice of identification $C_p \cong \ker R_p$, since they are both free of rank
$$n:=\mathrm{rank}\,C_p=\mathrm{rank}\,\ker R_p.$$
This also induces an isomorphism of the $u$-localisations $(C_p)_u\cong (\ker R_p)_u$. 

The composite $R_k':=Q_{k-1}'\circ \cdots \circ Q_p':C_p \to C_k$ for $k>p$ is an injective map of free $\ku$-modules, so we can apply the results from the previous Subsection:
$$
C_p \cong \ku^{n}, \qquad C_k \cong u^{-\ell_1(k)} \ku \oplus \cdots \oplus u^{-\ell_{n}(k)}\ku,
$$
arising from $\mathrm{adj}\,R_k' \otimes \tfrac{1}{\det R_k'}: C_k \to (C_p)_u$, where $(u^{\ell_1(k)},\ldots,u^{\ell_{n}(k)})$ are the invariant factors of $R_k'$ for $k\geq p$.
In particular, $0\leq \ell_1(k)\leq \cdots  \leq \ell_{n}(k)<\infty$, and $\ell_i(k)\leq \ell_i(k+1)$.

Combining all that with the previous discussion, we may view $V_k$ inside $(V_0)_u$ as follows: 
$$
V_k\cong V_k':= u^{-j_1(k)} \ku \oplus
\cdots \oplus u^{-j_{r-n}(k)} \ku \oplus
 u^{-\ell_1(k)} \ku \oplus \cdots \oplus u^{-\ell_{n}(k)} \ku \subset \kuu^r \cong (V_0)_u,
$$
noting that $\#\{i: j_i(k)=\infty\} = \mathrm{rank}\,\ker R_k = n$.

\begin{cor}\label{Lemma presentation Rj inverse of V noninjective case direct limit}

Assume only one map $Q_{p-1}:V_{p-1}\to V_{p}$ in \eqref{Equation V to V to V module maps Qj} fails to be injective. Then the direct limit of \eqref{Equation V to V to V module maps Qj} can be identified with
$$
\varinjlim Q_j \cong \bigcup_{k\geq p} V_k' \subset (V_0)_u.
$$
For $k\geq p$, the canonical map $V_k \to \varinjlim R_k$ is built from two maps, described previously:
\begin{enumerate}
    \item the map $A_k:V_k \to \mathrm{im}\,A_k\cong \bigoplus_{j_i(k)\neq \infty} u^{-j_i(k)} \ku$, and
    \item the composite map $V_k \to \mathrm{coker}\,R_k \to C_k\to (C_p)_u \cong (\ker R_p)_u$, with image $\oplus u^{-\ell_i(k)}\ku$. 
\end{enumerate} 
In particular, $\varinjlim Q_j$ viewed inside $(V_0)_u\cong \kuu^r$ gets identified with
$$
u^{-j_1}\ku \oplus \cdots \oplus u^{-j_{r-n-s}}\ku \oplus \kuu^{s} \oplus 
u^{-\ell_1}\ku \oplus \cdots \oplus u^{-\ell_{n-s'}}\ku \oplus \kuu^{s'} \subset \kuu^r,
$$
where $s$ is the number of limits $j_i=\lim j_i(k)$ which equal $\infty$, and $s'$ is the number of limits $\ell_i=\lim \ell_i(k)$ which equal $\infty$.
\end{cor}
\begin{proof}
This is now immediate from the previous discussion.   \end{proof}

\subsection{Canonical valuations for $\ku$-modules, and induced filtration}

\begin{de}\label{Definition induced filtration via a map c}
   A \textbf{valuation} $\mathrm{val}_u: W \to \Z \cup \{\infty\}$ on a $\ku$-module $W$ means the following:
   \begin{enumerate}
        \item $\mathrm{val}_u(w)=\infty$ if and only if $w$ is a $u$-torsion element ($u^kw=0$ for some $k\in \N$);
        \item $\mathrm{val}_u(u^kw)=k+\mathrm{val}_u(w)$;
        \item $\mathrm{val}_u(w+w')\geq \min \{ \mathrm{val}_u(w),\mathrm{val}_u(w')\}$, with equality if $\mathrm{val}_u(w)\neq \mathrm{val}_u(w')$.
    \end{enumerate} 
    A {\bf filtration} of a $\ku$-module $W$ is a collection of $\ku$-submodules $(F_j)_{j\in \Z}$ satisfying $$F_j \supset F_{j+1} \supset u F_j \qquad \textrm{ and}\qquad  W=\cup F_j.$$ 
    The filtration associated to a valuation is $F_j(W):=\{w\in W: \mathrm{val}_u(w)\geq j\}.$
\\
\noindent For any $\ku$-module homomorphism $c: V \to W$, a valuation on $W$ {\bf induces} a valuation on $W$ via $V \to W \to \Z\cup \{\infty\}$, and any filtration $F_j(W)$ on $W$ {\bf induces} a filtration on $V$ via 
\begin{equation}\label{Equation induced filtration definition original}
F_j^c=F_j^c(V):= c^{-1}(F_j(W)).
\end{equation}
We say $c$ is {\bf filtration-preserving} if $c(F_j(V))\subset F_j(W)$, equivalently $F_j(V)\subset F_j^c$.
    \end{de}

\begin{de}\label{Definition canonical valuation}\label{Definition of the filtration from the valuation}\label{Cor uvaluation on localisation}
Let $W$ be a free $\ku$-module.
The \textbf{canonical valuation and filtration} are:\footnote{With some abuse of notation: the convention here is that $\mathrm{div}_u(0)=\infty$.}
\begin{equation}\label{Equation valuation via divisibility}
\begin{split}
\mathrm{div}_u(x) &:= \max \{ n\in \N \cup \{\infty\}: x= u^n y \textrm{ for some }y\in W\},
\\
F_j(W)
&:= \{w\in W: \mathrm{div}_u(w)\geq j\} = \{w\in W: w \textrm{ is }u^j\textrm{-divisible}\}. 
\end{split}
\end{equation}
Equivalently: let $W_n:=u^nW$, so $W=W_0 \supset W_1 \supset W_2 \supset \cdots$, then $\mathrm{div}_u(x)=n\Leftrightarrow x\in W_n\setminus W_{n+1}$.

The {\bf canonical valuation and filtration} on a finitely generated $\ku$-module $W$ are induced by the quotient map $W \to W/T$, for the $u$-torsion submodule $T:=\{w\in W: u^kw=0 \textrm{ for some }k\in \N\}$, using the canonical valuation and filtration on the free\footnote{In this step, we are using the assumption that $W$ is finitely generated. E.g.\;this fails for $\kuu$.} module $W/T$. Explicitly: 
$$\mathrm{val}_u(w)\geq j \Longleftrightarrow w=u^j x+ t \textrm{ for some }x\in W,t\in T,$$
and equality occurs for the ``maximal'' such $j$, allowing infinity:\footnote{It was inevitable that $T$ lies inside all $F_j(W)$, because ``the'' free part of $W$ is not canonical. E.g.\,the difference of two elements with valuation $j$ could yield a torsion element.}  $F_{\infty}(W):=\cap_j F_j(W) = T$.

\noindent The {\bf canonically induced valuation and filtration} on the $u$-localisation $W_u:=W\otimes_{\ku}\kuu$ are 
$$F_j(W_u)=\{x\in W_u: \mathrm{val}_u(x)\geq j\},\; \textrm{ where } \; \mathrm{val}_u(x):=-m+\mathrm{val}_u(y) \textrm{ for any lift }y\in W\textrm{ of }u^m x,$$ 
i.e.\;via localisation $W \to W_u$, $y \mapsto u^m x$ (such $y$ exist for sufficiently large $m\in \N$ depending on $x$).
\end{de}

\begin{ex}\label{Example obvious valuation}
For $\ku$, we have 
$\mathrm{div}_u(u^k q_k + \textrm{higher powers of }u)=k\in \N$ if $q_k\neq 0 \in \k$. For $\ku^r$, $\mathrm{div}_u$ is the minimum of the canonical valuations on all factors. On a free $\kuu$-module there is no \textit{canonical} valuation, since everything is $u$-divisible, but if $\kuu^r$ arises from $u$-localising $\ku^r$ (or $\ku^r \oplus \{u\textrm{-torsion}\}$) then the canonical valuation on $\ku^r$ induces a canonical valuation on $\kuu^r$ for which the $u$-localisation map $\ku^r\to \kuu^r$ is valuation-preserving: we minimise over $\kuu$-summands the obvious valuations $\mathrm{div}_u(u^k q_k + \textrm{higher powers of }u)=k\in \Z$ if $q_k\neq 0 \in \k$.
\end{ex}

\begin{lm}[Functoriality]\label{Lemma valuation on free mod}
    Any $\ku$-module homomorphism is filtration-preserving for the canonical filtrations. Any $\ku$-module isomorphism preserves the canonical valuation and filtration.
\end{lm}
\begin{proof}
If $c: V \to W$ is a homomorphism of $\ku$-modules, and $v = u^j x + t$ for $x\in V$ and a torsion element $t\in V$, then $c(v)=u^j c(x) + c(t)$ where $c(t)\in W$ is torsion, thus $\mathrm{val}_u(c(v))\geq \mathrm{val}_u(v)$, proving the first claim.
If $c$ is an isomorphism, then considering $c^{-1}$ yields an equality on valuations.
\end{proof}

\begin{rmk}\label{Lemma canonical valuation on fg module}
Let $W$ be a finitely generated $\ku$-module $W$, so \eqref{Equation torsion piece notation} defines $k_j\in \N$.
On the free $\ku$-module $W/T\cong \ku^r$ the valuation and filtration are those from \cref{Example obvious valuation}; and for $W$:
$$
\mathrm{val}_u(w) = -m + \max \{ n\in \N \cup \{\infty\}: u^m w= u^n y \textrm{ for some }y\in W\}
\textrm{ for any }m\geq \max k_j.
$$
\end{rmk}

\begin{lm}
For an arbitrary $\ku$-module homomorphism $c: V \to W$, the filtration on $V$ induced by the canonical filtration on $W$ satisfies
\begin{equation}\label{Equation induced filtration defn}
\begin{split}
F_j^c  := c^{-1}(F_j( W &))  = \{ v\in V: \mathrm{val}_u(c(v))\geq j\}
\\
& \stackrel{\textrm{for }j\geq 0}{=} \{v\in V: c(v) \textrm{ is }u^j\textrm{-divisible modulo torsion}\}
\\
& 
\stackrel{\textrm{for }j\geq 0}{=} \ker \left(V \stackrel{c}{\longrightarrow} W/(T+u^jW)\right).
\end{split}
\end{equation}
In particular, $F_{0}^c=F_{j\leq 0}^c=V$ and $F_{\infty}^c=c^{-1}(T)=\ker (c: V \to W/T)$.
\end{lm}
\begin{proof}
This is immediate from unpacking Definitions \ref{Definition induced filtration via a map c} and \ref{Definition canonical valuation}.
\end{proof}

\begin{ex}\label{Example Toy model 0 of invariant factors}
For the multiplication $c: \ku \stackrel{u^3}{\to} \ku$,
$F_j^c:=\ku$ for $j\leq 3$, and $F_j^c=u^{j-3}\ku$ for $j\geq 3.$ The power $3$ is detected a posteriori by the quotients $F^c_{j+1}/uF^c_j$, using the notation \eqref{Equation torsion piece notation},
$$
\ldots,\;\;\,
F^c_2/uF^c_1 = T_1, \;\;\,
F^c_3/uF^c_2 = T_1, \;\;\,
F^c_4/uF^c_3 = 0,\;\;\,
F^c_5/uF^c_4 = 0,\;\;\,
\ldots
$$
Note $d_j:=\dim_{\k} F_j^c/F_{j+1}^c$ satisfy $d_{j\leq 3}=1$ and $d_{j\geq 4}=0$.
\end{ex}

\begin{lm}\label{Example Toy model of invariant factors}
For any $\ku$-module homomorphism $c:\ku^s\to \ku^r$, the values
$d_j:=\dim_{\k} F^c_{j}/uF^c_{j-1}$ detect both the rank and the invariant factors of the $\ku$-module $\mathrm{coker}\,c$.
\end{lm}
\begin{proof}
We may\footnote{after changing bases on $\ku^s,\ku^r$, using automorphism-invariance of valuations/filtrations from \cref{Lemma valuation on free mod}.} assume $c$ is in Smith normal form: $c=\mathrm{diag}(u^{j_1},\ldots,u^{j_a},0,\ldots,0)$ for $j_i\in \N$. Let $f_1,\ldots,f_s$ be the standard basis for $\ku^s$. Then $\ku^s\cong \ku^a \oplus \ku^b$ where $c^{-1}(0)=\mathrm{span}_{\ku}(f_{a+1},\ldots,f_s) \cong \ku^b$, $b=s-a$.
So $F^c_j\supset c^{-1}(0)$ and, in the notation \eqref{Equation torsion piece notation},
\begin{equation}\label{Equation successive quotients of Fcj}
F^c_{j}/uF^c_{j-1}
=
\;\;\Bigg(\;\;\bigoplus_{ i\leq a \textrm{ with } j\leq j_i}\! T_1 f_i\;\;\Bigg) \; \oplus \; T_1^b,
\end{equation}
where $T_1f_i\cong T_1$. For $j\gg 0$, \eqref{Equation successive quotients of Fcj} has rank $b$, so it detects $a=s-b$ and $r-a=\mathrm{rank}_{\ku}\mathrm{coker}\,c$:
$$\mathrm{coker}\, c\cong T_{j_1}\oplus \cdots \oplus T_{j_a} \oplus \ku^{r-a}.$$
Finally, notice that the number of invariant factors $u^k$ of $\mathrm{coker}\,c$ arising in the list $u^{j_1},\ldots,u^{j_a}$ is precisely the drop in $\dim_{\k} F^c_{j}/uF^c_{j-1}$ when going from $j=k$ to $j=k+1$ (cf.\,\cref{Example Toy model 0 of invariant factors}).
\end{proof}

\subsection{Multiplicative properties}
\label{Subsection Multiplicative properties}

In this Subsection, we temporarily assume that our $\ku$-modules are $\ku$-algebras over the ring $\ku$, for some product $\bullet$. 

\begin{rmk}
    The motivation is that $E^- QH^*(Y)$ and $E^- SH^*(Y)$ admit a unital product structure if we make $S^1$ act only on the target manifold (not on the free loopspace).
\end{rmk}

A filtration is {\bf multiplicative} if $F_i(W)\bullet F_j(W)\subset F_{i+j}(W)$, and the valuation is multiplicative if $\mathrm{val}_u(x\bullet y)\geq \mathrm{val}_u(x)+\mathrm{val}_u(y)$. If the filtration is multiplicative and $F_0(W)=W$, then $W\bullet F_j(W)=F_0(W)\bullet F_j(W)\subset F_j(W)$, so the $F_j(W)\subset W$ are ideals. 
Note $T=F_{\infty}(W)\subset W$ is always an ideal.

\begin{lm}
    The canonical valuation and filtration on any $\ku$-algebra $W$ is multiplicative, and  
    yields ideals $F_j(W) \subset W$. 
\end{lm}
\begin{proof}
For $t,s\in T$, $(u^j x + t)\bullet (u^i y + s)\in u^{j+i}x\bullet y + T$ 
    as $T$ is an ideal. %
    So the canonical valuation is multiplicative, hence the filtration is too. As $F_0(W)=W$, the $F_j(W)$ are ideals.
\end{proof}

\begin{cor}
  Let $c: V \to W$ be any $\ku$-algebra homomorphism. For any multiplicative filtration on $W$, the filtration $F_j^c$ on $V$ is multiplicative. If in addition $F_0^c=V$, then the $F_j^c\subset V$ are ideals. All these assumptions hold if we use the canonical filtrations on $V,W$. 
\end{cor}
\begin{proof}
If $v\in F_j^c,w\in F_i^c$ then $v\bullet w \in F_{j+i}^c$ since $c(v\bullet w)=c(v)\bullet c(w)\in F_j(W)\bullet F_i(W)\subset F_{j+i}(W).$
\end{proof}

\subsection{Canonical valuation on cohomology}
\label{Subsection Canonical u valuation on cohomology}

Let $(C^*,d)$ be a cochain complex of free $\ku$-modules. 
Choose a $\ku$-module basis for $C^*$, and take its $\k$-linear span $D^*\subset C^*$. Then $C^*\cong D^*\otimes_{\k}\ku$ as a $\ku$-module (not as a complex).
The differential $d$ on $C^*$ is a $\ku$-module homomorphism, so
$$
d = \delta_0 \otimes \mathrm{id} + \delta_1 \otimes u\cdot\mathrm{id}  + \delta_2 \otimes u^2\cdot\mathrm{id} +\cdots 
$$
for some $\delta_j:D^* \to D^{*+1-j|u|}$, in particular $\delta_0: D^* \to D^{*+1}$ is a cochain differential. More concretely, the basis determines identifications $C^*\cong \ku^n$, $D^*\cong \k^n \subset \ku^n$, and we simplify notation to:
\begin{equation}
\label{Equation for the differential expanded}
d = \delta_0 + u \delta_1 + u^2 \delta_2 + \cdots : C^* \to C^{*+1},
\end{equation}
where $\delta_j:C^*\to C^{*+1-j|u|}$ is a $\ku$-linear homomorphism whose matrix does not involve $u$.

Assume that\footnote{It would suffice to work degree-wise, and assume finite rank degree-wise. Most arguments generalise even further.} $\mathrm{rk}\, C^*<\infty$. Let $\mathbf{chval}_u$ be the canonical valuation on $C^*$. 
 The homology $W:=H(C^*)$ is a finitely generated $\ku$-module. We endow $W$ with the canonical valuation and filtration.

\begin{lm}\label{Lemma ker respects valuation}
    $\ker d \subset C^*$ is a free $\ku$-submodule of rank $\leq \mathrm{rk}\,C^*$, such that the inclusion $\ker d\subset C^*$ respects the canonical valuations.\footnote{The second claim is not a general property of submodules, e.g.\,consider $u\ku\subset \ku$.}
\end{lm}
\begin{proof}
    The first part of the statement is a general fact about submodules of free modules when working over PIDs. Secondly, if $x\in \ker d$ were of the form $x=u^k y \in C^*$ for $k\geq 1$, then $0=d(x)=u^k d(y)$ forces $d(y)=0$ since $C^*$ is $u$-torsion-free, so $y\in \ker d$. The claim follows.
\end{proof}

\begin{cor}\label{Corollary uvaluation maximises chain rep}
For any $w\in W=H(C^*)$, the valuation maximises chain representative valuations
$$
\mathrm{val}_u(w)
=
\max \{\mathrm{chval}_u(x): w=[x]\in W\}.
$$
\end{cor}
\begin{proof}
We give a very explicit proof, as we will reuse these ideas later.
By the automorphism-invariance of the valuation in \cref{Lemma valuation on free mod}, we may consider the Smith normal form of the $m\times n$ matrix\footnote{
See \cref{Example of d is ucubed} for a simple illustration.
In general, we use \cref{Lemma ker respects valuation} to see that $m\leq n$ and that the canonical valuation on $\ker d\cong \ku^m$ is the one inherited from $C^*$.} 
\begin{equation}\label{Equation A matrix from differential}
A: \ku^n\cong C^* \stackrel{d}{\longrightarrow} \ker d \cong \ku^m.
\end{equation}
Say\footnote{The invariant factors $u^k$ are the natural choice of generators in the PID $\ku$, as $q_0+u^{\geq 1}$ is a unit in $\ku$ for $q_0\neq 0 \in k.$ Note 
the changes of basis on domain/target may be unrelated when building the Smith normal form.} $A=\mathrm{diag}(u^{k_1},\ldots,u^{k_{m-r}},0,\ldots,0): \ku^n \to \ku^m$, for some $k_j\geq 0$. 
Let $e_1,\ldots,e_n$ be the standard basis of $\ku^n$, and $\widetilde{e}_1,\ldots,\widetilde{e}_m$ the standard basis of $\ku^m$. Using \eqref{Equation torsion piece notation}, 
\begin{equation}\label{Equation W is T plus kur}
T\cong T_{k_1} \oplus \cdots \oplus T_{k_{m-r}} \qquad \textrm{ and } \qquad W\cong T \oplus \ku^r,
\end{equation}
where $\mathrm{span}_{\ku}(\widetilde{e}_{m-r+1},\ldots,\widetilde{e}_m)\subset W$ identifies the final $r$ coordinates of $\ku^m$ with $\ku^r$.
Note $\mathrm{chval}_u$ is multi-valued on homology classes $w$ because of boundaries,
$$\im\, d = \mathrm{span}_{\ku}(u^{k_1}\widetilde{e}_1,\ldots,u^{k_{m-r}}\widetilde{e}_{m-r}) \cong \ku^{m-r}.$$
Maximising $\mathrm{chval}_u$ over chain level representatives of a class $w\in W$ corresponds%
\footnote{
Example: $\mathrm{div}_u(u^i u^{k_1} \widetilde{e}_1 + u^5 \widetilde{e}_{m-r+1})=5$ for all $i\geq 5-k_1.$
} 
to ignoring that $\ku^{m-r}$ summand, and applying $\mathrm{div}_u$ to the $\ku^r$-summand.
The claim follows.
\end{proof}

\begin{lm}\label{Lemma ulocalising W- is Winfty}
 We can identify $W_u=W\otimes_{\ku}\kuu=H(C^*_u)$, and then canonically filter it using \cref{Cor uvaluation on localisation}.
\end{lm}
\begin{proof}
The first claim follows because $u$-localisation is a flat functor. Explicitly, we get the canonical valuation on $W_u\cong \kuu^r$ as described in \cref{Example obvious valuation}.
\end{proof}

\subsection{Filtration on cohomology induced by a chain map}
\label{Subsection induced filtration for V and W}
Suppose $c:B^*\to C^*$ is a ($\ku$-linear) chain map defined on a free $\ku$-module $B^*$ of finite rank, inducing $$c:V:=H(B^*)\to W:=H(C^*)$$ on homology.
Assume $V$ is free. So there is no $u$-torsion in $V$, but $W$ may have $u$-torsion: $T\subset W$.

\begin{rmk}
    The motivation comes from the homomorphism $c$ from (a suitable model of) equivariant quantum cohomology to equivariant Hamiltonian Floer cohomology for a given slope value $\lambda>0$.
\end{rmk}

\begin{thm}\label{Lemma Fj V minus description}
There are $f_1,\ldots,f_{a+b}\in F_0^c(V)$ giving summand-preserving isomorphisms
$$
V \cong (\ku f_1 \oplus \cdots \oplus \ku f_a) \oplus (\ku f_{a+1}\oplus \cdots \oplus  \ku f_{a+b}) \cong
\ku^a \oplus c^{-1}(T) \cong \ku^a \oplus \ku^b,
$$
and the submodule $F_j^c\subset V$ is the following, when viewed inside the above summand notation,
$$
F_j^c \cong (\ku u^{\phi(j-j_1)} f_1 \oplus \cdots \oplus \ku u^{\phi(j-j_a)} f_a) \oplus \ku^b,
$$
where $\phi(n)=\max\{n,0\}$ and $j_i\geq 0$.  %
Thus the $d_j:=\dim_{\k} F_{j}^c/uF_{j-1}^c$ define a sequence which decreases from $d_0=d_{j\leq 0} = a+b$ down to $d_{j>\max j_i} = b$, satisfying%
\footnote{The invariant factors $u^{j}$ of the homomorphism $c$ are those arising in the Smith normal form, so $j\geq 0$ as $u^0=1$ may arise, whilst the invariant factors of the f.g.\;$\ku$-module $\mathrm{coker} (c: V\to W)$ are precisely the $u^j$ with $j>0$.
}
\begin{align*}
d_j & = b+ \#\{i: j \leq j_i\} \leq b+a = \mathrm{rank}_{\ku} V,\\
d_j-d_{j+1} & = \#\{i: j = j_i\} = \#\{\textrm{invariant factors } u^{j} \textrm{ of }c: V \to W\},\; \textrm{and }{\textstyle\sum} (d_j-d_{j+1})=a.
\end{align*}
If the $u$-localisation $c_u:V_u \to W_u$ is an isomorphism, then $b=0$.
\end{thm}
\begin{proof}
Submodules of a free module over a PID of rank $k$ are free and of rank at most $k$. We apply that to the $\ku$-submodule $c^{-1}(T) = c^{-1}(F_{\infty}(W)) = \cap c^{-1}(F_j(W)) = \cap F_j^c(V)\subset V$. Then we reduce to \cref{Example Toy model of invariant factors} by considering the induced map $\widetilde{c}:\ku^a \cong V/c^{-1}(T) \hookrightarrow W/T\cong \ku^r$. The $j_i\geq 0$ in the claim arise from the invariant factors in the Smith normal form $\mathrm{diag}(u^{j_1},\ldots,u^{j_a})$ of $\widetilde{c}$.
The final claim about localisation is immediate as, compatibly with the above summand notation, 
\begin{equation}\label{Equation the u localised cu}
c_u: V_u \cong \kuu^a \oplus c_u^{-1}(0) \cong \kuu^a \oplus \kuu^b \to \kuu^r \oplus 0 \cong W_u. \qedhere
\end{equation}

\end{proof}

\subsection{Finitely $u$-filtered $\ku$-modules}
\label{Subsection Finitely u filtered ku modules}

Let $V$ be a $\ku$-module, and $T\subset V$ its torsion submodule.

\begin{de}\label{Definition finitely ufiltered}
Say that $V$ is {\bf finitely $u$-filtered}, if 
$V= F_0 \supset F_1 \supset F_2 \supset \cdots$
has a filtration by $\ku$-submodules $F_j$
satisfying:
\begin{enumerate}
    \item $P:=V/uV$ is finite dimensional as a $\k$-vector space.
    \item $uF_j = F_{j+1}\cap uV$ (a strengthening of the condition $uF_j \subset F_{j+1}$).
\end{enumerate}
We abbreviate by 
$
\mathrm{ev}_0: V \to P
$
the $\ku$-module homomorphism induced by quotienting $V$ by $uV$.
\end{de}
\begin{rmk}
    If $V$ is a free $\ku$-module of finite rank, then $V=\oplus \ku x_j$ after a choice of basis $x_j\in V$, and we can identify $P:=\oplus \k x_j$ and view $\mathrm{ev}_0$ as evaluation at $u=0$. A change of basis for $V$ induces a $\k$-linear change of basis on $P$.
\end{rmk}

\begin{lm}\label{Lemma finite invariants from Fj1 over u Fj}
 If $V$ is finitely $u$-filtered, then $\mathrm{ev}_0$ induces $\k$-linear embeddings $F_{j+1}/uF_j\hookrightarrow P$, yielding finite invariants $\dim_{\k} (F_{j+1}/uF_j)\leq \dim_{\k}P$.  
\end{lm}
\begin{proof}
The restriction $\mathrm{ev}_0:F_{j+1} \to P$ has kernel $uV \cap F_{j+1}=uF_j$ by the assumptions. %
\end{proof}

\begin{prop}\label{Prop Fjc is finitely u filtered} Let $c: V \to W$ be a homomorphism of $\ku$-modules, where $V$ is finitely generated. Endow $W$ with the canonical filtration. Then $V$ is finitely $u$-filtered by $F_j^c$.
\end{prop}
\begin{proof}
By assumption, there is a surjection $\ku^N \to V$ for some $N\in \N$. It induces a surjection $u\ku^N \to uV$, so 
$P:=V/uV$ has $\dim_{\k} P \leq N <\infty$. We now show
$u F_{j}^c=uV\cap F_{j+1}^c$. 
Note $uF_{j}^c\subset uV\cap F_{j+1}^c$ as $c(uF_{j}^c)=uc(F_j^c)$. Conversely, suppose $uv \in uV\cap F_{j+1}^c$. Then $u^{j+1}w=c(uv)=uc(v)$ for some $w\in W$. So  $t:=c(v)-u^j w\in W$ satisfies $ut=0$. Since $t$ is a torsion element, $t\in F_i(W)$ for all $i$. As $u^jw\in F_j(W)$, also $c(v)=u^jw + t\in F_j(W)$. Hence $v\in F_{j}^c$, which proves $uV\cap F_{j+1}^c\subset uF_{j}^c$.
\end{proof}

\begin{lm}
If $V$ is finitely $u$-filtered, and $T\subset F_j$ for all $j$, then $u^k F_{j-k} = F_{j}\cap u^k V$ for all\footnote{$j-k\leq 0$ is not problematic, put $F_{\leq 0}:=F_0$. The result also holds for $F_j$ labelled by $\Z$, without assuming $V=F_0$.} $j,k\in \N$. E.g.\;this applies to canonical filtrations and to $F_j^c$ in \cref{Prop Fjc is finitely u filtered}.
\end{lm}
\begin{proof}
Note $u^k F_{j-k} \subset F_{j}\cap u^k V$, so we check the reverse inclusion. We explain the case $k=2$, as the general proof is then clear. Using self-explanatory notation, say $u^2 v = f_j$. Then $u^2v \in uV \cap F_j = u F_{j-1}$, so $u^2v = u f_{j-1}.$ Thus $uv - f_{j-1} \in T$. By assumption, $T \subset F_{j-1}$. So redefine $f_{j-1}$ so that $uv=f_{j-1}$. Then $uv \in uV \cap F_{j-1}=uF_{j-2}$. So $uv = uf_{j-2}$. Thus $u^2v = u^2f_{j-2}\in u^2 F_{j-2}$ as required.
\end{proof}

\subsection{Associated graded for finitely $u$-filtered modules}

The associated graded determined by a filtration on $W$ is
    \begin{equation}\label{Equation associated graded}
    \mathrm{gr}(W):=T \oplus \bigoplus_{j\in \Z} F_j(W)/F_{j+1}(W).
    \end{equation}
If the filtration is multiplicative, in the sense of \cref{Subsection Multiplicative properties}, then this will be a graded ring, for the natural induced product (the $j$-summand times the $i$-summand lands in the $(j+i)$-summand; the torsion submodule $T$ is preserved). In general, it is just a $\ku$-module.

\begin{prop}\label{Prop about filtration on V mod uV}
Let $V$ be a finitely $u$-filtered $\ku$-module. Abbreviate $F_j:=F_j(V)$. Then there is a natural filtration on $P:=V/uV$ given by
$$ 
F_j(P):= F_{j}/uF_{j-1} = F_j/(uV\cap F_j), \textrm{ and } F_{\infty}(P):= T/uT,
$$
where $d_j:=\dim_{\k} F_j(P)<\infty$,
whose associated graded as a $\ku$-module is naturally
$$
\mathrm{gr}(P) \cong T/uT \oplus \bigoplus_{j\in \Z} F_j/(F_{j+1}+uF_{j-1}).
$$
If the filtration $F_j$ is multiplicative, then so is $F_j(P)$, and the above is a ring isomorphism.
\end{prop}
\begin{proof}
    $V$ is filtered as $\cdots\supset F_j \supset F_{j+1}  \supset \cdots$ (with $T$ inside all $F_j$).
    Thus $P$ is filtered as $\cdots \supset F_j/(uV \cap F_{j})\supset F_{j+1}/(uV \cap F_{j+1})  \supset \cdots$ (with $T/uT$ inside all of those).
    By definition of finitely $u$-filtered, $uV \cap F_{j+1} = uF_j$, so $F_{j+1}/(uV \cap F_{j+1}) = F_{j+1}/uF_{j}$.
    The summands of the associated graded of that filtration on $P$ are therefore $T/uT$ and the quotients
    $$
    (F_j/uF_{j-1})/(F_{j+1}/uF_j)\cong (F_j/uF_{j-1})/[(F_{j+1}+uF_{j-1})/uF_{j-1}]\cong  F_j/(F_{j+1}+uF_{j-1}),
    $$
    using the ``second'' and ``third'' isomorphism theorems respectively (this also holds for rings).
\end{proof}

\begin{de}\label{Definition slice series}
 The {\bf slice series} of a finitely $u$-filtered $\ku$-module $V$ is a formal series, 
$$
s_V(t):=\sum_{j\geq 0} \dim_{\k} (F_j(V)/uF_{j-1}(V))\; t^j = \sum_{j\geq 0} d_j t^j.
$$
When we use the $F_j^c$ filtration from \cref{Prop Fjc is finitely u filtered} for $c:V \to W$, we write $s_c:=s_V$.

For general filtrations on $V$, one may wish to keep track also of some or all of the terms $d_jt^j$ for $j<0$. However, for the canonical filtration on $V$, or for  $F_j^c(V)$ for a $\ku$-module homomorphism $c: V \to W$ and the canonical filtration on $W$, we have $d_j = d_0$ for $j\leq 0$.
\end{de}

\begin{ex}
In \cref{Example Toy model 0 of invariant factors}, $s_c(t)=1+t+t^2+t^3.$ 
See also \cref{Lemma slice dimensions give coeffs of filtration polynomial} for an example.
\end{ex}

\begin{thm}\label{Corollary filtration polynomial}
For $c: V \to W$ as before, using the notation from \cref{Lemma Fj V minus description},
   \begin{equation}\label{Equation filtration series defn}
   s_c(t) = (a+b)\cdot 1 + \sum_{1\leq j \leq \max j_i} (b+\#\{i:j\leq j_i\}) t^{j} \;\; + t^b+t^b+\cdots
   \end{equation}
   \begin{enumerate}
    \item The constant term $s_c(0)=\mathrm{rank}_{\ku} V = \dim_{\k} P,$ where $P:=V/uV.$
    \item The $t^j$-coefficient $d_j$ is the number of invariant factors $u^{\leq j}$ of $c$, so $d_j\leq d_{j+1}$.
    \item $b = \mathrm{rank}_{\ku}\ker c$.
\end{enumerate}
\end{thm}

\begin{de}
    We also associate to $c:V \to W$ a {\bf reduced slice polynomial}
\begin{equation}\label{Equation filtration poly defn}
\widetilde{s}_c(t):=\sum_{j\geq 0} \dim_{\k} (F_j(V)/(uF_{j-1}(V)+\ker c))\; t^j = \sum_{j\geq 0} (d_j-b) t^j =  a\cdot 1 + \sum_{1\leq j \leq \max j_i} \#\{i:j\leq j_i\} t^{j}.
\end{equation}
If $c$ is injective, then $s_c=\widetilde{s}_c$, and we call it {\bf slice polynomial}.
\end{de}

\subsection{Filtrations arising from $u$-local isomorphisms}
\label{Subsection Filtrations arising from $u$-local isomorphisms}

Let $V$ be a free $\ku$-module of finite rank.
Suppose $W=\varinjlim W_{\lambda}$ is a direct limit of $\ku$-modules, where $\lambda$ is labelled by an unbounded subset of $\R_{> 0}$, whose directed system is compatible with $\ku$-module homomorphisms $c_{\lambda}: V \to W_{\lambda}$. Let $c:=\varinjlim c_{\lambda}:V \to W$ be their direct limit.
We use the canonical filtrations on $V$ and $W_{\lambda}$.
Abbreviate:
\begin{align*}
F_j^{\lambda} &:= F_j^{c_{\lambda}}=c_{\lambda}^{-1}(F_j(W_{\lambda})), \hspace{5ex} P_j^{\lambda}:=F_{j}^{\lambda}/uF_{j-1}^{\lambda}, \hspace{5ex} d_j^{\lambda}: = \dim_{\k} P_j^{\lambda},
\\
F_{j}^{\infty}& := F_j^{c}=c^{-1}(F_j(W)), \hspace{6.2ex} P_j^{\infty}:=F_{j}^{\infty}/uF_{j-1}^{\infty}, \hspace{3.7ex} d_j^{\infty}: = \dim_{\k} P_j^{\infty}.
\end{align*}
Note that $F_{\infty}^{\lambda}=c_{\lambda}^{-1}(T(W_{\lambda}))$ and $F_{\infty}^{\infty}=c^{-1}(T(W))$, where $T(\cdots)$ means taking the torsion submodule.
\begin{rmk}
    The motivation is the canonical equivariant map $E^-c_{\lambda}^*: E^-QH^*(Y)\to E^{-}HF^*(H_{\lambda})$, where equivariant formality will ensure that $E^-QH^*(Y)$ is a free $\ku$-module of finite rank.
\end{rmk}

\begin{lm}\label{Lemma dj inequality from Ps}
    $F_j^{\lambda}\subset F_j^{\mu}\subset V$ for all $\lambda\leq \mu$, and those inclusions induce embeddings
    $$P_j^{\lambda}\hookrightarrow P_j^{\mu} \hookrightarrow P:=V/uV,$$
    so in particular $d_j^{\lambda} \leq d_j^{\mu} \leq \dim_{\k}P$.
\end{lm}
\begin{proof}
   By compatibility, the composite $V \to W_{\lambda} \to W_{\mu}$ equals $V \to W_{\mu}$.
    By \cref{Lemma valuation on free mod}, $F_j(W_{\lambda}) \to F_j(W_{\mu})$ via $W_{\lambda} \to W_{\mu}$.
    Thus, via that composite, $F_j^{\lambda} \to F_j(W_{\lambda}) \to F_j(W_{\mu})$, so $F_j^{\lambda} \subset F_j^{\mu}$. By \cref{Lemma finite invariants from Fj1 over u Fj}, we may embed $P_j^{\lambda}\hookrightarrow P$, so viewed inside $P$ we have $P_j^{\lambda}\subset P_j^{\mu} \subset P$ because $F_j^{\lambda}\subset F_j^{\mu}$.
\end{proof}

\subsection{The injectivity property, and the Young diagrams}
\label{Subsection The injectivity property, and the Young diagrams}
\label{Remark structure of the dj dimensions in applications}

\begin{prop}[Injectivity Property]\label{Proposition when iso after localise get finite invariants}
Suppose the $u$-localisation of $c$ is an isomorphism. Then 
 $c: V \to W$ and all $c_{\lambda}: V \to W_{\lambda}$ are injective.
\end{prop}
\begin{proof}
Recall $V_u=V\otimes_{\ku}\kuu$ denotes localisation at $u$.
Consider the commutative diagram
$$
\begin{tikzcd}
V 
\arrow[rr,""',"c", bend left=30]
 \arrow[r, "c_{\lambda}"] 
 \arrow[hook,d, "\mathrm{id}\otimes 1"']
 &
 W_{\lambda}
\arrow[r, "\varinjlim"] \arrow[d, "\mathrm{id}\otimes 1"']
& W \arrow[d, "\mathrm{id}\otimes 1"] 
\\
V_u \arrow[r,"c_{\lambda}\otimes \mathrm{id}"]
\arrow[rr,"\cong"',"c\otimes \mathrm{id}", bend right=30]
&
 (W_{\lambda})_u
\arrow[r, "\varinjlim\otimes \mathrm{id}"] 
& (W)_u
\end{tikzcd}
$$
where the left vertical arrow is injective as $V$ is a free $\ku$-module; and the bottom composite horizontal arrow is an isomorphism by the hypothesis.
So the top horizontal composite is injective.
\end{proof}

In applications, $c_{\lambda}$ will be an isomorphism for small $\lambda>0$; there is only a discrete sequence of slopes $0=\lambda_0<\lambda_1<\lambda_2<\cdots \to \infty$ for which $c_{\lambda},W_{\lambda}$ are not defined; $W_{\lambda}\to W_{\mu}$ in the directed system is an isomorphism if $\lambda,\mu \in (\lambda_k,\lambda_{k+1})$ for some $k\in \N$; and \cref{Proposition when iso after localise get finite invariants} applies. Abusively write $d_j^{\lambda_k}$ to mean $d_j^{\lambda}$ for $\lambda_k<\lambda<\lambda_{k+1}$.
Abbreviate $a=\mathrm{rank}_{\ku}V$.
We obtain partitions
\begin{equation*}
    \begin{split}
        D^{\lambda_1}&:= d_1^{\lambda_1} + d_2^{\lambda_1}+\cdots\\
        \cdots\\
        D^{\lambda_k}&:= d_1^{\lambda_k} + d_2^{\lambda_k}+\cdots
    \end{split}
\end{equation*}
for a non-decreasing sequence $0=D^{\lambda_0}\leq D^{\lambda_1}\leq D^{\lambda_2}\leq \cdots$.

As in \cref{Example Young diagram and dual},
the $D^{\lambda_k}$-partition yields a {\bf Young diagram} of shape $(d_1^{\lambda_k},d_2^{\lambda_k},\ldots)$, so row-lengths $d_1^{\lambda_k}\geq d_2^{\lambda_k} \geq \cdots \geq 0$ where we ignore the tail of zeros at the end.
Let's call $y_k$ the number of columns in the Young diagram for $D^{\lambda_k}$. Swapping rows/columns defines the {\bf dual Young diagram} (in French notation), with $y_k$ rows of lengths, say, $j_{a-y_k+1}^{\lambda_k}\leq \cdots\leq  j_{a-1}^{\lambda_k} \leq j_a$. These are the $j_i$ indices that determine the invariant factors $u^{j_i}$; the missing indices are $j_1^{\lambda_k}=\cdots =j_{a-y_k}^{\lambda_k}=0$. 

\cref{Lemma dimension calculation for sum of ds} will imply that the sum of the first $j$ row-lengths of the Young diagram is:
$$d_1^{\lambda_k} + d_2^{\lambda_k}+\cdots+ d_j^{\lambda_k} = \dim_{\k} F_j^{\lambda_k}/u^j V,$$ 
which stabilises for $j$ large enough, and \cref{Cor sum of coeffs of filtration poly} will describe the total size: 
$$D^{\lambda_k}=d_1^{\lambda_k} + d_2^{\lambda_k}+\cdots=j_1^{\lambda_k}+\cdots +j_a^{\lambda_k}.$$

By \cref{Lemma dj inequality from Ps}, $d_i^{\lambda_k}\leq d_i^{\lambda_{k+1}}$, therefore the 
Young diagrams are contained inside each other as we increase $k$. Therefore the dual Young diagrams are contained inside each other as well (in this case, due to the French notation, when $y_k$ increases, new rows can be inserted at the top); this also follows from the definition of the invariant factors $u^{j_i}$: the $j_i$ cannot decrease in the slope parameter, so $j_i^{\lambda_k}\leq j_i^{\lambda_{k+1}}$.
In each case, we get a directed path in Hasse's diagram of Young's lattice. We call these the {\bf dimension lattice for $V$} and the {\bf invariant factors lattice for $V$}.

\section{Algebra II: the models $V^-$, $V^{\infty}$, and $V^+$}
\label{Section Algebraic preliminaries II}

\subsection{The $W^-,W^{\infty},W^+$ models}\label{Subsection the Wminus plus and infinity models}
Continuing with $C^*$ as in \cref{Subsection Canonical u valuation on cohomology}, consider the $\ku$-modules
\begin{equation}\label{defn F C^+ and W^+}
    \begin{split}
W^-& :=W = H(C^*),\\
W^{\infty} &:= H(C^*_u) \cong H(C^*\otimes_{\ku} \kuu) \cong W^- \otimes_{\ku} \kuu \cong W_u^-, \\
W^+ &:=H(C^+) \qquad \textrm{ where }\quad C^+:=C^*_u/uC^*\cong C^*\otimes_{\ku} \mathbb{F} \quad \textrm{ and }\quad \mathbb{F}:=\kuu/u\ku.
    \end{split}
\end{equation}
\begin{ex}\label{Example of d is ucubed} For $d=\left(\begin{smallmatrix} 0 & 0 \\ u^3 & 0 \end{smallmatrix}\right)$ on $C^*=\ku\oplus \ku$, $\ker d = 0 \oplus \ku$ and $\mathrm{im}\,d = 0\oplus u^3\ku$.
Then \eqref{Equation A matrix from differential} is
the row matrix $A=\left( u^3 \,\; 0 \right):\ku\oplus \ku \to \ku$. We get $W^-=T\cong T_3$ and $W^{\infty}=0$. For $C^+$, we get $\ker d=\textrm{span}_{k}\{u^{-2},u^{-1},u^{0}\}\oplus \mathbb{F}$ and $\mathrm{im}\,d=0\oplus \mathbb{F}$, so $W^+=\textrm{span}_{k}\{u^{-2},u^{-1},u^{0}\} \cong T_3$.
\end{ex}

\subsection{The long and short exact sequences}
\label{Subsection the long and short exact sequences}

In what follows, we keep track of gradings as if they were $\Z$-gradings, placing $u$ in degree $|u|=2$, and $W^-[-2]$ shifts gradings so that $(W^-[-2])_m=W^-_{m-2}$ and $u\otimes 1:W^-[-2]\to W^{\infty}$ is grading-preserving.

\begin{lm}\label{Lemma LES for W modules}
    There is a long exact sequence
\begin{equation}\label{Equation LES on W modules}
\begin{tikzcd}
\cdots \arrow[r,""]
& W^-[-2] 
\arrow[r,"u\otimes 1"] %
&
 W^{\infty}
\arrow[r, ""] 
& W^+  \arrow[r, "\textrm{connecting}","\textrm{map}"'] 
&[2em] W^-[-1] \arrow[r,""]
& \cdots
\end{tikzcd}
\end{equation}
The connecting map $W^+\to W^-[-1]$ is formally $u^{-1}d$ where $d$ is the $C^*_u$-differential at chain level.

Explicitly, the LES \eqref{Equation LES on W modules} decomposes (non-canonically) into the exact sequence
\begin{equation}\label{Equation LES on W modules 2}
\begin{tikzcd}
0 \arrow[r,""]
& u\ku^r \arrow[r,"\subset"] %
&
 \kuu^r
\arrow[r, ""] 
& W^+ \cong T[-1] \oplus \mathbb{F}^r \arrow[r,"\textrm{project}"]
& T[-1] \arrow[r,""]
& 0
\end{tikzcd}
\end{equation}
where the outside terms are the summands in \eqref{Equation W is T plus kur}, and  $\mathbb{F}^r\cong \kuu^r/u\ku^r$ is $\mathrm{coker}\,(u\otimes 1)$.

In particular, if $W$ is free then \eqref{Equation LES on W modules} is a short exact sequence $0\to W^-[-2]\to W^{\infty} \to W^+\to 0$. 
\end{lm}
\begin{proof}
The short exact sequence $0 \to \ku \stackrel{u}{\to} \kuu \to \mathbb{F} \to 0$ tensored with the free $\ku$-module $C^*$ yields a short exact sequence of cochain complexes, which induces the LES in \eqref{Equation LES on W modules}. 

For $c\in C^+$, $c$ is a cycle if and only if $d(c)\in uC^{*-1}$, meaning it lies in the image of $u:C^*\to C^*_u$. The connecting map is by definition the $C^{*-1}$-factor of that, i.e.\,formally $u^{-1}d(c)$.

The second claim essentially follows from \cref{Lemma canonical valuation on fg module}. We first describe $T$ and the connecting map more explicitly. Let $x=q_0+uq_1+u^2q_2+\cdots$ with $q_j\in \k$. If $x \in C^*$ is a cycle whose homology class is $u^k$-torsion (for $k\geq 0$), then $u^k(q_0+uq_1+\cdots)=d(b_0+ub_1+\cdots)$ for some $b_j \in \k$.
That exactness is detected earlier by $W^{\infty}$ since $x = d(u^{-k}b_0+u^{1-k}b_1+\cdots)$ in $C^*_u$. Via the $u^{-1}d$ connecting map $W^+\to T[-1]\subset W^-[-1]$ we have
$u^{1-k}b_0+u^{2-k}b_1+\cdots\mapsto x$.
We can describe $W^+ \to W^-[-1]$ more explicitly. Recall the matrix expression of the chain differential in the proof of \cref{Corollary uvaluation maximises chain rep}: $d=\mathrm{diag}(u^{k_1},\ldots,u^{k_{m-r}},0,\ldots,0):\ku^n \to \ku^m$. 
By considering the induced map $d: \mathbb{F}^n \to \mathbb{F}^m$, we see that the last $r$ coordinates yield an $\mathbb{F}^r\cong \kuu^r/u\ku^r$ summand in $W^+$ as in the claim, and it remains to discuss the first $n-r$ coordinates. 
By the proof of \cref{Corollary uvaluation maximises chain rep}, 
$d(e_i)=u^{k_i}\widetilde{e}_i$ yields a $\ku/(u^{k_i})$-summand for $W^+$ generated by $u^{-k_i+1}e_i$. 
The connecting map $u^{-1}d$ maps $u^{-k_i+1}e_i$ to $\widetilde{e}_i\in T$, which is the generator of the $T_{k_i}$-summand in \eqref{Equation W is T plus kur} for $W^-=W$.
\end{proof}

\begin{lm}\label{Lemma WF and the SES for W}
   The $\mathbb{F}^r$-summand in \eqref{Equation LES on W modules 2} is canonically identifiable with the $\ku$-submodule
\begin{equation}\label{Equation definition of WplusF}
W^+_{\mathbb{F}}:=\bigcap_{n\in \N} u^n W^+ = \mathrm{im}(W^{\infty}\to W^+)\subset W^+,
\end{equation}
and \eqref{Equation LES on W modules} yields the short exact sequence $0 \to (W^-/T)[-2] \to W^{\infty} \to W^+_{\mathbb{F}}\to 0.$
\end{lm}
\begin{proof}
Note $u: \mathbb{F}\to \mathbb{F}$ is surjective, so $\mathbb{F}^r \subset T[-1]\oplus \mathbb{F}^r$ can be canonically recovered as the collection of elements that admit $u^n$-preimages for arbitrary $n\in \N$.
\end{proof}

\begin{rmk}\label{Remark about free part of W+}
Note that $W^+$ is entirely $u$-torsion, in particular there is no interesting valuation on $W^+$.
Unlike $W^+_{\mathbb{F}}$, the $T[-1]$-summand in $W^+$ has no intrinsic description: it involves choices.\footnote{One could consider a splitting of $0\to \ker d \to C^* \to \mathrm{im}\,d \to 0$ by lifting a basis from $\mathrm{im}\, d$ to produce a free submodule $D^*\subset C^*$ complementary to $\ker d$, and finally require $T[-1]$ to have its representatives in $D^*$. 
This would allow us to define a filtration on $W^+$ by declaring this choice of $T[-1]$ to lie inside $F_j(W^+)$ for all $j\in \Z$.} 
\end{rmk}

Let $W_{\lambda}=H^*(C^*_{\lambda})$, labelled by an unbounded subset of $\R_{>0}$, together with chain maps $c_{\mu,\lambda}:C^*_{\lambda} \to C^*_{\mu}$ for $\lambda<\mu$ so that the induced maps $W_{\lambda} \to W_{\mu}$ are compatible under composition. Abbreviate $W:=\varinjlim W_{\lambda}$.
Let $T,T_{\lambda}$ denote the torsion of $W,W_{\lambda}$ respectively. 
By the construction in \cref{Subsection the Wminus plus and infinity models}, using the obvious notation, we obtain $W=W^-=\varinjlim W^-_{\lambda}$, $W^{\infty}=\varinjlim W^{\infty}_{\lambda}$,  and $W^+=\varinjlim W^+_{\lambda}.$

\begin{cor}\label{Cor LES for ESH abstractly}
The torsion satisfies $T = \varinjlim T_{\lambda}$.
There is a long exact sequence 
\begin{equation}\label{Equation LES connecting map}
\begin{tikzcd}
\cdots \arrow[r,""]
& W^-[-2] 
\arrow[r,"u\otimes 1"] %
&
W^{\infty}
\arrow[r, ""] 
& W^+  \arrow[r, "\textrm{connecting}","\textrm{map}"'] 
&[2em] W^-[-1] \arrow[r,""]
& \cdots
\end{tikzcd}
\end{equation}
and a short exact sequence 
\begin{equation}
\begin{tikzcd}
0 \arrow[r,""]
& W^-[-2]/T 
\arrow[r,"u\otimes 1"] %
&
W^{\infty}
\arrow[r, ""] 
& \varinjlim W^+_{\lambda,\mathbb{F}}  \arrow[r,""]
& 0,
\end{tikzcd}
\end{equation}
where $W^+_{\lambda,\mathbb{F}}:=\cap u^{-n}(W^+_{\lambda})$, and
$\varinjlim W^+_{\lambda,\mathbb{F}} \subset W^+_{\mathbb{F}}=\cap u^{-n}(W^+)$, which is not necessarily an equality.
\end{cor}
\begin{proof}
    If $t\in T$, then $rt=0$ for some $r\neq 0 \in \ku$. Say $w_{\lambda}\in W_{\lambda}$ represents $t$ (some $\lambda$).
    Then $rw_{\lambda}$ represents $rt=0$, so $rw_{\mu}=0$ for some $\mu\geq \lambda$, where $w_{\mu}:=c_{\mu,\lambda}(w_{\lambda})=0$ via some map $c_{\mu,\lambda}:W_{\lambda}\to W_{\mu}$ of the directed system. Thus, $t$ is represented by $w_{\mu}\in T_{\mu}.$
    Conversely, if $w_{\mu}\in T_{\mu}$, then $rw_{\mu}=0$ for some $r\neq 0 \in \ku$, thus $rt=0$ for the class $t\in T$ represented by $w_{\mu}.$ 
Thus, $T=\varinjlim T_{\lambda}$.
    
The long exact sequence follows by taking the direct limit of \eqref{Equation LES on W modules}.

The short exact sequence follows from \eqref{Equation LES on W modules 2}, using that $T=\varinjlim T_{\lambda}.$
\end{proof}

\begin{rmk}
    The directed system $W_n=\ku/u^n$ with multiplication maps $u:W_n\to W_{n+1}$ will yield $W=\varinjlim W_n=\mathbb{F}$, with $W_{\mathbb{F}}=W$ even though $W_{n,\mathbb{F}}=0$, in particular $\varinjlim W_{n,\mathbb{F}}\subsetneq W_{\mathbb{F}}$.
\end{rmk}

\subsection{The $u$-adic filtration, and the $W^+$-spectral sequence}
\label{Subsection uadic filtration}
Recall $C^*$ is a free $\ku$-module with a differential $d=\delta_0 + u^{>0}$ of the form \eqref{Equation for the differential expanded}. 
The {\bf ordinary complex} is the quotient chain complex 
$$
\ord\,C:= C^*/uC^* \textrm{ with differential } \ord\,d:=[\delta_0] \textrm{ and cohomology }\ord\,W:=H(\ord\,C,\ord\,d).
$$
We call $\ord\,W$ the {\bf ordinary cohomology}, as opposed to $W:=H(C^*,d)$.

\begin{lm}\label{Lemma choice of basis to get ev0 complex}
 There is a natural long exact sequence
 $$
 \cdots \to \mathrm{ord}\,W[-1] \to W[-2] \stackrel{u}{\to} W \to \mathrm{ord}\,W\to \cdots
 $$
\end{lm}
\begin{proof}
 This uses $[d]=[\delta_0]:C^*/uC^* \to C^*/uC^*$, and that $0\to C^*[-2] \stackrel{u}{\to} C^* \to C^*/uC^* \to 0 $ is a short exact sequence (using that $C^*$ is free).
\end{proof}

We use the notation $D^*\subset C^*$ from \cref{Subsection Canonical u valuation on cohomology}, so  $C^*=D^*\otimes_{\k} \ku$ as $\ku$-modules (not as chain complexes).
The chain complex $(D^*,\delta_0)$ yields a canonical chain isomorphism $D^* \to \mathrm{ord}\,C^*$ (induced by the quotient map $C^*\to C^*/uC^*=\mathrm{ord}\,C^*$). So $H^*(D^*,\delta_0)\cong \mathrm{ord}\,W$. 
Two such choices $D_1^*,D_2^*$ of $D^*$ are related by a $\k$-linear chain isomorphism $\psi_0:  D_1^* \to D_2^*$ compatible with the canonical isomorphisms to $\mathrm{ord}\,C^*$, where $\psi_0$ is the $u^0$-part of the relevant change of basis isomorphism $\psi: C^* \to C^*$, $\psi=\psi_0 + u \psi_1 + u^2\psi_2 + \cdots$, where $\psi_j$ are $\k$-linear maps that preserve the order of $u$-factors.

\begin{lm}\label{Lemma vanishing of the E+ version}
    The $u$-adic filtration induces a convergent spectral sequence of $\k$-vector spaces
    \begin{equation}
\label{Lemma spectral sequence SH to ESH} 
E_1^{pq} = \left\{ \begin{smallmatrix} \displaystyle H^{q-p}(D^*,\delta_0)\otimes u^p, & \displaystyle  \textrm{ if }p\leq 0\\ \displaystyle  0, & \displaystyle  \textrm{ if } p>0 \end{smallmatrix}\right.\Longrightarrow W^+:=H^*(C^+,d).
\end{equation}
 In particular, $\mathrm{ord}\,W=0$ implies $W^+=0$. (The analogous statement often fails for $W^-$ and $W^{\infty}$)
\end{lm}
\begin{proof}
By the same observation as in \cite[Sec.(8b)]{Sei08}, the $u$-adic filtration on $C^+$ (which as a $\k$-vector space is identifiable with $\oplus_{i\leq 0} D^* u^{i}$) is bounded below and exhausting, so the associated spectral sequence on cohomology converges to the cohomology of $C^+$.
\end{proof}

\subsection{Specialisation to ordinary cohomology}\label{Subsection Specialisation to ordinary cohomology}
Recall $C^-:=C^*$, $C^+:=C^*_u/uC^*_u$.
We have two chain maps, a quotient map and a subcomplex inclusion:
$$
\mathrm{ev}_0:C^- \to \ord\,C \qquad \textrm{ and } \qquad\mathrm{inc}:\ord\,C\to C^+,
$$
induced respectively by the 
quotient $\ku \to \ku/u\ku\cong \k$,
and the inclusion of the $u^0$-part $\k \cong \ku/u\ku \hookrightarrow \kuu/u\ku=\mathbb{F}$. Via \cref{Lemma choice of basis to get ev0 complex}, and identifying $C^*/uC^*\cong D^* \cong \oplus \k\subset \oplus \ku \cong C^*$, the map $\mathrm{ev}_0$ can be thought of as an evaluation at $u=0$ map $\oplus \ku \to \oplus \k.$

\begin{lm}\label{Lemma ev0 commutative diagram general case}
Let $c:B^*\to C^*$ be as in \cref{Subsection induced filtration for V and W} (in particular, $V:=H(B^*)$ is free). Then there are long exact sequences commuting with the induced $c$-maps, as follows
$$
\begin{tikzcd}[column sep=0.3in,row sep=0.2in]
0
\arrow[r, ""] 
 \arrow[d, ""]
& 
V^-[-2]
 \arrow[r, "u"] 
 \arrow[d, "c"]
 &
 V^-
\arrow[r, "\mathrm{ev}_0"] 
\arrow[d, "c"]
&  
\ord\,V 
\arrow[d, "\ord\,c"] 
\arrow[r, ""]
&
0
\arrow[d, ""] 
\\
\ord\,W[-1] 
\arrow[r, ""] 
&
W^-[-2]
 \arrow[r, "u"] 
 &
 W^-
  \arrow[r, "\mathrm{ev}_0"]
 &  
 \ord\,W 
 \arrow[r, ""] 
 &
 W^-[-1]
\end{tikzcd}
$$
and
$$
\begin{tikzcd}[column sep=0.3in,row sep=0.2in]
0
\arrow[r, ""]
\arrow[d, ""]
&
\ord\,V[-2]
\arrow[r, "\mathrm{inc}"]
\arrow[d, "\ord\,{c}"] 
&
V^+[-2]
\arrow[r, "u"]
\arrow[d, "{[c_u]}"] 
&
V^+
\arrow[r, ""]
\arrow[d, "{[c_u]}"] 
&
0
\arrow[d, ""]
\\
W^+[-1]
\arrow[r, ""]
&
 \ord\,W[-2]
 \arrow[r, "\mathrm{inc}"]
&
W^+[-2]
\arrow[r, "u"]
&
W^+
\arrow[r, ""]
&
 \ord\,W[-1] 
\end{tikzcd}
$$
\end{lm}
\begin{proof}
This follows from the short exact sequences
$0 \to\ku \stackrel{u}{\to} \ku \stackrel{\mathrm{ev_0}}{\to} \ku/u\ku\to 0$ and 
$0 \to \ku/u\ku \stackrel{inc}{\to} \mathbb{F} \stackrel{u}{\to} \mathbb{F} \to 0$.
Surjectivity of the top $\mathrm{ev}_0$ map in the claim (which causes the top LES to split) is due to the LES in \cref{Lemma choice of basis to get ev0 complex}, using that $u:V\to V$ is injective since $V$ is free.\end{proof}

\begin{cor}\label{Cor ordW zero iff Wplus zero}
    $\mathrm{ord}\,W = 0 \Leftrightarrow W^+ = 0$ (so also $\mathrm{ord}\,W \neq 0 \Leftrightarrow W^+\neq 0$). 
\end{cor}
\begin{proof}
If $\mathrm{ord}\,W=0$, 
\cref{Lemma vanishing of the E+ version} implies $W^+=0$.
If $W^+=0$, \cref{Lemma ev0 commutative diagram general case} implies $\mathrm{ord}\,W= 0$.
\end{proof}
\section{Algebra III: filtrations on $V^-$, $V^{\infty}$, and $V^+$}
\label{Section Algebraic preliminaries III}

\subsection{ %
Filtration of the exact sequence}
\label{Subsection filtrations on the W modules}

Filter $W^-,W^{\infty}$ by the canonical filtrations (\cref{Definition canonical valuation}). 

\begin{de}
The {\bf $\mathbb{F}$-filtration} on $W^+_{\mathbb{F}}$ is the filtration
  \begin{equation}\label{Lemma definition of the filtration from the uvaluation}
  F_{j}(W^+_{\mathbb{F}})
:=  \ker (u^{|j|+1}: W^+_{\mathbb{F}}\to W^+_{\mathbb{F}}) \;\textrm{ for }\,j\leq 0, \quad \textrm{ and } \quad F_{j}(W^+_{\mathbb{F}})=\{0\} \;\textrm{ for }\,j\geq 1.
    \end{equation}
    \end{de}
\begin{lm}
    $uF_j(W^+_{\mathbb{F}}) = F_{j+1}(W^+_{\mathbb{F}})  \subset F_j(W^+_{\mathbb{F}})$.    
\end{lm}
\begin{proof}
The second inclusion is obvious.
The inclusion $uF_j(W^+_{\mathbb{F}}) \subset F_{j+1}(W^+_{\mathbb{F}})$ is immediate.
The first equality holds for $j=0$: $F_0(W^+_{\mathbb{F}})=\ker u$, so $uF_0(W^+_{\mathbb{F}})=0=F_1(W^+_{\mathbb{F}})$.
Now suppose $j\leq -1$ and $w\in F_{j+1}(W^+_{\mathbb{F}})$, so $u^{|j|}w=0$. As $u: W^+_{\mathbb{F}}\to W^+_{\mathbb{F}}$ is surjective, we may write $w=uw'$. Thus $u^{|j|+1}w'=0$, so $w'\in F_{j}(W^+_{\mathbb{F}})$, so $w=uw'\in uF_{j}(W^+_{\mathbb{F}})$, as required.
\end{proof}

\begin{lm}[Functoriality]\label{Lemma functoriality Ffiltration}
    Any $\ku$-module homomorphism $\psi:W_1^+\to W_2^+$ preserves $\mathbb{F}$-filtrations: $\psi(F_{-j}(W_{1,\mathbb{F}}^+))\subset F_{-j}(W_{2,\mathbb{F}}^+).$ This is an equality when $\psi_{\mathbb{F}}:\psi:W_{1,\mathbb{F}}^+\to W_{2,\mathbb{F}}^+$ is an isomorphism.
\end{lm}
\begin{proof}
This is the analogue of \cref{Lemma valuation on free mod}, and it is straightforward.
\end{proof}

To simplify the notation, we will from now on omit the grading shifts from the previous section.

\begin{lm}\label{Lemma exact sequence on filtered pieces of W modules}
    \cref{Equation LES on W modules} induces 
    a filtered short exact sequence: for all $j\in \Z$,
\begin{equation}\label{Equation LES on W modules filtered}
\begin{tikzcd}
0 \arrow[r,""]
& F_{j-1}(W^-)/T 
\arrow[r,"u\otimes 1"] %
&
 F_{j}(W^{\infty})
\arrow[r, ""] 
& F_{j}(W^+_{\mathbb{F}}) 
\arrow[r,""]
& 0.
\end{tikzcd}
\end{equation}
After making explicit choices of bases as in \cref{Lemma LES for W modules}, this becomes explicitly:
\begin{align}
\label{Equation explicit SES case j positive}
\textrm{For }j\geq 1: \qquad
&
0 
\longrightarrow 
  u^{j-1}\ku^r 
\stackrel{u}{\longrightarrow}
u^{j}\ku^r
\longrightarrow 
0
\longrightarrow 
0.
\\
\label{Equation explicit SES case j negative}
\textrm{For }-j\leq 0: \qquad
&
0 
\longrightarrow 
\ku^r
\stackrel{u}{\longrightarrow} 
 u^{-j}\ku^r
\longrightarrow 
u^{-j}\ku^r/u\ku^r
\longrightarrow 
 0.
\end{align}

\end{lm}
\begin{proof}
These claims follow immediately from the proof of \cref{Lemma LES for W modules}.
\end{proof}

\subsection{Exact sequence for induced filtrations}
\label{Subsection Exact sequence for induced filtrations}

Continue with the notation from \cref{Subsection induced filtration for V and W}
and \cref{Subsection the Wminus plus and infinity models}. So for some $s\in \N$, we have%
\; $V^{-}=V\cong \ku^s$,\; $V^{\infty}=V_u\cong \kuu^s$, and $V^+\cong \mathbb{F}^s.$

\begin{cor}\label{Cor SES over LES}
We have 
commutative diagrams of $\ku$-modules, with exact horizontal rows,
\begin{equation*}
\begin{tikzcd}
0 \arrow[r, ""] & V^-
 \arrow[r, "u\otimes \mathrm{id}"] 
 \arrow[d, "c"]
 &
 V^{\infty}
\arrow[r, ""] \arrow[d, "c_u:=c\otimes \mathrm{id}"]
& V^+ \arrow[d, "{[c_u]}"] 
\arrow[r, ""]
& 0
&
0 \arrow[r, ""] & V^-
 \arrow[r, "u\otimes \mathrm{id}"] 
 \arrow[d, "c"]
 &
 V^{\infty}
\arrow[r, ""] \arrow[d, "c_u:=c\otimes \mathrm{id}"]
& V^+ \arrow[d, "{[c_u]}"] 
\arrow[r, ""]
& 0
\\
& W^- \arrow[r,"u\otimes \mathrm{id}"]
&
 W^{\infty}
\arrow[r, ""] 
& W^+ &
&
0 \arrow[r, ""] & W^-/T \arrow[r,"u\otimes \mathrm{id}"]
&
 W^{\infty}
\arrow[r, ""] 
& W^+_{\mathbb{F}} \arrow[r, ""] & 0
\end{tikzcd}
\end{equation*}
\end{cor}
\begin{proof}
This follows by \cref{Lemma LES for W modules} 
and \cref{Lemma WF and the SES for W}.
\end{proof}

\begin{rmk}
    The motivation comes from the three models of constructing $S^1$-equivariant cohomology. The top row will be the equivariant quantum cohomologies, the lower row is part of the long exact sequence for the three models of equivariant Hamiltonian Floer cohomology for some slope $\lambda>0$.
\end{rmk}

\begin{lm}\label{Lemma Vplus preimage of T in Wplus is not interesting}
 $[c_u]: V^+ \to W^+$ lands in $W^+_{\mathbb{F}}$. Explicitly, $W^+\cong T[-1]\oplus \mathbb{F}^r$, $\mathrm{image}\,[c_u]\subset 0\oplus \mathbb{F}^r$, and $\ker [c_u]$ contains a canonical submodule $K_c^+:=(\ker c)_u/u\ker c\cong \mathbb{F}^b$ \emph{(}and $b=\mathrm{rank}_{\ku}\ker c$, see \eqref{Equation the u localised cu}\emph{)}. 
\end{lm}
\begin{proof}
    Since $V^+\cong \mathbb{F}^s$, the elements of $V^+$ have arbitrary $u$-preimages: for $v\in V^+$, $v\in u^n V^+$ for all $n\in \N.$ Thus $c(v)\in u^n W^+$ for all $n\in \N$. 
    The rest follows easily, noting $\ker c\cong \ku^b$.
\end{proof}

\begin{de}\label{Definition induced filtration via c maps}
The filtrations 
in \cref{Definition of the filtration from the valuation} and \eqref{Lemma definition of the filtration from the uvaluation} induce filtrations on $V^-,V^{\infty}$, as in \eqref{Equation induced filtration defn}:
$$
F_j(V^-) :=c^{-1}(F_j(W^-)) \qquad \textrm{ and } \qquad
F_j(V^{\infty})  :=c_u^{-1}(F_j(W^{\infty})). 
$$
For $V^+$ we define $F_{j}(V^+):=[c_u]^{-1}(F_{j}(W^+_{\mathbb{F}}))$ for $j\leq 1$. For $j\geq 1$ this would constantly give $F_1(V^{+})=\ker [c_u]$, which would ruin later results when $j\geq 2$. The definition, adjusted for $j\geq 2$, is:
$$
F_{j\leq 0}(V^+):=[c_u]^{-1}(F_{j}(W^+_{\mathbb{F}}))= \ker (u^{|j|+1}[c_u])
\qquad \textrm{ and } \qquad
F_{j\geq 0}(V^+) := 
u^{j-1}\ker [c_u].
$$
\end{de}

\begin{ex}\label{Example cube map showing induced filtrations}
Suppose $c:=u^3:V^-=\ku \to \ku = W^-$,  $V^{\infty}=W^{\infty}=\kuu$, $V^+=W^+=\mathbb{F}$. 
$$
\small
\renewcommand{\arraystretch}{1.1}
\begin{array}{|c|c|c|c|c|}
\hline
& 
j\leq 0 
&
1\leq j \leq 3 & 
j > 3\\
\hline
F_{j-1}(V^-)
& 
\ku
&
\ku
&
\;u^{j-1-3}\ku\;
\\
\hline
F_{j}(V^{\infty})
&
u^{j-3}\ku
&
u^{j-3}\ku
&
u^{j-3}\ku
\\
\hline
F_{j}(V^{+})
&
\;u^{j-3}\ku/u\ku\;
&
\;u^{j-3}\ku/u\ku\;
&
0
\\
\hline
\end{array}
\;
\begin{array}{|c|c|c|c|c}
\hline
& 
j\leq 0 
&
1\leq j \leq 3 & 
j > 3\\
\hline
F_{j-1}(W^-)
&
\ku
&
u^{j-1}\ku
&
u^{j-1}\ku
\\
\hline
F_{j}(W^{\infty})
&
u^{j}\ku
&
u^{j}\ku
&
u^{j}\ku
\\
\hline
F_{j}(W^{+})
&
u^{j}\ku/u\ku
&
0
&
0
\\
\hline
\end{array}
\renewcommand{\arraystretch}{0.8}
$$
\end{ex}

\begin{lm}\label{Lemma Vplus description}
$u F_{j}(V^+) = F_{j+1}(V^+)$ 
for all $j\in \Z$ (in particular, $F_{j}(V^+)=u^j F_0(V^+)$ for $j\geq 0$).
\end{lm}
\begin{proof}
The inclusion $\subset$ is obvious (and a requirement of filtrations), and equality is obvious when $j\geq 0$. For $j\leq 0$, say $v\in F_{j+1}(V^+)$, so $u^{|j|}[c_u]v=0$.
The $u$-action on $\mathbb{F}$ is surjective and so the same holds for $V^+$, thus we may write $v=u\widetilde{v}$. Then $u^{|j|+1}[c_u]\widetilde{v}=0$, so $\widetilde{v}\in F_{j}(V^+)$ and $v=u\widetilde{v} \in uF_{j}(V^+)$.
\end{proof}

\begin{rmk}
The $u$-action on $V^+$ is surjective as $V^+\cong \mathbb{F}^s$, but there is no $u^{-1}$-action.
However, once an explicit presentation $V\cong \ku^s$ is chosen, we can identify $V^+ \cong \mathbb{F}^s$ with the vector space $\oplus_{i\leq 0} u^{i} \k$, and then write $u^{-j}\bullet v$ to mean that we shift by $-j\leq 0$ all powers of $u$ in $v\in \oplus_{i\leq 0} u^{i} \k$.
The ambiguity caused by the choice of presentation lies in $\mathbb{F}_{j-1}:= \oplus_{-j< i\leq 0} u^{i} \k$. More invariantly, $\mathbb{F}_{j-1}=\ker u^j$, where $u^j$ acts by multiplication on $V^+$.
Then \cref{Lemma Vplus description} is more explicitly the statement:
$$
F_{-j}^{\lambda}(V^+) = (u^{-j}\bullet \ker [(c_{\lambda})_u]) + \mathbb{F}_{j-1} \qquad (\textrm{for }-j\leq 0).
$$
\end{rmk}

\begin{lm}\label{Lemma SES on filtered V modules}
For each $j\in \Z$, \cref{Cor SES over LES} gives rise to the following exact sequence
\begin{equation}\label{Equation exact seq for Fj of V modules}
0\longrightarrow F_{j-1}(V^-) \stackrel{u}{\longrightarrow} F_j(V^{\infty}) \longrightarrow F_j(V^{+}) \longrightarrow 0.
\end{equation}
\end{lm}
\begin{proof}
The diagram chase can be done abstractly. More concretely: by \cref{Example Toy model of invariant factors} we may assume $c$ is in Smith normal form. This reduces the problem to noticing the claim holds in \cref{Example cube map showing induced filtrations} with $3$ replaced by a general $k\geq 0$ (and for $c=0:\ku \to \ku$ the claim becomes the top row of \cref{Cor SES over LES}).
\end{proof}

\begin{rmk}\label{Remark description of Vplus}
The interesting summands labelled by $f_{i\leq a}$ in \cref{Lemma Fj V minus description} map in \eqref{Equation exact seq for Fj of V modules} by:
$$0 \longrightarrow u^{\phi(j-1-j_i)}\ku f_i \stackrel{u}{\longrightarrow} 
 u^{j-j_i}\ku f_i \longrightarrow  (u^{j-j_i}\ku/u^{1+\phi(j-1-j_i)}\ku) [f_i] \longrightarrow 0.$$
The third module is zero for $j>j_i$, and for $j\leq j_i$ it is $u^{j-j_i}\ku/u\ku$ (so $\mathrm{span}_{\k}\{u^{j-j_i},\ldots,u^0\}$ as a $\k$-vector space).  
 The non-interesting summands $f_{i>a}$ (spanning $\ker c$) map in \eqref{Equation exact seq for Fj of V modules} by \eqref{Equation explicit SES case j positive} and \eqref{Equation explicit SES case j negative}.

\end{rmk}

\subsection{Relation between the filtrations for $V^-$ and $V^+$}
\label{Subsection Relation between the filtrations Vminus and Vplus}
\begin{de}
For $v\in V^-$, the inclusion $V^-\subset V^{\infty}$ determines
$u^{-j}v\in V^{\infty}$ (so $v\otimes u^{-j}$ in $V^-\otimes_{\ku} \kuu = V^{\infty}$). Denote $[u^{-j}v]\in V^+$ its image via $V^{\infty} \to V^+$. 
Note that $V^+=\cup_{j\geq 0} [u^{-j}V^-]$.
\end{de}

Observe that the $F_j(V^-)$ filtration is induced by the following {\bf valuation}
    $$
    \mathrm{val}(v) := \max \{ j\in \N\cup \{\infty\} : c(v)\in  W^-/T \textrm{ is }u^j\textrm{-divisible}\}.
    $$
\begin{thm}\label{Lemma relation between filtrations Vminus and Vplus}
For $v\in V^-$ and $j\geq 0$,
\begin{equation}\label{Equation lambda valuation theorem}
\begin{split}
u^{-j}v \in F_1(V^+) \Longleftrightarrow
[c_u][u^{-j}v]=0 \in W^{+} \Longleftrightarrow (v \in F_{j+1}^{\lambda}(V^-)) \Longleftrightarrow (\mathrm{val}^{\lambda}(v)\geq j+1),
\\
\mathrm{val}^{\lambda}(v)=j \Longleftrightarrow \left(\;[c_u][u^{-j}v]\neq 0  \;\textrm{ in }W^+,\;\textrm{ and } \;\; u\cdot [c_u][u^{-j}v]=0\; \textrm{ in }W^+\;\right).
\end{split}
\end{equation}
\end{thm}
\begin{proof}
 We first prove the $\Rightarrow$ direction.  By \cref{Lemma exact sequence on filtered pieces of W modules}, we have a commutative diagram of exact rows,
$$
\begin{tikzcd}[column sep=0.3in,row sep=0.2in]
0 \arrow[r, ""] & V^-
 \arrow[r, "u\otimes \mathrm{id}"] 
 \arrow[d, "c"]
 &
 V^{\infty}
\arrow[r, ""] \arrow[d, "c_u"]
& V^+ \arrow[d, "{[c_u]}"] 
\arrow[r, ""]
& 0
&
&
\strut
&
u^{-j}v
\arrow[r, "",mapsto]
\arrow[d, "",mapsto] 
&
{[u^{-j}v]}
\arrow[r, "",mapsto]
\arrow[d, "",mapsto] 
&
0
\\
0 \arrow[r, ""] & W^-/T \arrow[r,"u\otimes \mathrm{id}"]
&
 W^{\infty}
\arrow[r, ""] 
& W^+_{\mathbb{F}} 
\arrow[r, ""]
& 0
&
0
\arrow[r, "",mapsto]
&
\exists! w^-
\arrow[r, "u\otimes 1",mapsto]
&
w
\arrow[r, "",mapsto]
&
0
\arrow[r, "",mapsto]
&
0
\end{tikzcd}
$$
Note that $v\in V^-$ maps to $uv = u^{j+1}(u^{-j}v)$ via $V^- \to V^{\infty}$. So, from the diagram chase, it follows that $c(v)=u^{j+1} w^-\in F_{j+1}(W^-/T)$. Therefore $v\in F_{j+1}^{\lambda}(V^-):=c^{-1}(F_{j+1}(W^-)).$

 We now prove the $\Leftarrow$ direction. If $v\in F_{j+1}^{\lambda}$ then $c(v)=u^{j+1}w^- \in W^-/T$ for some $w^-\in W^-.$ 
 Let $w\in W^{\infty}$ denote the image $(u\otimes 1)(w^-)$ of $w^-$ via $W^- \to W^{\infty}$.
 By exactness, $w\mapsto 0$ via $W^{\infty} \to W^+_{\mathbb{F}}$.
 The same diagram chase on the right applies: commutativity shows $u c_u(v)=c_u((u\otimes 1)v)=u^{j+1}w \in W^{\infty}$ and therefore $u^{-j}v \mapsto w$ in the middle arrow, thus $[u^{-j} v] \in V^+$ maps to $0$ via $[c_u]$.
\end{proof}

\subsection{Relation between the filtrations for $V^-,V^{\infty},V^+$}
\label{Subsection Relation between the filtrations}
We continue to assume $V\cong \ku^s$ is free, and we now work with directed modules $W_{\lambda}$ as in \cref{Subsection the long and short exact sequences} together with maps $c_{\lambda}:V \to W_{\lambda}$ satisfying {\bf compatibility}:  $c_{\mu} = c_{\mu,\lambda} \circ c_{\lambda}:  V \to W_{\lambda} \to W_{\mu}$,
for all $\lambda<\mu$. The labelling $\lambda$ belongs to an unbounded subset of $\R_{>0}$; we call these values the {\bf generic slopes}.

\begin{de}
Define $F_j^{\lambda}(V^-),F_j^{\lambda}(V^{\infty}),F_j^{\lambda}(V^+)$ via \cref{Definition induced filtration via c maps}
for the map $c_{\lambda}$. We extend this definition to all slope values $p\in \R_{\geq 0}$ by intersecting (where it is always understood that $\lambda$ is generic):
\begin{equation}\label{Definition of Fjlambda}
\FF_j^{p}(V^{-})  := \bigcap_{\lambda \geq p} F_j^{\lambda}(V^-),
\qquad
\FF_j^{p}(V^{\infty})  := \bigcap_{\lambda \geq p} F_j^{\lambda}(V^{\infty}),
\qquad
\textrm{ and }
\qquad
\EE_{j}^p(V^{+})  := \bigcap_{\lambda \geq p} F_j^{\lambda}(V^+). 
\end{equation}
For $p$ generic, $\FF_j^{p}(V^{-}) = F_j^{p}(V^{-})$, etc.
In applications, there is a non-empty interval $(\lambda,\mu)$ of generic slopes $\lambda'$ for which $c_{\lambda',\lambda}$ is an isomorphism. Write $\lambda^+$ for any such $\lambda'$. Thus, $\FF_j^{p}(V^{-})  = F_j^{p^+}(V^-)$, etc. By compatibility, increasing $p$ yields inclusions: $\FF_j^{p}(V^{-})\subset \FF_j^{p}(V^{-})$ for $p\leq p'$, etc.

Define $F_j^{\infty}(V^-),F_j^{\infty}(V^{\infty}),F_j^{\infty}(V^+)$ via \cref{Definition induced filtration via c maps}
for the map $c=\varinjlim c_{\lambda}$; equivalently:
\begin{equation}\label{Definition of Fjinfty}
\FF_j^{\infty}(V^{-})  = \bigcup_{p} \FF_j^{p}(V^-),
\qquad
\FF_j^{\infty}(V^{\infty}) =  \bigcup_{p}  \FF_j^{p}(V^{\infty}),
\qquad
\textrm{ and }
\qquad
\EE_{j}^{\infty}(V^{+})  = \bigcup_{p} \EE_{j}^{p}(V^+). 
\end{equation}
\end{de}
In particular, $\EE_{1}^{\infty}=\ker [(c)_u]=\cup \EE_{1}^{p}$, where $\EE_{1}^{\lambda}=\ker [(c_{\lambda})_u]$.

 There is an obvious submodule $u^{j}V^- \subset \FF_j^{p}(V^{-})$, since $c_{\lambda}(u^{j}v)=u^{j}c_{\lambda}(v)\in F_{j}(W^-_{\lambda}).$

Let $T_{\lambda}=$\,Torsion$(W_{\lambda})$, then the filtration  $F_j^{p}(V^-)$ is induced by the {\bf valuation}
    \begin{align}\label{Definition lambda valuation}
    \begin{split}
    \mathrm{val}^{p}(v) := & \max \{ j\in \N\cup \{\infty\} : v\in \FF_j^{\lambda}(V^-) \textrm{ for all }\lambda \geq p\}
    \\
    =& \max \{ j\in \N\cup \{\infty\} : [c_{\lambda}(v)]\in  W^-_{\lambda}/T_{\lambda} \textrm{ is }u^j\textrm{-divisible for all }\lambda\geq p\},
    \end{split}
\end{align}
    The function $\mathrm{val}^{p}$ is non-decreasing in $p$, due to the compatibility condition.

Let $u^{j_i}(\lambda)$ be the invariant factors of $c_{\lambda}$.
By \cref{Remark description of Vplus},
$F_j^{\lambda}(V^+)$ has a summand $u^{j-j_i(\lambda)}\ku/u\ku$ for $j\leq j_i(\lambda)$ (so
$\mathrm{span}_{\k}\{u^{j-j_i(\lambda)},\ldots,u^0\}$ as a $\k$-vector space) and a summand $K_{c_{\lambda}}^+ =(\ker c_{\lambda})_u/\ker c_{\lambda}$ as in \cref{Lemma Vplus preimage of T in Wplus is not interesting}. We abbreviate 
$K_{\lambda}^+ := K_{c_{\lambda}}^+$ and $K_{p}^+=\cap_{\lambda \geq p} K_{\lambda}^+$. For $j\geq 0$, let 
$$\mathbb{F}_j:=\ker (u^{j+1}:V^+ \to V^+) \subset V^+,$$
so loosely the $\ku$-submodule of elements involving only terms of order $u^{\geq -j}$.

\begin{cor}\label{Cor relation to EEp}
For any $j\geq 0,m\in \N, p\in [0,\infty]$, 
there are isomorphisms of $\ku$-modules
\begin{align*}
 \FF_{j+1}^{p}(V^{-})/u^{j+1}V^- \hookrightarrow  \EE^p_1(V^{+}) \cap \mathbb{F}_j,\; & v \mapsto [u^{-j}v],
 \\
  \FF_{j+1}^{p}(V^{-})/u^{j+1}V^- \hookrightarrow  \EE_{m}^p(V^{+}) \cap \mathbb{F}_{j-m+1},\; & v \mapsto [u^{m-1-j}v].
\end{align*}

If $c_{\lambda}$ is injective, then for $p=\lambda$
those embeddings are isomorphisms for sufficiently large $j$, namely $$j\geq \max \{j_i(\lambda): u^{j_i(\lambda)} \textrm{ is an invariant factor of }c_{\lambda}\}.$$
Without assuming $c_{\lambda}$ is injective, for $j$ as above the map $v\mapsto [u^{-j}v]$ yields an isomorphism
$$\FF_{j+1}^{\lambda}(V^{-})/(u^{j+1}V^-+\ker c_{\lambda})
\to \EE^p_1(V^{+})/K_{\lambda}^+.$$
\end{cor}
\begin{proof}
   By \cref{Lemma relation between filtrations Vminus and Vplus}, the map $ \FF_{j+1}^{p}(V^{-}) \to \EE^p_1(V^{+})$, $v \mapsto [u^{-j}v]$ is a well-defined $\ku$-module homomorphism. By \cref{Lemma Fj V minus description}, $\ker c_{\lambda}\cong \ku^b$ contributes $\ku^b/u^j\ku^b$ to $\FF_{j}^{p}(V^{-})/u^{j}V^-$, 
and $K_{\lambda}^+\cong \mathbb{F}^b$ to $\EE^p_1(V^{+})$. The map $v \mapsto [u^{-j}v]$ between them has image $u^{-j}\ku^b/u\ku^b \subset \mathbb{F}^b$. If $b\neq 0$ (equivalently, if $c_{\lambda}$ is not injective), that map is not surjective. 
The remainder of the proof will show that the map does however surject onto a complement of $\mathbb{F}^b\subset \EE^p_1(V^{+})$ for $j$ as large as in the claim.
   
   By \cref{Remark description of Vplus} we reduce to the case $c_{\lambda}=u^{j_i}:\ku \to \ku$ for $j\geq j_i$. Thus%
   \footnote{The spans are $\k$-vector space presentations. The notation is imprecise for $j_i=0$: $\FF_{j}^{p}(V^{-})/u^jV^-=\{0\}=\EE^p_1(V^{+})$.
   } 
   $\FF_{j}^{p}(V^{-})\!=\!u^{j-j_i}\ku$,
   $$\FF_{j+1}^{p}(V^{-})/u^{j+1}V^-\!=\!\mathrm{span}_{\k}\{u^{j+1-j_i}\!,u^{j+1-j_i+1}\!,\ldots\!,u^{j}\} \qquad \textrm{ and }\qquad \EE^p_1(V^{+})\!=\!\mathrm{span}_{\k}\{u^{1-j_i}\!,\ldots\!,u^0\}$$
   (cf.\;\cref{Example cube map showing induced filtrations} for the case $j_i=3$). Since the left basis maps precisely to the right basis via the map $v\mapsto [u^{-j}v]$, we get an isomorphism in this case.
   If $j<j_i$, the basis on the left becomes $\{u^0,\ldots,u^j\}$ and it maps to the basis of $\EE^p_1(V^{+})\cap \mathbb{F}_j=\mathrm{span}_{\k}\{u^{-j}\!,\ldots\!,u^0\}$. 

   In the statement involving $m$, recall $\EE_{m}^p(V^{+})=u^{m-1}\EE_{1}^p(V^{+})$.
   So in the case $c_{\lambda}=u^{j_i}:\ku \to \ku$ we get $\EE_{m}^p(V^{+})=\mathrm{span}_{\k}\{u^{m-j_i},\ldots,u^0\}$. The rest of the argument is then very similar.
\end{proof}

\begin{de}
    Abbreviate $\FF_{j}^{p}:=\FF_{j}^{p}(V^{-})$ and $V=V^-$. Define the $\k$-vector spaces
    $$
    \PP_j^p := \FF_{j}^{p}/u\FF_{j-1}^{p} = \FF_{j}^{p}/(\FF_{j}^{p} \cap uV), 
    \;\;\;\;\;\,
    \PP := V/uV,
    \;\;\;\;\;\,
    \mathrm{gr}_j^p:=\PP_j^p/\PP_{j-1}^p,
    \;\;\;\,\textrm{ and }
    \;\;\;\,
    \mathrm{gr}^p \PP := \oplus \mathrm{gr}_j^p,
    $$
    noting the properties: $\PP_j^p=\PP$ for $j\leq 0$, $\PP_{j+1}^p\subset \PP_j^p$, $\PP_j^p\subset \PP_j^{p'}$ for $p\leq p'$. Declare $\PP_{\infty}^p:=0$.
\end{de}

\begin{lm}\label{Lemma dimension calculation for sum of ds}
There is an identification of $\k$-vector spaces,
    $$
    \FF_{j}^{p}/u^j V \cong \PP_j^p \oplus \PP_{j-1}^p \oplus \cdots \oplus \PP_1^p.
    $$
    Thus, $\dim_{\k} \FF_{j}^{p}/u^j V = d_j^p + \cdots + d_1^p,$ where $d_j^p:=\dim_{\k} \PP_j^p$.
\end{lm}
\begin{proof}
    Abbreviate $F_j:=F_j^{\lambda}(V^-)$ and $V=V^-$. Recall $F_0=V,$ so $u^j V= u^j F_0$. As $\k$-vector spaces, 
    $$
    F_j/u^j V \cong (F_j/uF_{j-1} )\oplus( uF_{j-1}/u^2 F_{j-2} )\oplus \cdots \oplus (u^{j-1} F_1/u^j F_0).
    $$
    Since $V$ is torsion-free, $F_{j-i}/u F_{j-i-1}\cong u^i F_{j-i}/u^{i+1} F_{j-i-1}$ via $v\mapsto u^iv$.  The claim follows.
\end{proof}
Those $d_j^p$ are the coefficients of the slice series $s_p$ for the filtration $\FF_{j}^{p}$
on $V$ (\cref{Definition slice series}). 
Each $d_j^p$ contains a summand of $b_p=\mathrm{rank}_{\ku}\ker c_{\lambda}$ by \eqref{Equation filtration series defn}. We define the reduced slice polynomial $\widetilde{s}_p$ as in \eqref{Equation filtration poly defn} by reducing coefficients to $d_j^p-b_p$.
If $c_{\lambda}$ is injective, then $b=0$, $s_p=\widetilde{s}_p$, $K_p^+=0$. 

\begin{cor}\label{Cor sum of coeffs of filtration poly}
$\dim_{\k} \EE^p_1(V^{+})/K_p^+
= \sum (\textrm{coefficients of }\widetilde{s}_p)-\mathrm{rank}_{\ku} V$.
If $c_{\lambda}$ is injective, then:
$$\dim_{\k} \EE^p_1(V^{+}) = \sum (\textrm{coefficients of }t,t^2,t^3,\ldots \textrm{ in the slice polynomial }s_p(t)=\widetilde{s}_p(t) \textrm{ for }\FF_{j}^{p}).$$
\end{cor}
\begin{proof}
This follows by \cref{Lemma dimension calculation for sum of ds}, using the isomorphism in \cref{Cor relation to EEp} that holds for large $j$ (the proof explains the role of the $K_p^+\cong \mathbb{F}^{b_p}$ summand). We also use $d_0^p= \dim_{\k} V/uV=\mathrm{rank}_{\ku}V.$ 
\end{proof}

\section{Equivariant Floer theory: the three conventions}
\label{Section Equivariant Floer theory the three conventions}

\subsection{Background on equivariant Floer cohomology}

From now on $\k$ will denote the Novikov field. Recall that the {\bf Novikov field} is a field extension of a base field $\mathbb{B}$, whose choice depends on properties of $Y$ and on grading requirements. 
When $c_1(Y)=0$, we use
\begin{equation}\label{EqnSec2NovikovField}
\textstyle \mathbb{K} = \{\sum n_j T^{a_j}: a_j\in \R, a_j\to \infty, n_j\in \mathbb{B} \},
\end{equation}
where $T$ is a formal variable in grading zero. When $Y$ is monotone, so $c_1(Y)\in \R_{>0}\cdot [\omega]$, the same Novikov field can be used but $T$ will have a non-zero grading \cite[Sec.2A]{R16}. In other situations, e.g.\,the weakly-monotone setup, the Novikov field is more complicated \cite[Sec.5B]{R16}: we briefly discuss and use this more complicated field for example in \cref{Remark change of cappings} and \cref{Subsection Beyond cappings}.

Ordinary and Morse cohomologies are always understood to be computed with coefficients in $\k$.
Below, $(Y,\omega,\Fi)$ will denote any symplectic $\C^*$-manifold (over a convex base) with $\C^*$-action $\Fi$.

The equivariant parameter will be denoted $u$, and placed in grading $|u|=2$. One can view $\k[u]=H^*(\C \P^{\infty})$ because we work with $BS^1:=S^{\infty}/S^1=\C \P^{\infty}$.

Let $C^*_{\lambda}$ be the free $\ku$-module generated by the $1$-orbits of a chosen Hamiltonian $H_{\lambda}:Y \to \R$ of slope $\lambda$ at infinity (meaning $H_{\lambda}=\lambda H$ at infinity, where $H$ is the moment map for the $S^1$-action by $\Fi$). 
Our discussion applies equally well to the {\MBF} setup, in which case we refer to \cite{RZ2} for a description of the generators. 

The equivariant Floer differential is a $\ku$-module homomorphism, of the form \eqref{Equation for the differential expanded}, where
$\delta_0$ is the non-equivariant Floer differential, and $(\delta_i)_{i\geq 1}$ are equivariant corrections (cf.\;various implementations of this in the literature such as \cite{Sei08,Bourgeois-Oancea-S1, McLR18}). We defer to \cref{Subsection the choice of S1 action and equiv floer theory} the discussion of the choice of $S^1$-action used (i.e.\;how precisely we choose $S^1$ to act on $1$-periodic Hamiltonian orbits). 

We follow the conventions from  \cite{liebenschutz2020intertwining, liebenschutz2021shift} (but use more general $S^1$-actions), which we now briefly overview.
The $\delta_i$ involve counts of Floer trajectories for which the Hamiltonian $H_z$ and the almost complex structure $J_z$ are allowed to depend on the parameter $z\in S^{\infty}$ in the Borel construction $(S^{\infty} \times Y)/S^1$. One counts pairs $(v,u)$, as follows. First, we have a Morse trajectory $v:\R \to S^{\infty}$, using the pull-back of standard Morse-Smale data for $\C P^{\infty}$; in particular, the Morse function comes from a Morse function $\C \P^{\infty}:=S^{\infty}/S^1 \to \R$ with one critical point of index $2i$ for each $i\geq 0$.  Second, $u:\R \times S^1 \to Y$ is a solution of Floer's equation $\partial_s u + J_{v(s)} (\partial_t u - X_{H_{v(s)}})=0$. Loosely, $\delta_i$ counts such pairs $(v,u)$ which are isolated modulo the naturally induced $S^1$-action, for which the asymptotic critical points for $v$ involve a Morse index difference of $2i$. At chain level, $C^*_{\lambda}$ is a $\ku$-module freely generated by the asymptotics of such solutions $u$, but implicitly it is also keeping track of the asymptotics of such $v$: a factor $u^{i}$ indicates that the index $2i$ Morse critical point in $\C \P^{\infty}$ is being considered.

The above $\delta$ determines a well-defined differential on the three models discussed in \cref{Subsection the Wminus plus and infinity models},
\begin{equation}\label{defn F C^+ and W^+ for Floer case}
W^-_{\lambda}:=W_{\lambda} = H(C^*_{\lambda}),\qquad\quad
W^{\infty}_{\lambda} := H((C^*_{\lambda})_u) \cong (W_{\lambda})_u^-, \qquad\quad
W^+_{\lambda} :=H((C^*_{\lambda})^+).
\end{equation}

Equivariant Floer cohomology does not in general carry a product, so the above are just $\ku$-modules.
In \cref{Subsection weight 10 case product}, we explain under what circumstances it carries a product.

\subsection{Background on equivariant quantum cohomology}

The equivariant quantum cohomology at the $\k$-vector space level is ordinary equivariant Morse cohomology, and it always carries a product unlike the Floer case. 
Again, we refer to \cite{liebenschutz2020intertwining, liebenschutz2021shift} for a detailed description of these conventions, and just give the loose overview here.
 
The equivariant Morse model at chain level is freely generated by critical points associated to $S^1$-equivariant Morse data on $S^{\infty}\times Y$, working modulo $S^1$. The differential counts pairs of Morse trajectories $(v,u)$ in $S^{\infty}\times Y$ that are isolated modulo the naturally induced $S^1$-action. Here, the required equation for $u: \R \to Y$ is $\partial_s u = -\nabla f_{v(s)}$, where the Morse data $f_z$ depends on $z\in S^{\infty}$. 

This yields a chain complex $C^*(Y)[u]$, where $C^*(Y)$ denotes the Morse complex for $Y$ from the above construction. As before, a factor $u^i$ refers to the index $2i$ critical point of $\C \P^{\infty}$. There is no need for an algebraic completion in $u$, unlike the Floer context where arbitrarily high powers of $u$ can arise in the differential. The motivation to complete in $u$ is therefore Floer-theoretical,
$$C_{\mathrm{eq}}:=C^*(Y)[\![u]\!].$$
\begin{lm}\label{Lemma V- completion in u is harmless}
    $H(C_{\mathrm{eq}})\cong H^*_{S^1}(Y)\otimes_{\k[u]}\ku$ as a $\ku$-module, where $H^*_{S^1}(Y):=H((S^{\infty}\times Y)/S^1)$ is the ordinary $S^1$-equivariant cohomology defined using the Borel model $(S^{\infty}\times Y)/S^1$.
\end{lm}
\begin{proof}
The $E^-$-equivariant differential \eqref{Equation for the differential expanded} for equivariant Morse theory
has a well-defined decomposition $d=d \otimes \mathrm{id}$ via $C^*(Y)[\![u]\!]=C^*(Y)[u]\otimes_{\k[u]} \ku$, because $d$ has no contributions from $u^j\delta_j$ for $j>\dim_{\C} Y$ for grading reasons. 
\end{proof}
We obtain three models as in \cref{Subsection the Wminus plus and infinity models},
\begin{equation}
V^- :=V = H(C_{\mathrm{eq}}),\qquad\quad
V^{\infty} := H((C_{\mathrm{eq}})_u) \cong V_u^-, \qquad\quad
V^+ :=H(C_{\mathrm{eq}}^+).
\end{equation}
we will see in \cref{Subsection why V is free in equiv quantum case} that $V$ is a free $\ku$-module of finite rank $\mathrm{rk}\,V = \dim_{\k} H^*(Y)$, so the freeness assumptions from \cref {Subsection Exact sequence for induced filtrations} apply.

The equivariant quantum cup product $\star$ involves a count of pseudo-holomorphic spheres, involving almost complex structures $J_z$ that depend on $z\in S^{\infty}$ as in the Floer case. One counts the following triples $(v,w,u)$. A Y-shaped triple $v$ of half-infinite Morse trajectories in $S^{\infty}$ for suitable Morse data, meeting at a common point $v(0)\in S^{\infty}$. (More precisely, we use ``\upsidedownY-shaped graphs'', as we work with cohomological conventions). A $J_{v(0)}$-pseudo-holomorphic sphere $w:\C\P^1 \to Y$ with three marked points. A Y-shaped triple $u$ of half-infinite Morse trajectories in $Y$ that reach, at the finite ends, the $w$-images of the three marked points.

When $w$ is constant, that count defines the ordinary (Morse construction of) equivariant cup-product,
$$\bullet\cup_0 \bullet + u \bullet \cup_1\bullet + u^2 \bullet\cup_2\bullet + \cdots: C^*(Y)[u]^{\otimes 2} \to C^*(Y)[u].$$
The non-equivariant cup-product $\cup_0$ arises when the ``Y-shaped graph'' $v$ in $\C \P^{\infty}:=S^{\infty}/S^1$ is isolated, in particular it is additive on Morse indices: in self-evident notation, $u^{j_1}q_1 \star_0 u^{j_2} q_2 = u^{j_1+j_2} q_3$. 
When $w$ above is a non-constant sphere, it gives rise to quantum corrections to the $\cup_i$ involving strictly positive powers of the Novikov parameter $T$ in $\k$, $\star_i:=\cup_i + (T^{>0}$-terms$)$. We obtain a $\ku$-bilinear product,
\begin{equation}\label{Equation equiv quantum product expansion}
\bullet \star \bullet = \bullet \star_0 \bullet + u \bullet\star_1\bullet + u^2 \bullet\star_2\bullet + \cdots: C_{\mathrm{eq}}^{\otimes 2} \to C_{\mathrm{eq}}.
\end{equation}
Specifically, when $v$ above involves an asymptote of Morse index that is $2i$ higher than expected (so $u^{j_1+j_2+i}q_3$ in the previous notation), the triple $(v,w,u)$ contributes to $u^i \bullet \star_i \bullet$. Thus, the $u^0$-terms are non-equivariant quantum terms, and the $T^0$-terms are equivariant non-quantum terms, so $u^0,T^0$-contributions define the non-equivariant ordinary cup product.

\subsection{The maps from quantum to Floer cohomology}
Equivariant quantum cohomology is identifiable with equivariant Floer cohomology for $H_{\lambda_0}$ whenever the slope $\lambda_0>0$ is smaller than the non-zero $S^1$-periods. Composing with the equivariant Floer continuation map induced by the homotopy from $H_{\lambda}$ to $H_{\lambda_0}$, yields a $\ku$-linear chain map
$$
c^*_{\lambda}: C_{\mathrm{eq}} \to C^*_{\lambda_0} \to C^*_{\lambda}.
$$
So we may apply Sections \ref{Subsection induced filtration for V and W}, \ref{Subsection Filtrations arising from $u$-local isomorphisms}, and \ref{Subsection Exact sequence for induced filtrations}:
$$
c^*_{\lambda}: V^- \to W^-_{\lambda}, \qquad 
(c^*_{\lambda})_u: V^{\infty}_u \to (W^{\infty}_{\lambda})_u, 
\qquad 
[(c^*_{\lambda})_u]: V^+ \to W^+_{\lambda}.
$$
Taking the direct limit over equivariant Floer continuation maps, as $\lambda\to \infty$, yields
$$
c^*: V^- \to W^-, \qquad 
c^*_u: V^{\infty}_u \to W^{\infty}_u, 
\qquad 
[c^*_u]: V^+ \to W^+.
$$
The $\infty$-version is often called the \textbf{periodic complex/cohomology}.
We often use the explicit notation:
\begin{equation}\label{Equation explicit notation equivariant quantum and Floer}
\begin{array}{rclcrclcrcl}
     V^-&=& E^-QH^*(Y) 
     & \strut\qquad &
     V^{\infty}&=& E^{\infty}QH^*(Y) 
      & \strut\qquad &
     V^+&=&E^+QH^*(Y)\\[1mm]
    W^-_{\lambda}&=& E^-HF^*(H_{\lambda}) 
      & \strut\qquad &
     W^{\infty}_{\lambda}&=& E^{\infty}HF^*(H_{\lambda}) 
      & \strut\qquad &
     W^+_{\lambda}&=& E^+HF^*(H_{\lambda}) \\[1mm]
         W^-&=& E^-SH^*(Y) 
      & \strut\qquad &
     W^{\infty}&=& E^{\infty}SH^*(Y) 
      & \strut\qquad &
     W^+&=& E^+SH^*(Y).
\end{array}
\end{equation}
These fit into exact sequences as explained in \cref{Subsection the long and short exact sequences}.
We will see in \cref{Subsection vanishing results equiv case} that $W^-$ and $W^{\infty}$ are usually non-zero due to injective maps from $V^-,V^{\infty}$, whereas $W^+$ often vanishes.

\begin{cor}\label{Lemma SES for various versions of equiv coh}
    There is a commutative diagram, for each slope $\lambda>0$, whose horizontal lines are short exact sequences of $\ku$-modules,
    $$
    \begin{tikzcd}[column sep=0.3in,row sep=0.2in]
0 \arrow[r, ""] & E^-QH^*(Y)[-2]
 \arrow[r, "u\otimes \mathrm{id}"] 
 \arrow[d, "c_{\lambda}"]
 &
 E^{\infty}QH^*(Y)
\arrow[r, ""] \arrow[d, "(c_{\lambda})_u"]
& E^+QH^*(Y) \arrow[d, "{[(c_{\lambda})_u]}"] 
\arrow[r, ""]
& 0
\\
0 \arrow[r, ""] & 
E^-HF^*(H_{\lambda})[-2]/\textrm{Torsion} \arrow[r,"u\otimes \mathrm{id}"]
\arrow[d, "\varinjlim"]
&
 E^{\infty}HF^*(H_{\lambda})
\arrow[r, ""] 
\arrow[d, "\varinjlim"]
& E^+HF^*(H_{\lambda})_{\mathbb{F}} 
\arrow[r, ""]
\arrow[d, "\varinjlim"]
& 0
\\
0 \arrow[r, ""] & 
E^-SH^*(Y)[-2]/\textrm{Torsion} \arrow[r,"u\otimes \mathrm{id}"]
&
 E^{\infty}SH^*(Y)
\arrow[r, ""] 
& E^+_{\mathbb{F}\mathrm{lim}}SH^*(Y)  
\arrow[r, ""]
& 0
\end{tikzcd}
    $$
where $E^+_{\mathbb{F}\mathrm{lim}}SH^*(Y):=\varinjlim E^+HF^*(H_{\lambda})_{\mathbb{F}}$.
Any $\k$-vector space complement to $E^+_{\mathbb{F}\mathrm{lim}}SH^*(Y)$ inside $E^+SH^*(Y)$ maps isomorphically onto the torsion of $E^-SH^*(Y)[-1]$ via the connecting map in \eqref{Equation LES connecting map}. We will show later, in \cref{Theorem torsion freeness 2}, that this torsion vanishes. 
\end{cor}
\begin{proof}
    This follows by \cref{Cor SES over LES} and \cref{Cor LES for ESH abstractly}.
\end{proof}
\subsection{A remark about the telescope versus the exhaustive model}\label{Subsection A remark about the telescope versus the exhaustive model}

We always take the direct limit over continuation maps at the cohomology level, not at the cochain level. We call this the {\bf exhaustive model}. So $E^-SH^*(Y)$, $E^{\infty}SH^*(Y)$, and $E^+SH^*(Y)$ always refer to that approach.
The motivation for the name is that equivariant {\MBF} spectral sequences (see \cite{RZ2}, this involves a consistent construction of a cofinal sequence of $H_{\lambda}$) arise from an exhaustive filtration. 

There is a standard construction at cochain level to implement the direct limit over continuation maps for a sequence of Hamiltonians $H_{\lambda}$ as we increase the slopes $\lambda\to \infty$, e.g.\,see \cite[Sec.4.2]{ritter2017monotone} or \cite[Sec.7]{zhao2019periodic}.
This is called the {\bf telescope model}. The two models behave rather differently:

\begin{enumerate}
\item The equivariant {\MBF} spectral sequence converges for the exhaustive model, but may fail to do so for the telescope model (to implement the completed tensor product correctly, one would need to allow suitable infinite sums of elements from different columns).

\item For the telescope model, there is a spectral sequence from non-equivariant to equivariant symplectic cohomology, for general homological reasons \cite[Cor.2.4]{zhao2019periodic}. So the vanishing of $SH^*(Y)$ implies the vanishing of all versions of telescope equivariant symplectic cohomology.  
The same holds for $E^+SH^*(Y)$ by 
\eqref{Lemma spectral sequence SH to ESH}, see \cref{cor vanishing of E+ for SH}. However, it can fail for $E^-SH^*(Y)$ and $E^{\infty}SH^*(Y)$, e.g.\;for $Y=\C$. By \cite{RZ1}, $SH^*(Y)=0$ for any symplectic $\C^*$-manifold with $c_1(Y)=0$, so the exhaustive $E^+SH^*(Y)=0$ vanishes, as well as all telescope models $E^{\karo}_{\textrm{telescope}}SH^*(Y)=0$; whereas $c^*:E^-QH^*(Y)\hookrightarrow E^-SH^*(Y)$ and $c^*:E^{\infty}QH^*(Y)\hookrightarrow E^{\infty}SH^*(Y)$ are injective and therefore non-zero. 
\end{enumerate}

\subsection{The choice of $S^1$-action, and $(S^1\times S^1)$-equivariant Floer theory}
\label{Subsection the choice of S1 action and equiv floer theory}

There are two $S^1$-actions in play in Floer theory on $Y$, acting on $1$-periodic Hamiltonian orbits $x=x(t):S^1 \to Y$. The {\bf $\Fi$-action} induced from the $\Fi$ action on $Y$, and the {\bf loop-action} induced reparametrising the domain of $x$,
$$
(w\cdot x)(t) = \Fi_w(x(t)), \quad \textrm{ and } \quad (w\cdot x)(t) = x(t+\theta),
$$
where $w=e^{2\pi i \theta} \in S^1\subset \C^*.$
One can therefore define $(S^1\times S^1)$-equivariant complexes, with one equivariant parameter in degree two assigned to each of the two respective actions.
This enlarges the complex, but has the advantage
that one does not need to keep track of changes in the weights $(a,b)$ in \cref{Subsection Comparing different weights via equivariant Seidel isomorphisms}.
In this paper, we are instead following the approach from Liebenschutz-Jones \cite{liebenschutz2020intertwining,liebenschutz2021shift}: we pick an $S^1$-subgroup of $S^1\times S^1$, by choosing a {\bf weight} $(a,b)\in \Z^2$. More precise details of the action are shown in \cref{Subsection precise equivariant action}. 
This means we only have one equivariant parameter $u$, at the cost of the weights $(a,b)$ changing in \cref{Subsection Comparing different weights via equivariant Seidel isomorphisms}.

\begin{de}\label{Definition free weight}
We call $(a,b)$ a {\bf free weight} if $ma\notin b\Z_{\geq 1}$, for any positive weight $m\geq 1\in \N$ of the $\Fi$-action. This condition always holds if $b,a$ have opposite signs; or if one but not both $a,b$ are zero. 
\end{de} 

The motivation for the above is \cref{Lemma fixed points of the mixed action}, so on Morse-Bott manifolds of non-constant $1$-orbits in $Y$, the combined action does not fix a geometric $S^1$-orbit.
That result is desirable due to our general localisation theorem, \cref{Lemma MBF in equiv case and localisation}. We remark that
$$
(a,b)\textrm{ free }\Rightarrow (a-b,b)\textrm{ free}.\quad\qquad (a,b) \textrm{ free and }|a|\geq |b| \Rightarrow (a+b,b)\textrm{ free}.
$$
\subsection{Technical Remarks about the choice of Hamiltonians and almost complex structures}
\label{Subsection precise equivariant action}
For the $\C^*$-action $\Fi$ on $Y$, we simplify the notation $\Fi_{e^{2\pi i\theta}}$ to $\Fi_{\theta}$, using the parameter $\theta \in S^1:=[0,1]/(0\sim 1)$.
We also parametrise the domain of loops by $t\in [0,1]/(0\sim 1)$.
We mimic the conventions in \cite[Sec.4.3.1, Sec.5.1]{liebenschutz2020intertwining}. 
For a weight $(a,b)\in \Z^2$, we define the $S^1$-action on $x=x(t)\in \mathcal{L}Y$ by
\begin{equation}\label{Equation action on loops}
 \theta \cdot_{(a,b)} x : t\mapsto \Fi_{a\theta}(x(t-b\theta)).
\end{equation}
The Borel model for Floer theory involves pairs $(w,x)\in S^{\infty}\times \mathcal{L}Y$ identifying 
\begin{equation}\label{Equation Borel model identification} 
(\theta\cdot w,x)\sim (w,\theta\cdot_{(a,b)} x),
\end{equation} 
where $\theta\cdot w$ denotes a choice of free $S^1$-action on $S^{\infty}$.
We assume the reader is familiar with the meaning of equivariant Floer data $H^{\mathrm{eq}}_{w,t}(y)$, $J^{\mathrm{eq}}_{w,t}$ in the sense of Liebenschutz-Jones \cite[Sec.4.3]{liebenschutz2020intertwining}.
In our case, equivariant Hamiltonians and almost complex structures satisfy the following,\footnote{To clarify, $\theta$ acts on $S^1$ by $\theta\cdot t:= t+b\theta$, which is consistent with \eqref{Equation action on loops} because we define $\theta \cdot x = \Fi_{a\theta}\circ x \circ \theta^{-1}$.} for $y\in Y$, 
\begin{equation}\label{Equation equivariant Hamiltonians} 
H^{\mathrm{eq}}_{w,t}(y)=H^{\mathrm{eq}}_{\theta^{-1}\cdot w,t+b\theta}(\Fi_{a\theta}(y)), 
\qquad\qquad
J^{\mathrm{eq}}_{\theta^{-1}\cdot w,t+b\theta}=d\Fi_{a\theta}\circ J^{\mathrm{eq}}_{w,t}\circ d\Fi_{a\theta}^{-1}.
\end{equation}
At infinity the Hamiltonian is linear of slope $\lambda$, so
$H^{\mathrm{eq}}_{w,t}(y)=\lambda H(y)$ when $H(y)$ is large. When we write the equivariant Floer chain complex, we use the abbreviated notation $H_{\lambda},I$ rather than $H^{\mathrm{eq}}_{w,t},J^{\mathrm{eq}}_{w,t}$.

The Floer action functional is $\mathcal{A}(x,\overline{x}) = -\int \overline{x}^*\omega + \int H_{w,t}^\mathrm{eq}(x(t))\,dt$, for $x\in \mathcal{L}Y=C^{\infty}(S^1,Y)$, suitably defined on a cover of the free loop space $\mathcal{L}Y$, involving  cappings $\overline{x}:\mathbb{D} \to Y$ of $x$ for contractible loops, or choosing a homotopy $\overline{x}$ of maps $S^1 \to Y$ towards a chosen reference loop in the free homotopy class $[x]\in [S^1,Y]$; we omit these standard details.

\begin{lm}\label{Lemma fixed points of the mixed action}
Let $x_0$ be any $\Z/m$-torsion point of $\Fi$ that is not a fixed point of $\Fi$. Consider the Hamiltonian flow of $\lambda H$ on the slice $H=H(x_0)$, for the slope $\lambda=\tfrac{k}{m}$, $k\geq 1\in \N$. Let $x$ be the Hamiltonian $1$-orbit with initial point $x_0$.
Then:
\begin{enumerate}
    \item $a = \lambda b \Longleftrightarrow x$ is a fixed point of the action \eqref{Equation action on loops}.
    \item The only fixed points of \eqref{Equation action on loops} are the constant orbits and any orbits arising 
     with slope $\lambda=\tfrac{a}{b}$.
    \item If $(a,b)$ is a free weight, then \eqref{Equation action on loops} has no fixed points 
    except for constant orbits in $\mathfrak{F}=\mathrm{Fix}(\Fi)$.
\end{enumerate}
\end{lm}
\begin{proof}
  By assumption, $x(t) = \Fi_{kt/m}(x_0)$.
  By definition, $\theta$ acts by $(\theta\cdot_{(a,b)} x)(t) =  \Fi_{ a \theta+  k(t-b\theta)/m}x_0.$
  As $x_0$ is not fixed by $\Fi$, the equation $\theta\cdot x = x$ is equivalent to:
  $$a \theta+  \tfrac{k}{m}(t-b\theta) = N+\tfrac{kt}{m} \qquad \textrm{ for some }N\in \Z,$$
that is: $(a-\tfrac{bk}{m})\theta = N$. As this holds for all $\theta\in \R$, it forces $N=0$, so $a=\tfrac{bk}{m}$. The claims follow.
\end{proof}

\section{Structural results for equivariant quantum and Floer cohomology}

\subsection{Equivariant quantum cohomology is free}
\label{Subsection why V is free in equiv quantum case} 

\begin{prop}[Equivariant formality]\label{Prop equivariant formality for QH}
There are non-canonical $\ku$-module isomorphisms
    $$
    V^-\cong H^*(Y)\otimes_{\k} \ku
    ,\quad
    V^{\infty}\cong  H^*(Y) \otimes_{\k} \kuu
    ,
    \quad
    V^{+}\cong H^*(Y)\otimes_{\k} \mathbb{F}.
    $$
\end{prop}
\begin{proof}
First we recall a general result due to Kirwan \cite[Prop.5.8]{Ki84}. Given any Hamiltonian $S^1$-action on a closed symplectic manifold $M$ (it also generalises to Hamiltonian actions by compact connected Lie groups), and working over rational coefficients, there is a non-canonical isomorphism $H^*_{S^1}(M)\cong H^*(M)\otimes H^*(BS^1)$, where $H^*(BS^1)\cong \Q[u]$. The same holds over $\k$ coefficients, since we assume $\k$ has characteristic zero.
Kirwan's result may not always hold in non-compact settings \cite[Sec.9]{Ki84}. However, in our setting, we may assume that the moment map is proper and bounded below (see \cite{RZ1}), which ensures the ordinary Morse inequalities which are needed in Kirwan's proof (which does not otherwise rely on compactness). 

It follows that the equivariant differential $d$ on $C^*(Y)[u]=C^*(Y)\otimes_{\k}\k[u]$ is chain homotopic to the ordinary differential $\delta_0 \otimes \mathrm{id}$. Say $d - \delta_0 = K \circ d + d \circ K$, where $K$ is a chain homotopy for the complex $C^*(Y)[u]$.
Then on $C^*(Y)[\![u]\!]=C^*(Y)[u]\otimes_{\k[u]}\ku$ we have $d - \delta_0 = K \circ d  + d \circ K$ (omitting $\cdots \otimes \mathrm{id}$ factors in the notation).
This implies not just $V^-\cong H^*(Y)\otimes_{\k} \ku$, but also the analogous results for $V^{\infty}$ and $V^+$, because that chain homotopy equation remains well-defined if we tensor the coefficient system as needed in \eqref{defn F C^+ and W^+} to obtain the chain complexes for $V^{\infty},V^+$.
\end{proof}

Consider the spectral sequence induced by the $u$-adic filtration on the Morse-Bott chain complex for $H^*_{S^1}(Y)$ over $\Q$, arising from the differential $d=\delta_0 + u\delta_1+\cdots + u^n \delta_n$ where $n=\dim_{\C}Y$ (cf.\cref{Lemma V- completion in u is harmless}).

\begin{cor}\label{Cor equiv formality via spectral sequence}
The $u$-adic spectral sequence for $QH^*_{S^1}(Y)$ over $\Q$ converges on the $E_1$-page, yielding an isomorphism $H^*(Y;\Q)\otimes_{\Q} \Q[u]\cong QH^*_{S^1}(Y)$. This induces isomorphisms as in \cref{Prop equivariant formality for QH}, by passing to $\k$ coefficients.
A basis for $H^*(Y;\Q)$ gives rise to a basis for $E^-QH^*(Y)$ which at chain level involves only coefficients from the base field $\mathbb{B}$, so not involving the formal Novikov variable $T$ of $\k$, and whose $u^0$-part is the same at the chain level as for the basis of $H^*(Y;\Q)$.
In particular, the unit $x_0$ of the ring $E^-QH^*(Y)$ is represented at chain level by the unique minimum of a suitable auxiliary Morse function on the Morse-Bott manifold $\F_{\min}:=\min H$ (a component of the $\C^*$-fixed locus \cite{RZ1}). 
\end{cor}
\begin{proof}
The $E_1$-page of the spectral sequence associated to the $u$-adic filtration is $H^*(Y;\Q)\otimes_{\Q} \Q[u]$, as $\delta_0$ is the non-equivariant differential computing $H^*(Y;\Q)$.
Consider the dimensions of the $\Z$-graded $\Q$-vector space $QH^*_{S^1}(Y)$, starting from degree $0$ and inductively moving up. By Kirwan's result in the proof of \cref{Prop equivariant formality for QH}, these dimensions match up with those for $H^*(Y;\Q)\otimes_{\Q} \Q[u]$, therefore there can be no non-trivial edge differentials from the $E_1$-page onwards.
For the final claim: under the isomorphism in the claim, a classical cycle $x$ may get mapped to $x+u^{\geq 1}$-terms. This map is grading preserving, so for the minimum $x_0\in H^0(Y)$ there are no $u^{\geq 1}$-corrections for grading reasons.   
\end{proof}

\subsection{Specialisation from equivariant to non-equivariant}

\begin{lm}\label{Lemma ev0 commutative diagram}
There are specialisation maps $\mathrm{ev}_0$ that fit into a commutative diagram of $\k$-linear maps
$$
\xymatrix@R=12pt@C=20pt{
 E^-QH^*(Y) 
\ar@{->}[r]^-{E^-c_{\lambda}^*} 
\ar@{->>}[d]_-{\textrm{ev}_0}  
&
E^-HF^*(H_{\lambda})
\ar@{->}^-{\textrm{ev}_0}[d] 
\\
QH^*(Y)\cong E^-QH^*(Y)/uE^-QH^*(Y) 
\ar@{->}[r]^-{c_{\lambda}^*} 
& HF^*(H_{\lambda}),
}
$$
in particular $E^-QH^*(Y) \to QH^*(Y)$ is surjective and corresponds to quotienting by $uE^-QH^*(Y)$.
\end{lm}
\begin{proof}
    This follows by \cref{Lemma ev0 commutative diagram general case}, using \cref{Prop equivariant formality for QH}.
\end{proof}

\subsection{The family of filtrations, and the slice polynomial}

\begin{de}
Abbreviate
$QH^*_{\Fi^a}(Y):=E^-QH^*(Y)$ to emphasize that the $\ku$-module $E^-QH^*(Y)$ depends on the $\Fi^a$-action on $Y$, but not on $b$ (the loop-action is trivial on constant $1$-orbits).
\end{de}

Abbreviate $E_{\lambda}^-:=E^-HF^*(H_{\lambda})$, $T_{\lambda}=\mathrm{Torsion}(E_{\lambda}^-)$, and $c_{\lambda}:=E^-c_{\lambda}^*$.

Abbreviate by $\FF_{j}^{\lambda}$ the filtration on 
$QH^*_{\Fi^a}(Y)$ induced by $c_{\lambda}:QH^*_{\Fi^a}(Y) \to E_{\lambda}^-$, so explicitly:
$$\FF_{j}^{\lambda}:=c_{\lambda}^{-1}(F_j(E^-HF^*(H_{\lambda}))), \textrm{ where }F_j(E^-HF^*(H_{\lambda}))\textrm{ is the canonical }u\textrm{-filtration}.
$$
Above it is understood that $\lambda>0$ is generic, whereas for general $p\in \R$ we let $\FF_j^p:=\cap_{\lambda>p}\FF_j^{\lambda}.$

\begin{thm}\label{Corollary curlyE is zero}
$QH^*_{\Fi^a}(Y)$ is finitely $u$-filtered by $\FF_{j}^{\lambda}$ (\cref{Definition finitely ufiltered}), which satisfies
    \begin{equation}
    \FF_{j}^{\lambda}=\ker \left(QH^*_{\Fi^a}(Y) \stackrel{c_{\lambda}}{\longrightarrow} E^-_{\lambda}/(u^jE^-_{\lambda}+T_{\lambda})\right), \quad \FF_{0}^{\lambda}=QH^*_{\Fi^a}(Y), \quad \FF_{\infty}^{\lambda}=\ker c_{\lambda},
    \end{equation}
   $u\FF_j^{\lambda}\subset \FF_{j+1}^{\lambda} = \FF_{j}^{\lambda} \cap uQH^*_{\Fi^a}(Y)  \subset \FF_j^{\lambda}$, and $\FF_j^{\lambda}\subset \FF_j^{\lambda'}$ for $\lambda\leq \lambda'$. 
    
This yields a {\bf $\lambda$-valuation} $\mathrm{val}^{\lambda}(v)$ on $QH^*_{\Fi^a}(Y)$ satisfying \eqref{Equation relating valuation and plus theory}, and we deduce \cref{Theorem relation between F fil and E fil}.

    We obtain finite invariants $0\leq d_j^{\lambda} \leq a:=\mathrm{rank}_{\ku}QH^*_{\Fi^a}(Y)=\dim_{\k}H^*(Y;\k)$ for each $j\in \N$, 
    $$d_{j}^{\lambda}:=\dim_{\k} \FF_{j}^{\lambda}/u\FF_{j+1}^{\lambda},$$
    yielding a {\bf slice series} $s_{\lambda}(t)$ (which becomes a polynomial when $c_{\lambda}$ is injective),
    $$
    s_{\lambda}(t) := \sum_{j\geq 0} d_j^{\lambda} t^j.
    $$
    Moreover, ordinary cohomology $P:=H^*(Y;\k)$ is filtered by $$F_j^{\lambda}(P):=\FF_{j}^{\lambda}/u\FF_{j-1}^{\lambda}, \quad F_{\infty}^{\lambda}(P):=0,$$ with associated Hilbert-Poincar\'{e} series $f_{\lambda}(t)$ whose coefficients are $f_j^{\lambda}=d_j^{\lambda}-d_{j+1}^{\lambda}$.
\end{thm}
\begin{proof}
In the notation from \cref{Definition finitely ufiltered}, $P:=E^-QH^*(Y)/uE^-QH^*(Y)\cong H^*(Y;\k)$ (using \cref{Lemma MBF in equiv case and localisation}) is a finite dimensional $\k$-vector space, so \cref{Prop Fjc is finitely u filtered} implies $E^-QH^*(Y)$ is finitely $u$-filtered by $\FF_j^{\lambda}$. 
The claim about the valuation and \cref{Theorem relation between F fil and E fil} follow from \cref{Lemma relation between filtrations Vminus and Vplus} and \cref{Cor relation to EEp}.
 The properties of $s_{\lambda}$ follow from \cref{Corollary filtration polynomial}.
The final claim follows by 
\cref{Prop about filtration on V mod uV} (where $T=0$ by \cref{Prop equivariant formality for QH}, so $\dim_{\k}P=\mathrm{rank}_{\ku}\,E^-QH^*(Y)$).
\end{proof}

Following \cref{Subsection filtrations on the W modules},
we have filtrations on $E^{-}HF^*(H_{\lambda})$, $E^{\infty}HF^*(H_{\lambda})$ and $(E^+HF^*(H_{\lambda}))_{\mathbb{F}}$, which induce filtrations on $E^{\karo}QH^*(Y)$ respectively for $\karo=-,\infty,+$, by \cref{Definition induced filtration via c maps}.
Denote these
$$
\quad \FF_{j,(a,b)}^{\karo,\lambda} \subset E^{\karo}_{(a,b)} QH^*(Y)
\quad \textrm{(often abbreviated }
\FF_j^{\karo,\lambda}, \textrm{ also abbreviate }\FF_j^{\lambda}:=\FF_{j}^{-,\lambda} \textrm{ and } \EE_j^{\lambda}:=\FF_{j}^{+,\lambda}),
$$
where $(a,b)$ is the weight. Intersecting over all generic $\lambda\geq p$, we define $\FF_{j,(a,b)}^{\karo,p}$ for $p\in [0,\infty]$.

The three models of equivariant quantum cohomology fit into a short exact sequence by \cref{Lemma SES for various versions of equiv coh}.

\begin{cor}
    The short exact sequence for equivariant quantum cohomologies is filtered:
    \begin{equation}\label{Equation exact seq for Fj of V modules in QH case}
0\longrightarrow \FF_{j-1}^{p}[-2] \stackrel{u}{\longrightarrow} \FF_j^{\infty,p} \longrightarrow \EE_j^{p} \longrightarrow 0,
\end{equation}
These filtrations satisfy the relations described in \cref{Subsection Relation between the filtrations}.
\end{cor}
\begin{proof}
     This follows by \cref{Lemma SES for various versions of equiv coh} and \cref{Lemma SES on filtered V modules}.
\end{proof}

\subsection{Vanishing and non-vanishing theorems}
\label{Subsection vanishing results equiv case}

\begin{cor}[Non-vanishing result]
$SH^*(Y)\neq 0 \Leftrightarrow E^+SH^*(Y)\neq 0$.\\
\indent
This occurs for many monotone examples from \cite{R14,R16}, such as $\mathcal{O}_{\C P^m}(-k)$ for $k=1,2,\ldots,m$. 
\end{cor}
\begin{proof}
    This follows by \cref{Cor ordW zero iff Wplus zero}.
\end{proof}

\begin{cor}[Vanishing result]\label{cor vanishing of E+ for SH}
If $SH^*(Y)=0$ (e.g.\;this holds if $c_1(Y)=0$), then
    $$E^+SH^*(Y)=0\;\;\;\;\textrm{ and }\;\;\;\;E^-SH^*(Y)\cong E^{\infty} SH^*(Y).$$
In particular $E^-SH^*(Y)$ is a torsion-free $\ku$-module, as $E^{\infty} SH^*(Y)$ is a $\kuu$-module.
\end{cor}
\begin{proof}
    This follows by \cref{Lemma vanishing of the E+ version} and \cref{Cor LES for ESH abstractly}.
    If $c_1(Y)=0$, this can be proved more directly as in \cite{RZ1}: the Hamiltonian $kH$ has $1$-orbits of arbitrarily negative grading for large $k$. As the grading of $\mathbb{F}$ is bounded above by $0$, it follows that $E^+HF^*(H_{k^+})$ is supported in arbitrarily negative degrees for large $k$, in particular it vanishes in the direct limit as $k\to \infty.$
    A third proof (although theoretically related to the second) is to use \cref{ESH is direct limit of maps on EQH}:
 $E^+HF^*(H_{N^+})$ lies in the same degrees as $E^+QH^*(Y)[2N\mu]$, but the latter lies in degrees $\leq -2N\mu + \dim_{\R} Y$.
\end{proof}

\begin{cor}\label{Corollary torsion freeness of E- and Einfty}
Assume $c_1(Y)=0$.
Given $q\in \Z$, for sufficiently large slope $k$ we have 
$$E^+HF^m(H_{k^+})=0 \quad \textrm{ and } \quad E^-HF^{m-2}(H_{k^+})\cong E^{\infty} HF^m(H_{k^+}) \qquad \textrm{ for all }m\geq q,$$
 in particular $E^-HF^{\geq q-2}(H_{k^+})$ is torsion-free.
\\[1mm]
\emph{Remark.} 
We will see later that
$E^{\infty} HF^m(H_{k^+}) \cong E^{\infty} SH^m(Y)$ whenever $k>\tfrac{a}{b}$. If $(a,b)$ is a free weight (\cref{Definition free weight}) then $E^{\infty} HF^m(H_{k^+}) \cong E^{\infty} SH^*(Y)\cong E^{\infty}QH^m(Y)$.
\end{cor}
\begin{proof}
By the second proof of \cref{cor vanishing of E+ for SH},
$E^+HF^{\geq q}(H_{k^+})=0$ for all sufficiently large $k$ (as $q$ is fixed).
The rest follows by the long exact sequence, see \cref{Lemma LES for W modules}.
The Remark will be a consequence of \cref{Prop vanishing of column in infinity spectral sequence} and
\cref{Prop vanishing of spectral sequence for infinity page}.
\end{proof}
\subsection{The weak localisation theorem}

\begin{lm}\label{Lemma weak localisation}
$E^{-}SH^*(Y)_u \cong
E^{\infty}SH^*(Y)$, indeed the $u$-localisation of $c^*: E^-QH^*(Y)\to E^{-}SH^*(Y)$ is canonically identifiable with $c^*:E^{\infty}QH^*(Y) \to E^{\infty}SH^*(Y)$.

    If $c_1(Y)=0$, then $E^{-}SH^*(Y)_u \cong E^-SH^*(Y) \cong E^{\infty}SH^*(Y)$.
\end{lm}
\begin{proof}
    Localising at $u$ makes the last vertical line in \cref{Lemma SES for various versions of equiv coh} vanish, using the following: localisation commutes with colimits since it is a left adjoint; $E^+HF^*(H_{\lambda})$ is $u$-torsion so $E^+HF^*(H_{\lambda})_u=0$; thus $E^+SH^*(Y)_u=0$ and so its submodule $\varinjlim (E^+HF^*(H_{\lambda})_{\mathbb{F}})_u$ also vanishes. 
    
    The final claim follows from this, by combining with \cref{cor vanishing of E+ for SH}.
\end{proof}

\subsection{The localisation theorem}

Recall $(a,b)$ is the weight from \cref{Subsection precise equivariant action}.

\begin{prop}\label{Prop vanishing of column in infinity spectral sequence}
Consider the $E_1$-page of the spectral sequence for $E^{\infty}SH^*(Y)$.
The zeroth column is $E^{\infty}QH^*(Y)$.
Any higher column arising with $S^1$-period $p\neq \tfrac{a}{b}$ must vanish.
In particular, $E^{\infty} HF^m(H_{k^+}) \cong E^{\infty} SH^m(Y)$ whenever $k>\tfrac{a}{b}$.
\end{prop}
\begin{proof} 
We first recall the Borel localisation theorem: for any $S^1$-action on a smooth compact manifold $B$, the inclusion of the fixed locus $\mathrm{Fix}\hookrightarrow B$ induces an isomorphism: $$\textrm{modulo }H^*_{S^1}(\mathrm{pt})\textrm{-torsion we have }\;\;H^*_{S^1}(B)\cong H^*_{S^1}(\mathrm{Fix})\cong H^*(\mathrm{Fix})\otimes H^*_{S^1}(\mathrm{pt}).$$
Identify $H^*_{S^1}(\mathrm{pt})=\k[u]$, and call $\k(u)$ its $u$-localisation. Then the above becomes
$$H^*_{S^1}(B) \otimes_{\k[u]} \k(u)\cong H^*(\mathrm{Fix})\otimes_{\k} \k(u),$$
i.e.\;the $u$-torsion in $H^*_{S^1}(B)$ gets killed. The $u$-completion is
$H^*_{S^1}(B) \otimes_{\k[u]} \kuu\cong H^*(\mathrm{Fix})\otimes_{\k} \kuu$.

We now apply the Borel localisation theorem to the action \eqref{Equation action on loops} on any (closed) {\MBF} manifold $B_{p,\c}$ of non-constant $S^1$-orbits in $Y$ of period $p=c'(H(B_{p,\c}))$ (see \cite{RZ1} for a detailed discussion of these manifolds).
By the assumption on the slope $p$, \cref{Lemma fixed points of the mixed action} implies that this action has an empty fixed locus. Thus 
$$H^*_{S^1}(B_{p,\c}) \otimes_{\k[u]} \kuu=0.$$
In the remainder of the proof, we will be referring to two tricks:\\
{\bf Trick 1.} $u$-localisation is a flat functor, so the $u$-localisation of a cohomology module is the cohomology of the $u$-localised cochain complex.
\\
{\bf Trick 2.} $u$-localisation commutes with direct limits, since $u$-localisation is a left adjoint, and left adjoint functors always commute with colimits.

By Trick 2, $E^-SH^*(Y)_u\cong \varinjlim E^-HF^*(H_{\lambda})_u$. Observe that the equivariant {\MBF} spectral sequence converging to $E^-HF^*(H_{\lambda})$ collapses at some finite page number since only finitely many columns of the spectral sequence are non-zero -- notice that because we reduced the problem to $H_{\lambda}$, we have an upper bound $p<\lambda$ on the $S^1$-periods $p$ for the Morse-Bott submanifolds $B_{p,\c}$ of Hamiltonian $1$-orbits that are involved in the chain complex $E^-CF^*(H_{\lambda})$. 
As each page is the cohomology of the previous page, we may apply Trick 1 finitely many times to deduce that we can $u$-localise the $E_1$-page to obtain a spectral sequence converging to $E^-HF^*(H_{\lambda})_u$.

As explained in \cite{RZ2}, each column of the $E_1$-page arises from a so-called energy spectral sequence starting with the (equivariant) cohomology of the {\MB}-manifolds of $1$-orbits $B_{p,\c}$ relevant to the given column. In our setup, that means it starts with the $\ku$-module $H^*_{S^1}(B_{p,\c})$ (there is also a degree shift, but we omit this for simplicity, as it does not affect the argument). 
That energy spectral sequence
converges on a finite page number: this is essentially a consequence of the fact that $H^*_{S^1}(B_{p,\c})$ is a finitely generated $\ku$-module, but it also involves an observation arising from Gromov compactness. The detailed proof in the non-equivariant case was carried out in \cite[Prop.B.6]{RZ2}, and the same argument applies here.
By applying Trick $1$ finitely many times, we deduce that we may $u$-localise $H^*_{S^1}(B_{p,\c})$. By the Borel localisation argument above, $H^*_{S^1}(B_{p,\c})_u=0$ for all non-zero slopes $p\neq \tfrac{a}{b}$.

The final statement follows because the spectral sequence for $E^{\infty}HF^*(H_{k^+})$ is the one for $E^{\infty}SH^*(Y)$ after omitting the columns with slopes $p>k$.
If $k>\tfrac{a}{b}$ then we just saw that those new columns vanish on the $E_1$-page, so that omission causes no change.
\end{proof}

\begin{prop}\label{Prop vanishing of spectral sequence for infinity page}
Let $(a,b)$ be a free weight (\cref{Definition free weight}). 
All columns in the $E_1$-page of the spectral sequence for $E^{\infty}SH^*(Y)$ vanish except the zeroth column, $E^{\infty}QH^*(Y)$. 

Thus we have an isomorphism $$c^*:E^{\infty}QH^*(Y)\to E^{\infty}SH^*(Y).$$
\end{prop}
\begin{proof} 
By \cref{Lemma fixed points of the mixed action} and \cref{Prop vanishing of column in infinity spectral sequence}, only the zeroth column in the $E_1$-page is non-zero, and it is $E^{\infty}QH^*(Y)$.
Thus $E^{\infty}QH^*(Y)\cong E^{\infty}SH^*(Y)$.
\end{proof}

\begin{rmk}
The proof also implies that in the $E_1$-page of the spectral sequence for $E^-SH^*(Y)$,
each column (except the zeroth, $E^-QH^*(Y)$) is a finite dimensional $\k$-vector space.\footnote{
The $u$-torsion $\k[u]$-module $H^*_{S^1}(B_{p,\beta})$ from the proof is a finite dimensional $\k$-vector space.
Indeed, a Morse-theoretic chain model for $H^*_{S^1}(B_{p,\beta})$ is the free $\k[u]$-module generated by critical points of a Morse function. There are finitely many critical points as $B_{p,\beta}$ is compact. Thus, $H^*_{S^1}(B_{p,\beta})$ is a finitely generated $\k[u]$-module. 
Consider a finite generating set.
As the module is $u$-torsion, multiplying generators by powers of $u$ yields a finite $\k$-linear spanning set.}
However, the $u$-action on the $E_1$-page does not usually agree with the $u$-action on the limit of the spectral sequence.
\end{rmk}

\begin{rmk}[Non-free weights]\label{Remark non-free weight MB contributions}
The case $(a,b)=(0,0)$ is not interesting: this corresponds to applying $\cdot \otimes_{\k}E^{\karo}H^*(\mathrm{point})$ to the non-equivariant Floer cohomologies. The case $(a,0)$ for $a\neq 0$ is free.

Consider a non-free $(a,b)\neq (0,0)$, so $b\neq 0$. Let $p=a/b$. Then there are Morse-Bott manifolds $B_{p,\beta}$ of $1$-orbits, arising for slope $p$. They consist entirely of fixed points of the action \eqref{Equation action on loops}, by \cref{Lemma fixed points of the mixed action}.
By Borel's theorem, their $E^{\infty}$-equivariant cohomology is
$$H^*_{S^1}(B_{p,\beta}) \otimes_{\ku} \kuu\cong H^*(B_{p,\beta})\otimes_{\k} \kuu.$$
These contribute in the column for slope $p$ in the $E_1$-page for $E^{\infty}SH^*(Y)$. Just like in the non-equivariant case explained in \cite{RZ2}, one must actually consider the energy spectral sequence in that column, whose $E_1$-page is $\oplus_{\beta} H^*(B_{p,\beta})(\!(u)\!)$: some of these classes may kill each other due to local Floer trajectories not detected by Morse theory. The surviving classes make up the column for slope $p$ of the $E_1$-page for $E^{\infty}SH^*(Y)$. These unexpected classes could kill something in the zeroth column (on the relevant page of the spectral sequence admitting an edge differential pointing to the zeroth column); otherwise they survive to the limit of the spectral sequence. A priori, the $\kuu$-vector spaces $E^{\infty}SH^*(Y)$ and $E^{\infty}QH^*(Y)$ may therefore differ in dimension, but at most by $\sum_{\beta} \dim_{\k} H^*(B_{p,\beta})$.
\end{rmk}

\begin{thm}[Localisation theorem]\label{Lemma MBF in equiv case and localisation}\label{Theorem localisation theorem}
Let $(a,b)$ be a free weight. The following diagram of isomorphisms commutes, where the top row is the $u$-localisation of $E^-c^*$.
    $$
\begin{tikzcd}[column sep=0.6in]
E^-QH^*(Y)_u 
 \arrow[r, "(E^-c^*)_u","\cong"'] 
 \arrow[d, "\cong","u\otimes 1"']
 &
 E^{-} SH^*(Y)_u  
\arrow[d, "\cong"',"u\otimes 1"]
\\
E^{\infty}QH^*(Y)
\arrow[r, "E^{\infty}c^*","\cong"'] 
& E^{\infty} SH^*(Y)
\end{tikzcd}
$$
\end{thm}
\begin{proof}
The commutative diagram is the first square in \cref{Cor SES over LES} in the notation \eqref{Equation explicit notation equivariant quantum and Floer}.
The left-vertical isomorphism follows by \cref{Prop equivariant formality for QH}.
The right-vertical isomorphism follows by \cref{Lemma weak localisation}. 
The lower-horizontal map is an isomorphism by \cref{Prop vanishing of spectral sequence for infinity page}.
The claim follows.
\end{proof}

\subsection{Injectivity, and periodicity}

We emphasize that we do not assume $c_1(Y)=0$. Even in the $c_1(Y)=0$ case (where much is known about the direct limits by \cref{cor vanishing of E+ for SH}), we will get a torsion-freeness result that is stronger than \cref{Corollary torsion freeness of E- and Einfty}.

\begin{thm}[Injectivity theorem]\label{Corollary E minus c star maps are injective}
Let $(a,b)$ be a free weight. The following maps are injective,
$$
c^*_{\lambda}: E^-QH^*(Y) \to E^-HF^*(H_{\lambda}) \quad\textrm{and}\quad
c^*: E^-QH^*(Y) \to E^-SH^*(Y).
$$
\end{thm}
\begin{proof}
This follows by \cref{Proposition when iso after localise get finite invariants}, using \cref{Lemma MBF in equiv case and localisation}.
\end{proof}

\begin{rmk}
For other weights,
\cref{Lemma MBF in equiv case and localisation}/\cref{Corollary E minus c star maps are injective}
hold if  $E^{\infty}c^*:E^{\infty}QH^*(Y)\to E^{\infty}SH^*(Y)$ is an isomorphism.
Indeed $\ker E^{\infty}c^*\neq 0 \Rightarrow \ker (E^-c^*)_u\neq 0\Rightarrow \ker c^*\neq 0$ in \cref{Corollary E minus c star maps are injective}.
\end{rmk}

\begin{prop}[Periodicity property]\label{Prop Periodicity property}
Keeping track of the $(a,b)$-weight, for any model $\karo\in \{-,\infty,+\}$, abbreviating by $\mu$ the Maslov index of the $S^1$-action $\Fi$, we have an isomorphism
$$
\mathcal{S}:E_{(a,b)}^{\karo} HF^*(H_{\lambda+1})\cong E_{(a-b,b)}^{\karo} HF^*(H_{\lambda})[2\mu].
$$
These isomorphisms are compatible with continuation maps, see \eqref{Equation compatibility of seidel iso with continuation}, and with the filtrations.
\end{prop}
\begin{proof}
This is an immediate consequence of \cref{Theorem Seidel iso}, and \cref{Cor Seidel isos preserve filtration}.
\end{proof}

In \cref{Properties of the filtration polynomial} we will revisit the above isomorphisms in greater detail.

\begin{cor}[Periodicity property 2]\label{Prop Periodicity property 2}
For $k\in \N$, we have an isomorphism of $\ku$-modules
\begin{equation}\label{Equation from HFN to QH shifted}
 \mathcal{S}^{k}: E_{(a,b)}^{\karo}HF^*(H_{k^+}) \cong E_{(a-kb,b)}^{\karo}QH^*(Y)[2k\mu],
\end{equation}
such that the continuation map $\psi_k$ from slope $k^+$ to $(k+1)^+$, for weight $(a,b)$, becomes the map 
\begin{equation}\label{Equation defining rab maps}
 r_{a-kb,b}:=\mathcal{S}^{k+1} \circ \psi_k \circ \mathcal{S}^{-k} : E_{(a-kb,b)}^{\karo} QH^*(Y)[2k\mu]
\to
E_{(a-(k+1)b,b)}^{\karo} QH^*(Y)[2(k+1)\mu].
\end{equation}
This map satisfies $r_{a-kb,b}=\mathcal{S}\circ c_{1^+}^*$ where $c_{1^+}^*: E_{(a-kb,b)}^{\karo} QH^*(Y)\to  E_{(a-kb,b)}^{\karo}HF^*(H_{1^+})$. 
\\ When $b=0$, the map $r_{a-kb,b}=r_{a,0}=\mathcal{S}\circ \psi_0$ is independent of $k$.

The $u^0$-part of \eqref{Equation defining rab maps} is the non-equivariant rotation map $r: QH^*(Y)\to QH^{*}(Y)[2\mu]$ from Ritter \cite{R14}: it is
quantum multiplication by a class $Q_{\Fi}\in QH^{2\mu}(Y)$ described in \cite{RZ1} (cf.\;\cref{Remark QFi class}).
\end{cor}
\begin{proof}
Apply $k$ times \cref{Prop Periodicity property}:
$
S_{\Fi}^k:E_{(a,b)}^{\karo}HF^*(H_{k^{+}})\cong E_{(a-kb,b)}^{\karo}HF^*(H_{0^+})[2k\mu],
$
then use the isomorphism $c^*_{0^+}:E^{\karo}_{(a-kb,b)}QH^*(Y)\to E^{\karo}_{(a-kb,b)}HF^*(H_{0^+})$ (we simplified the notation for $(c^*_{0^+})^{-1}\circ \mathcal{S}_{\Fi}^k$ in \eqref{Equation from HFN to QH shifted}).
Via \eqref{Equation from HFN to QH shifted}, the continuation map $\psi_k:E_{(a,b)}^{\karo}HF^*(H_{k^{+}})\to E_{(a,b)}^{\karo}HF^*(H_{(k+1)^{+}})$ becomes the map \eqref{Equation defining rab maps}.
The compatibility \eqref{Equation compatibility of seidel iso with continuation} implies $\mathcal{S}^{k}\circ \psi_k \circ \mathcal{S}^{-k} = c_{1^+}^*$. The second claim follows. The claim about the case $b=0$ is then immediate. The final claim follows by definition of $r$ \cite{R14,RZ1}, observing that the $u^0$-part of an equivariant chain map is the non-equivariant version of the map.
\end{proof}

\begin{rmk}\label{Remark QFi class}
We note $Q_{\Fi}\in QH^{2\mu}(Y)$, and it satisfies $Q_{\Fi^k}=Q_{\Fi}^k$.
If $c_1(Y)=0$, this class is nilpotent for degree reasons (see \cite{R14}), unlike the maps in \eqref{Equation sequence of maps that give direct lim} which are typically injective (\cref{Theorem injectivity theorem 2}).
We recall that $Q_{\Fi}$ was computed in a family of examples in \cite[Proposition 1.32]{RZ1}: for $Y$ K\"{a}hler with $c_1(Y)\in \R_{\geq 0}[\omega]$, assuming all non-zero weights of $\F_{\min}:=\min H$ are $1$ (equivalently, $\mu=\mathrm{codim}_{\C}\F_{\min}$), 
\begin{equation}\label{Equation QFi class}
Q_{\Fi}=\mathrm{PD}[\F_{\min}]+T^{>0}\textrm{-terms}\in QH^{2\mu}(Y).
\end{equation}
Here, $[\F_{\min}]$ is a locally finite cycle; its Poincar\'{e} dual $\mathrm{PD}[\F_{\min}]$ is the Euler class of the normal bundle of $\F_{\min}\subset Y$.
For example, $Q_{\Fi}\neq 0$ for all Conical Symplectic Resolutions of weight $1$  \cite[Ex.1.33]{RZ1}.
\end{rmk} 

\begin{thm}[Injectivity theorem 2]\label{Theorem injectivity theorem 2}
Recall $(a,b)$ is the weight. 
Let $k\in \N$, let $\lambda\in (0,\infty)$ be a generic slope with $k<\lambda$.
The following continuation maps are injective:\footnote{When $b=0$, $a\neq 0$, the weight $(a,0)$ is free, and all the maps in (1)--(3) are injective, for all $k,\lambda$. When $(a,b)=(0,0)$ no claim is being made: the equivariant theory becomes the non-equivariant theory up to tensoring with $\otimes_{\k} E^{\karo}H^*(\mathrm{point})$.}
\begin{enumerate}
    \item $c_{\lambda,k^+}^*: E^-HF^*(H_{k^+}) \hookrightarrow E^-HF^*(H_{\lambda})$ provided $\tfrac{a}{b}\notin (k,\lambda)$;

    \item $c^*_{\infty,k^+}: E^-HF^*(H_{k^+}) \hookrightarrow E^-SH^*(Y)$
provided $\tfrac{a}{b}\notin (k,\infty)$,

\item $c_{\lambda}^*: E^-QH^*(Y)
\to E^-HF^*(H_{\lambda})$
provided $\tfrac{a}{b}\notin (0,\lambda)$,

    \item if $(a,b)$ is a free weight, then all three maps above are injective.

\item if $E^-HF^*(H_{p^+})$ is torsion-free, then 
(1) and (2) also hold for $k$ replaced by $p\in [0,\infty]$. 
\end{enumerate}
\end{thm}
\begin{proof}
  We prove this by applying \cref{Proposition when iso after localise get finite invariants} to $V=E^-HF^*(H_{k^+})$ (which is free by \cref{Prop Periodicity property 2}) and $W_{\lambda}=E^-HF^*(H_{\lambda})$. Then $V_u\cong E^{\infty}HF^*(H_{k})$ (\cref{Lemma ulocalising W- is Winfty} and \eqref{defn F C^+ and W^+}). The $E_1$-page for $E^{\infty}HF^*(H_{k^+})$ is obtained from the $E_1$-page for $E^{\infty}SH^*(Y)$ after discarding columns involving $S^1$-periods $p>k^+$.
  A similar observation applies to  $E^{\infty}HF^*(H_{\lambda})$ and $p>\lambda$.
  Also, $c_{\lambda,k^+}$ corresponds to the inclusion of columns on $E_1$-pages.
 \cref{Prop vanishing of column in infinity spectral sequence} implies that all columns with period $p\neq \tfrac{a}{b}$ will vanish.
The assumption $\tfrac{a}{b}\notin (k,\lambda)$ ensures that those inclusions of $E_1$-pages become identity maps after $u$-localisation (and edge differentials on $E_1$ and higher pages will agree).
 Therefore the $u$-localisation of $c_{\lambda,k^+}$ is an isomorphism.
 Thus \cref{Proposition when iso after localise get finite invariants} applies, and claim (1) follows. 
 For claim (2) we apply the same argument, but we use $W=E^-SH^*(Y)$.
 For claim (3) we use $V=E^-QH^*(Y)$ (the zeroth column of the $E_1$-page) and $W=E^-HF^*(H_{\lambda})$.
 When $(a,b)$ is a free weight, all higher columns vanish under $u$-localisation by \cref{Prop vanishing of spectral sequence for infinity page}, so no condition on $\tfrac{a}{b}$ is needed.
 Claim (5) follows by the same proof as (1),(2) since it only used that $E^-HF^*(H_{k^+})$ was free.
  \end{proof}

\subsection{Torsion-freeness}

\begin{thm}[Torsion-freeness 1]\label{Theorem torsion freeness of Eminus HF} 
 $E^-HF^*(H_{k^+})$ is torsion-free for any $k\in \N$, and we have a SES:
$$
0 \longrightarrow E^-HF^*(H_{k^+})[-2] \stackrel{u\otimes 1}{\longrightarrow} E^{\infty}HF^*(H_{k^+}) \longrightarrow 
E^+HF^*(H_{k^+}) \longrightarrow 0.
$$
\end{thm}
\begin{proof}
The first claim follows from \cref{Prop Periodicity property 2} combined with \cref{Prop equivariant formality for QH}.
The SES follows from the LES in \cref{Lemma LES for W modules}, using that $E^-HF^*(H_{k^+})$ is free.
\end{proof}

  \begin{rmk}
Note the above result is only for a ``full rotation'' slope $k\in \N$. We have not been able to exclude the existence of torsion in $E^-HF^*(H_{p^+})$ for $S^1$-periods $0<p<1$. By \cref{Prop Periodicity property}, if $E^-HF^*(H_{p^+})$ had torsion, then it would repeat in $E^-HF^*(H_{(k+p)^+})$ for all $k\in \N$, and the torsion would map to zero via the continuation $E^-HF^*(H_{(k+p)^+})\to E^-HF^*(H_{(k+1+p)^+})$ since it factors through the torsion-free $HF^*(H_{(k+1)^+})$.
In particular, this torsion lies in $\ker c_{\infty,p^+}^*$.
\end{rmk}

\begin{cor}\label{Corollary torsion freeness 1 free case}
Let $(a,b)$ be a free weight, and $k\in \N$. There is a short exact sequence 
$$
0 \longrightarrow E^-HF^*(H_{k^+})[-2] \longrightarrow E^{\infty}QH^*(Y) \longrightarrow 
E^+HF^*(H_{k^+}) \longrightarrow 0.
$$
So $E^-HF^*(H_{k^+})[-2]$ is a free submodule of $E^{\infty}QH^*(Y)$ (whereas  
$E^+HF^*(H_{k^+})$ is always $u$-torsion).

$E^{+}c^*_{k^+}: E^+QH^*(Y)\to E^+HF^*(H_{k^+})$ is surjective and induces an isomorphism
\begin{equation}\label{Equation description of EplusHF when free}
\begin{split}
& E^+HF^*(H_{k^+})\cong E^+QH^*(Y)/\EE^k_1, \;\textrm{ where} \\
& \EE^k_1:=\ker E^+c^*_{k^+} \cong \mathrm{coker}\left(E^-c^*_{k^+}:E^-QH^*(Y)\hookrightarrow E^-HF^*(H_{k^+})\right)[2].
\end{split}
\end{equation}
\end{cor}
\begin{proof}
The first claim follows from \cref{Theorem torsion freeness of Eminus HF}, using \cref{Prop vanishing of spectral sequence for infinity page}.
The surjectivity of $E^+c^*$ follows by factorising the surjection $E^{\infty}QH^*(Y)\cong  E^{\infty}HF^*(H_{k^+}) \to E^+HF^*(H_{k^+})$ through $E^+QH^*(Y)$.
We can identify $E^+QH^*(Y)$ with the cokernel of $q: E^-QH^*(Y)[-2] \stackrel{u\otimes 1}{\hookrightarrow}  E^{\infty}QH^*(Y)$, which can be factorised as a composite of injections: 
$$q:E^-QH^*(Y)[-2]\stackrel{\;\;E^-c^*_{k^+}}{\hookrightarrow} E^-HF^*(H_{k^+})[-2] \stackrel{u\otimes 1}{\hookrightarrow} E^{\infty}HF^*(H_{k^+}) \cong E^{\infty}QH^*(Y).$$
Quotienting this sequence of maps by $E^-QH^*(Y)$ (and its injective images), yields an injection
$E^-HF^*(H_{k^+})/E^-QH^*(Y) \hookrightarrow \mathrm{coker}\,q[2] \cong E^+QH^*(Y)[2]$. By the SES in the first claim, the image of that map is precisely $\ker E^+c^*_{k^+}$.
\end{proof}

\begin{thm}[Torsion-freeness 2]\label{Theorem torsion freeness 2}
     In \cref{Lemma SES for various versions of equiv coh}, 
     $$\mathrm{Torsion}(E^-SH^*(Y))=0 \quad \textrm{ and }\quad E^+SH^*(Y)=E^+_{\mathbb{F}\mathrm{lim}}SH^*(Y)=\varinjlim E^+HF^*(H_{\lambda})_{\mathbb{F}},$$ and the LES from
     \cref{Cor LES for ESH abstractly} becomes a short exact sequence:
$$
0 \longrightarrow E^-SH^*(Y)[-2] \stackrel{u\otimes 1}{\longrightarrow} E^{\infty}SH^*(Y) \longrightarrow 
E^+SH^*(Y) \longrightarrow 0.
$$
where $u\otimes 1$ is $u$-localisation followed by $u$-multiplication.
\end{thm}
\begin{proof}
The torsion-freeness of the direct limit $E^-SH^*(Y)$ follows from \cref{Theorem torsion freeness of Eminus HF}, because any given torsion class would have a torsion class representative in  
$E^-HF^*(H_{k^+})$ for sufficiently large $k\in \N$. The rest follows from\footnote{more directly, 
\cref{Lemma LES for W modules} becomes a SES because the connecting map $E^+SH^*(Y) \to E^-SH^*(Y)[-1]$ must vanish since $E^+SH^*(Y)$ is torsion (being a direct limit of torsion modules) whereas $E^-SH^*(Y)$ is free.} \cref{Lemma SES for various versions of equiv coh}.
\end{proof}

\begin{cor}\label{Corollary torsion freeness 2 free case}
Let $(a,b)$ be a free weight. There is a short exact sequence 
$$
0 \longrightarrow E^-SH^*(Y)[-2] \longrightarrow E^{\infty}QH^*(Y) \longrightarrow 
E^+SH^*(Y) \longrightarrow 0.
$$
So $E^-SH^*(Y)[-2]$ is a free submodule of $E^{\infty}QH^*(Y)$ (whereas  
$E^+SH^*(Y)$ is always $u$-torsion).

$E^{+}c^*: E^+QH^*(Y)\to E^+SH^*(Y)$ is surjective and induces an isomorphism
\begin{equation}\label{Equation description of EplusHF when free in direct limit}
\begin{split}
& E^+SH^*(Y)\cong E^+QH^*(Y)/\EE^{\infty}_1, \;\textrm{ where} \\
& \EE^{\infty}_1:=\ker E^+c^* \cong \mathrm{coker}\left(E^-c^*:E^-QH^*(Y)\hookrightarrow E^-SH^*(Y)\right)[2]
\end{split}
\end{equation}
\end{cor}
\begin{proof}
This follows from \cref{Theorem torsion freeness 2} by the same argument used to deduce  \cref {Corollary torsion freeness 1 free case} from \cref{Theorem torsion freeness of Eminus HF}, by replacing the symbols $HF^*(H_{k^+})$ by $SH^*(Y)$, and $c^*_{k^+}$ by $c^*$.
\end{proof}

\subsection{Positive equivariant symplectic cohomology}
From the quotient complex of $E^-CF^*(H_{\lambda})$, by quotienting out the subcomplex generated by constant $1$-orbits, and then taking the direct limit in $\lambda$, one obtains {\bf $S^1$-equivariant positive symplectic cohomology}, denoted $E^-SH^*_{\myplus}(Y)$. 
This is analogous to the non-equivariant construction carried out in \cite{RZ1,RZ2}.
That same quotient complex, via \eqref{defn F C^+ and W^+}, defines $E^{\karo}SH^*_{\myplus}(Y)$ for all $\karo\in \{-,\infty,+\}.$
The long exact sequence induced by the inclusion of the subcomplex is 
\begin{equation}\label{Equation LES for ESH plus}
\cdots \to E^{\karo}SH^{*}_{\myplus}(Y)[-1]\to E^{\karo}QH^*(Y) \stackrel{c^*}{\to} E^{\karo}SH^*(Y)\to E^{\karo}SH^{*}_{\myplus}(Y)\to \cdots
\end{equation}

By \cref{Cor LES for ESH abstractly}, we also obtain the long exact sequence between the three models:
\begin{equation}\label{Equation three models SH plus LES}
\cdots \longrightarrow
 E^-SH^*_{\myplus}(Y)[-2]
\stackrel{u\otimes 1}{\longrightarrow}
E^{\infty} SH^*_{\myplus}(Y)
\longrightarrow
 E^{+} SH^*_{\myplus}(Y) 
 \longrightarrow
 E^{-} SH^*_{\myplus}(Y)[-1] 
 \longrightarrow
 \cdots
\end{equation}
We now rephrase \cref{Theorem localisation theorem}. Recall $\EE^{\infty}_1:=\ker ([c_u]:E^+QH^*(Y)\to E^+SH^*(Y))$, see \eqref{Definition of Fjinfty}.

\begin{cor}\label{Cor ESH plus discussion}
    Let $(a,b)$ be a free weight. Then we have $\ku$-module isomorphisms 
\begin{equation}
    \label{Equation positive ESH result}
    E^{\infty}SH^*_{\myplus}(Y)=0
    \quad\textrm{ and }\quad
    E^{+}SH^*_{\myplus}(Y)\cong E^{-}SH^{*}_{\myplus}(Y)[-1]
        \quad\textrm{ and }\quad
  E^{+}SH^*_{\myplus}(Y)\cong
   \EE^{\infty}_1[1].
   \end{equation}
    In particular, $E^-SH^*_{\myplus}(Y)$ is $u$-torsion.
The LES \eqref{Equation LES for ESH plus} becomes the SES of $\ku$-modules: 
\begin{equation}\label{Equation direct sum decompositions of EplusSH}
0\to E^-QH^*(Y)  \stackrel{c^*}{\to} E^-SH^*(Y) \to E^-SH^{*}_{\myplus}(Y) \to 0,
\end{equation}
which as $\k$-vector spaces\footnote{Not as $\ku$-modules: e.g.\;in CY examples, $E^-SH^*(Y)$ is torsion-free, whereas $E^{-}SH^{*}_{\myplus}(Y)$ is torsion.} yields a non-canonical isomorphism $E^-SH^*(Y) \cong H^*(Y)[\![u]\!]\oplus E^-SH^{*}_{\myplus}(Y)$.

If $SH^*(Y)=0$ (e.g.\,if $c_1(Y)=0$), then $E^+SH^*(Y)=0$, $E^-SH^*(Y)\cong E^{\infty}QH^*(Y)\cong E^{\infty}SH^*(Y)$, 
$$E^+SH^*_{\myplus}(Y)\cong E^+QH^{*}(Y)[1] \quad \textrm{ and } \quad  E^-SH^*_{\myplus}(Y)\cong E^+QH^{*}(Y)[2].$$
\end{cor}
\begin{proof}
The spectral sequence for $E^{\infty}SH^*_{\myplus}(Y)$ is obtained by the one for $E^{\infty}SH^*(Y)$ by omitting the zeroth column. Therefore \cref{Prop vanishing of spectral sequence for infinity page} implies $E^{\infty}SH^*_{\myplus}(Y)=0$ in \eqref{Equation positive ESH result}. Then \eqref{Equation three models SH plus LES} implies the second equation  in \eqref{Equation positive ESH result}.
Taking $\karo=+$ in \eqref{Equation LES for ESH plus} yields the exact sequence
$$
0 \to E^+QH^*(Y)/\EE^{\infty}_1 \to E^+SH^*(Y) \to E^+SH^*_{\myplus}(Y) \to \EE^{\infty}_1[1] \to 0.
$$
By \cref{Corollary torsion freeness 2 free case} the map to $E^+SH^*(Y)$ is surjective, hence the third equation in \eqref{Equation positive ESH result} follows.

As $E^+SH^*_{\myplus}(Y)$ arises as a direct limit, any class has a representative in some $E^+HF^*_{\myplus}(H_{k^+})$. The latter arises from a chain complex which is a direct sum of copies of $\mathbb{F}$: so already at chain level any element is $u$-nilpotent. By \eqref{Equation positive ESH result}, also $E^-SH^*_{\myplus}(Y)$ is $u$-torsion.

 The injectivity of $c^*$ from \cref{Corollary E minus c star maps are injective} implies that \eqref{Equation LES for ESH plus} becomes \eqref{Equation direct sum decompositions of EplusSH}.

For the final claim, we combine \cref{cor vanishing of E+ for SH} and \cref{Theorem localisation theorem}, then we use \eqref{Equation LES for ESH plus} for $\karo=+$ to compute $E^+SH^*_{\myplus}(Y)$, which in turn computes $E^-SH^*_{\myplus}(Y)$ by using the second equation in \eqref{Equation positive ESH result}.
 \end{proof}

\section{Properties of the slice series, the invariant factors, and the filtration polynomial}\label{Properties of the filtration polynomial}
\subsection{Rephrasing the direct limit}

\begin{cor}\label{ESH is direct limit of maps on EQH}
$E_{(a,b)}^{\karo}SH^*(Y)=\varinjlim E_{(a,b)}^{\karo}HF^*(H_{k^+})$ is isomorphic to the direct limit of $\ku$-modules 
\begin{equation}\label{Equation sequence of maps that give direct lim}
\begin{tikzcd}[row sep=large, column sep=large]
E_{(a,b)}^{\karo} QH^*(Y)
\arrow[r, "r_{a,b}"] 
&
E_{(a-b,b)}^{\karo} QH^*(Y)[2\mu]
\arrow[r, "{r_{a-b,b}}"] 
&
E_{(a-2b,b)}^{\karo} QH^*(Y)[4\mu]
\arrow[r, "{r_{a-2b,b}}"] 
&
\cdots
\end{tikzcd}
\end{equation}
using the maps from \eqref{Equation defining rab maps}.
Moreover, \eqref{Equation sequence of maps that give direct lim} can be identified non-canonically with a direct limit  
\begin{equation}\label{Equation simplifying the direct limit}
H^*(Y)\otimes_{\k} \ku
\to
H^*(Y)\otimes_{\k} u^{-\mu}\ku
\to
H^*(Y)\otimes_{\k} u^{-2\mu}\ku
\to
\cdots
\end{equation}
where these maps are typically not the naive inclusions of $\ku$-submodules (cf.\;\cref{Remark example of slice dimensions})

When $(a,b)$ is a free weight, the maps in \eqref{Equation sequence of maps that give direct lim} and \eqref{Equation simplifying the direct limit} are injective.
\end{cor}
\begin{proof}
We apply \cref{Prop Periodicity property 2}.
Pass to non-equivariant quantum cohomology via \cref{Lemma ev0 commutative diagram}, $\mathrm{ev}_0: E_{(a-kb,b)}^{\karo} QH^*(Y) \to QH^*(Y)$. We obtain the non-equivariant analogue of \eqref{Equation sequence of maps that give direct lim} 
from \cite{R14,RZ1}:
\begin{equation}
QH^*(Y)
\stackrel{Q_{\Fi}}{\longrightarrow}
 QH^*(Y)[2\mu]
\stackrel{Q_{\Fi}}{\longrightarrow}
QH^*(Y)[4\mu]
\stackrel{Q_{\Fi}}{\longrightarrow}
\cdots
\end{equation}
To get \eqref{Equation simplifying the direct limit} from \eqref{Equation sequence of maps that give direct lim}, we use the non-canonical identification from \cref{Prop equivariant formality for QH},
 $$E_{(a-kb,b)}^{\karo} QH^*(Y)[2k\mu]\cong H^*(Y)\otimes_{\k} u^{-k}\ku\subset H^*(Y)\otimes_{\k} \kuu.$$
 The final claim follows from \cref{Theorem injectivity theorem 2}.(4).
\end{proof}

\begin{ex}
For $Y=\C$ and a free weight, we will see  $E^-HF^*(H_{k^+})=\ku\cdot [x_k]$, with grading $|x_k|=-2k$ (\cref{Prop ESH of C}).
In \cref{ESH is direct limit of maps on EQH}, the continuation $E^-HF^*(H_{0^+})\to E^-HF^*(H_{1^+})$ is first viewed as a map
$E^-QH^*(\C)\to  E^-QH^*(\C)[-2]$, and then as a new map $\ku \to u^{-1}\ku$ via 
$$\ku\cong \ku\cdot [x_0] \to \ku \cdot [x_1]\cong \ku \cdot [x_0][2]\cong u^{-1}\ku \subset \kuu,$$ 
where $u^{-1}\cdot 1$ corresponds to a multiple of $[x_1]$.
In the $[x_0],[x_1]$ bases, 
$r_{a,b}=n_{a,b} u:\ku \to \ku$, where $n_{a,b}\in \k^{\times}$ (although $r=0$, as $Q_{\Fi}=0$ for degree reasons, we know $r_{a,b}$ is injective by \cref{Theorem injectivity theorem 2}). The new map becomes $u^{-1}r_{a,b}=n_{a,b}: \ku \to u^{-1}\ku$, so it is an inclusion up to $\k^{\times}$-rescaling.
The units $n_{a,b}$ play a non-trivial role when working over $\Z$ \cite[Thm.8.1]{liebenschutz2020intertwining} (cf.\,\cref{Subsection C The general free weight case}).

Similarly, $E^-HF^*(H_{k^+})\cong u^{-k}\ku$, and $\varinjlim u^{-k}\ku \cong \kuu$ recovers $E^-SH^*(\C)\cong E^{\infty}QH^*(\C)$.
\end{ex}

\subsection{The Hamiltonians $k^+H$}
\label{Subsection The Hamiltonians NH}

Recall that $E^-HF^*(H_{k^+})$ depends on a choice of Hamiltonian $H_{k^+}$ of slope $k^+$ at infinity. Changes in this choice correspond to unique continuation isomorphisms, which are grading-preserving $\ku$-module isomorphisms.
We now choose $k^+H$, so the only $1$-orbits are constant orbits at points of the fixed locus $\F=\sqcup \F_{\a}$ (cf.\,\cite{RZ1}).
In the {\MBF} model, the chain level generators are critial points $x_{\a,\beta}$ of auxiliary Morse functions $f_{\a}:\F_{\a} \to \R$ (cf.\,\cite{RZ2}).

The Seidel isomorphisms \eqref{Equation Seidel iso} give an identification
\begin{equation}
 \label{Equation Seidel map for kH Hams}
\mathcal{S}^k:
E^-_{(a,b)}HF^*(k^+H)
\to E^-_{(a-kb,b)}HF^*(0^+H)[2k\mu]\equiv
E^-_{(a-kb,b)}QH^*(Y)[2k\mu],
\end{equation}
which at chain level is the identity map shifted in grading by $2k\mu$: the constant orbits $x_{\a,\beta}$ with their constant caps satisfy \cref{Definition of canonical lift in terms of caps}, noting that in the target their grading has been shifted down by $2k\mu$. 
Unpacking the notation from \cref{ESH is direct limit of maps on EQH}: 
$$
r_{a,b} = \mathcal{S}\circ c_{1^+}:
E^-_{(a,b)}QH^*(Y) \to 
E^-_{(a-b,b)}QH^*(Y)[2\mu], 
$$
where $c_{1^+}$ is the continuation map $E^-_{(a,b)}QH^*(Y)\cong E^-_{(a,b)}HF^*(0^+H) \to E^-_{(a,b)}HF^*(1^+H)$.

We make two basic observations: \begin{enumerate}
    \item The Seidel isomorphism \eqref{Equation Seidel map for kH Hams} 
    is the identity map, shifted in grading, so it does not introduce negative powers of $T,u$.
    \item The continuation map $c_k^+:E^-_{(a,b)}QH^*(Y)\cong E^-_{(a,b)}HF^*(0^+H) \to E^-_{(a,b)}HF^*(k^+H)$ for $k\in \N$ can be constructed using a homotopy $H_s=\lambda_s H$ which is monotone,\footnote{$\partial_sH_s=(\partial_s \lambda_s)H\leq 0$ since we (may) assume $H\geq 0$ and $\partial_s\lambda_s\leq 0.$} so Floer continuation solutions are counted with non-negative powers of $T$.
\end{enumerate}
\begin{de}\label{Definition ERk maps definition rotation}
The continuation map $c_{k^+}:E^-_{(a,b)}QH^*(Y) \to E^-_{(a,b)} HF^*(H_{k^+})$, suitably composed with Seidel isomorphisms $\mathcal{S}:=\mathcal{S}_{\Fi}$ from \eqref{Equation Seidel iso 2}, defines a $\ku$-linear map:
\begin{equation}\label{Equation ERk map}
ER_k:=\mathcal{S}^k\circ c_{k^+} = r_{a-(k-1)b,b} \circ \cdots \circ r_{a,b}
: E^-_{(a,b)}QH^*(Y) \to E^-_{(a-kb,b)}QH^*(Y)[2k\mu].
\end{equation}
We call these {\bf rotation maps}, to distinguish them from the Seidel isomorphisms: they need not be isomorphisms, nor injective in general, since we composed Seidel isomorphisms with continuation maps.

The same definition is used also for the other models $\karo\in \{-,\infty,+\}$: $$ER_k^{\karo}:=\mathcal{S}^k\circ E^{\karo}c_{k^+} = r_{a-(k-1)b,b}^{\karo} \circ \cdots \circ r_{a,b}^{\karo}
: E^{\karo}_{(a,b)}QH^*(Y) \to E^{\karo}_{(a-kb,b)}QH^*(Y)[2k\mu].$$
\end{de}

\begin{cor}\label{Lemma slice poly for ERk is same as for ck}
The following properties hold:
\begin{enumerate}
    \item If $(a,b)$ is a free weight, then $ER_k$ is injective.
    \item The slice series $s_k$ for $E^-c_{k^+}$ is the same as the slice series for $ER_k$.
    \item The $u^0$-part of $ER_k^{\karo}$ is $r^k: QH^*(Y)\to QH^{*}(Y)[2k\mu]$.
\end{enumerate}
\end{cor}
\begin{proof}
(1) follows from \cref{ESH is direct limit of maps on EQH}; (2) follows immediately because $ER_k,E^-c_{k^+}$ differ by applying an isomorphism $\mathcal{S}^k$; and (3) follows by \cref{Prop Periodicity property 2}.
\end{proof}

Consider bases for $E^-_{(a,b)}QH^*(Y)$ as described in \cref{Cor equiv formality via spectral sequence}.
We abusively refer to it as a basis $y_i$ for $H^*(Y)$, even though the $\ku$-module isomorphism $H^*(Y)[\![u]\!]\cong E^-_{(a,b)}QH^*(Y)$ at chain level means that the $y_i$ may have $u^{\geq 1}$-correction terms that depend on $(a,b)$ (in the same grading as $y_i$). The matrix for $ER_k$ in such bases cannot involve negative powers of $T,u$. 

\begin{rmk}[Calabi-Yau and Montone cases]\label{Remark CY and monotone cases ERk matrix} 
The $2k\mu$ grading often imposes further constraints:
\begin{enumerate}
    \item When $c_1(Y)=0$, the matrix for $ER_k$ has $(i,j)$ entry $k_{ij} u^{m}$, where
    $m=(2k\mu-|y_i|+|y_j|)/2$ is required to be in $\N$ (otherwise the entry is zero), and $k_{ij}\in \k$ involves only $T^{\geq 0}$-terms. Note that for $c_1(Y)=0$ the Novikov variable $T$ lies in grading $|T|=0$, so $\k$ lies in grading zero.
    \item When $c_1(Y)=a [\omega]$ for $a>0\in \R$, we place the Novikov variable $T$ in grading $|T|=2a$. Abbreviate $t:=T^{1/a},$ so $|t|=2$. This time, the $(i,j)$ entry is a $\mathbb{B}$-linear combination of the monomials $t^m,t^{m-1}u,\ldots,u^m$ where $m=(2k\mu-|y_i|+|y_j|)/2$, again requiring $m\in \N$.
\end{enumerate}
\end{rmk}

\subsection{Growth rate of the slice polynomial (free weight case)}\label{Subsection Growth rate of the filtration polynomial}

Assume $(a,b)$ is a free weight.
Then
we are in the setting of \cref{Corollary E minus c star maps are injective} and \cref{Corollary torsion freeness 2 free case}, so \cref{Prop equivariant formality for QH} yields:
\begin{equation}\label{Equation succession of injections}
\ku^r \cong E^-QH^*(Y) 
\!\hookrightarrow\! 
E^-HF^*(H_{k^+})
\!\hookrightarrow\! 
E^-SH^*(Y) 
\!\hookrightarrow\! 
E^{\infty}SH^*(Y)\cong E^{\infty}QH^*(Y)\cong \kuu^r,
\end{equation}
where $r=\dim_{\k} H^*(Y)$.
By \cref{Cor structure theorem 2}, for $k\in \N$,
\begin{equation}\label{Equation SH case of direct limit structure theorem 2}
\begin{split}
E^-HF^*(H_{k^+}) &\cong u^{-j_1(k)}\ku \oplus \cdots \oplus u^{-j_{r}(k)} \ku,
\\
E^-SH^*(Y) &\cong u^{-j_1}\ku \oplus \cdots \oplus u^{-j_{r-s}} \ku
\oplus \kuu^s,
    \end{split}
\end{equation}
so that $E^-QH^*(Y)\cong \ku^r$ is identified with the submodules $\ku^r$ up to isomorphism, and where 
$$0\leq j_1(k)\leq \cdots \leq j_r(k),\quad 
j_i(k)\leq j_i(k+1),
\quad 
\lim_{k\to \infty} j_i(k)=j_i \textrm{ for }i\leq r-s,
\quad 
\lim_{k\to \infty} j_i(k)=\infty \textrm{ for }i> r-s.
$$
The invariant factors of the continuation map $c_{k^+}:E^-QH^*(Y) \hookrightarrow E^-HF^*(H_{k^+})$ are $(u^{j_1(k)},\ldots,u^{j_r(k)})$. These are well-defined invariants as $c_{k^+}$ is a homomorphism of free $\ku$-modules (cf.\;\cref{Subsection Basic prelimiaries about ku}).

The slice polynomial $s_k$ determined by the continuation map $c_{k^+}$ satisfies
$$
s_k = \sum d_j(k)t^j, 
\qquad 
d_j(k)=\#\{i:-j_i \leq -j\},
\qquad
0\leq d_j(k)\leq d_j(k+1) \leq r,
\qquad
\lim_{k\to \infty} d_j(k) = d_j.
$$
As $k\to \infty$, the $s_k$ limit to a series $\sum d_j t^j$ with tail end $s(t^{j_{r-s}+1}+t^{j_{r-s}+2}+\cdots)$. 
This series is a polynomial precisely when $s=0$, i.e.\;when $E^-SH^*(Y)$ is a free $\ku$-module.

In the notation from \cref{Definition slice dimensions} and \cref{Lemma dimension calculation for sum of ds},
\begin{equation}\label{Equation SH slice dims}
\begin{split} 
d_j(k) & = \dim_{-j} E^-HF^*(H_{k^+}) = \dim_{\k} \FF_j^{k}/u\FF_{j-1}^{k},
\\
d_j &= \dim_{-j} E^-SH^*(Y), \qquad (\textrm{in particular }d_0=r)
\\
s &=\dim_{-\infty} E^-SH^*(Y),
    \end{split}
\end{equation}
By \cref{Corollary torsion freeness 2 free case},
\begin{equation}
\EE^k_1 = \ker (E^+c_{k^+}:E^+QH^*(Y)\to E^+HF^*(H_{k^+})) \cong \mathrm{im}(u\otimes 1)/u\ku^r,
\end{equation}
where $u\otimes 1: E^-HF^*(H_{k^+})\hookrightarrow E^{\infty}HF^*(H_{k^+})\cong E^{\infty}QH^*(Y)$.

By \eqref{Equation description of EplusHF when free}, \cref{Cor sum of coeffs of filtration poly} and \cref{Lemma slice dimensions give coeffs of filtration polynomial},
\begin{equation}\label{Equation formula for Ek1}
\begin{split}
\dim_{\k} \EE^k_1 &= \dim_{\k} \, E^-HF^*(H_{k^+})/E^-QH^*(Y)
\\
&=
s_k(0)
\\
&=
d_1(k)+d_2(k)+\cdots 
\\ &=
j_1(k) + j_2(k) + \cdots + j_r(k). 
\end{split}
\end{equation}
In greater detail, \eqref{Equation description of EplusHF when free} yields 
\begin{equation}
\begin{split}
 \EE^k_1 & \cong
\left( E^-HF^*(H_{k^+})/ E^-QH^*(Y)\right)[2] 
\\
&
 \cong u\cdot E^-HF^*(H_{k^+})/ u\cdot E^-QH^*(Y)
 \\
&\cong
 (u^{1-j_1(k)}\ku \oplus \cdots \oplus u^{1-j_{r}(k)} \ku)/u\ku^r,
 \\
 E^+HF^*(H_{k^+})
 &\cong E^+QH^*(Y)/\EE^k_1 
  \\
 &\cong 
 \mathbb{F}^r/ (u^{1-j_1(k)}\ku \oplus \cdots \oplus u^{1-j_{r}}(k))
\\
&\cong
\mathbb{F}[2j_1(k)]
\oplus
\mathbb{F}[2j_2(k)]
\oplus 
\cdots
\oplus
\mathbb{F}[2j_r(k)].
\end{split}
\end{equation}
and more precisely it yields isomorphisms
\begin{equation}
\begin{split}
 \EE^k_1 & \cong
\left( E^-HF^*(H_{k^+})/ E^-QH^*(Y)\right)[2] 
\\
&
 \cong u\cdot E^-HF^*(H_{k^+})/ u\cdot E^-QH^*(Y)
 \\
&\cong
 (u^{1-j_1(k)}\ku \oplus \cdots \oplus u^{1-j_{r}(k)} \ku)/u\ku^r,
 \\
 E^+HF^*(H_{k^+})
 &\cong E^+QH^*(Y)/\EE^k_1 
  \\
 &\cong 
 \mathbb{F}^r/ (u^{1-j_1(k)}\ku \oplus \cdots \oplus u^{1-j_{r}}(k))
\\
&\cong
\mathbb{F}[2j_1(k)]
\oplus
\mathbb{F}[2j_2(k)]
\oplus 
\cdots
\oplus
\mathbb{F}[2j_r(k)].
\end{split}
\end{equation}

Recall $\EE^{\infty}_1=\ker (E^+QH^*(Y)\to E^+SH^*(Y))$, so $\dim_{\k} \EE^{\infty}_1=\lim \dim_{\k} \EE^k_1=d_1+d_2+\cdots$. Thus, $\dim_{\k} \EE^{\infty}_1=\infty$ for $s\neq 0$, whereas for $s=0$:
$$
\dim_{\k} \EE^{\infty}_1= 
d_1+d_2+\cdots = j_1+j_2+\cdots + j_r.
$$
Finally, observe that \eqref{Equation SH case of direct limit structure theorem 2} and the SES in \cref{Corollary torsion freeness 2 free case} imply
$$
E^+SH^*(Y) \cong E^{\infty}SH^*(Y)/(u\otimes 1)(E^-SH^*(Y)) \cong \mathbb{F}[2j_1]\oplus \mathbb{F}[2j_2]\oplus \cdots \oplus \mathbb{F}[2j_{r-s}],
$$
where the latter is the $\ku$-module $\mathbb{F}^{r-s}$ but the $u^0\in \mathbb{F}$ in the $m$-th coordinate of $\mathbb{F}^{r-s}$ arises from $[u^{-j_m}]\in \kuu/u\cdot u^{-j_m}\ku$, so actually lies in grading $-2j_m$, yielding the shifted summand $\mathbb{F}[2j_m]$.
\begin{rmk}\label{Remark example of slice dimensions}
$E^-HF^*(H_{k^+})\cong H^*(Y)\otimes_{\k} u^{-2k\mu}\ku$ as $\ku$-modules (cf.\,\cref{ESH is direct limit of maps on EQH}), but the continuation from $H_{k^+}$ to $H_{(k+1)^+}$ may not be an inclusion $H^*(Y)\otimes_{\k} u^{-2k\mu}\ku\hookrightarrow H^*(Y)\otimes_{\k} u^{-2(k+1)\mu}\ku$,
so $d_n(k)=\dim_{-n}\,E^-HF^*(H_{k^+})$ are not just a naive Betti number count. 
The following picture is the $\ku$-module $A:=E^-HF^*(H_{0^+})\cong E^-QH^*(T^*\C P^1)\cong \ku\oplus \ku[-2]$: the fat dots generate $H^*(T^*\C P^1)$ as a $\k$-vector space, the dots above are their images under the $u$-action, and the dots below are the two new generators needed for $B:=E^-HF^*(H_{1^+})\cong E^-QH^*(T^*\C P^1)[2]\cong \ku[2]\oplus \ku$.
The numbers are the $\Z$-grading. Depending on how $A$ injects into $B$, we either get invariant factors $(1,u^2)$ (so $j_1=0$, $j_2=2$) or $(u,u)$ (so $j_1=j_2=1$).
The $(u,u)$ case is the naive summand-preserving inclusion $\ku\oplus \ku[-2] \subset
\ku[2]\oplus \ku$. 
We will show that $(1,u^2)$ occurs for all free weights $(a,b)$. 
\begin{center}
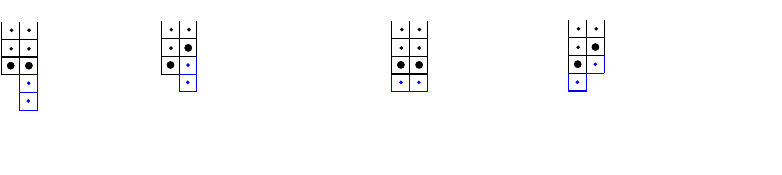
\end{center}
\end{rmk}

\subsection{Growth rate of the slice series (non-free weight case)}\label{Subsection Growth rate of the filtration polynomial nonfree weight case}

Consider the case discussed in \cref{Remark non-free weight MB contributions}: $(a,b)\neq (0,0)$ which is not free, so $p=a/b$ causes contributions from Morse-Bott manifolds $B_{p,\beta}$ at slope $p$ which can cause $E^{\infty}c^*:E^{\infty}QH^*(Y)\to E^{\infty}SH^*(Y)$ to fail being injective and/or surjective. We apply the machinery from \cref{Subsection Non-injective homomorphisms}.
Only one of the continuation maps 
$$
E^-QH^*(Y)[2(k-1)\mu] \cong E^- HF^*(H_{(k-1)^+}) \to E^-HF^*(H_{k^+}) \cong E^-QH^*(Y)[2k\mu]
$$
could fail to be injective: the one for $k=\wp:=\lfloor p \rfloor$, and it plays the role of the map $Q_p$ in \cref{Lemma presentation Rj inverse of V noninjective case direct limit} (in addition, we will see grading shifts of $2k\mu=2\wp \mu$ appearing below because of the grading shifts of the $E^-QH^*(Y)$ above). We deduce from \cref{Lemma presentation Rj inverse of V noninjective case direct limit} that, for $k\geq p$,
\begin{equation*}
\begin{split}
& E^-HF^*(H_{k^+}) \cong \left(u^{-j_1(k)}\ku \oplus \cdots \oplus u^{-j_{r-n}(k)} \ku\right)
\oplus \left(u^{-\ell_1(k)}\ku \oplus \cdots \oplus u^{-\ell_{n}(k)} \ku\right)[2\wp\mu],
\\
& E^-\!SH^*(Y) \!\cong\! \left( u^{-j_1}\ku \!\oplus\! \cdots \!\oplus\! u^{-j_{r-n-s}} \ku
\!\oplus\! \kuu^s\right)
\!\oplus\! 
\Big(u^{-\ell_1}\ku \!\oplus\! \cdots\! \oplus\! u^{-\ell_{n-s'}} \ku
\!\oplus\! \kuu^{s'}\Big)[2\wp\mu],
\\
& E^{\infty}SH^*(Y) \cong E^-SH^*(Y)_u \cong \kuu^{r-n} \oplus \kuu^n[2\wp \mu]
\cong
\kuu^r,
    \end{split}
\end{equation*}
where $r=\mathrm{rank}\,E^-QH^*(Y)=\dim_{\k}H^*(Y)$, and $n=\mathrm{rank}\ker E^-c_{p^+}^*=\dim_{\kuu} \ker E^{\infty}c^*$.

Viewing the above presentation of $E^-HF^*(H_{k^+})$ as embedded in $\kuu^{r-n}\oplus \kuu^{n}$, there is an identification $E^-QH^*(Y)\cong \ku^{r-n}\oplus \ku^n$ so that $E^-c_{k^+}^*:E^-QH^*(Y) \to E^-HF^*(H_{k^+})$ becomes 
\begin{equation}\label{Equation nonfree weight computation}
\mathrm{incl} \oplus 0: \;\;\ku^{r-n}\oplus \ku^{n}
\;\;\stackrel{\mathrm{incl} \oplus 0}{\longrightarrow}
\;\;\bigoplus_{i=1}^{r-n} u^{-j_i(k)}\ku \oplus \bigoplus_{i=1}^{n} u^{-\ell_i(k)}\ku[2\wp \mu]
\;\;\subset\;\;
\kuu^{r-n}\oplus \kuu^n.
\end{equation}
Localising at $u$, we get that $E^{\infty}c^*:E^{\infty}QH^*(Y)\to E^{\infty}SH^*(Y)$ gets identified with the map
$$
\mathrm{id} \oplus 0: \;\;\kuu^{r-n}\oplus \kuu^{n}
\;\;\stackrel{\mathrm{id} \oplus 0}{\longrightarrow}
\;\; \kuu^{r-n}\oplus \kuu^{n} \cong E^{\infty}SH^*(Y).
$$
By \cref{Theorem torsion freeness 2}, $E^{+}SH^*(Y)$ is the cokernel of the map $u\otimes 1: E^-SH^*(Y)\to E^{\infty}SH^*(Y)$, so 
$$
E^{+}SH^*(Y) \cong 
\mathbb{F}[2j_1] \oplus \cdots \oplus \mathbb{F}[2j_{r-n-s}]
\oplus
\mathbb{F}[2\ell_1+2\wp \mu]
\oplus
\cdots 
\oplus
\mathbb{F}[2\ell_{n-s'}+2\wp \mu],$$
and the map $E^{+}c^*:E^+QH^*(Y)\to E^+SH^*(Y)$ becomes the obvious one,%
\footnote{it vanishes on the $\mathbb{F}^s$ and $\mathbb{F}^{s'}$ summands, and the ``obvious map'' $\mathbb{F}\to \mathbb{F}[2m]$ is the one sending $u^{-m+1},\ldots,u^0$ to zero.}
after we use the induced identification $E^+QH^*(Y)=E^{\infty}QH^*(Y)/uE^-QH^*(Y)\cong \mathbb{F}^{r-n-s}\oplus \mathbb{F}^s \oplus \mathbb{F}^{n-s'}\oplus \mathbb{F}^{s'}$. 

See \cref{Example nonfree weight for Cn} for an explicit example ($Y=\C^n$ with $(a,b)=(1,1)$).

The slice series (see \cref{Corollary filtration polynomial}) for the map $E^-c_{k^+}^*$ can be read off from \eqref{Equation nonfree weight computation}:
$$
s_k(t) = r\cdot 1 \;\;+\;\; \sum_{j=1}^{j_{r-n}(k)} (n+\#\{i: j\leq j_i(k)\})\,t^j + \quad t^n + t^n + \cdots
$$
\subsection{Dependence of the slice series on the $(a,b)$-weight}\label{Subsection Dependence of the filtration polynomial on the (a,b)-weight}

Denote by $s_{\lambda}^{(a,b)}=\sum d_{j}^{(a,b)}(\lambda)\,t^j$ the slice series for $c_{\lambda}^{(a,b)}:E^-_{(a,b)}QH^*(Y)\to E^-_{(a,b)}HF^*(H_{\lambda})$ (\cref{Corollary filtration polynomial}), where we now keep track of $(a,b)$ in the notation. By \eqref{Equation compatibility of seidel iso with continuation}, we obtain a commutative diagram:
\begin{equation}\label{Equation compatibility of seidel iso with continuation application to filtration poly}
\begin{tikzcd}[column sep=0.6in]
E^-_{(a,b)}QH^*(Y)
 \ar[rr, "c_{\lambda}^{(a,b)}", bend left=10]
 \arrow[r, "c_{1^+}^{(a,b)}"'] 
 &
E^{-}_{(a,b)}HF^*(H_{1^+})
 \arrow[r, "\textrm{continuation}"'] 
 \arrow[d, "\cong","\mathcal{S}_{\Fi}"']
 &
E^{-}_{(a,b)}HF^*(H_{\lambda})
\arrow[d, "\cong","\mathcal{S}_{\Fi}"']
\\
&
E^{-}_{(a-b,b)}QH^*(Y)[2\mu]
 \arrow[r, "c_{\lambda-1}^{(a-b,b)}"]  
 &
E^{-}_{(a-b,b)}HF^*(H_{\lambda-1})[2\mu]
\end{tikzcd}
\end{equation}

This diagram establishes constraints about how much $s_{\lambda}^{(a,b)}$ can differ from $s_{\lambda-1}^{(a-b,b)}$. 
There is no clean statement, because the invariant factors of a composition of two maps are not determined by the invariant factors of the individual maps.
Using the $j_i(k)$-notation from \cref{Subsection Growth rate of the filtration polynomial}, 
$$
j_i^{(a,b)}(\lambda) = j_i^{(a-b,b)}(\lambda-1) + \gamma_i,
\qquad\textrm{ where }\qquad \alpha:=j_1^{(a,b)}(1) \leq \gamma_i \leq j_r^{(a,b)}(1):=\beta,
$$
in words: the least \& biggest invariant factors of $c_{1^+}^{(a,b)}$ give crude bounds on the invariant factors of the top composite horizonal map.
It follows that $d_{\alpha+j}^{(a,b)}(\lambda) \geq d_{j}^{(a-b,b)}(\lambda-1)$ and $d_{\beta+j}^{(a,b)}(\lambda) \leq d_{j}^{(a-b,b)}(\lambda-1)$.

In the case $b=0$, the notation simplifies as all weights in play become $(a,0)$, and we will see in \cref{Subsection weight 10 case product} that for $\lambda=k^+$, $k\in \N$, the map $c_{k^+}$ is the $k$-fold product by a certain class $EQ$. The crude estimate above says that the change in indices is bounded $\alpha\leq j_i(k+1)-j_i(k) \leq \beta$ in terms of the lowest and largest indices for the map given by multiplication by $EQ$.

\subsection{Seidel isomorphisms and rotation maps are compatible with the filtration}\label{Subsection The rotation maps are compatible with the filtration}

\begin{cor}\label{Cor Seidel isos preserve filtration}
 The Seidel isomorphisms in \cref{Theorem Seidel iso} preserve filtrations, e.g.\;\eqref{Equation Seidel iso 2} yields
$$
\mathcal{S}\left( F_j(E_{(a,b)}^{\karo}HF^*(H_{\lambda})) \right)
=
F_j(E_{(a-b,b)}^{\karo}HF^*(H_{\lambda-1})).
$$
\end{cor}
\begin{proof}
This follows by \cref{Lemma valuation on free mod} and \cref{Lemma functoriality Ffiltration}, since $\mathcal{S}:=\mathcal{S}_{\Fi}$ is a $\ku$-module isomorphism.
\end{proof}

Recall from  \cref{Cor Seidel isos preserve filtration} the filtration $\FF_{j,(a,b)}^{\karo,\lambda}\subset E^{\karo}_{(a,b)}QH^*(Y)$.
Recall the maps \eqref{Equation defining rab maps} and \eqref{Equation ERk map}:  
\begin{equation}
\begin{split}
 r_{a,b}&:=\mathcal{S} \circ c_{1^+} : E_{(a,b)}^{\karo} QH^*(Y)
\to
E_{(a-b,b)}^{\karo} QH^*(Y)[2\mu],
\\
ER_k&:=\mathcal{S}^k\circ c_{k^+} = r_{a-(k-1)b,b} \circ \cdots \circ r_{a,b}
: E^-_{(a,b)}QH^*(Y) \to E^-_{(a-kb,b)}QH^*(Y)[2k\mu].
\end{split}
\end{equation}

\begin{prop}
The map $r_{a,b}$ and the rotation map $ER_k$ satisfy
$$
r_{a,b}^{-1}\left(\FF_{j,(a-b,b)}^{\karo,p-1}\right)
=
(\FF_{j,(a,b)}^{\karo,p})[-2\mu]
 \qquad \textrm{ and } \qquad
ER_k^{-1}\left(\FF_{j,(a-kb,b)}^{\karo,p-k}\right)
=
\FF_{j,(a,b)}^{\karo,p}[-2k\mu],
$$
in particular 
$
r_{a,b}:\FF_{j,(a,b)}^{\karo,p}
\to
\FF_{j,(a-b,b)}^{\karo,p-1}[2\mu] \qquad \textrm{ and } \qquad
ER_k:\FF_{j,(a,b)}^{\karo,p}\to
\FF_{j,(a-kb,b)}^{\karo,p-k}[2k\mu].
$
\end{prop}
\begin{proof}

Consider the commutative diagram
$$
\begin{tikzcd}[column sep=0.6in]
E^{\karo}_{(a,b)}QH^*(Y)
\arrow[r, "c_{k^+}"'] 
\arrow[rd, "ER_k"'] 
\arrow[rr,bend left=10,"c_{p^+}"]
&
E^{\karo}_{(a,b)}HF^*(H_{k^+})
 \arrow[r, "\textrm{continuation}"'] 
 \arrow[d, "\cong","\mathcal{S}^k"']
 &
E^{\karo}_{(a,b)}HF^*(H_{p^+})
\arrow[d, "\cong","\mathcal{S}^k"']
\\
&
E^-_{(a-kb,b)}QH^*(Y)[2k\mu]
 \arrow[r, "c_{(p-k)^+}"]  
 &
E^{\karo}_{(a-kb,b)}HF^*(H_{(p-k)^+})[2k\mu]
\end{tikzcd}
$$
where the top three maps commute because  continuation maps are compatible on cohomology;
the triangle commutes by definition of $ER_k$; and the square commutes by  \eqref{Equation compatibility of seidel iso with continuation}.
This diagram preserves filtrations in the following sense:
$$
\begin{tikzcd}[column sep=0.6in]
\FF_{j,(a,b)}^{\karo,p}
\arrow[r, "c_{k^+}"'] 
\arrow[rd, "ER_k"'] 
\arrow[rr,bend left=10,"c_{p^+}"]
&
\mathcal{S}^{-k}(\FF_{j,(a-kb,b)}^{\karo,p-k}[2k\mu])
 \arrow[r, "\textrm{continuation}"'] 
 \arrow[d, "\cong","\mathcal{S}^k"']
 &
F_j(E^{\karo}_{(a,b)}HF^*(H_{p^+}))
\arrow[d, "\cong","\mathcal{S}^k"']
\\
&
\FF_{j,(a-kb,b)}^{\karo,p-k}[2k\mu]
 \arrow[r, "c_{(p-k)^+}"]  
 &
F_j(E^{\karo}_{(a-kb,b)}HF^*(H_{(p-k)^+}))[2k\mu]
\end{tikzcd}
$$
where we start in the bottom-left corner and move anti-clockwise; in the right-vertical arrow we use \cref{Cor Seidel isos preserve filtration} applied $k$ times; and we use the definition $(c_{p^+})^{-1}(F_j(E^{\karo}_{(a,b)}HF^*(H_{p^+})))=\FF_{j,(a,b)}^{\karo,p}$ to deduce the top-left entry, proving the second claim. The first claim is the case $k=1$.
\end{proof}

\subsection{The minimal classes $x_k$ generated by the unit}

The unit $x_0 \in QH^0_{\Fi^a}(Y)=E^-_{(a,b)}QH^0(Y)$ was described in \cref{Cor equiv formality via spectral sequence}.
Although there is such an $x_0\in QH^0_{\Fi^a}(Y)\cong E^-HF^0_{(a,b)}(H_{0^+})$ for each weight $(a,b)$, they have the same representative at chain level and they all map to the unit $1\in QH^0(Y)$ via the specialisation map
$\mathrm{ev}_0$ from \cref{Lemma ev0 commutative diagram}.

We now apply the Seidel isomorphisms from the Appendix to define the 
{\bf minimal classes} $x_k$,
\begin{equation}\label{Definition xk class}
x_k := \mathcal{S}^{-k}(x_0)
\in 
E^-HF^*_{(a,b)}(H_{k^+})[2k\mu],
\end{equation}
by repeatedly using \eqref{Equation Seidel iso 2}: each $\mathcal{S}:=\mathcal{S}_{\Fi}$ changes $(a,b)$, so that $x_0$ actually involved weight $(a-kb,b)$.

If $c_1(Y)=0$, then we have a $\Z$-grading on Floer cohomologies and the Novikov field $\k$ lies in degree zero,
so $x_0,x_k$ can be characterised uniquely, up to $\k^{\times}$-rescaling, as the minimal degree classes of $E^-HF^*_{(a,b)}(H_{0^+})$ and $E^-HF^*_{(a,b)}(H_{k^+})[2k\mu]$ respectively, in degrees $|x_0|=0$, $|x_k|=-2k\mu.$

Recall that the sequence of Hamiltonians $H_{k^+}$ can be built so that the continuation maps are inclusions at chain level (the non-equivariant argument in  \cite{RZ2} holds essentially verbatim for the equivariant case). So we may abusively write $x_j \in E^-HF^*_{(a,b)}(H_{k^+})$ for $k\geq j$, when we actually mean that we applied to $x_j$ the continuation map going from $H_{j^+}$ to $H_{k^+}$.

\begin{lm}\label{Lemma xj are compatible with seidel iso}
    For $x_j\in E^-HF^*_{(a,b)}(H_{k^+})$ we have $x_{j+1}=\mathcal{S}^{-1}(x_j)\in E^-HF^*_{(a+b,b)}(H_{(k+1)^+})$.
\end{lm}
\begin{proof}
By construction, $\mathcal{S}^{-1}x_j=x_{j+1}$ if we view $x_j\in E^-HF^*_{(a,b)}(H_{j^+})$. 
Now apply \eqref{Equation compatibility of seidel iso with continuation}, for $\lambda = j$ and $\lambda'=k-j$, to get the equation $\mathcal{S}^{-1}x_j=x_{j+1}$
 when viewing $x_j\in E^-HF^*_{(a,b)}(H_{k^+})$.
\end{proof}

\begin{cor}
    If $x_0\in \FF_{j,(a,b)}^{p}$,
    then $x_k \in \FF_{j,(a+kb,b)}^{p+k}$.
In particular, if $c_1(Y)=0$ and $b=0$, then
 \begin{equation}\label{Equation filtration relation x0 and xk}
   x_0 \in \FF_{\mu k,(a,b)}^k \Longrightarrow
   x_0 \in \FF_{\mu N k,(a,b)}^{Nk} \; \textrm{ for all }N\geq 1 \in \N, 
   \end{equation}
so in words:  if $x_0$ minimises the valuation at slope $k^+$ then it continues to do so at slope $(Nk)^+$.

If we just assume $c_1(Y)=0$, then
 \begin{equation}\label{Equation filtration relation x0 and xk general version}
   x_0 \in \FF_{\mu k,(a,b)}^k \cap 
   \FF_{\mu k,(a-kb,b)}^k \cap \cdots \cap 
   \FF_{\mu k,(a-Nkb,b)}^k
   \Longrightarrow
   x_0 \in \FF_{\mu N k,(a,b)}^{Nk}.
   \end{equation}
\end{cor}
\begin{proof}
\cref{Cor Seidel isos preserve filtration} implies the first claim, by applying $\mathcal{S}^{-k}$.
For the second claim, we can exploit the characterisation of $x_0,x_k$ as grading-minimisers since $c_1(Y)=0$. The assumption implies $x_0 \in \k^{\times}\cdot u^{\mu k} x_k\subset E^-_{(a,0)}HF^*(H_{k^+})$.
Thus $x_j \in \k^{\times}\cdot u^{\mu k} x_{k+j}\subset  E^-_{(a+kb,b)}HF^*(H_{(k+j)^+})$ by applying $\mathcal{S}^{-j}$.
For $j=k$ we get $x_k=\mathrm{unit}\cdot u^{\mu k} x_{2k}$. When $b=0$, this $x_k$ and the original $x_k$ coincide, as the weight $(a,0)$ agrees. 
So $x_0\in \k^{\times}\cdot u^{\mu k} x_k \subset \k^{\times}\cdot u^{2\mu k} x_{2k}$. Now repeat the argument for $j=2k$, $j=3k$, etc.
When $b\neq 0$, we adjust the argument: we use the assumption
$x_0\in  \FF_{\mu k,(a-mkb,b)}^k$ to ensure that the inductive equation $x_{mk}=\mathrm{unit}\cdot u^{\mu k}x_{(m+1)k}$ holds for weight $(a,b).$
\end{proof}

\section{Product structure in equivariant Floer theory for weights $(a,0)$}
\label{Subsection weight 10 case product}

\subsection{The lack of an equivariant Floer product in general}
For weights $(a,b)$ with $b\neq 0$ there is an obstruction to constructing a pair-of-pants product in equivariant Floer theory. At the domain level one would need an $S^1$-action on a three-punctured sphere which near the three punctures acts with winding number $b,b,-b$ respectively, which is impossible for homotopy reasons unless $b=0$. The weaker requirement of a $QH^*_{S^1}(Y)$-module structure on equivariant Floer cohomology is also obstructed for $b\neq 0$: it requires a non-trivial $S^1$-action on a twice-punctured sphere with an $S^1$-fixed marked point.

\subsection{The Floer theory product in the weight $(a,0)$ case}

\begin{prop}\label{Prop product structure in weight 1 0 case}
For the (free) weight $(a,b)=(a,0),$
the models $\karo\in \{-,\infty\}$ admit a unital pair-of-pants product, compatibly with equivariant quantum product, and additive in grading:
$$
\begin{tikzcd}[column sep=0.3in,row sep=0.2in]
E^{\karo}QH^*(Y) \otimes E^{\karo}QH^*(Y)
 \arrow[r, ""] 
 \arrow[d, "c_{\lambda_1}\otimes c_{\lambda_2}"]
 &
 E^{\karo}QH^*(Y)
 \arrow[d, "c_{\lambda_1+\lambda_2}"]
 &
 E^{\karo}QH^*(Y)^{\otimes 2}
 \arrow[r, ""] 
 \arrow[d, "c\otimes c"]
 &
 E^{\karo}QH^*(Y)
 \arrow[d, "c"]
\\
E^{\karo}HF^*(H_{\lambda_1})
\otimes
E^{\karo}HF^*(H_{\lambda_2})
 \arrow[r, ""] & 
E^{\karo}HF^*(H_{\lambda_1+\lambda_2})
&
E^{\karo}SH^*(Y)^{\otimes 2}
 \arrow[r, ""] & 
E^{\karo}SH^*(Y),
\end{tikzcd}
$$
where the diagram on the right is the direct limit of the diagrams on the left.

$E^{\karo}HF^*(H_{\lambda})$ is an $E^{\karo}QH^*(Y)$-module, and $c^*_{\lambda}:E^{\karo}QH^*(Y) \to E^{\karo}HF^*(H_{\lambda})$ is an $E^{\karo}QH^*(Y)$-module homomorphism.
The direct limit of $c^*_{\lambda}(1)\in E^{\karo}HF^*(H_{\lambda})$ is the multiplicative unit in $E^{\karo}SH^*(Y)$. 

The diagram in \cref{Lemma SES for various versions of equiv coh} is compatible with the above ring structures.

The specialisation diagram in \cref{Lemma ev0 commutative diagram} respects the ring structures.
\end{prop}
\begin{proof}
This is the equivariant analogue of the constructions from \cite{R13}, and the proof is essentially the same up to the additional bookkeeping of the auxiliary equivariant data defining the complexes.

For the second claim, we take $\lambda_1=0^+$ so that $E^{\karo}HF^*(H_{\lambda_1})\cong E^{\karo}QH^*(Y)$, and we take $\lambda_2=\lambda$ (and it is understood that we work with generic slopes $\lambda>0$).

The third claim follows by observing the following detail about the construction for the three models.
Analogously to \eqref{Equation for the differential expanded},
the equivariant pair-of-pants product at chain level is a $\ku$-bilinear map 
$$\mu^2(\cdot,\cdot)=\mu^2_0 + u \mu^2_1 + u^2\mu^2_2+\cdots$$ for $\k$-bilinear maps $\mu^2_j$ not involving $u$, where $\mu^2_0$ is the non-equivariant pair-of-pants product.
As usual, this is constructed first on the original chain complex $C^-:=C^*$, then it is $u$-localised to a map on the chain complex $C^{\infty}:=C^*_u$, and finally it descends to the chain complex $C^+:=C^*_u/uC^*_u.$

The fourth claim follows by noting that the equivariant quantum product in \eqref{Equation equiv quantum product expansion}
is $\ku$-bilinear by construction, and the equivariant $c^*$-map is $\ku$-linear and has an expansion like \eqref{Equation equiv quantum product expansion} with $u^0$-part $c_{\mathrm{ne}}^*$ given by the non-equivariant $c^*$-map. So the $u^0$-part of the $c^*$-image of an equivariant quantum product, $c^*(x \star y)$, is $c_{\mathrm{ne}}^*(x_0 \star_0 y_0)$, where $x_0=\mathrm{ev}_0(x),y_0=\mathrm{ev}_0(y)$.
\end{proof}

\subsection{Computation of $E^-SH^*(Y)$ in the weight $(a,0)$ case}

For the (free) weight $(a,b)=(a,0),$ the maps in \eqref{Equation sequence of maps that give direct lim} are repeated application of the ``same'' injective $\ku$-module homomorphism $r_{a,0}$ up to degree-shifts (cf.\,\cref{Prop Periodicity property 2}), which we now describe.

\begin{prop}\label{Prop a0 case multiplication by EQ}
For $(a,b)=(a,0)$, \eqref{Equation sequence of maps that give direct lim} is equivariant quantum multiplication by a class $EQ$,
$$
E^-QH^{*}(Y)
\stackrel{EQ\,}{\longrightarrow}
E^-QH^{*}(Y)[2\mu]
\stackrel{EQ\,}{\longrightarrow}
E^-QH^{*}(Y)[4\mu]
\stackrel{EQ\,}{\longrightarrow}
\cdots
$$
The element $EQ\in E^-QH^{2\mu}(Y)$ satisfies $EQ^k\neq 0$ for all $k\in \N$, and has the form
$$
EQ = Q_{\Fi} + q u^m \neq 0 \in E^-QH^{2\mu}(Y), \qquad \textrm{ for some }m\geq 1, q\in E^-QH^{2\mu-2m}(Y),
$$
where $Q_{\Fi}$ is the non-equivariant rotation class from \cite{RZ1} (cf.\,\cref{Remark QFi class}). 

The $x_k$ from \eqref{Definition xk class} satisfy $x_k \star x_{\ell}=x_{k+\ell}$.
If $c_1(Y)=0$, then $x_k=n_k\cdot EQ^k$ for some $n_k\in \k^{\times}$.

If $c_1(Y)=0$, then $Q_{\Fi}$ is nilpotent, whereas 
$EQ$ becomes invertible in $E^-QH^*(Y)_u\cong E^{\infty}QH^*(Y)$.
\end{prop}
\begin{proof}
We use a general property of Seidel isomorphisms: in the presence of a pair-of-pants product in Floer theory, the Seidel isomorphism acts as a module isomorphism (so  $\mathcal{S}(x * y) = \mathcal{S}(x)*y$) given by quantum product by the image $\mathcal{S}(x_0)$ of the identity element $x_0$, and $\mathcal{S}(x_0)$ is often called the Seidel element (see \cite[Prop.6.3]{Sei97} and \cite[Thm.23]{R14}).
In our case, $EQ:=\mathcal{S}(x_0).$
Note $EQ^k$ corresponds to $c_{k^+}(x_0) \in E^-_{(a,0)}HF^*(H_{k^+})$, so $EQ^k\neq 0$ follows by injectivity of $c_{k^+}$ (cf.\,\cref{Corollary E minus c star maps are injective}).

Using the property $
\mathcal{S}(x * y) = \mathcal{S}(x)*y,
$
the $x_k$ from \eqref{Definition xk class} satisfy:
$$
x_k * x_{\ell} = \mathcal{S}^{-k}(x_0)*\mathcal{S}^{-\ell}(x_0)
=\mathcal{S}^{-k}(x_0*\mathcal{S}^{-\ell}(x_0))
=\mathcal{S}^{-k}(\mathcal{S}^{-\ell}(x_0))
=\mathcal{S}^{-k-\ell}(x_0)
=x_{k+\ell}.
$$
Suppose $c_1(Y)=0$. Both $x_k$ and $EQ^k$ lie in minimal grading so they differ by $\k^{\times}$-rescaling.
Nilpotency of $Q_{\Fi}$ is a consequence of a grading argument due to Ritter \cite{R14}, see also \cite{RZ1}. The last claim uses that $u$-localisation and $EQ$-localisation agree: $E^-QH^*(Y)_u\cong E^-SH^*(Y) \cong E^-QH^*(Y)_{EQ}$.
\end{proof}

\begin{cor}
$E^-SH^*(Y) \cong E^-QH^*(Y)_{EQ}$, so that $E^-c^*:E^-QH^*(Y) \to E^-SH^*(Y)$ becomes the $EQ$-localisation map $E^-QH^*(Y) \to E^-QH^*(Y)_{EQ}$.
\end{cor}
\begin{proof}
This follows from \cref{Prop a0 case multiplication by EQ} by using \cref{Lemma direct limit of V copies}.
\end{proof}
\subsection{The filtration by ideals in the weight $(a,0)$ case}

\begin{cor}
For $\karo\in \{-,\infty\}$ the filtration
$\FF_{j,(a,0)}^{\karo,p} \subset E^-_{(a,0)}QH^*(Y)$ consists of ideals for equivariant quantum product. It also defines ideals $\mathrm{ev}_0(\FF_{j,(a,0)}^{-,p})\subset QH^*(Y)$ for ordinary quantum product.
\end{cor}
\begin{proof}
    The first claim holds because $c_{\lambda}$ is a $\ku$-linear $E^{-}QH^*(Y)$-module homomorphism by \cref{Prop product structure in weight 1 0 case}.
    The second claim follows, as $\mathrm{ev}_0: E^-QH^*(Y) \to QH^*(Y)$ is a surjective ring homomorphism by combining \cref{Lemma ev0 commutative diagram} and \cref{Prop product structure in weight 1 0 case}.
\end{proof}

\section{Properties of the equivariant spectral sequences}

\subsection{Conventions and notation}\label{Subsection Spectral sequence conventions}
We assume the reader has familiarity with the discussion of {\MBF } spectral sequences from \cite{RZ2}, which readily adapts to the equivariant setting.

Abbreviate {\bf $E_r^-$-page} to mean the $E_r$-page of the {\MBF } spectral sequence converging to $E^-SH^*(Y)$ (or $E^-HF^*(H_{k^+})$, if we ignore columns above the $k$-th one). Similarly, define the $E_r^{\infty}$- and $E_r^{+}$-page.
The zeroth column is the one involving the generators of equivariant quantum cohomology.
We call {\bf higher columns} all columns after (and excluding) the zeroth column. 

Observe \cref{Eq sp seq for C all together} in \cref{Example intro equivariant Floer theory for C}: when we draw equivariant spectral sequences, a bold dot will denote a $\k$-summand with $u^0$-factor. A smaller dot denotes a $\k$-summand with a $u^{\neq 0}$-factor.

\subsection{Structural results for the {\MBF} spectral sequences}\label{Subsection Spectral sequence trick}
\begin{prop}\label{Prop spectral seq tricks}
Let $(a,b)$ be a free weight.\footnote{Many of the claims extend to general weights when considering columns of period $p\neq \tfrac{a}{b}$; compare the proof of \cref{Theorem injectivity theorem 2}. We omit these statements.} 
The following apply to the {\MBF} spectral sequences for equivariant symplectic cohomology as well as for equivariant Floer cohomology of $H_{k^+}$.
\strut
\begin{enumerate}
\item On any $E_r^-,E_r^{\infty}$ page, for $r\geq 1$, all edge differentials landing in the zeroth column must vanish.

More generally, for any $k\in \N,$ on the $E_r^-$ page for $r\geq 1$ all edge differentials from columns with slopes $\lambda>k$ landing in the columns with slopes $\leq k$ must vanish.

\item
The edge differential on the $E_0^{\infty}$-page kills all higher columns, and its arrows cannot decrease $u$-orders
(this can fail on higher pages, \cref{Subsection Spectral sequence trick 2}).
Thus, the higher columns of $E_1^{\infty}$ are zero. Therefore $E_r^{\infty}$ has already converged to its zeroth column: $E^{\infty}_1\cong E^{\infty}QH^*(Y)\cong E^{\infty}SH^*(Y)$.

\item 
Forgetting smaller dots in $E_0^{\infty}$, the page of bold dots with all arrows between them is the $E_0$-page of the non-equivariant spectral sequence converging to $SH^*(Y)$ $($which is $0$ if $c_1(Y)=0)$.

\item The differential on higher columns for the $E_0^{-}$, $E_0^+$ pages is determined by the $E_0^{\infty}$-page.

\item The following commutative diagram is induced by $E_0$-pages. The rows are long exact sequences. The numbers on the arrows indicate the degree of the map when it is a non-zero degree.
$$
\xymatrix{
\cdots \ar@{->}[r]^-{} 
& 
E_1^+ \ar@{->}[r]^-{-1}
\ar@{->}[d]_-{d_1}
& 
E_1^- \ar@{->}[r]^-{+2}_{u\otimes 1}  
\ar@{->}[d]_-{d_1}
& 
E_1^{\infty} \ar@{->}[r]^-{}
\ar@{->}[d]_-{d_1=0}
& 
E_1^+ \ar@{->}[r]^-{-1} 
\ar@{->}[d]_-{d_1}
& 
\cdots 
\\
\cdots \ar@{->}[r]^-{} 
& 
E_1^+ \ar@{->}[r]^-{-1} 
& 
E_1^- \ar@{->}[r]^-{+2}_{u\otimes 1}  
& 
E_1^{\infty} \ar@{->}[r]^-{}
& 
E_1^+ \ar@{->}[r]^-{-1} 
& 
\cdots 
}
$$
Warning: it is a LES of $\ku$-modules, but the $\ku$-module structure on $E_1$-pages often does not agree with that on the modules that the spectral sequence converges to (see \cref{Example intro equivariant Floer theory for C}). 
\item The zeroth column of $E_1^+$ is identifiable with  $\mathrm{coker}\,(u\otimes 1: E_1^- \to E_1^{\infty})$. 

The higher columns of $E_1^+$ are identifiable with the higher columns of $E_1^-$ shifted up by $1$ degree, via the explicit $\k$-linear isomorphism $u^{-1}d_0^{\infty}$
using the edge differential $d_0^{\infty}$ of $E_0^{\infty}$.

The higher columns of $E_1^-$ are precisely the $u$-torsion part of the $\ku$-module $E_1^-$.
\end{enumerate}
\end{prop}
\begin{proof}
The first part of Claim 1 follows by \cref{Corollary E minus c star maps are injective}, because inclusion of the zeroth column corresponds (after taking the limit of the spectral sequence) to the continuation map.
The second part of Claim 1 essentially follows by \cref{Theorem injectivity theorem 2} parts (1) and (4). The columns with slopes $\leq k$ compute $HF^*(H_{k^+})$, and the columns with slopes $<\lambda$ compute $HF^*(H_{\lambda})$. Just like in the non-equivariant setup from \cite{RZ2}, the continuation map $HF^*(H_{k^+})\to HF^*(H_{\lambda})$ can be built as the inclusion of the subcomplex of the first set of columns into the second set of columns (on the zero-page, but using the original differential of the chain complex).
Suppose there is a chain $y$ in a column with slope $>k$ on the $E_r^-$ page, with edge differential $d_r(y)=x$ landing in a column with slope $\leq k$.
There are two cases: (1) either $x$ was going to survive to the limit $E_{\infty}^-\cong HF^*(H_{k^+})$ in the spectral sequence where we ignore columns with slopes $>k$; or (2) there was going to be a non-trivial edge differential $d_{r'}(x)=z\neq 0$ on a later page $r'\geq r$ (where $z$ also lies in a column of slope $\leq k$). In case (1), $x$ represents a class of $HF^*(H_{k^+})$ that is killed in $HF^*(H_{\lambda})$, contradicting \cref{Theorem injectivity theorem 2}.(1). In case (2), we get $d_{r'}(d_r(y))=d_{r'}(x)=z\neq 0$: if $r'=r$ then this contradicts that $d_r\circ d_r=0$; if $r'>r$ then this contradicts the well-definedness of $d_{r'}$ on the full spectral sequence page $E_{r'}^-$ since $[x]$ represents the zero class on $E_{r'}^-$ (in view of $x$ having been killed by $y$ on the earlier page $E_r^-$).

Claim 2 follows by \cref{Prop vanishing of spectral sequence for infinity page},
using that \cref{Equation for the differential expanded} does not increase the $u$-order, thus the same holds on quotients when building the $E_0^{\infty}$-edge differential.
For claim 3, we use \cref{Lemma ev0 commutative diagram general case}.
Claim 4 follows from naturality of the construction of the $E_0$ page, and the fact that at the chain complex level the equivariant differential is the ``same'' in all three models except for the fact that different coefficients are used, respectively: $\ku,\kuu,\mathbb{F}$ (and over $\kuu$ all information about the differential is known), see \eqref{defn F C^+ and W^+}.
Claim 5 follows by naturality of the construction of spectral sequences, as follows. We consider the long exact sequence in \cref{Lemma LES for W modules} in the case where $C^*,C^*_u,C^+$ are the $E_0$-pages with their differentials (thus the $W^-,W^{\infty},W^+$ are the $E_1$-pages). 

To prove Claim 6, observe that the map $E_1^-\to E_1^{\infty}$ is induced by the natural inclusion on $E_0$-pages after multiplying by $u$ (using the natural $\ku$-module action on the $E_1^-$-page, viewed as the cohomology of the $\ku$-module $E_0^-$, so $u$ acts as a degree $2$ vertical map in the spectral sequence). 
Note that the higher columns of $E_1^-$ are $u$-torsion, because in high degrees the $E_0^-$-page coincides with the $E_0^{\infty}$-page, and we showed that the higher columns of the $E_1^{\infty}$-page vanish. Whereas the zeroth column of $E_1^-$, which is $E^-QH^*(Y)$, is a free $\ku$-module by \cref{Prop equivariant formality for QH}.
Therefore, via $E_1^-\to E_1^{\infty}$ the zeroth column will inject (after $u$-shifting), whereas higher columns of $E_1^-$ map to zero in $E_1^{\infty}$ as the higher columns of $E_1^{\infty}$ vanish.
By exactness of the LES, those (torsion) higher columns of $E_1^-$ that map to zero must come from the higher columns of $E_1^+$ (shifted down in degree by $1$).
Note this is consistent with the general discussion of torsion in \cref{Equation LES on W modules 2}.
The claim about $u^{-1}d_0^{\infty}$ comes from the description of the connecting map in \cref{Lemma LES for W modules}.
\end{proof}

\subsection{Spectral sequence edge differentials viewed as zig-zags}\label{Subsection Spectral sequence trick 2}
\strut\\[1mm]
We briefly remind the reader how edge differentials arise as {\bf zig-zags} in the spectral sequence associated to a filtration $F_0\subset F_1 \subset \cdots \subset C^*$ by subcomplexes. We work with Massey's exact couples as in Bott-Tu \cite[Chp.III.Sec.14]{BottTu}. Bott and Tu describe the case of a bicomplex, using very intuitive zig-zags
 \cite[Equation (14.12) p.164]{BottTu}. Our setting is more general, so we explain what we mean.
We use the usual convention for spectral sequences that $p$ denotes the column number. To avoid confusing $p$ with the period of $S^1$-orbits, we instead let $\lambda_0=0<\lambda_1<\lambda_2<\ldots$ denote the slopes for which $1$-orbits appear, corresponding to $\lambda_p$-periodic $S^1$-orbits of the $\Fi$ action.  
In our notation (compare \cref{Eq sp seq for C all together} in \cref{Example intro equivariant Floer theory for C}), $F_p$ corresponds to the subcomplex of $1$-orbits arising with slopes $\leq \lambda_p$, since those are the ones which appear as representatives for classes in columns $0,1,\ldots,p$ of the $E_0$-page.

For $r\geq 1$, the edge differential $D_r$ on a class $[f_p]$ in the $p$-column of the $E_r$-page is
$$
D_r [f_p] = [f_{p-r}] \;\;\textrm{ if }\;df_p \in f_{p-r} + dF_{p} \subset F_p,
$$
where $d: F_p \to F_p$ is the original differential of $C^*$.
In the case we are most interested in, when $p=r$ so $D_r[f_p]$ hits the zeroth column,
this can be thought of as a zig-zag of arrows, because we seek classes $[g_0],[g_1],[g_2],\ldots,[g_{r}]$ in columns $0,1,2,\ldots,r$ of the $E_0$-page, such that 
$$
df_r = f_0 + dg_0 + dg_1 + dg_2 + \cdots + dg_{r} \qquad \textrm{ as an equation in }F_r.
$$
If we abusively draw all arrows $\delta_i$ on the $E_0$-page caused by $d=\delta_0+u\delta_1+\cdots$, and not just the vertical edge differential $D_0=[\delta_0]$, then the above equation is a zig-zag of arrows on the $E_0$-page from the classes $[g_0],\ldots,[g_r],[f_r]$ which all cancel out except for the output $f_0$.
(The more general case, when $r<p$, is similar, with the $[g_0],\ldots,[g_{p-r-1}]$ playing the role of cancelling out any error terms arising in the columns to the left of the $(p-r)$-column, in which the desired output $[f_{p-r}]$ lives).

It is crucial in the above description to use the $E_0$-page, rather than the $E_1$-page because classes that die when passing from $E_0$ to $E_1$ may be crucial to link up a zig-zag as above.

\begin{ex}\label{Example zigzag}
The following picture is the piece of the $E_0^+$-page from \cref{Eq sp seq for T^CP^1 all together} in \cref{Example intro equivariant Floer theory for C} arising in rows $-2$ and $-3$ (which is the total degree value $p+q$ of $(E_0^+)^{p,q}$, in our convention). Each dot corresponds to a copy of the base field $\k$, and we label a basis by letters as shown.
\begin{center}
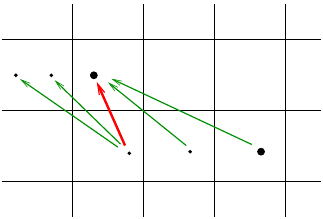
\end{center}
The bold red arrow is the $D_0=[\delta_0]$ differential of the $E_0^+$-page, the green arrows are all the potentially-nonzero higher order $(\delta_i)_{i\geq 1}$ contributions to $d$ in the setting of \cref{Eq sp seq for T^CP^1 all together}. In general we do not know which of these arrows are non-zero, although a posteriori one finds that most are non-zero: we later prove that $g_1$ kills $h_1$ on $E_0^+$, and then since $E^+\Rightarrow E^+SH^*(T^*\C P^1)=0$, we deduce that $g_2$ must kill some $1$-dimensional subspace $S\subset \k [h_0]\oplus \k [f_0]$ on $E_2^+$, and finally $f_3$ must kill the $1$-dimensional space $(\k [h_0]\oplus \k [f_0])/S$ on $E_3^+$. We explain how $D_3[f_3]=(\textrm{non-zero})\cdot [f_0]$ is possible, even though there are no arrows $f_3\to f_0$.
To illustrate a simple example, let us assume that $\k=\Z/2$ and that all arrows are isomorphisms, except $g_2 \to h_1$ is zero (so $dg_2=h_0$ and $D_2[g_2]=[h_0]$). Then:
$$
D_3[f_3] = [f_0], \qquad \textrm{ since } df_3 = f_0 + dg_1 + dg_2.
$$
Note: on the $E_1$-page, only $h_0,f_0,g_2,f_3$ survive: in particular $h_1$ is killed, so the $E_1$-page would not detect the existence of that zig-zag anymore.

Over a general base field $\k$ it is more complicated. Write $\sim$ to mean equality up to a non-zero constant. We will see later that $D_0[g_1]\sim [h_1]$, and we know the spectral sequence converges to zero. 
If $D_2[g_2]\sim [h_0]$, then $D_3[f_3]\sim [f_0]$ and it would follow that the arrow $g_2 \to h_1$ is zero, the arrow $g_1 \to h_0$ can be anything, and the remaining arrows are isomorphisms. Conversely, if $g_2 \to h_1$ is zero, then $D_2[g_2]\sim [h_0].$
However, if $g_2 \to h_1$ is non-zero, a zig-zag with $g_2$ shows $D_3[f_3]\sim [h_0]$, and $D_2[g_2]=\alpha [h_0] + \beta [f_0]$ for some $\alpha,\beta \in \k$ with $\beta \neq 0$, due to a zig-zag involving $g_1.$
\end{ex}
\section{Example 1: $\C$}\label{Example intro equivariant Floer theory for C}
\subsection{Notation}
\label{Example Y is C Notation}
Let $Y=\C$, with the standard $S^1$-action. Recall $\mathrm{char}\,\k=0$.
One can construct $H_{\lambda}=c(H)$ to have a critical point $x_0$ at $0$, and {\MB} manifolds $B_k\cong S^1$ of non-constant $1$-orbits arising at $c'(H)=k\in \{1,2,\ldots,\lfloor \lambda \rfloor\}$. 
Let $x_k$ and $y_k$ denote the generators of $H^*(B_k)[2k]$, in degrees $|x_k|=-2k$, $|y_k|=-2k+1$.
As $c_1(Y)=0$, we have a $\Z$-grading on Floer complexes.

\begin{lm}\label{Lemma equiv floer diff in C case}
The Floer differential on $E^{\infty}FC^*(H_{\lambda})$ is a $\ku$-module homomorphism
with
$$
d(x_k)\in \k y_k
\;\; \textrm{ and } \;\; d(y_k)\in \k x_{k-1} \oplus \k ux_k.
$$
In particular, $\delta_i=0 \textrm{ for } i\geq 2, \textrm{ so } d=\delta_0+u\delta_1$ in \eqref{Equation for the differential expanded}, with $\delta_1(x_k)=0.$
\end{lm}
\begin{proof}
The Floer differential increases grading by $1$ and cannot increase $H$ by the filtration results from \cite{RZ1}, so $d(x_k)\in \mathrm{span}_{\ku}(u^{-(k-1)}y_1,\ldots,u^{-1}y_{k-1},y_k)$, $d(y_k)\in \mathrm{span}_{\ku}(u^{-(k-1)}x_0,\ldots,x_{k-1},ux_k)$.
However, $d$ cannot decrease $u$-powers by \eqref{Equation for the differential expanded}, so the claim follows.
\end{proof}
 
\subsection{The case $\mathbf{(a,b)=(0,1)}$}
This case was studied in detail by Zhao \cite[Sec.8.1]{zhang2017multiplicativity}:
\begin{equation}\label{Equation d by Zhao}
d(x_{k})=0 \quad\textrm{ and }\quad d(y_k)=x_{k-1}+kux_{k}.
\end{equation}
The non-equivariant case (equivalent to setting $u=0$) gives $HF^*(H_{k^+})=\k[x_{k}]$, so $SH^*(\C)=0$.
For $\diamondsuit \in \{+,-,\infty\}$, the continuation map
$E^{\karo}HF^*(H_{k^+})\to E^{\karo}HF^*(H_{(k+1)^+})$ is therefore
\begin{equation}\label{Equation C cont map}
[x_{k}]\mapsto [x_{k}]=-(k+1)u[x_{k+1}].
\end{equation}
Thus, for $\karo=-,\infty,+$, the map \eqref{Equation C cont map} is respectively: an inclusion, an isomorphism, and zero (using $\char\,\k=0$). Thus the
direct limit is $E^-\!SH^*(\C)=\kuu$, 
$E^{\infty}\!SH^*(\C)=\kuu$, and
$E^+\!SH^*(\C)=0$ respectively. 
By contrast, 
$E^{\diamondsuit}_{\textrm{telescope}}SH^*(\C)=0$ for all $\diamondsuit$: this follows from $SH^*(\C)=0$ for homological reasons \cite[Cor.2.4]{zhao2019periodic};
e.g.\;$x_0$ is a boundary as follows: $x_0=d(y_1-uy_2+2!u^2y_3-3!u^3y_4+\cdots)$.

The same calculation by {\MBF} spectral sequence methods is as follows. The $E_1$-page in \cref{Eq sp seq for C all together}  arose from the ``vertical differential'': $d_0(x_k)=0$, $d_0(y_{k})=kux_{k}$ (compare \cref{Example S1 equiv cohomologies calculation}). For $E^{\infty}\!SH^*(\C)$, all columns in $E_1$ vanish except the $0$-th column $H^*(Y)\otimes \kuu $ (using $\mathrm{char}\,\k=0$).
For $E^{-}\!SH^*(\C)$, the $k$-th column is $\k\cdot x_{k}$, so the spectral sequence collapsed to $E_1=E_{\infty}=\k[\![u]\!]\cdot [x_0] \oplus \k\cdot [x_{1}] \oplus \k\cdot [x_{2}]\oplus \cdots$. The $u$-action is $u\cdot [x_{k}]=[ux_{k}]=-\tfrac{1}{k}[x_{k-1}]$  (for $k\geq 1$), so we obtain the $\ku$-module $E^-\!SH^*(\C)\cong \kuu$. 

\subsection{The general free weight case $(a,b)$}\label{Subsection C The general free weight case}
We will assume $(a,b)$ is free; for $Y=\C$ this means $a \notin b\Z_{\geq 1}$. We recall that the Maslov index of the $S^1$-action on $Y=\C$ is $\mu=1$.

Note that Zhao's local Floer computation \eqref{Equation d by Zhao} was only for the case  $(a,b)=(0,1)$.
For weights $(-a,1)$ with $-a\leq 0$, Liebenschutz-Jones \cite[Sec.8.1]{liebenschutz2020intertwining}, working over $\Z$-coefficients, combined Zhao's result with properties of the Seidel isomorphisms to obtain a result for $(-a,1)$. It turns out that $r_{-a,1}:E^-_{(-a,1)}QH^*(Y)\to E^-_{(-a-1,1)}QH^*(Y)$, as a map $\Z[u] \to \Z[u]$, is multiplication by $(a+1)u$ (similarly, for $Y=\C^n$ one gets multiplication by $[(a+1)u]^n$).

We wish to instead discuss the general free weight case, over characteristic zero, without appealing to local computations, as a means to illustrate the general theory.

\begin{prop}\label{Prop ESH of C}
Let $(a,b)$ be any weight with $a \notin b\Z_{\geq 1}$. 
Then
\begin{equation*}
\begin{array}{ccccl}
\!\!\!E^-\!SH^*(\C)\cong \kuu, 
&
\strut\;
&
\!\!\!\!\!\!E^{\infty}SH^*(\C)\cong \kuu, 
&
\strut\;
&
E^{+}\!SH^*(\C) = 0,
\\
E^-HF^*(H_{k^+})\cong \ku\!\cdot\! [x_k],
&
\strut\;
&
\!\!E^{+}HF^*(H_{k^+})= \kuu\!\cdot\! [x_0],
&
\strut\;
&
E^{+}HF^*(H_{k^+}) =
(\kuu/u^{1-k}\ku) \!\cdot\! [x_0]
\end{array}
\end{equation*}
so $E^{\infty}HF^*(H_{k^+})  \cong \mathbb{F}[2k].$
The $x_k$ agree up to $\k^{\times}$-rescaling with the $x_k$ from  \eqref{Definition xk class}, for degree reasons.
\end{prop}
\begin{proof}
This follows immediately from the general theory in \cref{Subsection Growth rate of the filtration polynomial}.
\end{proof}

We deduced \cref{Prop ESH of C} from the general theory. We now show how it can also be read-off from the {\MBF} spectral sequence, in \cref{Eq sp seq for C all together}.
\begin{figure}[ht]%
				\centering
				{
					\includegraphics[scale=1.2]{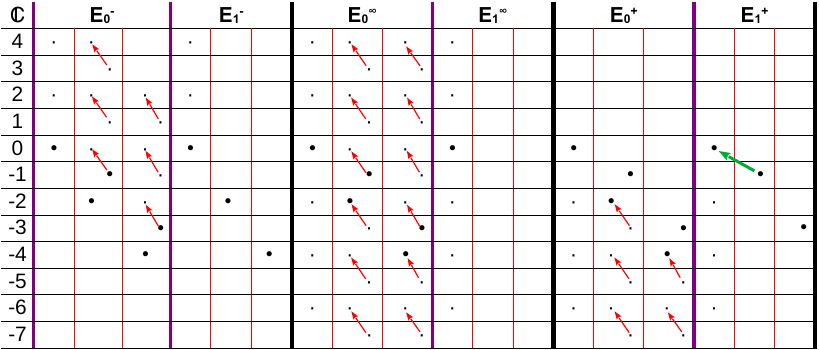}
					\caption{\label{Eq sp seq for C all together}
     $E_0$-page and $E_1$-page for the $S^1$-equivariant spectral sequence for $E^{\karo}SH^*(\C)$ for the three models $\karo=-,\infty,+$, and $Y=\C$. The smaller dots indicate that the $1$-orbit generator comes with a factor involving a non-zero power of $u$.
     }
     }
\end{figure}
\begin{lm}\label{Lemma direction of d0 edge on infty page for C}
 The edge differential on higher columns of the $E_0^{\infty}$ page goes from odd to even degrees, making higher columns acyclic, as drawn in \cref{Eq sp seq for C all together}. 
\end{lm}
\begin{proof}
By \cref{Prop spectral seq tricks}(2), each higher column of $E_0^{\infty}$ is acyclic, so one possibility is that the arrows are as drawn in \cref{Eq sp seq for C all together}: each arrow denotes an isomorphism $\k \to \k$.
However, we need to exclude another possibility: the arrows increase the degree by 1, but could a priori go from even to odd classes within a column.
Consider, by contradiction, the minimal $N\geq 1$ such that column $N$ has the arrows going from even to odd degrees. It follows that the arrows in column $N$ preserve the power of $u$, so they arise from $\delta_0$ in \eqref{Equation for the differential expanded}. Apply \cref{Prop spectral seq tricks}(3): remove all smaller dots, and consider just the resulting $E_0$-page of fat dots. By assumption, there is a (non-trivial) arrow between the two fat dots in column $N$. Therefore column $N$ does not kill the lowest degree fat dot in column $N-1$, which therefore will survive to the limit as nothing else can cancel it for degree reasons (here the minimality assumption on $N$ ensures that there is no cancellation occurring within column $N-1$ on the $E_0$-page). This contradicts that the $E_0$-page of fat dots converges to zero: $SH^*(Y)=0$.
\end{proof}

By \cref{Prop spectral seq tricks}, from the $E_0^{\infty}$ page we deduce the rest of the information about $E_1^{\infty},E_0^{\pm},E_1^{\pm}$ shown in \cref{Eq sp seq for C all together}. 
The columns up to column number $k\in \N$ give the spectral sequence for the Hamiltonian $H_{k^+}$.
Because all generators in the $E_1^-$, $E_1^{\infty}$ pages are in the same parity, those spectral sequences have already converged. By the vanishing theorem, \cref{cor vanishing of E+ for SH}, we know that everything in the $E^+$-spectral sequence will die eventually in the limit, and we know that $E^-SH^*(Y)\cong E^{\infty}SH^*(Y)$. The $E_1^-$-page has a different $u$-module structure than $E^-SH^*(Y)$, nevertheless we can reconstruct the correct $u$-module structure a posteriori: we know that the $\k$-vector space generators must organise themselves so that $E^-SH^*(Y)\cong E^{\infty}SH^*(Y)\cong E^{\infty}QH^*(Y)$ is a $\ku$-module isomorphism.
From this observation, \cref{Prop ESH of C} follows, noting that $E^+HF^*(H_{k^+})$ is the cokernel of $u\otimes 1:E^-HF^*(H_{k^+}) \to E^{\infty}HF^*(H_{k^+})$ by \cref{Theorem torsion freeness of Eminus HF}.

We now show that, a posteriori, we can  partly reconstruct the Floer differential, analogously to \eqref{Equation d by Zhao}.

\begin{cor}
   The equivariant Floer differential in \cref{Lemma equiv floer diff in C case} satisfies
    \begin{equation}\label{Equation equiv diff in C case}
d(x_{k})=0 \quad\textrm{ and }\quad d(y_k)=\alpha_k x_{k-1}+\beta_k ux_{k},
\end{equation}
where $\alpha_k,\beta_k \in \k$ are non-zero.
\end{cor}
\begin{proof}
Recall that the $E_0^{\infty}$-page differential is $\delta_0$.
By \cref{Lemma equiv floer diff in C case},
$d(x_{k})=\delta_0(x_k)\in \k y_k$, and there are no red arrows out of $[x_k]$ on the $E_0^{\infty}$-page, so $d(x_k)=0$.
By \cref{Lemma equiv floer diff in C case},
$\delta_0(y_k)=\alpha_k x_{k-1}$ and $\delta_1(x_k)=\beta_k x_k$ (the latter is also known as the BV operator $\Delta(x_k)$).
The $E_0^{\infty}$-page differential $d_0$ is induced by $\delta_0$, and in \cref{Eq sp seq for C all together} we have a red arrow that implies $d_0([y_k])=\beta_k u [x_k]$ for $\beta_k\neq 0$ (since the arrow is an isomorphism).
Finally, $\delta_0(y_k)=\alpha_k x_{k-1}\neq 0$ because $\delta_0$ is the differential from the non-equivariant Floer complex, and only $y_k$ can kill $x_{k-1}$, which it must since $SH^*(Y)=0$.
\end{proof}

\begin{rmk}
A second proof of $\alpha_k\neq 0$ follows from the $\ku$-module structure.
The $\ku$-module structure on the equivariant spectral sequence does not agree with the one on the limit. For  $k\geq 1$,
\begin{equation}\label{Equation u torsion equation for C case}
u\cdot [x_k]=[ux_k]=[d_0(\b_k^{-1}y_k)],
\end{equation}
so $[x_k]$ is $u$-torsion when viewed in $E_1^-$ with its natural $u$-action. However, by \cref{cor vanishing of E+ for SH}, we know that 
$E^-SH^*(\C)$ has no $u$-torsion, and the only way to prevent the analogue of \eqref{Equation u torsion equation for C case} in $E^-SH^*(\C)$ is to ensure that $d(\b_k^{-1}y_k)$ has a non-zero term involving $\alpha_k\neq 0$. It follows that the columns
``link up'' via the $u$-action, when we instead consider the $u$-action on $E^-SH^*(\C)$:
$$u\cdot [x_k]=[ux_k]=[d(\beta_k^{-1}y_k)-\alpha_k \beta_k^{-1}x_{k-1}]= -\alpha_k \beta_k^{-1}\cdot [x_{k-1}] \textrm{ where }\alpha_k \beta_k^{-1}\neq 0.$$
\end{rmk}

\begin{rmk}
A third proof of $\alpha_k\neq 0$ uses the $E^+$-spectral sequence.
 On the $E_1^+$-page, by \eqref{Equation equiv diff in C case},
$$
[y_1] \mapsto d_1[y_1]=[d(y_1)]=
[\alpha_1 x_{0}+\beta_1 ux_{1}]
=
\alpha_1 [x_{0}],
$$
observing that $ux_1=0$ over $\mathbb{F}$-coefficients. So the fact that $[y_1],[x_{0}]$ kill each other off on the $E_1^+$-page is equivalent to the non-vanishing result $\alpha_1\neq 0$.
For $y_2$, the class $[y_2]$ survives to the $E_2^+$-page, and one might be tempted by \eqref{Equation equiv diff in C case} to think that it will continue to survive. However, on the $E_2^+$ page, $[y_2]$ will kill $[u^{-1}x_0]$ because the cohomology relations allow us to build a ``zig-zag'' map (\cref{Subsection Spectral sequence trick 2}):
$$
[y_2]\mapsto [d(y_2)]
=
[\alpha_2 x_1 + \beta u x_2]
= 
[\alpha_2 x_1]
=
[d(\alpha_2\beta_1^{-1}u^{-1}y_1)-
\alpha_2\alpha_1\beta_1^{-1}u^{-1}x_0
]
=
-\alpha_2\alpha_1\beta_1^{-1} \cdot [u^{-1}x_0].
$$
A similar argument will build a zig-zag map from $[y_k]$ to $[u^{1-k}x_0]$ (up to $\k^{\times}$-rescaling) contributing to the differential on the $E_k^+$-page.
Using $\beta_k\neq 0$, this discussion yields:
\begin{equation}\label{Equation killing E+ page for C}
\left([u^{-k}x_0]\textrm{ gets killed on the }E_{k+1}^+\textrm{-page for each }k\geq 0\right) \Longleftrightarrow \left(\alpha_k\neq 0 \textrm{ for all }k\geq 0\right).
\end{equation}
If the left-hand side of \eqref{Equation killing E+ page for C} failed, then for grading reasons that $[u^{-k}x_0]$ would survive to $E_{\infty}^+\cong E^+SH^*(Y)$, contradicting the vanishing in \cref{cor vanishing of E+ for SH}.

\end{rmk}

\subsection{The filtration polynomial for $Y=\C$}

\begin{prop}\label{Prop ESH of C filtration poly}
Let $(a,b)$ be any weight with $a \notin b\Z_{\geq 1}$. 
Let $\lambda \in [k,k+1)$, for $k\geq 0 \in \N.$

Then $\FF_j^{\lambda}=E^-QH^*(\C)=\ku\cdot x_0$ for $j\leq k$, and $\FF_j^{\lambda}=\ku \cdot u^{j-k}x_0$ for $j>k.$

The $\lambda$-valuation  of $x_0$ is
$
\mathrm{val}^{\lambda}(x_0) = k =\lfloor \lambda \rfloor
$
(cf.\;\cref{Definition lambda valuation}).

The slice polynomial is:
$
s_{\lambda}(t) = 1+t+\cdots+t^k
$
(cf.\;\cref{Corollary filtration polynomial}).
\end{prop}
\begin{proof}
By \cref{Prop ESH of C},
$\FF_j(E^-HF^*(H_{k^+}))=(u^j$-divisible elements of $\ku\cdot [x_k])$, so $\ku\cdot u^j x_k$.
For degree reasons, inside $E^-HF^*(H_{k^+})$ we have:
$$\ku\cdot u^j x_k = \ku \cdot x_{k-j} \supset \ku \cdot x_0 \;\textrm{ for }j\leq k, \quad \textrm{ and } \quad \ku\cdot u^j x_k = \ku \cdot u^{j-k} x_0 \;\textrm{ for }j\geq k.$$
By definition, $\FF_j^{\lambda}=\FF_j^{c_{\lambda}}(E^-QH^*(\C))$
viewed inside $E^-HF^*(H_{k^+})$
is the intersection between $\ku\cdot x_0$ and $F_j(E^-HF^*(H_{k^+}))$, so the claim follows.

The slice polynomial also follows. It can also be computed directly using \cref{Corollary filtration polynomial}: 
for grading reasons, $c_{\lambda}: \ku\cdot x_0 \to \ku\cdot x_k$, $x_0 \mapsto \gamma u^k x_k$ for some $\gamma\in \k$, and $\gamma\neq 0$ by \cref{Corollary E minus c star maps are injective}, so $c_{\lambda}$ has precisely one invariant factor: $u^{k}$.
\end{proof}
\section{Example 2: $T^*\CP^1$}
\subsection{The general free weight case for $Y=T^*\C P^1$}
Let $Y=T^*\CP^1 \cong \mathrm{Tot}(\mathcal{O}_{\CP^1}(-2))$ with the standard $\C^*$-action on fibres. Note $c_1(Y) = 0$ (cf.\,\cite{R14}).
The weight $(a,b)$ is free $\Leftrightarrow a \notin b\Z_{\geq 1}$. 

The {\MBF} manifolds of non-constant $1$-orbits are $B_k\cong \R P^3$, in columns $k=1,2,\ldots$ of the $E_1$-page, and they involve a degree shift of $2k\mu=2k$. Denote $[x_k]$ the minimum of $B_k$ (for an auxiliary Morse function), with grading $|x_k|=-2k$. The zeroth column has generators $[x_0],[x_{-1}]$ in degrees $0,2$, generating the cohomology of the zero section, $\C P^1$. 

\begin{prop}\label{Prop TCP1 calculation}
Let $(a,b)$ be any weight with $a \notin b\Z_{\geq 1}$. 
Then
\begin{equation*}
\strut\hspace{-5ex}\begin{array}{ccccc}
\strut\hspace{-16ex}E^-\!SH^*(T^*\CP^1)\cong \kuu^2, 
&
\strut\quad\!\!\!\!\!\!\!\!\!\!\!\!\!\!\!
&
\hspace{3ex}E^{\infty}\!SH^*(T^*\CP^1)\cong \kuu^2, 
&
\strut\quad\!\!\!
&
\hspace{-13ex}E^{+}\!SH^*(T^*\CP^1) = 0,
\\
E^-\!HF^*(H_{k^+})\cong \ku\cdot [x_k]\oplus \ku\cdot [x_{k-1}],
&
\strut\quad\!\!\!\!\!\!\!\!\!\!\!\!\!\!\!
&
E^{\infty}\!HF^*(H_{k^+})\cong \kuu^2,
&
\strut\quad\!\!\!
&
E^{\infty}\!HF^*(H_{k^+}) \cong \mathbb{F}[2k]\oplus \mathbb{F}[2k-2]
\end{array}
\end{equation*}
The $x_k$ agree up to $\k^{\times}$-rescaling with the $x_k$ from \eqref{Definition xk class}, for degree reasons.
\end{prop}
\begin{proof}
This follows immediately from the general theory in \cref{Subsection Growth rate of the filtration polynomial}.
\end{proof}
\begin{figure}[H]%
				\centering
				{
					\includegraphics[scale=1]{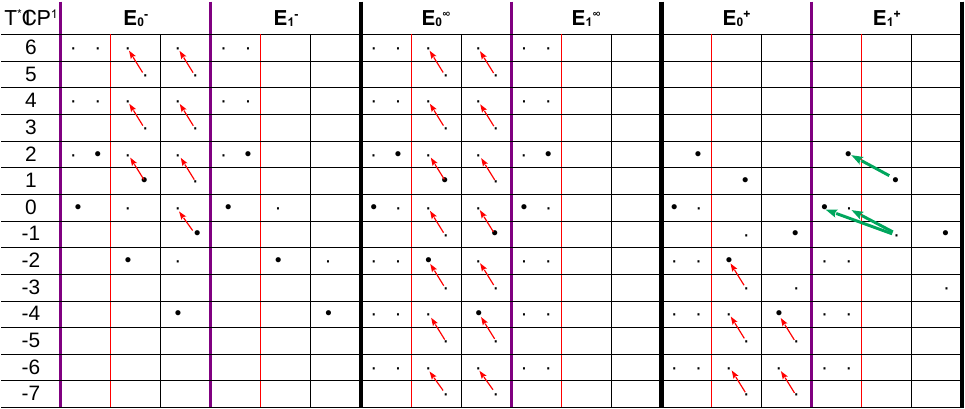}
					\caption{Pages $E_0$ and $E_1$ of equivariant spectral sequences for $\karo=-,\infty,+$ on $T^*\CP^1$    
     }
\label{Eq sp seq for T^CP^1 all together}
				}
\end{figure}
\cref{Prop TCP1 calculation} can also be read off directly from the {\MBF} spectral sequence, in \cref{Eq sp seq for T^CP^1 all together}, just like in the discussion for $Y=\C$ below \cref{Prop ESH of C}.
For $Y=T^*\C P^1$, the analogue of \cref{Lemma direction of d0 edge on infty page for C} is immediate: the arrows on higher columns of $E_0^{\infty}$ cannot go from even to odd classes, because that would decrease the $u$-power, contradicting \cref{Prop spectral seq tricks}(2).
The zeroth column bold dots are the two generators of $H^*(\C P^1)\cong H^*(T^*\C P^1)$ (with coefficients in $\k$), and each of the other columns has bold dots that correspond to a copy of $H^*(\R P^3)$, shifted in degrees appropriately (down by $2k\mu$, where $k$ is the column number, and $\mu=1$ is the Maslov number of the $S^1$-action). 
All columns (except the zeroth) on the $E_1^{\infty}$-page must vanish by \cref{Prop spectral seq tricks}(2). Thus, the red arrows on the $E_0^{\infty}$-page in \cref{Eq sp seq for T^CP^1 all together} are isomorphisms.
This determines the differentials on the $E_0^-$, $E_0^+$ pages. As the $E_1^-$, $E_1^{\infty}$ pages lie in even degree, they have converged. In the $E^+$-spectral sequence, everything eventually dies by \cref{cor vanishing of E+ for SH}. The edge differentials are not entirely determined, 
for example the double arrow shown on the $E_1^+$-page is an injective map $\k \to \k^2$, but its $1$-dimensional image is not determined.

The $E_1^-$-page has a different $u$-module structure than the module $E^-SH^*(Y)$ it has converged to; if we just consider columns up to column $k$, it has converged to $E^-HF^*(H_{k^+})$.
As $E^-HF^*(H_{k^+})$ is torsion-free (\cref{Theorem torsion freeness of Eminus HF}), or using that $E^-SH^*(Y)\cong E^{\infty}QH^*(Y)$ is free, we deduce a posteriori much of the $u$-action by degree considerations: $[x_k],[x_{k-1}]$ must generate $E^-HF^*(H_{k^+})$ as a $\ku$-module. 
However, how precisely the $u$-action relates $[x_k],\ldots,[x_0]$ is unclear, due to boundary cycle relations, with the exception of\footnote{no $\k$-linear combination of $x_{k-1}$ and $ux_k$ arises as a boundary cycle in $E^-CF^*(H_{k^+})$ for degree reasons.} $u\cdot [x_k] = [ux_k]$ in $E^-HF^*(H_{k^+})$.

\subsection{The rotation elements for $Y=T^*\C P^1$}\label{Subsection The rotation elements for TCP1}

We now describe $r_{a,b}$ from \cref{ESH is direct limit of maps on EQH}.

We use the Hamiltonian $k^+ H$, see \cref{Subsection The Hamiltonians NH}. The $1$-orbits are the constant orbits at points of $\F_{\min}=\C P^1$ (zero section of $\mathcal{O}_{\C P^1}(-2)$). We pick an auxiliary Morse function on $\F_{\min}$ with two critical points, the minimum $x_{\min}$ and the maximum $x_{\max}$. 
Their grading is $|x_{\min}|=-2k\mu$ and $|x_{\max}|=-2k\mu+2$, where $\mu=1$ for $Y=T^*\C P^1$ with the standard $\C^*$-action on fibres.
The $\ku$-module basis $x_{\min},x_{\max}$ of $E^-HF^*(k^+H)$ plays a role analogous to the basis $x_k,x_{k-1}$ of $E^-HF^*(H_{k^+})$ (also in degrees $-2k$, $-2k+2$), but need not coincide.\footnote{via the continuation isomorphism $E^-HF^*(k^+H)\cong E^-HF^*(H_{k^+})$, $x_k$ is a $\k^{\times}$-rescaling of $x_{\min}$ due to being in minimal grading, but $x_{k-1}$ could be a non-obvious $\k$-linear combination of $ux_{\min}$ and $x_{\max}$.}
The critical points $x_{\min},x_{\max}$ with those gradings can also be used as chain level generators for $E^-QH^*(Y)[2k\mu]$.

\begin{lm}
In the bases $x_{\min},x_{\max}$ (with respective gradings), $r_{a,b}$ is a matrix
$$
r_{a,b}=
\left(\begin{smallmatrix}
Au & Bu^2\\
-2(1+\lambda_+) & Cu
\end{smallmatrix}\right)
$$
for some $A,B,C\in \k$,
where $\lambda_+\in \k$ involves only $T^{>0}$-terms and arises as $\mathcal{Q}_{\Fi}=(1+\lambda_+)\mathrm{PD}[\C P^1]$ in \eqref{Equation QFi class}, and $\mathrm{PD}[\C P^1]=-2x_{\max}$.
The invariant factors of $r_{a,b}$ are $(1,u^2)$, cf.\,the first case in \cref{Remark example of slice dimensions}, in particular the first map in \eqref{Equation simplifying the direct limit} is not a naive inclusion since $r_{a,b}(x_{\min})= \textrm{unit}\cdot x_{\max} + Au\cdot x_{\min}$.
\end{lm}
\begin{proof}
On the domain of $r_{a,b}$, the basis $x_{\min},x_{\max}$ lies respectively in grading $0,2$, whereas in the codomain it lies in grading $-2,0$. The powers of $u$ are therefore dictated by the grading (see the $c_1(Y)=0$ discussion in \cref{Remark CY and monotone cases ERk matrix}).
The $Q_{\Fi}$ class is known from \cite{R14}.
Using \cref{Example invariant factors in the 2x2 case}: the lower left entry of $r_{a,b}$ implies $1$ is an invariant factor; the other invariant factor is $u^2$ as the determinant is of order $u^2$ and is non-zero since $r_{a,b}$ is injective by \cref{ESH is direct limit of maps on EQH}.
\end{proof}

\subsection{The slice polynomial for $Y=T^*\C P^1$}

The slice polynomial $s_{\lambda}=s_k$ for $T^*\C P^1$ only depends on $k$ for $\lambda \in [k,k+1)$, where $k\geq 0 \in \N.$ We will narrow it down to two  possible cases:
\begin{equation*}
\begin{split}
N_k(t) &= 2+2t+\cdots+2t^{k-1}+t^k+t^{k+1},
\qquad (\textrm{so }c_{k^+}\textrm{ has invariant factors }(u^{k-1},u^{k+1})),\\
Z_k(t) &=  2+2t+\cdots+2t^{k-1}+2t^k,
\qquad\qquad\;\; (\textrm{so }c_{k^+}\textrm{ has invariant factors }(u^{k},u^{k})).
\end{split}
\end{equation*}
The two cases are not distinguished by the dimension $\dim_{\k} \EE^k_1=2k$ (see \cref{Subsection Growth rate of the filtration polynomial}).
 
\begin{prop}\label{Prop ESH of TCP1 filtration poly}
Let $(a,b)$ be any weight with $a \notin b\Z_{\geq 1}$. 
The slice polynomial is
$$
s_k=Z_k \Leftrightarrow x_0 \in \FF_{k}^{\lambda} \Leftrightarrow [u^{-k+1}x_0] = 0 \in E^+HF^*(H_{k^+}),
\textrm{ otherwise: } s_k=N_k \textrm{ and } x_0 \in \FF_{k-1}^{\lambda} .
$$
If $s_k=Z_k$ then $s_{k+1}=N_{k+1}$. 
So $s_k=N_k$ holds for infinitely many $k\in \N.$

If $b=0$ and $s_k=Z_k$, then $s_{mk}=Z_{mk}$ for all $m\in \N.$
\end{prop}
\begin{proof}
The (injective) image $c_{\lambda}(\FF_j^{\lambda})\subset E^-HF^*(H_{k^+})$ cannot be easily determined as the $u$-action can relate $x_{-1},x_0,x_{k-1},x_k$ in a complicated way (for $k\geq 1$):
\begin{equation}\label{Equation x0xminus1 formula}
\begin{array}{rclccc}
x_{0} & = & A_k u^{k} x_{k} & + & B_k u^{k-1} x_{k-1}\\
x_{-1} & = & C_k u^{k+1} x_{k} & + & D_k u^{k} x_{k-1},
\end{array}
\end{equation}
where $A_k,B_k,C_k,D_k$ are in the underlying field $\mathbb{B}$ (so do not involve $u$).
This identifies $c_{\lambda}$ with 
\begin{equation}\label{Equation clambda matrix in TP1 case}
u^{k-1}\cdot \left(\begin{smallmatrix} A_k u  & C_k u^2  \\ B_k    & D_k u \end{smallmatrix}\right) : \ku^2 \to \ku^2, \;\; \textrm{ with } A_kD_k-B_kC_k\neq 0,
\end{equation}
using injectivity of $c_{\lambda}$ (\cref{Corollary E minus c star maps are injective}).
By \cref{Example invariant factors in the 2x2 case}, the invariant factors of
$c_{\lambda}$ are
$(u^{k-1},u^{k+1})$ when $B_k\neq 0$, or $(u^k,u^k)$ when $B_k=0$.

Using \cref{Corollary filtration polynomial}, we deduce the two possible slice polynomials:
$$
B_k=0\Rightarrow s_k = Z_k, \textrm{ and } B_k\neq 0 \Rightarrow s_k = N_K.
$$

There are $\alpha_{k+1},\beta_{k+1}$ in the underlying field determined by the linear dependency relation
\begin{equation}\label{Equation xkplus1 formula}
u^2 x_{k+1} = \alpha_{k+1} u x_k + \beta_{k+1} x_{k-1},
\end{equation}
and $\beta_{k+1}\neq 0$ since $x_k,x_{k+1}$ are $\ku$-linearly independent.
Thus, the continuation map $E^-HF^*(H_{k^+})\to E^-HF^*(H_{k+1}^+)$ in the bases $x_k,x_{k-1}$ and $x_{k+1},x_k$ can be viewed as the matrix $M_k$:
\begin{equation}\label{Equation continuation from k to kplus1 in TCP1 case}
M_k = \left(\begin{smallmatrix} 0  & \;\;\;\;\;\tfrac{1}{\beta_{k+1}} u^2  \\  1   & \;\;- \tfrac{\alpha_{k+1}}{\beta_{k+1}} u
\end{smallmatrix}\right) : \ku^2 \to \ku^2,
\textrm{ so }
c_{\lambda+1}=M_k\cdot \eqref{Equation clambda matrix in TP1 case} = u^{k} \left(\begin{smallmatrix} \tfrac{B_k}{\beta_{k+1}} u   & \;\;\tfrac{D_k}{\beta_{k+1}} u^2 \\  A_k- \tfrac{\alpha_{k+1}B_k}{\beta_{k+1}}    &  \;\;(C_k-\tfrac{\alpha_{k+1}D_k}{\beta_{k+1}}) u
\end{smallmatrix}\right).
\end{equation}
Thus, $B_{k+1}=A_k- \tfrac{\alpha_{k+1}B_k}{\beta_{k+1}}$.
The second claim follows:%
\footnote{Another proof is that $s_k=Z_k$, $s_{k+1}=Z_{k+1}$ would imply that $x_0\in \k^{\times} u^k x_k \cap \k^{\times} u^{k+1}x_{k+1}$, causing $x_k,x_{k+1}$ to be $\ku$-linearly dependent, contradiction.}
if $B_k=0$ then $B_{k+1}=A_k\neq 0$, using \eqref{Equation clambda matrix in TP1 case}.

Now the first claim. Note $u^{k+1}x_{k+1}=\alpha_{k+1} u^k x_k + \beta_{k+1} u^{k-1} x_{k-1}$ by \eqref{Equation xkplus1 formula}. By the previous paragraph, $B_{k+1}=0$ precisely if $(A_k,B_k)$ and $(\alpha_{k+1},\beta_{k+1})$ are proportional, and so, by \eqref{Equation x0xminus1 formula}, precisely if $x_0$ and $u^{k+1}x_{k+1}$ are proportional (meaning $x_0\in \FF_{k+1}^{\lambda+1}$). So $s_k=Z_k \Leftrightarrow B_k=0 \Leftrightarrow x_0 \in \FF_{k}^{\lambda}$. The last condition is equivalent to $[u^{-k+1}x_0] = 0 \in E^+HF^*(H_{k^+})$ by \eqref{Equation relating valuation and plus theory}.

The third claim uses \eqref{Equation filtration relation x0 and xk}:
$s_k=Z_k \Leftrightarrow x_0 \in \FF_{k}^{k^+}
\Rightarrow x_0 \in \FF_{mk}^{(mk)^+}
\Leftrightarrow
s_{mk}=Z_{mk}$.
\end{proof}

\begin{rmk}
    For small $k$: 
$s_0=Z_0=2$ and $s_1=N_1=2+t+t^2$. The first is because $c_{0^+}$ is an isomorphism, the second uses $B_1=1$ from \eqref{Equation continuation from k to kplus1 in TCP1 case}. We do not know $s_2$, as we do not know how $u^{-1}x_0$ dies in the spectral sequence for $E^+SH^*(Y)$. This involves precisely the zig-zags discussed in \cref{Example zigzag}: if the arrow $g_2 \to h_1$ in that example is zero then $u^{-1}x_0$ is killed by $u^{-1}x_2$ (via the arrow $g_2 \to h_0$) so we are in the case $B_2=0$ and $s_2=Z_2$, otherwise we are in the case $B_2\neq 0, s_2=N_2.$

More generally, $s_k=Z_k$ requires a very special circumstance: $x_0 \in \ku^{\times}\cdot u^k x_k$; intuitively $x_0$ is ``more divisibile'' than expected.
The $u$-divisibility of the image of $y\in E^-QH^{2d}(Y)$ in $E^-HF^*(H_{k^+})$ is at least $u^{d+k-1}$ since $y$ lies in the $\ku$-linear span of $x_{k-1},ux_k$. However, some $y$ have one extra degree of $u$-divisibility: they lie in $\ku\cdot x_k$.
Viewed in the $2$-plane $A:=\mathrm{span}_{\k}(x_0,u^{-1}x_{-1})\subset E^{\infty}QH^0(Y),$
we get a line $M_k:=\ku\cdot x_k \cap A$ through $0$, consisting of classes that are ``more divisible'' than expected.
So $s_k=Z_k$ (so $B_k=0$ above) is equivalent to $x_0\in M_k$.
Note that $M_k \cap M_{k+1}=\{0\}$, otherwise we get a relation $u^kx_k = \textrm{unit}\cdot u^{k+1}x_{k+1}$ contradicting the $\ku$-linear independence of $x_k,x_{k+1}$. So the result ``$s_k=Z_k\Rightarrow s_{k+1}=N_k$'' is equivalent to the observation: $x_0\notin M_k\cap M_{k+1}$.
We are curious to know what configuration of lines $M_k$ in the plane one obtains as $k$ varies.
\end{rmk}

\section{Example $2'$: $T^*\CP^1$ with a twisted action}

Consider now $Y=T^*\C P^1$ with a $\C^*$-action which acts both on the fibres and on the base. This twisted action was described for $T^*\C P^{n-1}$ in \cite[Example 7.6]{RZ2}.
This was mostly already discussed in \cref{Example twisted TCP1}, but we add some additional explanations here.
In our case, the $\C^*$-fixed locus  $\F=\F_0 \sqcup \F_1$ consists of two fixed points lying in the zero section: $x_0=[1,0]$ and $y_0=[0,1]$.
The weights associated to $\F_{\min}=\F_0=\{x_0\}$ are $(1,1)$, and to $\F_1=\{y_0\}$ are $(-1,3).$
Those two points correspond to generators for the Morse and the Floer complexes, in grading $|x_0|=0$ and $|y_0|=2$, and they generate the $\ku$-module $E^-QH^*(Y)$ (see the two fat dots in the zeroth column in \cref{Eq sp seq for T^CP^1 twisted}).
The Maslov index is $\mu=2$ (the sum of the weights of any fixed component).
In particular, $Y$ is a Conical Symplectic Resolution of weight $s=2$.
The fiber over $y_0$ is a $\Z/3$-torsion line bundle $\mathcal{H}\to \F_1$ inside $T^*\C P^1$, giving rise to {\MB}-manifolds $B_{1/3},B_{2/3},B_{4/3},\ldots$ diffeomorphic to $S^1$.
The only other {\MB}-manifolds of orbits, are the ones corresponding to full rotations: $B_1,B_2,B_3,\ldots$ which are diffeomorphic to a sphere subbundle of $T^*\C P^1$, thus $\R P^3$ (cohomologically $S^3$, as $\mathrm{char}\,\k=0$).

We consider the equivariant Floer theory for any free weight $(a,b)$.
The {\MBF} spectral sequence pages $E_0^-$ and $E_1^-$ are drawn in \cref{Eq sp seq for T^CP^1 twisted} up to the slope value $4/3$.
As in the previous Section, the fat dots are the generators arising on the non-equivariant spectral sequence so these must kill each other out as $SH^*(Y)=0$; from this the red arrows all follow by using the fact that the higher columns of the $E^{\infty}_0$ spectral sequence page are acyclic (as we are in the free weight case).

By the general theory, $E^-SH^*(Y)\cong E^{\infty}SH^*(Y)\cong E^{\infty}QH^*(Y)$ and $E^+SH^*(Y)=0$, because $c_1(Y)=0$ and $(a,b)$ is free.
Abbreviate $E_{\lambda}:=E^-HF^*(H_{\lambda^+})$.
Using the injectivity of the continuation maps, the {\MBF} spectral sequence implies that
$$
E_{\frac{1}{3}}\cong \ku^2,\quad
E_{\frac{2}{3}}\cong \ku \oplus \ku[2],\quad E_{1}\cong \ku[2]\oplus \ku[4],\quad E_{\frac{4}{3}}\cong (\ku^2)[4], \;\textrm{ etc.}
$$
For slope values $p=k+\frac{1}{3}$, for $k\in \N,$ the filtration $\FF^p$ on $E^-QH^*(Y)$ is essentially the filtration by grading (up to a translation by $4\lfloor p \rfloor$ in the integer labelling). 
Whereas for slopes $p=k+\frac{2}{3}$ and $p\geq 1 \in \N$ it is not a filtration by grading: the $u$-valuation on $E^-QH^*(Y)$, compared to what it was for slope value $p^-$, has gone up by $1$ on a specific rank one $\ku$-submodule of $E^-QH^*(Y)$. 

We now explain the proof of \eqref{Equation slice poly twisted TCP1}.
Call $x_0,y_0$ the two fixed points, where $x_0$ is the minimum of the moment map, so the gradings are $|x_0|=0$, $|y_0|=2$ (the two fat dots in the zeroth column of 
\cref{Eq sp seq for T^CP^1 twisted}). Ignoring gradings, $\ku^2 \cong E^-QH^*(Y)\hookrightarrow E^-HF^*(H_{\lambda})\hookrightarrow E^-SH^*(Y)\cong \kuu^2$, so $c_{\lambda}^*:E^-QH^*(Y)\to E^-HF^*(H_{\lambda})$ corresponds to a $\ku$-module homomorphism $\ku^2\to \ku^2$.
If we now take into account gradings again, and observe \cref{Eq sp seq for T^CP^1 twisted}, we see that
$c_{(1/3)^+}^*$ must have invariant factors $(1,u).$
For $k\in \N$, the continuation map 
$$
\psi: \ku^2 \cong E^-HF^*(H_{(\frac{1}{3})^+})\to E^-HF^*(H_{(k+\frac{1}{3})^+}) \cong (\ku^2)[4k]
$$
must have invariant factors $(u^{2k},u^{2k})$ for grading reasons.
Since $c_{(k+\frac{1}{3})^+}=\psi\circ c_{(1/3)^+}$, it must have invariant factors $(u^{2k},u^{2k+1})$.
The continuation map 
$$
\psi': (\ku^2)[4k] \cong E^-HF^*(H_{(k+\frac{1}{3})^+})\to E^-HF^*(H_{(k+\frac{2}{3})^+}) \cong \ku[4k]\oplus \ku[4k+2] 
$$
must have invariant factors $(1,u)$, for grading reasons. Since $c_{(k+\frac{2}{3})^+} = \psi'\circ c_{(k+\frac{1}{3})^+}$, the invariant factors can be $(u^{2k},u^{2k+2})$ or $(u^{2k+1},u^{2k+1})$ (these give rise respectively to the slice polynomials $N_{2k+1}$ and $Z_{2k+1}$ in \eqref{Equation slice poly twisted TCP1}).
The map $ER_k:E^-QH^*(Y)\to E^-QH^*(Y)[2k\mu]$ has grading $2k\mu=4k$, so in the basis $x_0,y_0$ it has the form
$$
ER_k = \left(\begin{smallmatrix} 
A_k u^{2k} & C_k u^{2k+1} \\
B_k u^{2k-1} & D_k u^{2k}
\end{smallmatrix}\right),
$$
where $A_k,B_k,C_k,D_k\in \mathbb{B}.$
As it is an injective map, $\det\, ER_k\neq 0$, so: if $B_k\neq 0$ then the invariant factors are $(u^{2k-1},u^{2k+1})$, if $B_k=0$ then the invariant factors are $(u^{2k},u^{2k}).$
It follows that $s_k$ is either $N_{2k}$ or $Z_{2k},$ which proves the options for $s_{k+1}$ in \eqref{Equation slice poly twisted TCP1}.

Finally, we explain the ``either ... or ...'' in \eqref{Equation slice poly twisted TCP1}: the column with integer slope $k+1$ on the $E_1$-page in \cref{Eq sp seq for T^CP^1 twisted} has two generators: $x_{\min}$ and $u\cdot x_{\min}$ (in gradings $-4k-4$, $-4k-2$), where $x_{\min}$ is the minimum of an auxiliary Morse function on the {\MB} manifold $B_{k+1}\cong \R P^3.$ This implies that $u^2\cdot x_{\min}$ in $E^-HF^*(H_{(k+1)^+})$ is an element in $E^-HF^*(H_{(k+\frac{2}{3})^+})$ whose $u$-valuation has dropped by $2$. It also shows that the minimum of an auxiliary Morse function on the {\MB} manifold $B_{(k+\frac{2}{3})^+}\cong S^1$ has not changed $u$-valuation value.
This proves that only one invariant factor changes, and it changes by a $u^2$ factor, when  passing from slope $(k+\frac{2}{3})^+$ to $(k+1)^+$, so \eqref{Equation slice poly twisted TCP1} follows.

\section{Example 3: $\mathcal{O}_{\CP^1}(-1)$}
\label{Section Example 3 O of minus 1}

Let $Y=\mathrm{Tot}(\mathcal{O}_{\CP^1}(-1))$ be the total space  of the negative complex line bundle with Chern number $-1$, equivalently it is the blow-up of $\CP^1$ at one point. 
As explained in Ritter \cite{R14}, $Y$ is monotone with: 
$$c_1(Y)=[\omega]\neq 0, \;\; QH^*(Y)\cong \k[\omega]/(\omega^2 + T \omega), \;\; \textrm{ and } \;\; SH^*(Y)\cong QH^*(Y)/\ker [\omega] \cong \k.$$ 
Here $QH^*$ and $SH^*$ are $\Z$-graded, the Novikov variable $T\in \k$ has grading $|T|=2$, so $\k$ is graded. As $H^*(Y)\cong \k \oplus \k[-2]$, we have $SH^*_{\myplus}(Y)\cong \k[1]$ by the non-equivariant version of the LES \eqref{Equation LES for ESH plus}.
We use the standard weight one $\C^*$-action $\Fi$ on fibres. The weight $(a,b)$ is free precisely if $a\notin b\Z_{\geq 1}.$ 

For $\mathcal{O}(-1)\to \C P^n$, weights $(-a,1)$ with $-a\leq 0$, and working over $\Z$-coefficients, Liebenschutz-Jones \cite[Sec.8.3]{liebenschutz2020intertwining} gave a very explicit description of the $r_{(-a,1)}$ maps on equivariant quantum cohomology, and of $E^-SH^*(Y)$.
The key ingredient was Liebenschutz-Jones's intertwining formula.

We again wish to instead illustrate how the general theory applies, over charactersitic zero, for free weights.
Consider the first page of the {\MBF} spectral sequence for $E^-SH^*(Y)$. Column $0$ involves generators  $x_0,y_0\in H^*(\C P^1)$; in columns $k\in \{1,2,\ldots\}$ we have the cohomology generators $x_k,y_k$ for the {\MB} manifolds $B_k\cong S^3$ with grading shift $2k\mu=2k$.
Thus the $\Z$-grading is:
$$|x_0|=0,\; |y_0|=2,\; |x_k|=-2k,\; |y_k|=-2k+3,\; |T|=2,\; |u|=2.$$

\begin{lm}\label{Lemma TstarCP1 computation of ESH modules}
Let $(a,b)$ be any weight with $a\notin b\Z_{\geq 1}$.
\begin{equation}
\begin{array}{lcccc}
E^-\!SH^*(Y)\cong \ku \oplus \kuu, 
&
\strut
&
E^{\infty}\!SH^*(Y)\cong \kuu^2, 
&
\strut
&
E^{+}\!SH^*(Y) \cong \mathbb{F},
\end{array}
\end{equation}
where the $\kuu$-summand in $E^-SH^*(Y)$ is generated as a $\kuu$-module by $[\omega+T\cdot 1]\in E^-QH^*(Y)$, whereas the $\ku$ summand in $E^-SH^*(Y)$ is generated as a $\ku$-module by $1\in E^-QH^*(Y)$. The latter generator also spans the $\k u^0$-summand of $\mathbb{F}$.
\end{lm}
\begin{proof}
By \cref{Prop Periodicity property 2}, 
the $u^0$-part of $r_{a,b}$ is the non-equivariant $r: QH^*(Y)\to QH^*(Y)[2]$, which  by 
 \cite{R14} is multiplication by $Q=-[\omega]$.
We pick a basis of eigenvectors for $r$: 
$v_1=\omega$ (eigenvalue $T$), $v_2=\omega+T\cdot 1$ (eigenvalue $0$), thus  
$r=\left(\begin{smallmatrix} T & 0 \\ 0 & 0 \end{smallmatrix}\right)$.
Since $r_{a,b}$ has degree two, \cref{Remark CY and monotone cases ERk matrix}.(2) yields:
$$r_{a,b}=r+uA^{a,b}, \textrm{ for a }2\times 2\textrm{ matrix }A^{a,b}\textrm{ with entries }A_{ij}^{a,b}\in \mathbb{B} \textrm{ (so not involving }T).$$
By \cref{Corollary E minus c star maps are injective}, $r_{a,b}$ is injective, so $0\neq \det\, r_{a,b} \in A_{22}^{a,b} T u + \mathbb{B} u^2$.

By \cref{Example invariant factors in the 2x2 case}, $r_{a,b}$ either has invariant factors $(1,u)$ (when $A_{22}^{a,b}\neq 0$) or $(1,u^2)$ (when $A_{22}^{a,b}=0$).
The composition $ER_k$ from \eqref{Equation ERk map} of $k$ such maps $r_{a,b}$, with varying $a$, therefore has $u^0$-part $r^k$ causing an invariant factor $1$. By \cref{Example invariant factors in the 2x2 case}, the invariant factors of $ER_k$ are therefore $(1,u^{m_k})$ where $$m_k:=\nu(\det\, ER_k)=\nu(\det\,r_{a-(k-1)b,b})\cdot \ldots \cdot \nu(\det\,r_{a-b,b}) \cdot \nu(\det\,r_{a,b}).$$ We deduce that $m_{k+1}-m_k:=\nu(r_{a-kb,b})$ is either 1 or 2, depending on whether $A_{22}^{a-kb,b}$ is non-zero or zero, respectively. 
The claim follows from the general theory in \cref{Subsection Growth rate of the filtration polynomial}: in our case, 
$j_1(k)=0\to 0$ and $j_2(k)=m_k\to \infty$, as $k\to \infty$.
\end{proof}

\begin{lm}\label{Lemma TstarCP1 MBF calculation}
Let $(a,b)$ be any weight with $a\notin b\Z_{\geq 1}$.
Then as $\k$-vector spaces we have:
\begin{equation}\label{Equation Eminus SH for O minus 1}
\begin{split}
    E^-SH^*(Y) & \cong \ku x_0 \oplus \ku y_0 \oplus \bigoplus_{i\geq 1} \k x_i\\
     E^-HF^*(H_{k^+}) & \cong \ku x_0 \oplus \ku y_0 \oplus \bigoplus_{ 1 \leq i \leq k} \k x_i
\end{split}
\end{equation}
which inject into $E^{\infty}SH^*(Y) \cong E^{\infty}QH^*(Y)= \kuu x_0 \oplus \kuu y_0$, and these injections are equal to an inclusion on the subspace $\ku x_0 \oplus \ku y_0$.
\end{lm}
\begin{proof}
We prove this from the {\MBF} spectral sequence. This requires care because $\k$ is now graded, $|T|=2$, $|T^{-1}|=-2$ (this complication is paid off by ensuring a $\Z$-grading on Floer complexes).
In column zero, $x_0$ represents $1$, and $y_0$ represents $[\omega]$. 
This gives rise to $\ku$-modules $E^-QH^*(Y)= \ku x_0 \oplus \ku y_0$ and $E^{\infty}QH^*(Y)= \kuu x_0 \oplus \kuu y_0$. The claims about the inclusions of $\ku x_0 \oplus \ku y_0$ follow, because 
the natural inclusion $E^-QH^*(Y)\hookrightarrow E^{\infty}QH^*(Y)$ is the composite of the injections $E^-QH^*(Y)\hookrightarrow E^-HF^*(H_{k^+}) \hookrightarrow E^-SH^*(Y) \hookrightarrow E^{\infty}SH^*(Y)\cong E^{\infty}QH^*(Y)$, cf.\,the proofs of \cref{Corollary torsion freeness 1 free case} and \cref{Corollary torsion freeness 2 free case}.

In column one, we have generators $x_1,y_1$ coming from the non-equivariant spectral sequence. Thus, $\kuu x_1 \oplus \kuu y_1$ is the first column in $E_0^{\infty}$. We want to find the contributions to the edge differential $d_0$ for $E_0^{\infty}$ occurring within column 1 and caused by the equivariant differentials $u^i\delta_i$ for $i\geq 1$, see \eqref{Equation for the differential expanded}.
Note that the contributions to $u\delta_1(x_1),u\delta_1(y_1)$ within column 1 must land in $\mathbb{B}T^{>0}u x_1 \oplus \mathbb{B}T^{>0}u y_1$ because $\delta_1$ does not involve $u$, and because non-constant Floer trajectories arise with strictly positive energy $E>0$, so they contribute with a strictly positive power $T^E$ to $\delta_1$.
Note the grading $|\mathbb{B}T^{>0}u x_1|>0$ and $|\mathbb{B}T^{>0}u y_1|>3$.
The $\Z$-grading also implies $|d_0(y_1)|=|y_1|+1=2$ and $|d_0(x_1)|=|x_1|+1=-1$. Thus $u\delta_1 (y_1)\in \mathbb{B}Tux_1$ and $u\delta_1 (x_1)=0$.
By \cref{Prop spectral seq tricks}(2), column 1 in $E_0^{\infty}$ should be acyclic, therefore $u\delta_1 (y_1) =bTux_1$ with $b\neq 0\in \mathbb{B}$, so that a non-trivial arrow $y_1 \to ux_1$ and its $u^{\pm 1}$-translates make the column acyclic. Using \cref{Prop spectral seq tricks}.(4),
it follows that in column 1 of $E_0^-$ everything cancels except $x_1$ (note $u^{-1}T^{-1}y_1$ exists in $E_0^{\infty}$ and kills $x_1=u\delta_1(b^{-1}u^{-1}T^{-1}y_1)$, but it does not exist in $E_0^-$).
A similar argument applies to higher columns: copies $\k x_i$ survive to $E_1^-$. 
By \cref{Prop spectral seq tricks}.(1),\footnote{in this case the second part of \cref{Prop spectral seq tricks}.(1) is obvious: the $x_i$ cannot kill each other off because they are in in even grading.} $\oplus \k x_i$ survives to the limit, thus the $E_r^-$ spectral sequence has converged on page $E_1^-$. The first equation in \eqref{Lemma TstarCP1 MBF calculation} follows. The second equation follows because the spectral sequence, after omitting columns with slopes $>k$, converges to $E^-HF^*(H_{k^+}).$ 
\end{proof}

\begin{cor}\label{Cor Tstar CP1 slice poly}
Let $(a,b)$ be any weight with $a\notin b\Z_{\geq 1}$. The slice polynomial is $s_k = 2+t+\cdots+t^k$, $ER_k$ has invariant factors $(1,u^k)$, and the filtration polynomial is $f_k=1+t^k$.
\end{cor}
\begin{proof}
By the proof of \cref{Lemma TstarCP1 computation of ESH modules},
 $s_k = 2+t+\cdots+t^{m_k}$, using \cref{Lemma slice poly for ERk is same as for ck}. By \eqref{Equation formula for Ek1},
$$
m_k = s_k(0)=\dim_{\k} \EE^k_1= \dim_{\k} E^-HF^*(H_{k^+})/E^-QH^*(Y)=k,
$$
using \eqref{Equation Eminus SH for O minus 1} in the last equality. The rest follows.
\end{proof}

\begin{prop}\label{Prop ESH for O-1 over P1}
Let $(a,b)$ be any weight with $a\notin b\Z_{\geq 1}$.
\begin{equation}\label{Equation O-1 over P1 ESH computation}
\begin{array}{lcccc}
E^-\!SH^*(Y)\cong \ku \oplus \kuu, 
&
\strut%
&
\hspace{-3ex}E^{\infty}\!SH^*(Y)\cong \kuu^2, 
&
\strut\!\!\!
&
\hspace{-9ex}E^{+}\!SH^*(Y) \cong \mathbb{F},
\\
E^-\!HF^*(H_{k^+}\!)= \ku\!\cdot\! [x_0]\oplus \ku\!\cdot\! [x_k],
&
\strut%
&
E^{\infty}\!HF^*(H_{k^+}\!)\cong \kuu^2,
&
\strut\!\!\!
&
\hspace{1ex}E^{\infty}\!HF^*(H_{k^+}\!) \cong \mathbb{F} \oplus \mathbb{F}[2k]
\end{array}
\end{equation}
The $\ku$ summand in $E^-SH^*(Y)$ is generated as a $\ku$-module by $1$ 
$($or $[\omega]$, or indeed any element whose $u^0$-part is not in $\k^{\times}\cdot[\omega+T\cdot 1])$. That generator also spans the $\k u^0$-summand of $\mathbb{F}$.

The $\kuu$-summand in $E^-SH^*(Y)$ is generated as a $\kuu$-module by an element in $E^-QH^*(Y)$ with $u^0$-part $[\omega+T\cdot 1]$, and such an element is unique up to $\k^{\times}$-rescaling. 
\end{prop}
\begin{proof}
By the proof of \cref{Lemma TstarCP1 computation of ESH modules}, and using the notation $v_1,v_2$ from that proof, we know 
$ER_k(v_1)=T^kv_1 + u^{\geq 1}$-terms. So,
to build a Smith normal form for $ER_k$, we may use $v_1$ as first basis vector on the domain, and $T^{-k}ER_k(v_1)$ as first basis vector on the codomain: this puts the invariant factor $1$ in the first diagonal entry. \cref{Cor structure theorem 2} then implies that $v_1$ generates the $\ku$-summand in $E^-SH^*(Y)$.

Since $ER_k(v_2)$ only involves $u^{\geq 1}$-terms, the same argument shows that any $v_1'=\alpha v_1 + \beta v_2$ with $\alpha \in \ku^{\times}$, $\beta\in \ku$ can be used as first basis vector on the domain, by using $\alpha^{-1}T^{-k}ER_k(v_1')$ on the codomain. This proves the claim about the choice of generator for the $\ku$-summand in $E^-SH^*(Y)$.

A posteriori, we know there exists a $w\in E^-SH^*(Y)$ that generates the $\kuu$-summand as a $\kuu$-module. By $u^{\pm}$-rescaling, we may assume $w$ lies in $E^-QH^*(Y)$, indeed we may also ensure it has non-zero $u^0$-part (so $\mathrm{ev}_0(w)\neq 0\in QH^*(Y)=\k x_0 \oplus \k y_0$). With those requirements, $w$ corresponds to an element in $\{0\}\times \ku^{\times} \subset \ku \oplus \kuu \cong E^-SH^*(Y).$
Also, such a $w\in E^-QH^*(Y)$ is not $u$-divisible in $E^-QH^*(Y)=\ku x_0 \oplus \ku y_0$, whereas $ER_k(w)$ is $u^k$-divisible for all $k\geq 1$. 
We write $w=\alpha v_1 + \beta v_2$ in the basis $v_1,v_2$, where $\alpha,\beta\in \ku$. By the proof of \cref{Lemma TstarCP1 computation of ESH modules}, $ER_k(\alpha v_1 + \beta v_2)=\alpha T^k v_1 + u^{\geq 1}$-terms. Since $ER_k(w)=u^{\geq k}$-terms, we deduce $\alpha\in u\ku$ and $\beta\in \ku^{\times}.$ Thus $w$ lies in $E^-QH^*(Y)$ with $u^0$-part $v_2$ up to $\k^{\times}$-rescaling, as claimed.
Our construction of $w$ also ensures that
$$
\ku x_0 \oplus \ku w = \ku x_0 \oplus \ku y_0 =E^-QH^*(Y).
$$

We now prove $E^-HF^*(H_{k^+})=\ku x_0\oplus \ku x_k$.
The inclusion $E^-HF^*(H_{k^+})\hookrightarrow E^-HF^*(H_{k+1}^+)$ has%
\footnote{If we were working over $c_1(Y)=0$, so $\k$ lies in grading zero, then we would have expected a codimension two map, because of the Seidel isomorphism $E^-HF^*(H_{(k+1)^+})\cong E^-HF^*(H_{k^+})[2]$. This is not a contradiction here, because in the monotone case the Novikov field $\k$ is supported in all even integer gradings due to $|T^{\pm 1}|=\pm 2$.}
$\mathrm{codim}_{\k}=1$ in view of the new class $x_{k+1}$, see \eqref{Equation Eminus SH for O minus 1}.

By \cref{Cor Tstar CP1 slice poly}, the invariant factors $(1,u^k)$ of $ER_k$ imply that $E^-HF^*(H_{k^+})\cong \ku \oplus u^{-k}\ku$ (see the general theory in \cref{Subsection Growth rate of the filtration polynomial}).
So there is a generator of the $u^{-k}\ku$ summand of the form
$$
w_{-k}:=x_k + u^{\geq 1}\textrm{-terms, \qquad where }k\geq 1.
$$ 
Thus, $\k^{\times}w_{-k}=\{$all generators of the $u^{-k}\ku$-summand$\}$, where $\k^{\times}=\k\setminus \{0\}$. So, for $k\geq 2$:
$$
u\cdot w_{-k} \in \k^{\times} w_{-k+1}.
$$
A posteriori, all these $w_{-k}$ generate the $\kuu$-summand of $E^-SH^*(Y)$, therefore $uw_{-1}$ equals the above $w$ up to $\k^{\times}$-rescaling,
thus $u^k w_{-k} \in \k^{\times} w$.
We also deduce that $$u\cdot x_k \in \k^{\times} x_{k-1} + u^{\geq 1}\textrm{-terms, for }k\geq 2.$$
So $\ku x_0 \oplus \ku x_k$ contains $x_0,x_1,\ldots,x_k$.
The claim about $E^-HF^*(H_{k^+})$ follows, using \eqref{Equation Eminus SH for O minus 1}.
\end{proof}
\begin{rmk}
It is still unclear how the $x_k$ are related to the classes $x_k=\mathcal{S}^{-k}(x_0)$ from \eqref{Definition xk class}: they are both in degree $-2k$, and it is reasonable to expect that they agree up to $u^{\geq 1}$-corrections and up to $\k^{\times}$-rescaling. We know $\mathcal{S}^{-k}(x_0),\mathcal{S}^{-k}(y_0)$ generate the $\ku$-module $E^-HF^*(H_{k^+})$, but the grading does not imply any conclusions because $|T^{\pm 1}|=\pm 2$ can shift gradings arbitrarily.
\end{rmk}

\section{Equivariant symplectic cohomology in the monotone case}

\subsection{The CY case}

For symplectic $\C^*$-manifolds with $c_1(Y)=0$, we showed in \cref{cor vanishing of E+ for SH} that $E^-SH^*(Y)\cong E^{\infty}SH^*(Y)$ and $E^+SH^*(Y)=0$, for any weight $(a,b)$.

\subsection{The monotone case, for free weights}\label{Subsection The monotone case for free weights}

Now consider the monotone setting: $c_1(Y)\in \R_{>0}[\omega]$.
Let $(a,b)$ be a free weight, and $m:=\dim_{\k}\, H^*(Y)$. The Novikov field $\k$ is now graded.
Consider the $\k$-linear map $r$ given by (non-equivariant) quantum product by $Q_{\Fi}$ (\cref{Remark QFi class}):
$$
r:QH^*(Y)\to QH^*(Y)[2\mu], \qquad E_0:=\ker r^m, \qquad W:=r^m(QH^*(Y)).
$$
By \cite{RZ1}, leaning on \cite{R14,R16}, the following commutative diagram of $\k$-linear maps holds,
$$
\begin{tikzcd}[column sep=0.6in]
0
 \arrow[r, ""]
&
E_0
 \arrow[d, "\cong",""']
 \arrow[r,  "\mathrm{incl}"] 
 &
QH^*(Y)
 \arrow[d, "\cong",""']
 \arrow[r, "c^*"] 
&
SH^*(Y)
 \arrow[d, "\cong",""']
 \arrow[r, ""]
& 
 0
 \\
0
 \arrow[r, ""]
&
E_0
 \arrow[r, "\mathrm{incl}"] 
 &
W \oplus E_0
 \arrow[r, "\textrm{quotient}"] 
&
W
  \arrow[r, ""]
& 
 0
\end{tikzcd}
$$
where the rows are exact. Indeed $E_0:=\ker r^m$ is the generalised $0$-eigenspace of $r$; and the image $W:=r^m(QH^*(Y))$ has stabilised: $r(W)=W$, in other words $r: W \to W$ is an isomorphism.

\begin{thm}\label{Theorem u localisation for monotone}
Let $Y$ be a monotone symplectic $\C^*$-manifold, and $(a,b)$ a free weight. Then we have the following $\ku$-module isomorphisms, compatibly with the short exact sequence in \cref{Corollary torsion freeness 2 free case},
\begin{equation}\label{Equation ESH for monotone}
\begin{split}
E^-SH^*(Y)
&\cong 
SH^*(Y)[\![u]\!] \oplus E_0(\!(u)\!),
\\
E^{\infty}SH^*(Y)
&\cong
SH^*(Y)(\!(u)\!) \oplus E_0(\!(u)\!)
\cong
E^{\infty}QH^*(Y),
\\
E^{+}SH^*(Y)
&\cong SH^*(Y)(\!(u)\!)/u\cdot SH^*(Y)[\![u]\!] \cong 
SH^*(Y)\otimes_{\k}\mathbb{F}.
\end{split} 
\end{equation}
Thus, letting $t:=\dim_{\k}\,SH^*(Y)$, and $m:=t+s=\dim_{\k}\,H^*(Y),$
\begin{equation}\label{Equation iso for Eminus infty plus monotone}
E^-SH^*(Y)\cong \ku^t \oplus \kuu^s,
\quad
E^{\infty}SH^*(Y) \cong E^{\infty}QH^*(Y)\cong \kuu^{m},
\quad
E^+SH^*(Y) \cong \mathbb{F}^{t},
\end{equation}
and the map $E^-c^*:E^-QH^*(Y)\to E^-SH^*(Y)$ becomes the inclusion $\ku^t \oplus \ku^s\hookrightarrow \ku^t \oplus \kuu^s$. 
\end{thm}
\begin{proof}
To prove the claim, we construct the following commutative diagram of $\ku$-module homomorphisms (we omitted in the left column the grading shift $[-2]$):
$$
\begin{tikzcd}%
0
 \arrow[r, ""]
&
E^-SH^*(Y)
 \arrow[r, ""] 
 \arrow[d, "\cong",""']
 &
E^{\infty}SH^*(Y)
 \arrow[r, ""] 
\arrow[d, "\cong",""']
&
E^{+}SH^*(Y)
\arrow[d, "\cong",""']
 \arrow[r, ""]
& 
 0
\\
0
 \arrow[r, ""]
&
W[\![u]\!] \oplus E_0(\!(u)\!)
 \arrow[r, "u"] 
 &
W(\!(u)\!) \oplus E_0(\!(u)\!)
 \arrow[r, ""] 
&
W(\!(u)\!)/uW[\![u]\!]
 \arrow[r, ""]
& 
 0
 \\
0
 \arrow[r, ""]
&
W[\![u]\!] \oplus E_0[\![u]\!]
 \arrow[u, "\mathrm{id}\otimes \mathrm{localise}",""']
 \arrow[r, "u"] 
 &
W(\!(u)\!) \oplus E_0(\!(u)\!)
 \arrow[u, "\mathrm{id}",""']
 \arrow[r, ""] 
&
(W\oplus E_0)(\!(u)\!)/u(W\oplus E_0)[\![u]\!]
 \arrow[u, "\mathrm{quotient}",""']
 \arrow[r, ""]
& 
 0
  \\
0
 \arrow[r, ""]
&
E^-QH^*(Y)
 \ar[uuu, "c^*", bend left=102]
 \arrow[u, "\cong"']
 \arrow[r, ""] 
 &
E^{\infty}QH^*(Y)
\ar[uuu, "c^*_u", bend left=102]
 \arrow[u, "\cong"']
 \arrow[r, ""] 
&
E^+QH^*(Y)
\ar[uuu, "{[c^*_u]}", bend left=102]
 \arrow[u, "\cong"']
 \arrow[r, ""]
& 
 0
\end{tikzcd}
$$
where all rows are exact and involve the obvious maps (see \cref{Corollary torsion freeness 2 free case}).

We can pick a $\k$-vector space basis $w_1,\ldots,w_t,v_1,\ldots,v_s$ for $QH^*(Y)$, where $w_j$ are a basis for $W$, and $v_j$ is a basis for $E_0\subset QH^*(Y)$. This induces a $\k$-linear isomorphism $QH^*(Y)\cong W \oplus E_0$.
By \cref{Prop equivariant formality for QH}, we also obtain a non-canonical isomorphism $$E^{\karo}QH^*(Y) \cong QH^*(Y)\otimes_{\k} E^{\karo}(\mathrm{point}) \cong (W \oplus E_0) \otimes_{\k} E^{\karo}(\mathrm{point}),$$ 
for $\karo\in \{-,\infty,+\}$, giving rise to the bottom two rows of the above commutative diagram. 

In the basis $w_1,\ldots,w_t,v_1,\ldots,v_s$, the matrix for $r^m$ is a block-diagonal matrix: the first $t\times t$ block corresponds to the isomorphism $r^m: W \to W$, the second $s\times s$ block is the zero matrix corresponding to $r^m=0:E_0 \to E_0$.
It follows that $ER_m$ has the form described in \eqref{Equation id block plus higher u}. Similarly, any composite of $m$ maps of type $r_{a,b}$, for varying $a$, will have the form described in \eqref{Equation id block plus higher u}.
Since $ER_{km}$ is a composite of $k$ such composites, it has the form $f_k\circ \cdots \circ f_1$ described in \cref{Lemma invariant factors of id block uorder 1 blocks}. It follows that $ER_k$ has invariant factors $1,\ldots,1,u^{\geq k},\ldots,u^{\geq k}$, with $t$ copies of $1$. None of the $u^{\geq k}$ invariant factors are $u^{\infty}:=0$ because all $ER_k$ maps are injective by \cref{Lemma slice poly for ERk is same as for ck}, using that $(a,b)$ is a free weight.

Thus, we are in the setting of the general theory described in \cref{Subsection Growth rate of the filtration polynomial}, with: 
$$
j_1(km)=\cdots=j_t(km)=0, 
\qquad \textrm{ and } \qquad
k \leq j_{t+1}(km)\leq \cdots \leq j_{t+s}(km).
$$
The claim now easily follows from that general theory. 
\end{proof}

\begin{rmk}\label{Remark E0 basis clarification}
A basis of $E^-QH^*(Y)$ which (for some basis on the codomain) ensures that $E^-c^*:E^-QH^*(Y)\to E^-SH^*(Y)$ is in Smith normal form can be chosen to be $$w_1,\ldots,w_t,v_1+(u^{\geq 1}\textrm{-terms}),\ldots,v_s+(u^{\geq 1}\textrm{-terms}).$$
The vectors $v_i\in E_0$ typically require such $u^{\geq 1}$-corrections. This is the same reason why in \cref{Prop ESH for O-1 over P1} the generator of the $\kuu$-summand in $E^-SH^*(Y)$ is much trickier than the $\ku$-summand.
\end{rmk}

\subsection{The monotone case, for non-free weights}\label{Subsection The monotone case for non-free weights}

\begin{thm}\label{Theorem u localisation for monotone any weight case}
Let $Y$ be a monotone symplectic $\C^*$-manifold, and $(a,b)\neq (0,0)$ any non-free weight. 
Then \eqref{Equation ESH for monotone} and \eqref{Equation iso for Eminus infty plus monotone} hold.
However, $E^-c^*:E^-QH^*(Y)\to E^-SH^*(Y)$ may have kernel: 
it corresponds to a map $\ku^t \oplus \ku^s \to \ku^t \oplus \kuu^s$ with $\ker E^-c^*\subset \ku^s$. 
\end{thm}
\begin{proof}
This follows by the same proof as \cref{Theorem u localisation for monotone}, except that $ER_{km}$ may have kernel, which corresponds to $u^{\geq k}$ invariant factors being $u^{\infty}:=0$.
As $(a,b)\neq (0,0)$ is non-free, we have $b\neq 0$. 
Let $k:=\lceil \tfrac{a}{b}\rceil$.
\cref{Theorem injectivity theorem 2} implies that $ER_{k-1}$ is injective, since $k-1<\tfrac{a}{b}$.
$ER_{k}$ may have kernel, as the $E^{\infty}$-page for $E^{\infty}SH^*(Y)$ may have a non-trivial contribution in the column for slope $\tfrac{a}{b}$, as explained in \cref{Remark non-free weight MB contributions}.
From then on, $\ker ER_{k}=\ker ER_{k+\ell}$ for all $\ell \geq 1$, i.e.\;no new kernel is created, because of \cref{Theorem injectivity theorem 2}.(1).
Call $n$ the rank of that stabilised kernel.
We now assume that $km>\lceil \tfrac{a}{b}\rceil$.
We apply \cref{Remark basis for ker in Smith argument id plus ugeqone} to the map $ER_{km}$: part of the $v_1,\ldots,v_s$ basis for the $\ku^s$ summand, namely $v_{t+s-n+1},\ldots,v_{t+s}$, can be chosen to be a basis for the kernel.
To pass from $ER_{mk}$ to $ER_{m(k+\ell)}$ we compose with a continuation map that still has the form \eqref{Equation id block plus higher u}, and that can be turned into \eqref{Equation change of basis trick} by basis adjustments without changing the $v_i$ (so only adjusting the $w_j$ by adding higher corrections). In particular, the continuation map acting on the $v_i$ sub-basis, modulo the $w_j+$(higher correction) sub-basis, gives rise to invariant factors $u^{\geq \ell},\ldots,u^{\geq \ell}$, and none of these are $u^{\infty}$ because of \cref{Theorem injectivity theorem 2}.(1).  
Thus, although the kernel caused the death of $0\oplus \ku^n\subset \ku^s$, that summand $0\oplus \ku^n$ is reborn as the cokernel of $ER_{mk}$ inside $E_{(a-mkb,b)}^-QH^*(Y)[2mk\mu]\cong \ku^t \oplus \ku^s$. This cokernel contributes generators to the direct limit. The subsequent continuation maps (for slopes $\lambda>mk>\tfrac{a}{b}$) cause the $u$-localisation of the $\ku^s$-summand including the $u$-localisation of the reborn piece $0\oplus \ku^n\subset 0\oplus \ku^s.$
The claim follows.
\end{proof}

\appendix
\section{Equivariant Seidel isomorphisms}

\subsection{The general statement}
\label{Subsection Equivariant Seidel isomorphism} 
Recall our ongoing notation: $Y$ is any
symplectic $\C^*$-manifold with almost complex structure $I$, the $\C^*$-action $\Fi$ has moment map $H:Y \to \R$, and $\mathrm{Fix}(\Fi)=\F =\sqcup \F_{\a}$ denotes the connected components of the fixed locus, in particular $\F_{\min}:=\min H.$
 We use the time-parameter $t\in S^1=[0,1]/(0\sim 1)$, and abbreviate $\Fi_t$ to mean the $S^1$-flow by $\Fi$, whereas when we use the $\R_+$-action we will say so explicitly.
 Following \cite{RZ1}, we use Hamiltonians $H_{\lambda}=c(H):Y \to \R$ of slope $\lambda$ at infinity ($c'(H)=\lambda$ for large $H$).
 
Write $E^{\karo}_{(a,b)}FC^*(H_{\lambda},I)$ for the equivariant Floer complex for weight $(a,b)$ using the equivariant model of type $\karo\in \{-,\infty,+\}$. Recall, from \eqref{Equation action on loops}, this means Floer theory is defined equivariantly with respect to the $S^1$-action $(\theta \cdot_{(a,b)} x)(t) = \Fi_{a\theta}(x(t-b\theta))$ on free loops $x(t)$ in $Y$.
In particular, $\Fi$ is acting as $\Fi^a$ on the target of the reparametrised loop.
In this section, we will also use the symbol $E^{\karo}_{(\psi,b)}$: it means we change $\Fi$ to another $S^1$-action $\psi$ acting on $Y$, and the $S^1$-action on the free loopspace is
$$
(\theta \cdot_{(\psi,b)} x)(t) = \psi_{\theta}(x(t-b\theta)).
$$
Let $\sigma$ be a Hamiltonian $S^1$-action on $Y$, with moment map $K:Y \to \R$.
When we say $\sigma$ {\bf commutes with $\Fi$} we mean that it commutes with the $S^1$-part of the action $\Fi$.
This ensures that arbitrary products of $\sigma,\Fi$ are Hamiltonian $S^1$-actions.
We will prove that the following point exists,
\begin{equation}\label{Equation xmin definition}
x_{\min} 
\in 
 \mathrm{Min}\, K|_{\F_{\min}}
\subset
\F_{\min}\cap \mathrm{Fix}(\sigma).
\end{equation}
Here $\mathrm{Min}$ is the minimal locus in $Y$ where $K$ attains its minimal value $\min K|_{\F_{\min}}\in \R$.
Let $\mu_{\sigma}$ be the {\bf Maslov index} of the loop of symplectomorphisms $\sigma_t \in \mathrm{Aut}(T_{x_{\min}} Y).$
If $\sigma$ is {\ph} ($I \circ d\sigma_t = d\sigma_t \circ I$), then 
$\mu_{\sigma}$ is the sum of the weights of the linearised $S^1$-action on $T_{x_{\min}} Y$ (cf.\,\cite[Sec.5.1]{RZ1}). 
We define the ``pull-back'' data as follows:
$$
(\sigma^*x)(t) = \sigma_{-t}x(t), \qquad 
\sigma^*H_{\lambda} := 
H_{\lambda} - K, \qquad  \sigma^*I = d\sigma_t^{-1} \circ I \circ d\sigma_t.
$$
If $\sigma$ is {\ph} (e.g.\,the  $S^1$-part of a {\ph} $\C^*$-action) then $\sigma^*I=I$.

\begin{thm}\label{Theorem Seidel iso}\!\!\footnote{The left hand side of \eqref{Theorem Seidel iso} is well-defined, and it is a general feature of the construction of \eqref{Theorem Seidel iso} that the Floer complex on right hand side is automatically well-defined because there is a canonical bijection between $1$-orbits and between Floer trajectories, for the two respective choices of Floer data.
There is in fact a way to make both sides of the equation $\Z$-graded, and the map $\Z$-graded, by the capping discussion in \cref{Remark capping discussion Comparison with the literature}.
}
For any Hamiltonian $S^1$-action $\sigma$ that commutes with $\Fi$,
there is a chain isomorphism
    \begin{equation}\label{Equation Seidel iso}
    \mathcal{S}_{\sigma}: E^{\karo}_{(\Fi^a,b)}FC^*(H_{\lambda},I) \cong E^{\karo}_{(\sigma^{-b}\Fi^a,b)}FC^{*}(\sigma^*H_{\lambda},\sigma^*I)[2\mu_{\sigma}], \qquad (\textrm{1-orbit }x) \mapsto \sigma^*x,
    \end{equation}
    where $\mu_{\sigma}$ is the Maslov index of $\sigma$. These isomorphisms commute with continuation maps
in the sense of \cite[Cor.4.8]{Sei97} and \cite[Thm.18]{R14}.
In the case $\sigma=\Fi$, we get an isomorphism on cohomology,
    \begin{equation}\label{Equation Seidel iso 2}
    \mathcal{S}_{\Fi}: E^{\karo}_{(a,b)}HF^*(H_{\lambda},I) \cong E^{\karo}_{(a-b,b)}HF^{*}(H_{\lambda-1},I)[2\mu], \qquad (\textrm{1-orbit }x) \mapsto \Fi^*x,
    \end{equation}    
    where $\mu$ is the Maslov index of $\Fi$.
    At chain level, \eqref{Equation Seidel iso 2} holds if we replace $H_{\lambda-1}$ by $\Fi^*H_{\lambda} = H_{\lambda}-H$.
In this case, compatibility with continuation maps is the commutative diagram:
\begin{equation}\label{Equation compatibility of seidel iso with continuation}
\begin{tikzcd}[column sep=0.6in]
E^{\karo}_{(a,b)}HF^*(H_{\lambda})
 \arrow[r, "\textrm{continuation}"] 
 \arrow[d, "\cong","\mathcal{S}_{\Fi}"']
 &
E^{\karo}_{(a,b)}HF^*(H_{\lambda+\lambda'})
\arrow[d, "\cong","\mathcal{S}_{\Fi}"']
\\
E^{\karo}_{(a-b,b)}HF^*(H_{\lambda-1})[2\mu]
 \arrow[r, "\textrm{continuation}"]  
 &
E^{\karo}_{(a-b,b)}HF^*(H_{\lambda+\lambda'-1})[2\mu]
\end{tikzcd}
\end{equation}
    \end{thm}

\subsection{Comparison with the literature.}
\label{Remark capping discussion Comparison with the literature}
\label{Definition canonical lift}
The non-equivariant version of \eqref{Equation Seidel iso} for closed symplectic manifolds is due to Seidel \cite{Sei97}; the version for symplectic cohomology of convex symplectic manifolds is due to Ritter \cite{R14,R16}. These were adapted to the $E^-$-equivariant setup by Liebenschutz-Jones \cite{liebenschutz2020intertwining,liebenschutz2021shift}.
We are adapting it to symplectic $\C^*$-manifolds, for all three equivariant versions $E^{\karo}$, using the {\MBF} model setup from \cite[Appendix]{RZ2}.

Seidel \cite{Sei97} considers more generally loops of Hamiltonian diffeomorphisms. In that case $K_t$ is time-dependent, and $\sigma^*H_{\lambda}=H_{\lambda}\circ \sigma_t - K_{t}\circ \sigma_t$. In our setup, $K=K_t$ is time-independent, and the commutativity of $\sigma,\Fi$ ensures $\F_{\min}\cap \mathrm{Fix}(\sigma)\neq \emptyset$ (proved later); it follows that\footnote{e.g.\,see \cite[Proof of (5.2)]{liebenschutz2020intertwining}.} $\omega(X_H,X_K)=0$, so 
\begin{equation}\label{Equation Hams preserved by action}
H\circ \sigma_t=H \qquad \textrm{ and }\qquad  K\circ \Fi_t=K.
\end{equation}
In Seidel \cite{Sei97}, the Floer chain complex only allows contractible Hamiltonian $1$-orbits, and in fact the generators of the chain complex require a choice of {\bf capping}, a smooth map $\widetilde{x}:\mathbb{D}\to Y$ bounding the $1$-orbit $x: S^1 \to Y$.
One then identifies caps if their difference is a spherical classes in 
$$
G:= 
\pi_2(Y)/(\ker [\omega] \cap \ker c_1(Y)).
$$
More precisely, $\widetilde{x}$ lives in a (connected) covering space
\begin{equation}\label{Equation cover of loopspace}
\widetilde{\mathcal{L}_0Y} \to \mathcal{L}_0Y,\quad \widetilde{x} \mapsto x,
\end{equation}
of the space of contractible free loops in $Y$, with abelian deck group $G$. The Novikov ring used by Seidel is more complicated than in this paper, because it involves using $G$ \cite[Definition 3.10]{Sei97}. 
Seidel shows that $\sigma: \mathcal{L}_0 Y \to \mathcal{L}_0 Y$, $x\mapsto \sigma^*x$, admits a lift $\widetilde{\sigma}$ to $\widetilde{\mathcal{L}_0Y} \to \widetilde{\mathcal{L}_0Y}$, $\widetilde{x} \mapsto \widetilde{\sigma}(\widetilde{x})$. This requires showing that $\sigma$ preserves $\mathcal{L}_0 Y$, using the Arnol'd conjecture
for closed symplectic manifolds  \cite[Lemma 2.2]{Sei97}. In these conventions, the chain isomorphism \eqref{Equation Seidel iso} arises by linearly extending $\widetilde{x}\mapsto \widetilde{\sigma}(\widetilde{x}).$ 

In our setup, we do not need to appeal to the Arnol'd conjecture: $\sigma$ has a fixed point by \eqref{Equation xmin definition}, so $\sigma$ fixes the {\bf constant cap} $c_{x_{\min}}$ at $x_{\min}$, therefore $\sigma^*(\mathcal{L}_0Y)$ cannot land in any other component of $\mathcal{L}Y$ other than $\mathcal{L}_0 Y$. 
We choose a specific lift $\widetilde{\sigma}$, called {\bf canonical lift}, determined by the condition
\begin{equation}\label{Definition of canonical lift in terms of caps}
\widetilde{\sigma}(c_{x_{\min}})=c_{x_{\min}}.
\end{equation}
This lift does not depend on the choice of $
x_{\min}$,  because we will see that $\mathrm{Min}\, K|_{\F_{\min}}$ is path-connected.

In our case, the $1$-orbits of $H_{\lambda}$ correspond to orbits (of various periods $p\in \Q_{\geq 0}$) of the $S^1$-flow $\Fi_t$, and such orbits admit {\bf canonical caps} $\widetilde{x}_{\mathrm{can}}$ given by applying the $\C^*$-action $\Fi$ (see \eqref{Equation Cstar action cap}).
However, for orbits on the right in \eqref{Equation Seidel iso} there are no canonical caps because $\sigma^{-b}\Fi^a$ is just a Hamiltonian $S^1$-action, not a contracting $\C^*$-action.
We also use a simpler Novikov ring: as in \cite[Sec.2A,5B]{R16} we have a formal variable $T$ to keep track of $[\omega]$-values; when $c_1(Y)=0$ or $Y$ is monotone ($c_1(Y)\in \R_{>0}[\omega]$, placing $T$ in a suitable positive degree) this ensures a $\Z$-grading; beyond that, one would need another formal variable to keep track of $c_1(Y)$-values or one makes do with a $\Z/2$-grading. We comment further on these technical issues in the two Remarks below.

We then declare that, on the left in \eqref{Equation Seidel iso}, $x$ comes with its canonical cap $\widetilde{x}_{\mathrm{can}}$, whereas on the right $\sigma^*x$ comes with the {\bf induced cap} $\widetilde{\sigma}(\widetilde{x}_{\mathrm{can}})$ determined by the canonical lift $\widetilde{\sigma}$.
From \eqref{Equation Seidel iso}, one can recover the version for Seidel's Novikov ring (see \cref{Remark passing between two Novikov rings}), and the bijection on generators that induces \eqref{Equation Seidel iso} is correctly $\Z$-graded if we take into account these capping conventions.

\begin{rmk}\label{Remark change of cappings}
Suppose we wish to change cappings on the right in \eqref{Equation Seidel iso}. Abbreviate $y=\sigma^*x$, suppose we change induced caps $\widetilde{y}_{\mathrm{ind}}$ to new caps $\widetilde{y}_{\mathrm{new}}$. We get a spherical class $g_y:=\widetilde{y}_{\mathrm{ind}}\#(-\widetilde{y}_{\mathrm{new}})\in G$, with deck group action $\widetilde{y}_{\mathrm{ind}}=g_y(\widetilde{y}_{\mathrm{new}})$. Suppose we use the Novikov ring
$$\mathbb{K} (\!(S)\!)=\Big\{\textstyle\sum k_{n_i} S^{n_i}: k_{n_i}\in \mathbb{K}, n_i\in \Z, n_i\to \infty\Big\},$$
where $\mathbb{K}=\{\sum n_j T^{a_j}: a_j\in \R, a_j \to \infty, n_j\in \mathbb{B}\}$, and $\mathbb{B}$ is a choice of base field with $\mathrm{char}(\mathbb{B})=0$.
We place the formal variables $T,S$ in degrees $|T|=0$, $|S|=2$. The above corresponds to a change of basis on the right in \eqref{Equation Seidel iso},%
\begin{equation}\label{Equation change of cappings}
y\mapsto S^{c_1(Y)[g_y]}T^{[\omega](g_y)} y.
\end{equation}
This correctly keeps track of the $\Z$-grading, because the grading of $y$ thought of as $\widetilde{y}_{\mathrm{new}}$ has increased by $2c_1(Y)(g_y)$ compared to $\widetilde{y}_{\mathrm{ind}}$ (using a general property of Conley-Zehnder indices when one glues in a sphere).
When $c_1(Y)=0$, one does not need $S$.
When $c_1(Y)\in \R_{>0}[\omega]$, say $c_1(Y)=k[\omega]$,
one can omit $S$ by placing $T$ in degree $|T|=2k$.
In all other cases, one can omit $S$ at the cost of downgrading the $\Z$-grading to a $\Z/2$-grading.
\end{rmk}

\begin{rmk}\label{Remark passing between two Novikov rings}
We briefly explain the passage between Seidel's Novikov ring and ours. 
Suppose $u$ is a Floer trajectory from $x$ to $y$, for a Hamiltonian $H_{\lambda}$.
Its energy is $E(u)=\int u^*\omega + H_{\lambda}(x)-H_{\lambda}(y)$.
Omitting the orientation sign, it contributes 
$T^{E(u)}x$ to the Floer differential $d(y)$. Now consider cappings:
gluing a capping $\widetilde{x}$ with $u$ yields a capping of $y$: the one obtained by lifting $u$ to the cover starting from $\widetilde{x}$. 
If we already prescribed a capping $\widetilde{y}$, then over the Novikov ring using cappings the contribution to $d(\widetilde{y})$ is $g^{-1}\cdot \widetilde{x}$, where $g=\widetilde{x}\# u \# (-\widetilde{y})\in G.$ 
By construction, $g$ is determined by the $[\omega],c_1(Y)$ values on the spherical class $g$. The $c_1(Y)$ value is determined by grading considerations. The $[\omega]$ value is determined by $\int u^*\omega$, and $E(u)$ recovers $\int u^*\omega$. In conclusion, we recover $g$.
The Hamiltonian contributions in $E(u)$ are harmless as they are $u$-independent:
the change of basis $x\mapsto T^{H_{\lambda}(x)}x$ induces a chain isomorphism, after which the above contribution $T^{E(u)}x$ changes to $T^{\int u^*\omega}x$.

In conclusion, the passage from Seidel's Novikov ring to our Novikov ring, means the following replacement of generators on the two respective sides of \eqref{Equation Seidel iso}:
\begin{equation}\label{Equation cap to capless}
\widetilde{x}_{\mathrm{can}}
\mapsto
T^{-H_{\lambda}(x)}x
\qquad
\textrm{ and }
\qquad
\widetilde{\sigma}(\widetilde{x}_{\mathrm{can}}) \mapsto
T^{-(\sigma^*H_{\lambda})(\sigma^*x)}\sigma^*x.
\end{equation}
Therefore, the more ``natural'' way to write \eqref{Equation Seidel iso} (which corresponds to $\widetilde{x}_{\mathrm{can}}\mapsto \widetilde{y}_{\mathrm{ind}}$), taking into account those basis changes, is:
\begin{equation}\label{Equation Seidel iso with Ham corrections}
x\mapsto 
T^{H_{\lambda}(x)-(\sigma^*H_{\lambda})(\sigma^*x)}\,\sigma^*x.
\end{equation}
\end{rmk}
\subsection{Explanation of twist in the $S^1$-actions}

Recall the $S^1$-action on free contractible loops, 
\begin{equation}\label{Equation free loop space action}
\sigma^*: \mathcal{L}_0Y \to \mathcal{L}_0Y,\;
(\sigma^*x)(t) = (\sigma^{-1}x)(t)=\sigma_{-t} x(t).
\end{equation}

\begin{lm}\label{Lemma sigma permutes fixed components Falpha}
Let $\sigma$ be a Hamiltonian $S^1$-action commuting with $\Fi.$
Then 
\begin{enumerate}
    \item $\sigma_t(\F_{\a})=\F_{\a}$ for all $t$, for all connected components $\F_{\a}$ of $\F=\mathrm{Fix}(\Fi)$,
    \item $\sigma$ is a Hamiltonian $S^1$-action on the connected closed symplectic submanifold $\F_{\a}\subset Y$,
    \item  
    $\sigma^*: \{$constant $\Fi$-orbits in $\F_{\a}\}\to \{$possibly non-constant 1-orbits for $\sigma^{-1}$ in $\F_{\a}\}$ is surjective,
    \item $ \F_{\a}\cap \mathrm{Fix}(\sigma) = \sqcup ($connected components of $\mathrm{Crit}(K|_{\F_{\a}}))$, one of which is $\mathrm{Min}\, K|_{\F_{\a}}$,
    \item in particular $ \F_{\a}\cap \mathrm{Fix}(\sigma)\neq \emptyset$, and these are the constant orbits arising on the right in (3).
\end{enumerate}
\end{lm}
\begin{proof}
By \cite{RZ1}, $\F_\a\subset Y$ is a closed symplectic submanifold. Commutativity of $\sigma,\Fi$ implies that, $\sigma_t(\F_{\a})=\F_\a$ for each $t$: indeed, if $x\in \F_{\a}$, then $\sigma_t(x)=\sigma_t(\varphi_s(x))=\varphi_s(\sigma_t(x))$, so $\sigma_t(x)\in \mathrm{Fix}(\varphi)=\F$. Since $\sigma_0=\mathrm{id},$ it follows that $\sigma_t(x)$ is trapped in $\F_\a$. Thus $\sigma_t(\F_{\a})\subset \F_\a$. By considering $\sigma_{-t}$ we also obtain the reverse inclusion, so it is an equality.
Finally, $\mathrm{Min}\, K|_{\F_{\a}}$ is non-empty as $\F_{\a}$ is compact, and (as we assume $Y$ is connected) it is connected e.g.\,by \cite{Frankel59} (cf.\,\cite[Sec.1]{RZ1}).
\end{proof}

With $\sigma$ as in the Lemma, we get an action on the Borel model by $(w,x)\mapsto (w,\sigma^*x)$. This action does not in general preserve the relation \eqref{Equation Borel model identification}:
\begin{lm}
The following diagram commutes
$$
\xymatrix@C=50pt{
 (w,x)
\ar@{->}[r]^-{\sigma\textrm{-action}} 
\ar@{->}[d]_-{\theta\textrm{-action}}  
&
(w,\sigma^*x)
\ar@{->}^-{\theta\textrm{-action}}[d] 
\\
(\theta^{-1}\cdot w,\theta\cdot_{(a,b)}x)
\ar@{->}[r]^-{\sigma\textrm{-action}} 
& (\theta^{-1}w,\sigma^*(\theta\cdot_{(a,b)}x)) = (\theta^{-1}w,\theta\cdot_{(a-b,b)}\sigma^*x)
}
$$
Note that in the target instead of using the $\Fi^a$ action on $Y$, we are using the $\Fi^{a-b}$ action on $Y$.

Suppose either $a>b\geq 0$, or $a<b\leq 0$. Then the weight $(a,b)$ and the target weight $(a-b,b)$ are either both free, or both not free.
\end{lm}
\begin{proof}
This follows from 
$\sigma_t^{-1}\Fi_{a\theta}(x(t-b\theta))
=
\Fi_{(a-b)\theta}(\sigma_{t-b\theta}^{-1}x(t-b\theta))
$, using $\sigma_t = \Fi_t.$
The final claim is immediate if $b=0$, so assume $b\neq 0$.
If $(a,b)$ is not free, then $\tfrac{a}{b}=\tfrac{k}{m}$, some $k\geq 1$ and some weight $m\geq 1$ of $\Fi$, hence $\tfrac{a-b}{b}=\tfrac{k-m}{m}$, and $k-m>0$ (so $\geq 1$) by the assumptions on $a,b$, so $(a-b,b)$ is not free. If $(a-b,b)$ is not free, then $\tfrac{a-b}{b}=\tfrac{k'}{m}$, so $\tfrac{a}{b}=\tfrac{k'+m}{m}$, and $k'+m> k' \geq 1$, so $(a,b)$ is not free.
\end{proof}

Using the notation from \cite{liebenschutz2020intertwining} that we already used in \eqref{Equation equivariant Hamiltonians}, the key idea from \cite{Sei97} is that the pull-back data 
$
\sigma^*H_{w,t}^{\mathrm{eq}} := 
H_{w,t}^{\mathrm{eq}} \circ \sigma_t - K\circ \sigma_t$ and $\sigma^*I = d\sigma_t^{-1} \circ I \circ d\sigma_t,
$
relates the Floer action $1$-forms
$$
(dA_{H_{w,t}^{\mathrm{eq}}})|_x  = (dA_{\sigma^*H_{w,t}^{\mathrm{eq}}})|_{\sigma^*x}\circ d\sigma^*|_x: T_x\mathcal{L}Y \to \R.
$$
Thus there is a bijection of $1$-orbits, $\mathrm{Zeros}(dA_{H_{w,t}^{\mathrm{eq}}}) \to \mathrm{Zeros}(dA_{\sigma^*H_{w,t}^{\mathrm{eq}}})$, $x\mapsto \sigma^*x$. Similarly, the definition of $\sigma^*I$ ensures a bijection of Floer trajectories via $u\mapsto \sigma^*u$, with $(\sigma^*u)(s,t)=\sigma_{-t} u(s,t)$.

\begin{lm}
Let $\sigma$ be any Hamiltonian $S^1$-action $\sigma$ that commutes with $\Fi$.
Suppose we use the weight $b$ loop-action, and the action on $Y$ by $\sigma^*\Fi^a=\sigma^{-b}\circ \Fi^a$ (when $\sigma=\Fi$, this is $\Fi^{a-b}$).
Then the time-dependent Hamiltonian function $\sigma^*H^{\mathrm{eq}}_{w,t}: S^{\infty} \times Y \to \R$ is $S^1$-equivariant.
\end{lm} 
\begin{proof} By definition of all $\theta$-actions in play, using commutativity of $\sigma,\Fi$, \eqref{Equation equivariant Hamiltonians} and  \eqref{Equation Hams preserved by action}:
   $$
   \begin{array}{rcl}
   \sigma^*H^{\mathrm{eq}}_{\theta^{-1}\cdot w,t+b\theta} \circ (\sigma^*\Fi^a)|_{\theta} & = & H^{\mathrm{eq}}_{\theta^{-1}\cdot w,t+b\theta} \circ \sigma_{t+b\theta} \circ \sigma_{-b\theta}\circ\Fi_{a\theta} - K \circ \sigma_{t+b\theta} \circ 
   \sigma_{-b\theta}\circ\Fi_{a\theta} \\
   &=&  
   H^{\mathrm{eq}}_{\theta^{-1}\cdot w,t+b\theta} \circ \Fi_{a\theta} \circ \sigma_{t}  - K \circ  
   \Fi_{a\theta} \circ \sigma_{t} \\
   & = & H^{\mathrm{eq}}_{w,t} \circ \sigma_{t} - K \circ \sigma_{t}\\
   &=& \sigma^*H^{\mathrm{eq}}_{w,t}.\qedhere
   \end{array}
   $$
\end{proof}

\subsection{Further discussions of cappings.}
\label{Remark A discussion of cappings}
\label{Remark how to compute tildesigma}
We now continue the discussion from \cref{Remark capping discussion Comparison with the literature}. 
The cover \eqref{Equation cover of loopspace} consists of all classes of cappings $\widetilde{x}$, with deck group action given by gluing in spherical classes from $G$.
Since \eqref{Equation cover of loopspace} is a covering map, a ($G$-equivariant) lift 
$\widehat{\sigma}:\widetilde{\mathcal{L}_0Y}\to \widetilde{\mathcal{L}_0Y}$ of the action \eqref{Equation free loop space action} is uniquely determined by where it sends the constant cap $c_{x_{\min}}$ at the $x_{\min}$ from \eqref{Equation free loop space action}. So it is determined by a choice of capping $\widehat{\sigma}(c_{x_{\min}})=\widetilde{x_{\min}}$ of $x_{\min}$.
Explicitly, if $\widetilde{\nu}:[0,1]\to \widetilde{\mathcal{L}_0Y}$ is a path starting at $\widetilde{\nu}(0)=\widetilde{x_{\min}}$, then $\widehat{\sigma}(\widetilde{\nu}(s))=\widetilde{\nu}(s)$ for all $s$.
The {\bf canonical lift} $\widetilde{\sigma}$ is defined by requiring that it fixes the constant cap $c_{x_{\min}}$.
Since $\widetilde{x_{\min}}=gc_{x_{\min}}$ for some $g\in G$, we see that the lifts are classified as the $G$-images of the canonical lift: $\widehat{\sigma}=g\widetilde{\sigma}$ for $g\in G.$

\begin{center}
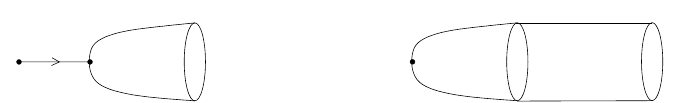
\end{center}

\begin{lm}\label{Lemma capping lemma 1}
Let $\widetilde{x}:\mathbb{D}\to Y$ be a capping of any loop $x$ in $Y$. 
Let $\gamma:[0,1]\to Y$ be a path from $x_{\min}$ to the centre of the cap, $z:=\widetilde{x}(0)$.
Let $\sigma^*\gamma$ be the capping of the loop $\sigma^*z=\sigma_{-t}z$ given by $\sigma_{-t}\gamma(s).$
Let $\sigma^*\widetilde{x} \in\widetilde{\mathcal{L}_0Y}$ be the obvious cylinder $\sigma_{-t}\widetilde{x}(re^{2\pi i t})$ joining $\sigma^*z$ and $\sigma^*x$.
Then the canonical cap satisfies
$$
\widetilde{\sigma}(\widetilde{x})=
(\sigma^*\gamma)\# \sigma^* \widetilde{x}.
$$
Also, $\sigma^*\gamma$ depends on $z$,$\sigma$, not on $x_{\min},\gamma$.
If $\widetilde{x}=c_x$ is the constant disc, then $\widetilde{\sigma}(c_x)=\sigma^*\gamma.$

If $x\in \mathrm{Fix}(\sigma)$, then $\widetilde{\sigma}(c_x)=A^{\sigma}_x\, c_x$ where $A^{\sigma}_x=\sigma^*\gamma\in G$ is a spherical class dependent only on $\sigma,x$.
\end{lm}
\begin{proof}
We first prove that $\widetilde{x} = \gamma \# \widetilde{x}$ in $\widetilde{\mathcal{L}_0Y}$.
The path $s_0 \mapsto (\gamma)|_{[s_0,1]} \# \widetilde{x}$ in $\widetilde{\mathcal{L}_0Y}$, from 
$\gamma \# \widetilde{x}$ to $\widetilde{x}$,
consists of cappings of $x$.
It must be a constant path because its projection via the covering map \eqref{Equation cover of loopspace} is a constant path at $x$. Thus $\widetilde{x} = \gamma \# \widetilde{x} \in \widetilde{\mathcal{L}_0Y}$.

Concatenate the path $s_0\mapsto \sigma^*(\gamma|_{[0,s_0]})$ with the path $r_0 \mapsto \sigma_{-t}\widetilde{x}(re^{2\pi it})|_{r\in [0,r_0],t\in S^1}$ to get a path $\widetilde{\nu}:[0,1]\to \widetilde{\mathcal{L}_0Y}$ starting at $\sigma^*(\gamma|_{s=0})=c_{x_{\min}}$. By the previous discussion, $\widetilde{\sigma}(\widetilde{x})=\widetilde{\sigma}(\gamma\#\widetilde{x})=\widetilde{\nu}(1)$, and the equation in the claim follows. It follows, a posteriori, from the equation that $\sigma^*\gamma$ does not depend on $\gamma$. Independence from $x_{\min}$ holds because $\mathrm{Min}\, K|_{\F_{\min}}$ is connected by \cref{Lemma sigma permutes fixed components Falpha}.
\end{proof}

The following picture should be helpful while reading the subsequent discussion.
\begin{center}
\includegraphics[scale=0.20]{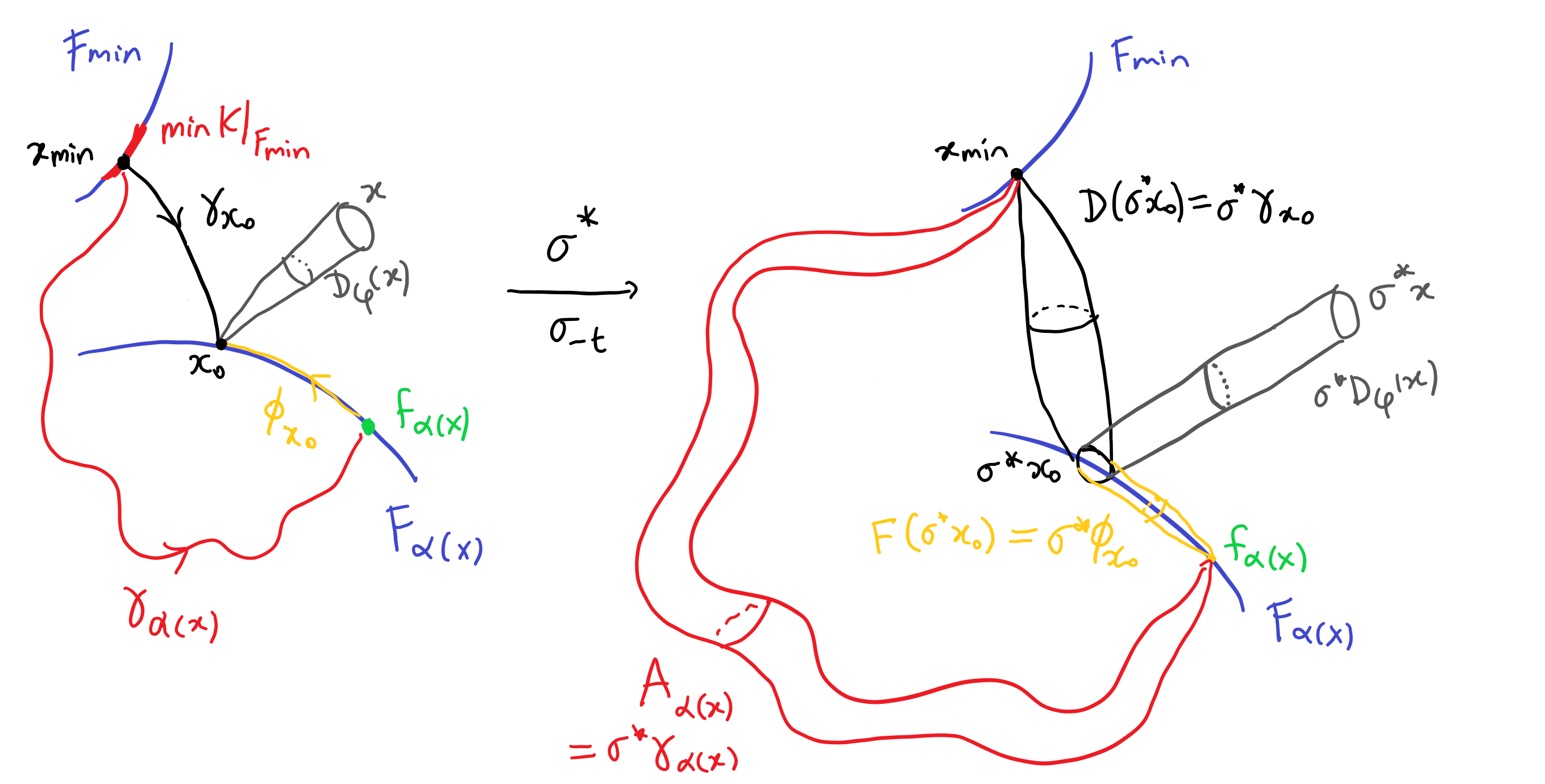}
\end{center}

Let $x$ be a $1$-orbit of $H_{\lambda}$, so a $1$-orbit of $\Fi_{pt}$ where $p=c'(H(x))$.
Following \cite[Sec.4.3]{RZ1}, the {\bf canonical capping} $D_{\Fi}(x)$ of $x:[0,1]\to Y$ is represented by the pseudo-holomorphic disc $\widetilde{x}_{\mathrm{can}}$ obtained by applying the $\R_{\leq 0}$-action; explicitly: we apply the $\C^*$-action by $\Fi$, 
\begin{equation}\label{Equation Cstar action cap}
\widetilde{x}_{\mathrm{can}}:\R_{\leq 0} \times [0,1] \to Y, \; (s,t)\mapsto e^{2\pi  p(s+it)}\cdot x(0),
\end{equation}
and we add in the {\bf convergence point} $x_0$ of $x$, which is the $\Fi$-fixed point obtained as $s\to -\infty$.
The convergence point $x_0$ lies in some component $\F_{\a}\subset \F=\mathrm{Fix}(\Fi)$. We denote this $\a$ by $\a(x)$.
Let $\gamma_{x_0}:[0,1]\to Y$ be a path from $x_{\min}$ to $x_0$. By \cref{Lemma capping lemma 1}, for 
$D(\sigma^*x_0):=\sigma^*\gamma_{x_0}$ and $\sigma^*D_{\Fi}(x)=\sigma^*\widetilde{x}_{\mathrm{can}}$: 

\begin{cor}\label{Corollary to Lemma capping lemma 1}
The cap $D(\sigma^*x_0)\in \widetilde{\mathcal{L}_0Y}$ only depends on $x_0$, $\sigma$, not on the choices $x_{\min},\gamma_{x_0}$, and
$$
\widetilde{\sigma}(D_{\Fi}(x)) =
D(\sigma^*x_0)\# \sigma^*D_{\Fi}(x).
$$
\end{cor}

\begin{rmk}
    If the $S^1$-action $\sigma$ commutes with the $\C^*$-action $\Fi$ (e.g.\,for commuting $\C^*$-actions $\sigma,\Fi$), then $\sigma^*D_{\Fi}(x)$ is the capping of $\sigma^*x$ obtained by applying the $\R_{\leq 0}$-action by $\Fi$ analogous to \eqref{Equation Cstar action cap}. 
\end{rmk}

By \cref{Lemma sigma permutes fixed components Falpha}, $\mathrm{Min}\, K|_{\F_{\a}}\subset \F_{\a}\cap \mathrm{Fix}(\sigma)$ is connected. 
Pick any $f_{\a}\in \mathrm{Min}\, K|_{\F_{\a}}$.
Pick a path $\gamma_{\a}:[0,1]\to Y$ from $x_{\min}$ to $f_{\a}$. By \cref{Lemma capping lemma 1} we get a spherical class, 
$$
A_{\a}:=A^{\sigma}_{f_\a}=\sigma^*\gamma_{\a} \in G.
$$
Pick a path $\phi_{x_0}:[0,1]\to \F_{\a(x)}$ from $f_{\a(x)}$ to $x_0$, to get a cap $F(\sigma^*x_0):=\sigma^*\phi_{x_0}\subset \F_{\a(x)}$ for $\sigma^*x_0$. 
It lies in $\F_{\a(x)}$, by \cref{Lemma sigma permutes fixed components Falpha}.(1).

\begin{cor}\label{Cor another formula for spherical class}
The spherical class $A_{\a}\in G$ only depends on $\a$, $\sigma$, not on the choices $x_{\min},f_{\a},\gamma_{\a}$.
The cap $F(\sigma^*x_0)\in \widetilde{\mathcal{L}_0Y}$ only depends on $x_0$, $\sigma$, not on the choices $f_{\a(x)},\phi_{\a(x)}$, and
\begin{equation}\label{Equation capping lemma 2}
\widetilde{\sigma}(D_{\Fi}(x)) =
A_{\a(x)}\# F(\sigma^*x_0) \# \sigma^*D_{\Fi}(x).
\end{equation}
\end{cor}
\begin{proof}
By \cref{Lemma capping lemma 1},
$\widetilde{\sigma}(c_{f_{\a}})=A_{\a}c_{f_{\a}}$ and $A_{\a}$ only depends on $\a,\sigma$. As $A_{\a}$ does not depend on the choice of path, we can replace $\gamma_{\a(x)}$ by $\gamma_{x_0}\#-\phi_{x_0}$ to deduce $A_{\a(x)}=D(\sigma^*x_0)\# -F(\sigma^*x_0)$ in $G.$
By \cref{Lemma capping lemma 1}, 
$D(\sigma^*x_0)$ does not depend on the choice of path, so we can replace $\gamma_{x_0}$ by $\gamma_{x_0}\#-\phi_{x_0}\#\phi_{x_0}$ to deduce 
$D(\sigma^*x_0) = D(\sigma^*x_0)\# (-F(\sigma^*x_0)) \#F(\sigma^*x_0)\in \widetilde{\mathcal{L}_0Y}$.
The claim follows by \cref{Corollary to Lemma capping lemma 1}.
\end{proof}

\subsection{Beyond cappings.}
\label{Subsection Beyond cappings}
We continue the discussion in Remarks \ref{Remark change of cappings}-\ref{Remark passing between two Novikov rings}.
Let $c$ be a closed differential $2$-form representing $2c_1(Y)$.
Let $v:\mathbb{D}\to Y$ be a smooth disc $v:\mathbb{D}\to Y$.
Abbreviate 
$$\textstyle
\omega(v):=\int v^*\omega, \qquad c(v):=\int v^*c, \qquad 
\intercal(v):=S^{c(v)} T^{\omega(v)}.
$$
If $\widetilde{x}$ is a capping of a $1$-orbit $x$, we can declare that 
$\intercal(\widetilde{x})x$ (over our Novikov ring) correponds to the generator $\widetilde{x}$ over Seidel's Novikov ring.
In this convention, $x$ is a generator over our Novikov ring but it typically does not correspond to a capping (it corresponds to the constant cap $c_x$ when $x$ is constant).
Under this convention,
abbreviating $\widetilde{x}_{\mathrm{can}}:=D_{\Fi}(x)$ and $\widetilde{y}_{\mathrm{ind}}:=\widetilde{\sigma}(\widetilde{x}_{\mathrm{can}})$,
the map \eqref{Equation Seidel iso} becomes
$x \mapsto \intercal(\widetilde{x}_{\mathrm{can}})^{-1} \intercal\!(\widetilde{y}_{\mathrm{ind}})\, y$.
For example, \eqref{Equation change of cappings} would arise from the observation 
$\intercal(\widetilde{y}_{\mathrm{ind}})\, y = 
\intercal(g_y)
\intercal\!(\widetilde{y}_{\mathrm{new}})
\,y.
$
If we wish to apply the convention \eqref{Equation Seidel iso with Ham corrections}, then  \eqref{Equation Seidel iso} becomes
\begin{equation}\label{Equation capless seidel map}
x \mapsto 
T^{H_{\lambda}(x)-(\sigma^*H_{\lambda})(\sigma^*x)}\,
\intercal(\widetilde{x}_{\mathrm{can}})^{-1} \intercal\!(\widetilde{y}_{\mathrm{ind}})\, y.
\end{equation}
This formula seems complicated, but has the advantage of being completely independent of cappings. Another motivation comes from applications where $c_1(Y)=0$ or $c_1(Y)\in \R_{>0}[\omega]$: there we can omit the $S$-variable, and it is the messy \eqref{Equation capless seidel map} which will in fact yield the neatest expression, in 
\eqref{Equation Seidel iso neat formula}.
To obtain this, we first require formulas for the $[\omega]$-values of cappings. 

\begin{lm}\label{Lemma omega area formulas}
For any $1$-orbit $x$ of $H_{\lambda}$, arising at slope $p=c'(H(x)),$ we have:
$$
\begin{array}{rcl}
\omega(D_{\Fi}(x)) &=&
H_{\lambda}(x)-pH(\F_{\a(x)}),
\\[1.5mm]
\omega(\sigma^*D_{\Fi}(x)) &=& 
(\sigma^*H_{\lambda})(\sigma^*x)-
pH(\F_{\a(x)})+K(x_0),
\\[1.5mm]
\omega(D(\sigma^*x_0))
&=&
\min K|_{\F_{\min}}-K(x_0),
\\[1.5mm]
\omega(\widetilde{\sigma}(D_{\Fi}(x))) &=&
\min K|_{\F_{\min}}+(\sigma^*H_{\lambda})(\sigma^*x)-
pH(\F_{\a(x)}).
\end{array}
$$
Moreover, 
$
\omega(A_{\a(x)})=
\min K|_{\F_{\min}}
-
\min K|_{\F_{\a(x)}}
$ 
and\;
$
\omega(F(\sigma^*x_0))=
\min K|_{\F_{\a(x)}}
-
K(x_0)
$.
\end{lm}
\begin{proof}
Let $\psi_t$ be the Hamiltonian flow for some time-independent Hamiltonian $K_{\psi}:Y \to \R$.
   Let $u(s,t)=(\psi^*\gamma)(s,t):=\psi_{-t}\gamma(s),$ 
   for a smooth curve $\gamma$ in $Y$ from $y_-$ to $y_+$.
Then $\partial_t u = - X_{K_{\psi}}|_{\gamma(s)}$, so
$$\textstyle
\omega(u)
:= \int u^*\omega 
= \int \omega(\partial_s u, \partial_t u)\, ds\,dt
= -\int  dK_{\psi}(\partial_s u)\,ds \,dt
= -\int \partial_s (K_{\psi}\circ u)\,ds\,dt
= K_{\psi}(y_-)-K_{\psi}(y_+).
$$
The case $D_{\Fi}(x)$ from \eqref{Equation Cstar action cap} corresponds to $\psi_{-t} = \Fi_{pt}$. We take $K_{\psi}=-pH$, the curve $\gamma(s)=e^{2\pi ps}\cdot x(0)$ for $s\in (-\infty,0]$, with $y_-=x_0$ (the convergence point as $s\to -\infty$), and $y_+=x(0)$. Then $\omega(D_{\Fi}(x)) =
pH(x(0))-pH(x_0)$, and the claim follows (the first equation).

The case $\sigma^*D_{\Fi}(x)$ corresponds to taking $\psi_{-t} = \sigma_{-t} \Fi_{pt}$, so we take $K_{\psi}=K-pH$, and the same data $\gamma$, $y_-$, $y_+$ as above. Then $\omega(\sigma^*D_{\Fi}(x)) =(pH-K)(x(0))-(pH-K)(x_0)$, which yields the claim.

The case $\omega(D(\sigma^*x_0))$ uses $\psi_t=\sigma_t$, $K_{\psi}=K$, $\gamma=\gamma_{x_0}.$ So $\omega(D(\sigma^*x_0))
=
K(x_{\min})-K(x_0).$

By \cref{Corollary to Lemma capping lemma 1}, $\omega(\widetilde{\sigma}(D_{\Fi}(x))) =\omega(\sigma^*D_{\Fi}(x))+\omega(D(\sigma^*x_0))$. The last claim uses \cref{Cor another formula for spherical class}.
\end{proof}

\begin{cor}
In the convention \eqref{Equation capless seidel map}, but ignoring the formal $S$ variable, \eqref{Equation Seidel iso} becomes
\begin{equation}\label{Equation Seidel iso neat formula}
    x \mapsto 
    T^{\min K|_{\F_{\min}}}\,\sigma^*x.
    \end{equation}
In particular, \eqref{Equation Seidel iso 2} becomes 
$x\mapsto T^{H(\F_{\min})}\,\Fi^*x$.
\end{cor}
\begin{proof}
    Omitting $S$,
    $\intercal(\widetilde{x}_{\mathrm{can}}) =
T^{H_{\lambda}(x)-pH(\F_{\a(x)})}$,
and
    $\intercal(\widetilde{y}_{\mathrm{ind}}) =
T^{\min K|_{\F_{\min}}+(\sigma^*H_{\lambda})(\sigma^*x)-
pH(\F_{\a(x)})}$, by \cref{Lemma omega area formulas}.
Then
$\intercal(\widetilde{x}_{\mathrm{can}})^{-1}\intercal(\widetilde{y}_{\mathrm{ind}}) =
T^{(\sigma^*H_{\lambda})(\sigma^*x)-H_{\lambda}(x)
    +\min K|_{\F_{\min}}}$.
Thus in
\eqref{Equation capless seidel map}, all $T$-factors clear except the $T^{\min K|_{\F_{\min}}}$ factor.
\end{proof}

\subsection{Three commuting actions and compositions.}
\label{Subsection Three commuting actions and compositions.}
Let $\sigma,\tau$ be commuting Hamiltonian $S^1$-actions, commuting with $\Fi$.
Let $K_{\sigma},K_{\tau}$ be their moment maps. Then $K_{\sigma\tau}:=K_{\sigma}+K_{\tau}$ is a moment map for $\sigma\tau$.
The restriction of these moment maps to $\F_{\min}$ may have different minimal loci in general.
What we used to denote $x_{\min},c_{x_{\min}}$ will now be denoted $x_{\sigma},c_{\sigma}$ to emphasize the dependence on $\sigma$.
By the same argument as in the proof of \cref{Lemma sigma permutes fixed components Falpha}, $\tau_t$ preserves
$M_{\sigma}:=\mathrm{Min}\, K_{\sigma}|_{\F_{\min}}$ for each $t$, and $\mathrm{Min}\, K_{\tau}|_{M_{\sigma}}\subset  \F_{\min}  \cap \mathrm{Fix}(\sigma)\cap \mathrm{Fix}(\tau) \subset \mathrm{Fix}(\sigma\tau)$ is a closed connected symplectic submanifold.
We pick:
$$
x_{\sigma}\in \mathrm{Min}\, K_{\tau}|_{M_{\sigma}}, 
\qquad
x_{\tau}\in \mathrm{Min}\, K_{\sigma}|_{M_{\tau}}, 
\qquad
x_{\sigma\tau}\in \mathrm{Min}\, K_{\sigma}|_{M_{\sigma\tau}}.
$$
All three points lie in $\F_{\min}$ and are fixed by $\sigma,\tau,\sigma\tau$.
Pick a path $\gamma^{\sigma}_{\sigma\tau}$ from $x_{\sigma}$ to $x_{\sigma\tau}$, and a path 
$\gamma^{\tau}_{\sigma\tau}$ from $x_{\tau}$ to $x_{\sigma\tau}$.
By \cref{Lemma capping lemma 1} we obtain the following spherical classes in $G$, independent of choices:
$$
A^{\sigma}_{\sigma\tau}:=\sigma^*\gamma^{\sigma}_{\sigma\tau},
\qquad 
A^{\tau}_{\sigma\tau}:=\tau^*\gamma^{\tau}_{\sigma\tau},
\qquad
A_{\sigma,\tau}:=A^{\sigma}_{\sigma\tau}+
A^{\tau}_{\sigma\tau}.
$$
\begin{lm}\label{Lemma sigma tau lemma}
The canonical lifts are related by the $G$-action as follows,$$
\widetilde{\sigma}\widetilde{\tau}=
A_{\sigma,\tau}\cdot  
\widetilde{\sigma\tau}.
$$
In particular, the maps from \eqref{Theorem Seidel iso} satisfy:
$$
\mathcal{S}_{\sigma}\mathcal{S}_{\tau}=
T^{[\omega](A_{\sigma,\tau})}
\mathcal{S}_{\sigma\tau}[a_{\sigma,\tau}],
$$
where we shift the image of the right map down in grading by $a_{\sigma,\tau}:=2c_1(Y)(A_{\sigma,\tau}).$

From the case $\tau=\sigma$ we deduce that, for any $N\in \N$,
$$
\widetilde{\sigma}^N=
\widetilde{\sigma^N}
\qquad\textrm{ and }\qquad
\mathcal{S}_{\sigma}^N
=
\mathcal{S}_{\sigma^N}.
$$
For $\tau=\sigma^{-1}$, we get $\widetilde{\sigma^{-1}} = A \widetilde{\sigma}^{-1}$ and $\mathcal{S}_{\sigma^{-1}}=T^{[\omega](A)} \mathcal{S}_{\sigma}^{-1}[a]$, where $A:=A^{\sigma}_{\sigma^{-1}}$, $a:=2c_1(Y)(A)$.
\end{lm}
\begin{proof}
By \cref{Lemma capping lemma 1},
$\widetilde{\tau}(c_{\sigma\tau})=
A^{\tau}_{\sigma\tau}\cdot c_{\sigma\tau}$,
$\widetilde{\sigma}(c_{\sigma\tau})=
A^{\sigma}_{\sigma\tau}\cdot c_{\sigma\tau},$
and
$\widetilde{\sigma\tau}(c_{\sigma\tau})=
c_{\sigma\tau}.$
By $G$-equivariance, 
$$\widetilde{\sigma}\widetilde{\tau}(c_{\sigma\tau})=
\widetilde{\sigma}(A^{\tau}_{\sigma\tau}\cdot c_{\sigma\tau})=
A^{\tau}_{\sigma\tau}\cdot \widetilde{\sigma}(c_{\sigma\tau})
=
A^{\tau}_{\sigma\tau}\cdot (
A^{\sigma}_{\sigma\tau}\cdot c_{\sigma\tau} )
=
A_{\sigma,\tau}\cdot  c_{\sigma\tau}
=
A_{\sigma,\tau}\cdot 
\widetilde{\sigma\tau}(c_{\sigma\tau}).$$
Since $A_{\sigma,\tau}\cdot 
\widetilde{\sigma\tau}$ and $\widetilde{\sigma}\widetilde{\tau}$ agree on $c_{\sigma\tau}$ and are both lifts of $\sigma\tau$, it follows that
$\widetilde{\sigma}\widetilde{\tau}=A_{\sigma,\tau}\cdot \widetilde{\sigma\tau}$, as required.

The grading shift in \eqref{Equation Seidel iso} is by twice the Maslov index $\mu_{\sigma}=\mathrm{Maslov}(\widetilde{\sigma})$ \cite{Sei97}. We will need a general property from \cite{Sei97} about these indices: for any $g\in G$ and any lift $\widehat{\sigma}$,
$$\mathrm{Maslov}(g\cdot \widehat{\sigma}) = c_1(Y)(g) + \mathrm{Maslov}(\widehat{\sigma}).$$
The maps $\mathcal{S}_{\sigma}\mathcal{S}_{\tau}$ and $\mathcal{S}_{\sigma\tau}$ are defined by the actions respectively of $\widetilde{\sigma}\circ \widetilde{\tau}$ and $\widetilde{\sigma\tau}$ on the covering \eqref{Equation cover of loopspace}.
So although \eqref{Equation Seidel iso} seems to suggest that both maps are just $x\mapsto (\sigma\tau)^*x=\sigma^*\tau^*x$ on chain level generators, this need not hold for $\mathcal{S}_{\sigma}\mathcal{S}_{\tau}$ because $\mathcal{S}_{\tau}$ outputs induced caps, not canonical caps.
The equation $\widetilde{\sigma}\widetilde{\tau}=
A_{\sigma,\tau}\cdot \widetilde{\sigma\tau}$ shows that the discrepancy between cappings is the action of an element $A_{\sigma,\tau}\in G$ independent of the input $1$-orbit $x$. Over Seidel's Novikov ring, this proves the equation $\mathcal{S}_{\sigma}\mathcal{S}_{\tau}=
A_{\sigma,\tau}\cdot
\mathcal{S}_{\sigma\tau}$.
Over our Novikov ring, if we use both $T,S$ variables, this $G$-action would be encoded by:
$\mathcal{S}_{\sigma}\mathcal{S}_{\tau}=
S^{c_1(Y)(A_{\sigma,\tau})}T^{[\omega](A_{\sigma,\tau})} 
\mathcal{S}_{\sigma\tau}$.
The $S$ factor keeps track of gradings, and since that factor is independent of the input $x$, we can omit the $S$-factor at the cost of applying a grading shift. We apply the above Maslov index property to the case $g=A_{\sigma,\tau}$ and $\widehat{\sigma}=\widetilde{\sigma\tau}$ to deduce that the difference between the grading shifts caused by $\mathcal{S}_{\sigma}\mathcal{S}_{\tau}$ and $\mathcal{S}_{\sigma\tau}$ is $2 c_1(Y)(A_{\sigma,\tau})$, as required.

For $\tau=\sigma$ we can pick $x_{\sigma}=x_{\tau}=x_{\sigma\tau}$, so the $A$-classes vanish, and by induction we obtain the statement about $\widetilde{\sigma}^N$.
For $\tau=\sigma^{-1}$, the $A$-shift is caused because $x_{\sigma^{-1}}\in  \mathrm{min}(-K_{\sigma})|_{\F_{\min}}= \mathrm{max}K_{\sigma}|_{\F_{\min}}$
whereas $x_{\sigma}\in \mathrm{min}K_{\sigma}|_{\F_{\min}}$.
We pick $x_{\sigma\tau}=x_{\mathrm{id}}=x_{\tau}$. So $A^{\tau}_{\sigma\tau}=0$, thus $A_{\sigma,\tau}=A^{\sigma}_{\sigma\tau}=A^{\sigma}_{\tau}$. The final claim now follows from the general one, using $\mathcal{S}_{\mathrm{id}}=\mathrm{id}$.
\end{proof}

\subsection{Adapting the Seidel isomorphism to the {\MBF} setup}

Let $\sigma$ be a Hamiltonian $S^1$-action commuting with $\Fi$. Using terminology and conventions from \cite{RZ2}, let $B=B_{p,\beta}$ be a {\MBF} manifold of $1$-orbits of $H_{\lambda}=c(H)$, where $p=c'(H)$ (called the $S^1$-period, or slope, of $B$). Recall that we can view $B$ in two ways: as a collection of $1$-orbits $x(t)$, or we can identify $B$ with the collection of initial points $B_0=(B_{p,\beta})_0:=\{x(0): x\in B\}\subset Y$. The reverse map is $B_0\to B$, $y\mapsto x_y$, where $x_y(t)=\Fi_{p\cdot t}(y)$.
We assume the reader is familiar with the {\MBF} model for Floer cohomology from the Appendix in \cite{RZ2}.
When we construct auxiliary Morse data for $B$, what we actually mean is Morse data for $B_0$: a generic Morse function $f=f_{p,\beta}:B_0\to \R$ and we consider the negative gradient flowlines of $f$ (using the ambient Riemannian metric $g:=\omega(\cdot,I\cdot)$ restricted to $B_0$). Equivalently, this means that we are working with $\Fi$-invariant constructions on $B$ (e.g.\,the only sections of $TB$ allowed are $\Fi$-invariant vector fields along $1$-orbits).

\begin{lm}\label{Lemma MBF trick about pullback Morse functions}
The obvious identification of {\MBF} manifolds
    $$
    B \to \sigma^*B:=\{\sigma^*x: x\in B\}, \;\; x\mapsto \sigma^*x,
    $$
corresponds to the identity map, once we identify these manifolds with the collections of initial points: 
$$\{y=x(0)\in Y :x\in B\}=:B_0=(\sigma^*B)_0:=\{ (\sigma^*x)(0)=x(0)\in Y: \sigma^*x\in \sigma^*B\}.$$ Under this identification, all auxiliary data (Riemannian metrics, Morse functions, negative gradient trajectories) used for $B,\sigma^*B$ can be chosen to be the same, by viewing it as data on $B_0=(\sigma^*B)_0$.
\end{lm}
\begin{proof}
As $\sigma$ is a Hamiltonian $S^1$-action, 
$\sigma_t$ satisfies $\sigma_0=\mathrm{id}$, so $(\sigma^*x)(0)=\sigma_{-0}x(0)=x(0).$ 
The claim is therefore an obvious formal consequence of the above discussion.\footnote{
A point of confusion may arise when viewing this as an equivariant construction along $1$-orbits.
The flow defining the {\MBF} manifold $\sigma^*B$ is $\psi:=\sigma^*\mathrm{Flow}^t_{X_{H_{\lambda}}}$, which restricts to $\psi_{p\cdot t}:=\sigma_{-t}\Fi_{p\cdot t}$ on $B_0$.
A negative gradient flowline $\gamma(s)$ in $B_0$ corresponds to a $\Fi$-invariant flowline $\varphi_{pt}\gamma(s)$ of $1$-orbits in $B$ (with $s$-derivative $-d\varphi_{pt} \nabla f$).
This maps to a $\psi$-invariant flowline $\psi_{p\cdot t}\gamma(s)$ of $1$-orbits in $\sigma^*B$ (with $s$-derivative $-d\psi_{pt} \nabla f$).
We do not need $\sigma^*I=I$ (i.e.\,$\sigma$-invariance of $g=\omega(\cdot,I\cdot)$), we only need $I,g$ to be invariant at $t=0$, which is clear: $\sigma_0=\mathrm{id}$ and $d\sigma_0=\mathrm{id}$.
If, at a given time $t$, one wishes $-d \psi_{pt} \nabla f$ to be the negative gradient of $(\psi^*f)_t:=f\circ \psi_{-p\cdot t}$, then in general one must use the $t$-dependent Riemannian metric 
$(\psi^*g)_t:=\omega(\cdot,(\sigma^*I)_t\cdot)$ which is time-independent only when $\sigma^*I=I$.
}
\end{proof}

\begin{cor}
    Assuming we use the same auxiliary {\MBF} data for the {\MBF} manifolds, in the sense of \cref{Lemma MBF trick about pullback Morse functions}, the map \eqref{Equation Seidel iso} is more precisely defined by linearly extending the map on generators determined by the identity maps $y \mapsto y$ for $y\in \mathrm{Crit}(f:B_0\to \R)$, equivalently:
    $$
    (\textrm{critical 1-orbit }x(t)=\Fi_{p\cdot t}(y)\textrm{ in }B_{p,\beta}) \mapsto
       (\textrm{critical 1-orbit }(\sigma^*x)(t)=\sigma_{-t}\Fi_{p\cdot t}(y)\textrm{ in }\sigma^*B_{p,\beta}),
    $$
    where we recall from \cite[Sec.A.1]{RZ2} that critical $1$-orbit $x(t)$ means $y=x(0)\in \mathrm{Crit}(f_{p,\beta}:B_{p,\beta}\to \R).$
\end{cor}

\bibliography{FZ}
\bibliographystyle{amsalpha}
\end{document}